\documentclass{amsart}
\usepackage{color}
\usepackage{amscd}
\usepackage{graphicx}
\usepackage{hyperref}
\usepackage{amssymb}
\usepackage{amsfonts}
\usepackage{euscript}
\usepackage[all]{xy}

\numberwithin{equation}{section}
\newcommand{\stareq}{\tag{\theequation H}\refstepcounter{equation}}
\renewcommand{\subsection}[1]{\hspace{-\parindent}\refstepcounter{subsection}{\bf (\arabic{section}\alph{subsection}) #1.}\addcontentsline{toc}{subsection}{\bf #1.}}

\newenvironment{nouppercase}{%
  \renewcommand{\uppercasenonmath}[1]{}}{}

\theoremstyle{plain}
\newtheorem{thm}{Theorem}[section]
\newtheorem{theorem}[thm]{Theorem}
\newtheorem{application}[thm]{Application}

\newtheorem{corollary}[thm]{Corollary}

\newtheorem{prop}[thm]{Proposition}
\newtheorem{assumption}[thm]{Assumption}

\newtheorem{remark}[thm]{Remark}

\newtheorem{proposition}[thm]{Proposition}

\newtheorem{lemma}[thm]{Lemma}
\newtheorem{setup}[thm]{Setup}

\newtheorem{conjecture}[thm]{Conjecture}

\newtheorem{discussion}[thm]{Discussion}

\newtheorem*{claim*}{Claim} 
\newtheorem*{lemma*}{Lemma}
\newtheorem*{theorem*}{Theorem}
\newtheorem*{conjecture*}{Conjecture}

\newcommand{\bC}{{\mathbb C}}

\newcommand{\bF}{{\mathbb F}}

\newcommand{\bK}{{\mathbb K}}

\newcommand{\bQ}{{\mathbb Q}}
\newcommand{\bR}{{\mathbb R}}

\newcommand{\bZ}{{\mathbb Z}}

\newcommand{\scrA}{\EuScript A}

\newcommand{\scrL}{\EuScript L}

\newcommand{\scrO}{\EuScript O}
\newcommand{\scrP}{\EuScript P}

\newcommand{\scrW}{\EuScript W}

\newcommand{\half}{{\textstyle\frac{1}{2}}}

\newcommand{\iso}{\cong}
\newcommand{\htp}{\simeq}
\newcommand{\smooth}{C^\infty}

\newcommand{\qabla}{\nabla\mkern-13mu/\mkern3mu}

\title[LEFSCHETZ FIBRATIONS]{Fukaya \protect{$A_\infty$}-structures associated to\\ Lefschetz fibrations. III}
\author{Paul Seidel}

\begin{document}
\begin{nouppercase}
\maketitle
\end{nouppercase}

\begin{abstract}
Floer cohomology groups are usually defined over a field of formal functions (a Novikov field). Under certain assumptions, one can equip them with connections, which means operations of differentiation with respect to the Novikov variable. This allows one to write differential equations for Floer cohomology classes. Here, we apply that idea to symplectic cohomology groups associated to Lefschetz fibrations, and establish a relation with enumerative geometry.
\end{abstract}


\section{Introduction}

This paper continues a discussion \cite{seidel12b,seidel14b,seidel15} of the Floer-theoretic structures associated to Lefschetz fibrations. The specific part under consideration is a linear differential equation introduced in \cite{seidel15}, for a pair of functions $(\rho,\sigma)$ of a formal parameter $q$:
\begin{equation} \label{eq:1st-order}
\partial_q \begin{pmatrix} \rho \\ \sigma \end{pmatrix} +
\begin{pmatrix} 0 & \psi \\ 4z^{(2)} \psi & \eta \end{pmatrix} \begin{pmatrix} \rho \\ \sigma \end{pmatrix} = 0.
\end{equation}
The geometric context is that we are looking at a Lefschetz fibration over $\bC P^1$, for which the fibre represents the first Chern class of the total space (this is the case for a rational elliptic surface; or more generally, for the fibrations obtained by blowing up the base locus of an anticanonical Lefschetz pencil). The coefficients $\psi$, $\eta$ and $z^{(2)}$ in \eqref{eq:1st-order} (each a function of $q$) are extracted from the genus zero Gromov-Witten invariants of the total space. In fact, a restriction on those Gromov-Witten invariants must be imposed: that is Assumption \ref{th:psi-eta} below, already extensively discussed in \cite{seidel15}. 

Let's look at \eqref{eq:1st-order} on an elementary level, as a formal differential equation (topologically-minded readers are hereby asked for some patience). In our context $\psi \neq 0$, so we can eliminate $\sigma = -\psi^{-1}\partial_q \rho $, which leaves the second order equation
\begin{equation} \label{eq:2nd-order}
\partial_q^2 \rho + \Big(\eta - \frac{\partial_q \psi}{\psi} \Big) \partial_q \rho - 4z^{(2)} \psi^2 \rho = 0.
\end{equation}
For the logarithmic derivative $\alpha = \rho^{-1}\partial_q \rho$, we get a nonlinear first order equation (a Riccati equation \cite{riccati24})
\begin{equation} \label{eq:nonlinear-1st-order}
\partial_q \alpha + \alpha^2 + \Big(\eta - \frac{\partial_q \psi}{\psi} \Big) \alpha - 4 z^{(2)} \psi^2 = 0.
\end{equation}
Alternatively, let's consider the action of \eqref{eq:1st-order} on the projective line, which concretely means on the quotient $\lambda = \rho^{-1}\sigma = -\psi^{-1}\alpha$. This satisfies an equation closely related to \eqref{eq:nonlinear-1st-order},
\begin{equation} \label{eq:projective-eq}
\partial_q\lambda - \psi \lambda^2 + \eta \lambda + 4 z^{(2)} \psi = 0. 
\end{equation}
The quotient $\theta = \rho_0^{-1}\rho_1$ of two solutions of \eqref{eq:2nd-order} satisfies 
\begin{equation} \label{eq:rho01}
\frac{\partial_q^2\theta}{\partial_q\theta} + \Big(\eta - \frac{\partial_q \psi}{\psi}\Big) + 2 \frac{\partial_q\rho_0}{\rho_0} = 0;
\end{equation}
from that and \eqref{eq:nonlinear-1st-order} one gets a nonlinear third order equation involving the Schwarzian operator \cite{kummer36} $S_q \theta = \partial_q (\partial_q^2 \theta/\partial_q \theta) - \half (\partial_q^2 \theta/\partial_q \theta)^2$,
\begin{equation} \label{eq:schwarz}
S_q \theta + \partial_q \Big(\eta - \frac{\partial_q\psi}{\psi}\Big) + \half \Big(\eta - \frac{\partial_q\psi}{\psi}\Big)^2 + 8 z^{(2)} \psi^2 = 0.
\end{equation}
Finally, one can use \eqref{eq:rho01} to write the solutions of \eqref{eq:2nd-order} in terms of $\theta$,
\begin{equation}
\rho_0 = \exp\Big(-\half \int \frac{\partial_q^2\theta}{\partial_q\theta} + \Big(\eta - \frac{\partial_q \psi}{\psi}\Big) \; \mathit{dq} \Big), \quad \rho_1 = \theta\rho_0.
\end{equation}
%
%

In the application to Lefschetz fibrations (more precisely, ones obtained by blowing up the base locus of an anticanonical Lefschetz pencil), \eqref{eq:projective-eq} has a straightforward enumerative meaning: it describes the $q$-dependence of the eigenvalues of quantum multiplication with the fibre class \cite[Section 3]{seidel15}. On a deeper level, which is where the main motivation lies, it has been conjectured that \eqref{eq:1st-order} governs the $q$-dependence of canonical natural transformations on the Fukaya category of the Lefschetz fibration \cite[Section 4b]{seidel15}; as a consequence, \eqref{eq:schwarz} would then describe a kind of mirror map for the fibre of the Lefschetz fibration. That idea was supported by a form of the Homological Mirror Symmetry conjecture, and by example computations, but lacked a direct geometric interpretation. Our task is to show that versions of \eqref{eq:1st-order}--\eqref{eq:projective-eq} arise in symplectic cohomology. More precisely, we remove a fibre from the Lefschetz fibration, so as to obtain an open symplectic manifold, whose symplectic cohomology we study. To obtain the differential equations above, one replaces the functions in them by distinguished symplectic cohomology classes, and $\partial_q$ by a connection (an operator of differentiation in $q$-direction) acting on symplectic cohomology. The relation between Gromov-Witten invariants and symplectic cohomology has already been the object of considerable study: see e.g.\ \cite{diogo12, diogo-lisi15}, and unpublished work of Borman-Sheridan (one can also consider Symplectic Field Theory as an intermediate theory, and then the relevant literature becomes much larger). What's new here is the appearance of connections, and their interplay both with Gromov-Witten invariants and with the algebraic structure of symplectic cohomology.

There are at least two perspectives on the material developed here. One can focus on the abstract framework surrounding the construction and properties of connections. This leads away from the specific geometric situation, towards more abstract and axiomatic (TQFT) considerations. Alternatively, one can emphasize the concrete computational aspects, in relation with Gromov-Witten theory. We will try to accomodate both tastes. The first part of the paper, consisting of Sections \ref{sec:results}--\ref{sec:motivation}, presents an overview of the constructions and results, including their motivation.
After that, in Sections \ref{sec:conf}--\ref{sec:bv}, we operate in a greatly simplified TQFT-style framework, and explain what can be seen at that level of abstraction. In contrast, Sections \ref{sec:floer}--\ref{sec:operations-on-floer-cohomology} are much more geometric, and set up just enough machinery in (finite-dimensional) Hamiltonian Floer cohomology groups to enable us to carry out the main computations. In Section \ref{sec:symplectic-cohomology}, we translate the outcome of those computations to the (infinite-dimensional) symplectic cohomology groups, and this fairly straightforward process completes our argument.

{\em Acknowledgments.} I would like to thank Nick Sheridan for an illuminating conversation about his joint work with Borman, and Amitai Netser-Zernik for a helpful observation concerning \eqref{eq:more-gw}. Partial support was provided by the Simons Foundation, through a Simons Investigator award, and by NSF grant DMS-1500954.

\section{The results\label{sec:results}}

This section summarizes the Floer-theoretic constructions and results which are the main topic of the paper. Technical details will be kept to a minimum.

\subsection{Connections\label{subsec:sh}}
Symplectic cohomology $\mathit{SH}^*(E)$ \cite{viterbo97a, cieliebak-floer-hofer95} is an invariant of a certain kind of non-compact symplectic manifold $(E^{2n},\omega_E)$. Our version of symplectic cohomology will be linear over a specific (characteristic $0$) field $\bK$, the single-variable Novikov field with real coefficients; elements $f \in \bK$ are formal series
\begin{equation} \label{eq:novikov}
f(q) = c_0 q^{d_0} + c_1 q^{d_1} + \cdots \qquad \text{with $c_i \in \bR$, $d_i \in \bR$, $\textstyle\;\lim_i d_i = \infty$.}
\end{equation}
Note that $\bK$ is closed under differentation $\partial_q$. We'll assume throughout that $c_1(E) = 0$. This allows us to lift the usual $\bZ/2$-grading of symplectic cohomology to a $\bZ$-grading, and it also simplifies the technical aspects of the construction. (Readers who wish for more specifics about the class of symplectic manifolds under consideration are referred to Section \ref{subsec:define-sh}, but should keep in mind that the same ideas apply to other situations as well.)
%

Symplectic cohomology has the structure of a BV (Batalin-Vilko\-vi\-sky) algebra, which one can think of as a Gerstenhaber algebra structure (consisting of a product, denoted here by $\bullet$, and a degree $-1$ Lie bracket) together with a BV operator $\Delta$ (see e.g.\ \cite{seidel07, ritter10}). It also comes with a map
\begin{equation} \label{eq:acc}
B: H^*(E;\bK) \longrightarrow \mathit{SH}^*(E).
\end{equation}
We like to think of $B$ as a version of the construction from \cite{piunikhin-salamon-schwarz94}, hence call it a PSS map. It is a map of BV algebras, if one equips $H^*(E;\bK)$ with the small quantum product (denoted here by $\ast_E$), as well as the zero bracket and zero BV operator. In particular, $1 \in H^0(E;\bK)$ maps to the unit for the product structure, which we denote by $e \in\mathit{SH}^0(E)$. We denote the image of $q^{-1}[\omega_E]$ under \eqref{eq:acc} by $k$, and call it the Kodaira-Spencer class (the terminology is motivated by mirror symmetry). 

There is also a reduced symplectic cohomology group $\mathit{SH}^*(E)_{\mathit{red}}$ (see e.g.\ \cite{bourgeois-oancea09}), which fits into a long exact sequence
\begin{equation} \label{eq:plus-sequence}
\cdots \rightarrow H^*(E;\bK) \stackrel{B}{\longrightarrow} \mathit{SH}^*(E) \longrightarrow \mathit{SH}^*(E)_{\mathit{red}} \rightarrow \cdots
\end{equation}
Based on the vanishing of the BV operator on $H^*(E;\bK)$, one can define a map $\Delta_{\mathit{red}}$ which fits into a diagram
\begin{equation} \label{eq:delta-minus}
\xymatrix{
\mathit{SH}^*(E) \ar[d]_-{\Delta} \ar[r]  &
\mathit{SH}^*(E)_{\mathit{red}} \ar[dl]^-{\Delta_{\mathit{red}}} 
\\
\mathit{SH}^{*-1}(E),
}
\end{equation}
and satisfies 
\begin{equation} \label{eq:delta-delta-red}
\Delta \Delta_{\mathit{red}} = 0.
\end{equation}

As should already be obvious from the discussion so far, we do not require that $[\omega_E] \in H^2(E;\bR)$ is trivial, or even that it vanishes on spherical homology classes (such assumptions occur in most, but not all, of the literature concerning symplectic cohomology; among the exceptions are \cite{ritter-smith12, harris13, groman15}). Eventually, we will impose a vanishing condition, but one which is weaker and not expressed in terms of classical topology:

\begin{assumption} \label{th:as-1}
The Kodaira-Spencer class vanishes. 
\end{assumption}

Whenever Assumption \ref{th:as-1} holds, we want to fix a choice of bounding cochain, in the following sense. Consider the chain complex underlying symplectic cohomology. Pseudo-holomorphic curve theory yields a cocycle in that complex which represents $k$. By assumption, this is nullhomologous, and the desired choice is that of a cochain which bounds it. Two such choices have the same effect if they differ by a nullhomologous cocycle, so the essential amount of freedom is an affine space over $\mathit{SH}^1(E)$. A choice of bounding cochain determines a class
\begin{equation} \label{eq:t-class}
t \in \mathit{SH}^1(E)_{\mathit{red}},
\end{equation}
which maps to $q^{-1}[\omega_E]$ under the connecting homomorphism from \eqref{eq:plus-sequence}; and that gives rise to a class
\begin{equation} \label{eq:a-class}
a = \Delta_{\mathit{red}} t \in \mathit{SH}^0(E),
\end{equation}
which, in view of \eqref{eq:delta-delta-red}, satisfies
\begin{equation} \label{eq:delta-a-is-zero}
\Delta a = 0.
\end{equation}
If one changes the bounding cochain by adding $\alpha \in \mathit{SH}^1(E)$, \eqref{eq:t-class} changes by the image of $\alpha$ in $\mathit{SH}^1(E)_{\mathit{red}}$, and therefore, the corresponding change in \eqref{eq:a-class} is
\begin{equation} \label{eq:change-a}
\tilde{a} = a + \Delta \alpha.
\end{equation}

\begin{prop} \label{th:1}
If Assumption \ref{th:as-1} holds, a choice of bounding cochain determines a connection $\nabla$ on symplectic cohomology, which is compatible with the Gerstenhaber algebra structure.
\end{prop}

Here, by a connection, we mean an operator of differentiation in $\partial_q$-direction: an additive (and grading-preserving) endomorphism $\nabla$ which satisfies
\begin{equation} \label{eq:connection-property}
\nabla( f x) = f \nabla x + (\partial_q f) x \quad \text{for $f \in \bK$.}
\end{equation}
Of course, any $\bK$-vector space can be equipped with a connection, simply by differentiating the coefficients with respect to some choice of basis. What makes Proposition \ref{th:1} nontrivial is compatibility with the Gerstenhaber algebra structure, which means that $\nabla$ acts as a derivation for both the product and Lie bracket: 
\begin{align}
& \label{eq:nabla-derivation} \nabla(x_2 \bullet x_1) = \nabla x_2 \bullet x_1 + x_2 \bullet \nabla x_1, \\
& \label{eq:nabla-derivation-2} \nabla [x_2,x_1] = [\nabla x_2,x_1] + [x_2,\nabla x_1].
\end{align}
Note that \eqref{eq:nabla-derivation} implies
\begin{equation} \label{eq:nabla-e-is-0}
\nabla e = 0.
\end{equation}
Our connection depends on the choice of bounding cochain. Changing that cochain by $\alpha$ affects it as follows:
\begin{equation} \label{eq:nabla-add}
\tilde\nabla x = \nabla x - [\alpha,x].
\end{equation}
Consider the interaction of the connection and the BV operator, as measured by the $\bK$-linear endomorphism $\nabla \Delta - \Delta \nabla$. From the general relation between the BV operator, the product, and the bracket, it follows that $\nabla \Delta - \Delta \nabla$ is a (degree $-1$) derivation for the product. The following result sharpens that statement, by saying that this derivation is inner:

\begin{prop} \label{th:a-element}
The connection from Proposition \ref{th:1} satisfies
\begin{equation} \label{eq:delta-nabla}
\nabla \Delta x = \Delta \nabla x - [a,x].
\end{equation}
\end{prop}

One can think of $\nabla$ as part of a family of connections,
\begin{equation} \label{eq:c-term}
\nabla^c x = \nabla x + ca \bullet x, \quad c \in \bK.
\end{equation}
In particular, from \eqref{eq:nabla-e-is-0} one gets
\begin{equation} \label{eq:nabla-c-e-2}
\nabla^c e = c\,a.
\end{equation}
One of those connections, $\nabla^{-1}$, is compatible with the BV operator (but not with the product or bracket, in general): by \eqref{eq:delta-a-is-zero}, \eqref{eq:delta-nabla}, and the definition \eqref{eq:c-term},
\begin{equation} \label{eq:-1-connection}
\nabla^{-1}(\Delta x) - \Delta(\nabla^{-1} x) = -[a,x] - a \bullet \Delta x + \Delta(a \bullet x) = (\Delta a) \bullet x = 0.
\end{equation}

\begin{remark} \label{th:c-is-minus-one}
From \eqref{eq:nabla-add}, one can derive formulae for how a change of bounding cochain affects $\nabla^c$. Again, the case $c = -1$ is of particular interest:
\begin{equation} \label{eq:-1-ambiguity}
\begin{aligned}
\tilde{\nabla}^{-1} x &= \tilde{\nabla}x - \tilde{a} \bullet x = 
\nabla x - [\alpha,x] - (a + \Delta \alpha) \bullet x \\ &= \nabla^{-1}x - \Delta(\alpha \bullet x) - \alpha \bullet \Delta x. \end{aligned}
\end{equation}
Combining \eqref{eq:-1-connection} and \eqref{eq:-1-ambiguity} yields the following: if one considers $\Delta$ as a differential on $\mathit{SH}^*(E)$, then $\nabla^{-1}$ induces a connection on its cohomology, and that induced connection is independent of the choice of bounding cochain.
\end{remark}

\subsection{Lefschetz fibrations\label{subsec:define-lefschetz}}
Take a symplectic Lefschetz fibration 
\begin{equation} \label{eq:lefschetz}
p: F \longrightarrow \bC P^1,
\end{equation}
where $F$ is a manifold without boundary, and the map $p$ is proper. Choose a smooth fibre $M$. We also assume that $c_1({F})$ is Poincar{\'e} dual to $[M]$. Consider genus $0$ pseudo-holomorphic curves in ${F}$ which have degree $k$ over the base of \eqref{eq:lefschetz}, and which come with one marked point. The evaluation map on the moduli space of such curves gives rise to Gromov-Witten invariants
\begin{equation} \label{eq:zk}
z^{(k)} \in H^{4-2k}({F};\bK), 
\end{equation}
which are zero if $k<0$. Let's remove a fibre from \eqref{eq:lefschetz} to get an open symplectic manifold 
\begin{equation} \label{eq:e}
E = F \setminus M. 
\end{equation}
By assumption, $c_1(E) = 0$. Our main Floer-theoretic object of study is the symplectic cohomology of $E$. 

\begin{lemma} \label{th:z1-vanish}
The cohomology class $z^{(1)}|E$ lies in the kernel of \eqref{eq:acc}.
\end{lemma}

The proof will exhibit a geometrically defined bounding cochain for $B(z^{(1)}|E)$. In analogy with the previous discussion of Assumption \ref{th:as-1}, the bounding cochain determines a class in $\mathit{SH}^1(E)_{\mathit{red}}$ which goes to $B(z^{(1)}|E)$ under the connecting map from \eqref{eq:plus-sequence}; and by applying $-\Delta_{\mathit{red}}$, another class
\begin{equation} \label{eq:s-element}
s \in \mathit{SH}^0(E).
\end{equation}
The last-mentioned class has an alternative and more direct description, in terms of pseudo-holomorphic thimbles going through $M$ (this goes back to unpublished work of Borman-Sheridan, see also \cite{ganatra-pomerleano}). With that description in mind, we call $s$ the Borman-Sheridan class.

\begin{remark}
From the perspective of \cite{seidel14b}, the Borman-Sheridan class can be viewed as a special case of a more general idea. Consider a Lefschetz fibration $p: E \rightarrow \bC$, which does not necessarily extend over $\bC P^1$. Let $\mu$ be the monodromy around $\infty$, which is a symplectic automorphism of the fibre. Its fixed point Floer cohomology $\mathit{HF}^*(\mu)$ comes with a map (which already played a role in \cite[Conjecture 6.1]{seidel00b})
\begin{equation} \label{eq:mu-map}
\mathit{HF}^*(\mu) \longrightarrow H^*(E;\bK),
\end{equation}
obtained by counting sections of $p$ with suitable asymptotic conditions. It turns out that the composition of this map and \eqref{eq:acc} is always zero (this is a generalization of Lemma \ref{th:z1-vanish}). Now consider the double composition
\begin{equation}
\mathit{HF}^*(\mu) \longrightarrow H^*(E;\bK) \xrightarrow{B} \mathit{SH}^*(E) \xrightarrow{\Delta} \mathit{SH}^{*-1}(E).
\end{equation}
This is zero for two different reasons: the one which we have just mentioned, and also because $\Delta$ vanishes on the image of the PSS map. By subtracting the two underlying nullhomotopies from each other, one gets a map
\begin{equation} \label{eq:secondary-map}
\mathit{HF}^*(\mu) \longrightarrow \mathit{SH}^{*-2}(E).
\end{equation}
In the specific situation arising from \eqref{eq:lefschetz}, $\mu$ is essentially the identity map; more precisely, because of our assumption on $c_1(F)$, one has $\mathit{HF}^*(\mu) \iso H^{*-2}(M;\bK)$. One plugs $1 \in H^0(M;\bK)$ into \eqref{eq:secondary-map} to obtain \eqref{eq:s-element}. \end{remark}

So far, our observations have not depended on making any assumptions on the enumerative geometry of $F$. We will now introduce the key restriction:

\begin{assumption} \label{th:psi-eta}
There are $\psi, \eta \in \bK$ (with $\psi$ necessarily nonzero), such that
\begin{equation} \label{eq:express-o}
q^{-1}[\omega_F] = \psi z^{(1)} - \eta [M] \in H^2(F;\bK).
\end{equation}
\end{assumption}

In order for \eqref{eq:express-o} to hold, $z^{(1)}$ must be nonzero, which implies that the Lefschetz fibration must have pseudo-holomorphic sections (for any choice of compatible almost complex structure which makes the map \eqref{eq:lefschetz} pseudo-holomorphic). As we will now explain, there is a natural source of examples where that is the case.

\begin{application} \label{th:pencil}
Start with a Lefschetz pencil of anticanonical hypersurfaces on a monotone symplectic manifold (or algebro-geometric Fano variety). Construct \eqref{eq:lefschetz} by blowing up the base locus of the pencil. Write $D \subset F$ for the exceptional divisor. The symplectic class can be chosen to be
\begin{equation} \label{eq:symplectic-blowup-class}
[\omega_F] = [D] + \gamma [M] \quad \text{for $\gamma \gg 0$.}
\end{equation}
The obvious ruling of $D$ by lines gives a contribution of $q^{\gamma-1} [D]$ to $z^{(1)}$, and all other terms have powers $q^{\gamma}, q^{\gamma+1},\dots$. If we assume additionally that 
\begin{equation} \label{eq:betti-2}
\mathrm{dim}\, H^2(F;\bR) = 2,
\end{equation}
then 
\begin{equation}
\begin{aligned}
q^{-\gamma} z^{(1)} 
= \, & (1 + \text{positive integer powers of $q$})\, q^{-1}[\omega_F] \\ & + \text{(nonnegative integer powers of $q$)} \,[M].
\end{aligned}
\end{equation}
Assumption \ref{th:psi-eta} will therefore be satisfied. One example is the algebro-geometric pencil of Calabi-Yau hypersurfaces on $\bC P^n$, $n \geq 3$. There is also a useful generalization, in which one replaces $H^2(F;\bR)$ with the subspace which is invariant under a suitable automorphism group \cite[Lemma 3.3]{seidel15}. That generalization applies e.g.\ to the pencil of elliptic curves on $\bC P^2$ (giving rise to a space $F$ which is a rational elliptic surface).
\end{application}

\begin{remark}
As explained in \cite{seidel15}, allowing formal bulk deformations of the symplectic form greatly extends the reach of Assumption \ref{th:psi-eta}, for appropriately twisted versions of all the Floer-theoretic structures. It would be possible to allow such deformations throughout the current paper as well, but we won't do that, in the interest of simplicity.
\end{remark}

\subsection{Main results\label{subsec:main}}
Returning to the main thread of our argument, we notice that, by Lemma \ref{th:z1-vanish}, Assumption \ref{th:psi-eta} implies Assumption \ref{th:as-1}. One can be a little more precise:

\begin{lemma} \label{th:bounding}
Suppose that Assumption \ref{th:psi-eta} holds. Then, for a suitable choice of bounding cochain, the class \eqref{eq:a-class} is given by
\begin{equation}
\label{eq:a-s} 
a = -\psi s.
\end{equation}
\end{lemma}

This is fundamentally not surprising: the definitions of the two sides of \eqref{eq:a-s} both start with the vanishing of certain other elements in symplectic cohomology, and the condition $q^{-1}[\omega_E] = \psi(z^{(1)}|E)$ relates those elements. In view of  \eqref{eq:nabla-c-e-2}, one can equivalently rewrite \eqref{eq:a-s} as
\begin{equation} \label{eq:nabla-c-e}
\nabla^c e = -c \,\psi s.
\end{equation}
We can go further, and describe how the connection acts on $s$:

\begin{theorem} \label{th:sigma1}
Suppose that Assumption \ref{th:psi-eta} holds. Then, one can choose the bounding cochain so that Lemma \ref{th:bounding} is satisfied, and with the additional property that
\begin{equation} \label{eq:nabla-s-squared}
\nabla s - \psi s \bullet s + \eta  s + 4 z^{(2)} \psi \, e = 0.
\end{equation}
\end{theorem}

This is a counterpart of \eqref{eq:projective-eq}, with $s$ instead of $\lambda$, and $\nabla$ instead of ordinary differentiation. We will now draw some immediate consequences from it. An equivalent statement is that $a$ satisfies the analogue of \eqref{eq:nonlinear-1st-order}:
\begin{equation} \label{eq:nonlinear-a}
\nabla a + a \bullet a + \Big(\eta - \frac{\partial_q \psi}{\psi} \Big) a - 4 z^{(2)} \psi^2 e = 0.
\end{equation}
In terms of $\nabla^c$, \eqref{eq:nabla-s-squared} takes on the following form:
\begin{equation} \label{eq:nablac-s}
\nabla^c s + (c-1) \psi s \bullet s + \eta s + 4 z^{(2)} \psi \, e = 0.
\end{equation}
In particular, setting $c = 1$ causes the nonlinear term to vanish,
\begin{equation} \label{eq:nabla1-2}
\nabla^1 s + \eta s + 4 z^{(2)} \psi \, e = 0.
\end{equation}
This equation and the $c = 1$ case of \eqref{eq:nabla-c-e} are counterparts of the components of \eqref{eq:1st-order}, with $\nabla^1$ replacing ordinary differentiation. Imitating the argument there, one can eliminate $s$ by substitution, and obtain the analogue of \eqref{eq:2nd-order} for $e$ alone:
\begin{equation} \label{eq:2nd-order-2}
\begin{aligned}
\nabla^1\nabla^1 e + \Big(\frac{\partial_q \psi}{\psi} - \eta\Big) \nabla^1 e - 4 z^{(2)} \psi^2 e = 0.
\end{aligned}
\end{equation}
We have now fulfilled the promise, made at the beginning of the paper, that all forms \eqref{eq:1st-order}--\eqref{eq:projective-eq} of the basic differential equation would occur in the context of symplectic cohomology. 

\section{A sketch of the equivariant theory\label{sec:equivariant}}

This (partly speculative) section is not strictly necessary for the rest of the paper, but expands our overall theme in a natural direction. There is an operation on $S^1$-equivariant symplectic cohomology, which we'll call the Gauss-Manin connection. It exists without any additional assumptions, and is canonical \cite{seidel17}. Even though it is therefore somewhat different from our $\nabla^c$, one expects the two to be related via a suitable intermediate object. More precisely, the relevant value is $c = -1$, as foreshadowed by Remark \ref{th:c-is-minus-one}. We will explain this picture, and how it leads to another approach towards proving our main results. Some of the steps in this approach remain conjectural, but those (specifically Conjecture \ref{th:equi-sheridan}) may actually be of independent interest.

\subsection{Background\label{subsec:sh-background}}
Let's start in the same overall situation as in our discussion of ordinary symplectic cohomology (Section \ref{subsec:sh}). The $S^1$-equivariant symplectic cohomology $\mathit{SH}^*_{\mathit{eq}}(E)$ is a $\bZ$-graded module over $\bK[[u]]$, where the formal variable $u$ has degree $2$. It fits into a long exact sequence
\begin{equation} \label{eq:u-sequence}
\cdots \rightarrow \mathit{SH}^{*-2}_{\mathit{eq}}(E) \stackrel{u}{\longrightarrow} \mathit{SH}^*_{\mathit{eq}}(E) \longrightarrow \mathit{SH}^*(E) \longrightarrow \mathit{SH}^{*-1}_{\mathit{eq}}(E) \rightarrow \cdots
\end{equation}
The BV operator can be recovered by composing the middle and right maps in this sequence (in the nontrivial order). There is also an equivariant version of the PSS map, which fits into a commutative diagram with the forgetful maps from \eqref{eq:u-sequence}:
\begin{equation}
\xymatrix{
\ar[d] H^*(E;\bK[[u]]) \ar[rr]^-{B_{\mathit{eq}}} && \mathit{SH}^*_{\mathit{eq}}(E) \ar[d] \\
H^*(E;\bK) \ar[rr]^-{B} && \mathit{SH}^*(E).
}
\end{equation}
There are actually different versions of the equivariant theory, which share the basic properties mentioned so far, but otherwise behave quite differently. The one convenient for the present discussion was introduced in \cite{albers-cieliebak-frauenfelder14, zhao14}; in that version, the underlying chain complex is not $u$-adically complete (as a consequence, vanishing of $\mathit{SH}^*(E)$ does not imply the corresponding result for $\mathit{SH}^*_{\mathit{eq}}(E)$). The main result of \cite{albers-cieliebak-frauenfelder14, zhao14} is a localisation theorem, which says that the equivariant PSS map induces an isomorphism
\begin{equation} \label{eq:zhao}
H^*(E;\bK((u))) \iso \mathit{SH}^*_{\mathit{eq}}(E) \otimes_{\bK[[u]]} \bK((u)).
\end{equation}
To be clear, the geometric situation here is not exactly the same as that in which \eqref{eq:zhao} was proved originally, but the proof goes through without any significant modifications. 

The ordinary cohomology of $E$ carries the quantum connection, which is defined in terms of the small quantum product, as follows:
\begin{equation}
\begin{aligned}
& D: H^*(E;\bK[[u]]) \longrightarrow H^{*+2}(E;\bK[[u]]), \\
& D(x) = u\partial_q x + q^{-1}[\omega_E] \ast_E x.
\end{aligned}
\end{equation}
In spite of the name, this (in our formulation) is not a connection in the strict sense, but rather what one would get after multiplying a connection by $u$. It turns out that a similar operation can be defined on equivariant symplectic cohomology:

\begin{theorem} \label{th:q2}
There is a canonical additive endomorphism
\begin{equation} \label{eq:eq-connection}
\begin{aligned}
& \Gamma: \mathit{SH}^*_{\mathit{eq}}(E) \longrightarrow \mathit{SH}^{*+2}_{\mathit{eq}}(E), \\
& \Gamma( f x) = f \Gamma(x) + u (\partial_q f) x \quad \text{for $f \in \bK[[u]]$,}
\end{aligned}
\end{equation}
which fits into a commutative diagram
\begin{equation} \label{eq:sh-gamma}
\xymatrix{ 
\ar[d]_-{B_{\mathit{eq}}} H^*(E;\bK[[u]]) \ar[rr]^-{D} && \ar[d]^-{B_{\mathit{eq}}} H^{*+2}(E;\bK[[u]])
\\
\ar[d] \mathit{SH}^*_{\mathit{eq}}(E) \ar[rr]^-{\Gamma} && \mathit{SH}^{*+2}_{\mathit{eq}}(E) \ar[d] 
\\
\mathit{SH}^*(E) \ar[rr]^-{k \bullet} && \mathit{SH}^{*+2}(E).
}
\end{equation}
\end{theorem}

This $\Gamma$ is what we call the Gauss-Manin connection. For the finite-dimensional Floer cohomology groups obtained from Hamiltonians with ``finite slope'', it is constructed in \cite{seidel17} (where the notation is $\Gamma_q$); the argument from \cite[Section 5f]{seidel17} could be adapted to prove its compatibility with continuation maps which ``increase the slope'', yielding the theorem in the form stated above.

Intuitively, \eqref{eq:sh-gamma} says that the Kodaira-Spencer class is the obstruction to being able to divide $\Gamma$ by $u$ in order to get an honest connection on the equivariant theory. A more precise form of this idea is the following (see \cite[Remark 5.3]{seidel17} for further discussion):

\begin{conjecture} \label{th:eq-connection}
Suppose that Assumption \ref{th:as-1} holds (and that a bounding cochain has been chosen). Then, 
$\mathit{SH}^*_{\mathit{eq}}(E)$ can be equipped with a connection $\nabla_{\mathit{eq}}$, which preserves gradings and satisfies the analogue of \eqref{eq:connection-property} for $f \in \bK[[u]]$. The Gauss-Manin connection is related to it by
\begin{equation}
\Gamma = u \nabla_{\mathit{eq}};
\end{equation}
moreover, $\nabla_{\mathit{eq}}$ is an equivariant version of $\nabla^{-1}$, meaning that it fits into the diagram
\begin{equation} \label{eq:eq-connection-2}
\xymatrix{
\ar[d] \mathit{SH}^*_{\mathit{eq}}(E) \ar[rr]^-{\nabla_{\mathit{eq}}} && \mathit{SH}^*_{\mathit{eq}}(E) \ar[d] \\
\mathit{SH}^*(E) \ar[rr]^-{\nabla^{-1}} && \mathit{SH}^*(E).
}
\end{equation}
\end{conjecture}

\subsection{A bit more Gromov-Witten theory\label{subsec:gw}}
We now return to the more specific geometric situation from Section \ref{subsec:define-lefschetz}. $H^*(F;\bK)$ also carries a quantum product, but that no longer preserves the grading. Let's write it as
\begin{equation} \label{eq:quantum}
x_1 \ast_F x_2 = \sum_k x_1 \ast^{(k)}_F x_2,
\end{equation}
where each term $\ast_F^{(k)}$ has degree $-2k$. The terms with $k<0$ are zero, while the $k = 0$ contribution is related to the quantum product on $E$ by restriction:
\begin{equation} \label{eq:star-ast}
(x_1 \ast^{(0)}_F x_2)|E = (x_1|E) \ast_E (x_2|E).
\end{equation}
For classes $x_k \in H^2(F;\bK)$, one can use the divisor axiom to reduce the quantum product to one-pointed invariants. In particular, the following products can be expressed in terms of \eqref{eq:zk}:
\begin{align} \label{eq:m-ast-m}
& [M ] \ast_F [M] = z^{(1)} + 4 z^{(2)}, \\
\label{eq:omega-ast-m}
& q^{-1}[\omega_{F}] \ast_F [M] = q^{-1}[\omega_{F}]\smile [M] +
\partial_q (z^{(1)} + 2 z^{(2)}), \\
& 
\begin{aligned}
q^{-1}[\omega_{F}] \ast_F q^{-1}[\omega_{{F}}] = & \, q^{-1}[\omega_F] \smile q^{-1}[\omega_F] \\ + & (q^{-1}\partial_q + \partial_q^2) (z^{(0)} + z^{(1)} + z^{(2)}),
\end{aligned}
\end{align}
where $\smile$ is the cup product. More relations between Gromov-Witten invariants come from the WDVV (associativity) equation. One instance which is relevant for us is this:
\begin{equation} \label{eq:wdvv}
\begin{aligned}
x \ast_F^{(0)} z^{(1)} & = x \ast_F^{(0)} ([M] \ast^{(1)}_F [M]) \\ & = (x \smile [M]) \ast^{(1)}_F [M] + (x \ast^{(1)}_F [M]) \smile [M].
\end{aligned}
\end{equation}

Gromov-Witten invariants can be generalized in various directions. The most relevant one for our purpose are relative invariants, and specifically ones with tangency conditions to $M$. Consider curves which intersect $M$ in a single point, with multiplicity $k$, and which carry an additional marked point. These give rise to classes
\begin{equation} \label{eq:tilde-z}
\tilde{z}^{(k)} \in H^2(F;\bK),
\end{equation}
with $\tilde{z}^{(1)} = z^{(1)}$. All these relative invariants can be reduced to ordinary Gromov-Witten theory, see \cite{maulik-pandharipande06} for a general discussion. We will only appeal to the simplest special case of such a reduction:
\begin{equation} \label{eq:relative-gw}
\tilde{z}^{(2)}|E = \half (z^{(1)} \ast_F^{(1)} [M])|E \in H^2(E;\bK).
\end{equation}
This is proved by first using a degeneration argument to relate $\tilde{z}^{(2)}$ to a Gromov-Witten invariant with a gravitational descendant, and then eliminating that descendant using the Topological Recursion Relation.

\subsection{The Borman-Sheridan class and enumerative geometry}
Let's return to our main object of study, the symplectic cohomology of \eqref{eq:e}. We will use the following (highly convenient, but not absolutely indispensable) fact. One can arrange that the Conley-Zehnder indices of periodic orbits are all nonnegative, which implies that
\begin{equation} \label{eq:vanishing}
\mathit{SH}^*(E) = 0, \quad \mathit{SH}^*_{\mathit{eq}}(E) = 0 \quad \text{for $\ast < 0$.}
\end{equation}
As a consequence of this and \eqref{eq:u-sequence}, we have
\begin{equation} \label{eq:forgetful-iso}
\mathit{SH}^0_{\mathit{eq}}(E) \iso \mathit{SH}^0(E).
\end{equation}
We will denote by $x_{\mathit{eq}}$ the unique equivariant lift of any class $x \in \mathit{SH}^0(E)$. 
Note that in several important cases, such as the unit class $e$ or the Borman-Sheridan class $s$, there are geometric constructions of equivariant lifts which do not depend on \eqref{eq:forgetful-iso}.
%

At this point, we are faced with an interesting puzzle. After multiplication by some power of $u$, any element of $\mathit{SH}^*_{\mathit{eq}}(E)$ comes from an ordinary cohomology class, by \eqref{eq:zhao}. Applying this to the Borman-Sheridan class and its powers, we get specific cohomology classes, which ought to have some enumerative meaning. An attempt to make this more precise leads to the following:

\begin{conjecture} \label{th:equi-sheridan}
In $\mathit{SH}_{\mathit{eq}}^{*}(E)$,
\begin{align} \label{eq:ueq}
& u\, s_{\mathit{eq}} = B_{\mathit{eq}}(z^{(1)}|E), \\
\label{eq:uueq}
& 
\begin{aligned}
u^2 \, (s \bullet s)_{\mathit{eq}} = \;
&
B_{\mathit{eq}}\big( {\textstyle \half} (z^{(1)} \ast_F^{(0)} z^{(1)} + u z^{(1)} \ast^{(1)}_F [M]) |E \\
& \qquad \qquad + 2u^2 z^{(2)} \big). \end{aligned}
%
\end{align}
\end{conjecture}

The first equation \eqref{eq:ueq} is closely related to the discussion in Section \ref{subsec:secondary} below. Concerning the somewhat less transparent \eqref{eq:uueq}, we have formulated it in a way which is convenient for our applications, but which is not necessarily the most conceptual one. To get a better picture, we should point out that the Borman-Sheridan class is the first in a sequence of classes 
\begin{equation} \label{eq:s-k}
\tilde{s}^{(k)} \in \mathit{SH}^0(E),
\end{equation}
defined by looking at thimbles with $k$-fold tangency to $M$. In the simplest nontrivial case, \eqref{eq:s-k} is related to the pair-of-pants square of \eqref{eq:s-element} by the formula (proved in Section \ref{subsec:higher-s} below)
\begin{equation} \label{eq:s22}
s \bullet s = \tilde{s}^{(2)} + 2 z^{(2)} e,
\end{equation}
where $e$ is the unit. (The classes $\tilde{s}^{(k)}$ have natural equivariant lifts, independent of \eqref{eq:forgetful-iso}, something which is not a priori true of $s \bullet s$.) Using \eqref{eq:s22}, as well as \eqref{eq:relative-gw} and \eqref{eq:wdvv}, one rewrites \eqref{eq:uueq} as
\begin{equation} \label{eq:more-gw}
\begin{aligned}
u^2 \tilde{s}^{(2)}_{\mathit{eq}} & =  B_{\mathit{eq}} \big( {\textstyle \half} ((z^{(1)} \smile [M] + uz^{(1)}) \ast^{(1)}_F [M]) |E\big) \\
& = B_{\mathit{eq}} \big( (\textstyle \half (z^{(1)} \smile [M]) \ast^{(1)}_F [M] + u\tilde{z}^{(2)})|E \big). 
\end{aligned}
\end{equation}

\begin{remark}
Figure \ref{fig:gw} gives a schematic picture of the two terms in the second line of \eqref{eq:more-gw}. One can think of them as corresponding to strata in the space of relative stable maps \cite{li} to $(F, M)$, which have degree $2$ over $\bC P^1$. In both cases, the curve has a component not drawn in Figure \ref{fig:gw}, which is a parametrized map to $(\bC P^1 \times M, \{0,\infty\} \times M)$ that is a double cover of $\bC P^1 \times \{\mathit{point}\}$, with one branch point at $(\infty, \mathit{point})$. The difference between the two cases lies in the position of the other branch point. More precisely, what one may want to think of is a parametrized version of the space of stable maps, which as in \cite{givental} can be interpreted as unparametrized maps in $\bC P^1 \times F$. In spite of this tentative description, we will not speculate on the infinite hierarchy of which the equations \eqref{eq:ueq} and \eqref{eq:more-gw} should be the start.
\end{remark}
%
\begin{figure}
\begin{centering}
\begin{picture}(0,0)%
\includegraphics{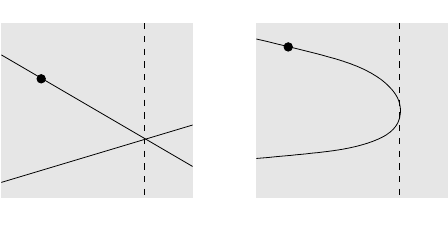}%
\end{picture}%
\setlength{\unitlength}{3355sp}%
\begingroup\makeatletter\ifx\SetFigFont\undefined%
\gdef\SetFigFont#1#2#3#4#5{%
  \reset@font\fontsize{#1}{#2pt}%
  \fontfamily{#3}\fontseries{#4}\fontshape{#5}%
  \selectfont}%
\fi\endgroup%
\begin{picture}(4213,2217)(-4361,-1630)
\put(-3074,464){\makebox(0,0)[lb]{\smash{{\SetFigFont{10}{12.0}{\rmdefault}{\mddefault}{\updefault}{\color[rgb]{0,0,0}$M$}%
}}}}
\put(-674,464){\makebox(0,0)[lb]{\smash{{\SetFigFont{10}{12.0}{\rmdefault}{\mddefault}{\updefault}{\color[rgb]{0,0,0}$M$}%
}}}}
\put(-1274,-1561){\makebox(0,0)[lb]{\smash{{\SetFigFont{10}{12.0}{\rmdefault}{\mddefault}{\updefault}{\color[rgb]{0,0,0}$\tilde{z}^{(2)}$}%
}}}}
\put(-4299,-1561){\makebox(0,0)[lb]{\smash{{\SetFigFont{10}{12.0}{\rmdefault}{\mddefault}{\updefault}{\color[rgb]{0,0,0}$(z^{(1)} \smile [M]) \ast^{(1)}_F [M]$}%
}}}}
\end{picture}%
\caption{\label{fig:gw}Holomorphic curves corresponding to terms in the second line of \eqref{eq:more-gw}.}
\end{centering}
\end{figure}
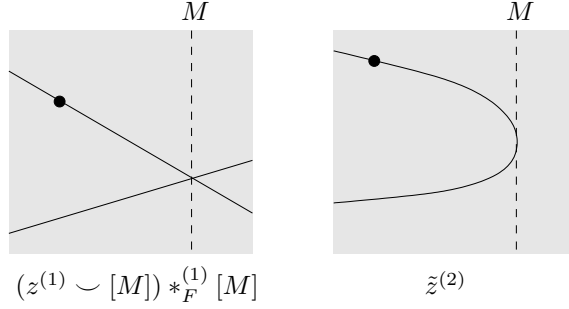

\begin{remark}
Here is another way to see the meaning of formulae such as \eqref{eq:ueq}. Suppose that $L \subset E \setminus \partial E$ is a closed Lagrangian submanifold ({\it Spin}, and with vanishing Maslov class). For the sake of argument, let's also assume that $L$ bounds only constant holomorphic discs in $E$ (for some choice of almost complex structure). Consider holomorphic discs in $F$ with boundary on $L$, and which intersect $M$ once. The boundary values of such discs gives rise to a class
\begin{equation} \label{eq:w-class}
w_{\mathit{eq}} \in H_{\mathit{eq}, n}(\scrL L;\bK),
\end{equation}
lying in the equivariant homology of the free loop space. More precisely, we use the equivariant homology theory from \cite{jones-petrack90}, which localises along constant loops. Taking Poincar{\'e} duality into account, the localisation theorem says that
\begin{equation} \label{eq:l-localisation}
H_{\mathit{eq},-*}(\scrL L;\bK) \otimes_{\bK[[u]]} \bK((u)) \iso H^{*+n}(L;\bK((u))).
\end{equation}
After localisation, \eqref{eq:w-class} contains only contributions from $S^1$-invariant (stable) discs, which are holomorphic spheres with a constant disc attached. Hence, under \eqref{eq:l-localisation},
\begin{equation} \label{eq:w-maps-to}
w_{\mathit{eq}} \longmapsto u^{-1} z^{(1)}|L \in H^0(L;\bK((u))).
\end{equation}
On the other hand, \eqref{eq:w-class} is the image of the Borman-Sheridan class under a canonical map $\mathit{SH}^*_{\mathit{eq}}(E) \rightarrow H_{\mathit{eq},n-*}(\scrL L;\bK)$ (see \cite[Theorem 1.1.9]{zhao16}). Hence, \eqref{eq:w-maps-to} is consistent with \eqref{eq:ueq}.
\end{remark}

\subsection{Implications for the connections $\nabla^c$}
We now combine the general theory above with Assumption \ref{th:psi-eta}. As an immediate consequence, the formulae \eqref{eq:ueq}, \eqref{eq:uueq} take on the following form:
\begin{equation}
\begin{aligned}
\label{eq:ueq-prime} 
& 
u \psi s_{\mathit{eq}} = B_{\mathit{eq}}(q^{-1}[\omega_E]), \\
& 
u^2 \psi^2\, (s \bullet s)_{\mathit{eq}}  = B_{\mathit{eq}}\Big( \half q^{-1}[\omega_E] \ast_E q^{-1}[\omega_E] \\ & \qquad \qquad \qquad + 
\half u \Big(\eta - \frac{\partial_q\psi}{\psi} - q^{-1}\Big) q^{-1}[\omega_E] + 2u^2\, z^{(2)} \psi^2 \Big).
\end{aligned}
\end{equation}
Using \eqref{eq:sh-gamma} as well as \eqref{eq:m-ast-m} and \eqref{eq:omega-ast-m}, one arrives at the following statement for the Gauss-Manin connection:
\begin{equation} \label{eq:ueq-doubleprime}
\begin{aligned} 
&
\Gamma( e_{\mathit{eq}}) = \Gamma(B_{\mathit{eq}}(1)) = B_{\mathit{eq}}( q^{-1} [\omega_E])
= u \psi s_{\mathit{eq}}, 
\\
&
u \Gamma (s_{\mathit{eq}}) = \Gamma( B_{\mathit{eq}}(\psi^{-1} q^{-1}[\omega_E]  ))
\\ & \qquad = B_{\mathit{eq}} \big( u\partial_q(\psi^{-1}q^{-1}[\omega_E]) + q^{-1}[\omega_E] \ast_E \psi^{-1}q^{-1}[\omega_E] \big)
\\
& \qquad = B_{\mathit{eq}} \Big(
u \psi^{-1} \big(-\frac{\partial_q\psi}{\psi} - q^{-1}\big) q^{-1}[\omega_E] \\ & \qquad \qquad \qquad \qquad +
\psi^{-1} (q^{-1}[\omega_E] \ast_E q^{-1}[\omega_E]) \Big) \\
& \qquad = 2u^2 \psi\, (s \bullet s)_{\mathit{eq}} - u^2 \eta s_{\mathit{eq}} - 4 u^2 z^{(2)} \psi e_{\mathit{eq}}.
\end{aligned}
\end{equation}
Assuming Conjecture \ref{th:eq-connection}, one would conclude that
\begin{equation} \label{eq:ueq-tripleprime}
\begin{aligned}  
&
\nabla_{\mathit{eq}} e_{\mathit{eq}} = \psi s_{\mathit{eq}} \; \text{modulo $\mathit{ker}(u)$,} \\
&
\nabla_{\mathit{eq}} s_{\mathit{eq}} = 2 \psi (s \bullet s)_{\mathit{eq}} - \eta s_{\mathit{eq}} - 4 z^{(2)} \psi e_{\mathit{eq}} \; \text{modulo $\mathit{ker}(u^2)$.}
\end{aligned}
\end{equation}
Let's suppose that we already knew that, for some undetermined constants $\gamma_{ij} \in \bK$,
\begin{equation} \label{eq:know}
\begin{aligned}
& \nabla^{-1} e = \gamma_{10} s + \gamma_{00} e, \\
& \nabla^{-1} s = \gamma_{21} (s \bullet s) + \gamma_{11} s + \gamma_{01} e.
\end{aligned}
\end{equation}
(It is not hard to arrive at such a qualitative statement, using a filtration as in \cite{mclean12}; however, we will not explain that argument here.) Because of \eqref{eq:eq-connection-2} and \eqref{eq:forgetful-iso}, the equivariant counterpart must then hold as well:
\begin{equation}
\begin{aligned}
& \nabla_{\mathit{eq}} e_{\mathit{eq}} = \gamma_{10} s_{\mathit{eq}} + \gamma_{00} e_{\mathit{eq}}, \\
& \nabla_{\mathit{eq}} s_{\mathit{eq}} = \gamma_{21} (s \bullet s)_{\mathit{eq}} + \gamma_{11} s_{\mathit{eq}} + \gamma_{01} e_{\mathit{eq}}.
\end{aligned}
\end{equation}
Comparison with \eqref{eq:ueq-tripleprime} yields
\begin{equation} \label{eq:ueq-quadrupleprime}
\begin{aligned}
& (\gamma_{10} - \psi) s_{\mathit{eq}} + \gamma_{00} e_{\mathit{eq}} \in \mathit{ker}(u), \\
& (\gamma_{21} - 2\psi) (s \bullet s)_{\mathit{eq}} + (\gamma_{11} + \eta) s_{\mathit{eq}} + (\gamma_{01} + 4 z^{(2)} \psi) e_{\mathit{eq}} \in \mathit{ker}(u^2).
\end{aligned}
\end{equation}
From \eqref{eq:ueq} and \eqref{eq:zhao}, one sees that $e_{\mathit{eq}}$ and $s_{\mathit{eq}}$ generate a free rank two $\bK[[u]]$-submodule of $\mathit{SH}^*_{\mathit{eq}}(E)$. Hence, the first expression in \eqref{eq:ueq-quadrupleprime} is not just $u$-torsion, but actually zero. If we additionally suppose that $\mathrm{dim}(E) > 4$, the same argument shows that $(e_{\mathit{eq}}, s_{\mathit{eq}}, (s \bullet s)_{\mathit{eq}})$ generate a free rank three submodule, leading to the same conclusion for \eqref{eq:ueq-quadrupleprime}. The outcome would then be that
\begin{equation} \label{eq:enum}
\begin{aligned}
&
\nabla^{-1} e = \psi s, \\
&
\nabla^{-1} s = 2 \psi (s \bullet s) - \eta s - 4 z^{(2)} \psi e,
\end{aligned}
\end{equation}
which are the $c = -1$ cases of \eqref{eq:nabla-c-e} and \eqref{eq:nablac-s}, respectively. By the same argument as in Section \ref{subsec:main} (taken in reverse), Lemma \ref{th:bounding} and Theorem \ref{th:sigma1} would then follow from this, in a way which is independent of the proofs given in this paper (conversely, the fact that we already have a proof of these results provides a strong check on Conjecture \ref{th:equi-sheridan}).

\begin{remark}
It is worthwhile reiterating the ingredients that enter into this (speculative) approach to \eqref{eq:enum}. One is Conjecture \ref{th:equi-sheridan}, which is a relation between equivariant symplectic cohomology and enumerative geometry (it does not require Assumption \ref{th:psi-eta}, and does not involve connections in any way). The second ingredient is the general theory of Gauss-Manin connections. Finally, one needs to know that a formula of the form \eqref{eq:know} exists, which is a qualitative statement about filtrations on symplectic cohomology. Altogether, this approach is somewhat complicated, but has the advantage of separating out different parts, each of which is meaningful in itself.
\end{remark}

\section{Motivation and context\label{sec:motivation}}

Like the previous one, this section is not strictly speaking needed for the rest of the paper. Its aim is to put the main results in a wider context. With that in mind, most of the arguments will just be outlined.

\subsection{Homological algebra}
Let $\scrA$ be an $A_\infty$-algebra over $\bK$. We denote by $\mathit{HH}_*(\scrA,\scrA)$ and $\mathit{HH}^*(\scrA,\scrA)$ its Hochschild homology and cohomology. $\mathit{HH}^*(\scrA,\scrA)$ is a Gerstenhaber algebra. $\mathit{HH}_*(\scrA,\scrA)$ is a Gerstenhaber module over that algebra, and also carries the Connes operator, an endomorphism of degree $-1$. The Kaledin class \cite{kaledin07,lunts10}
\begin{equation} \label{eq:kaledin}
k \in \mathit{HH}^2(\scrA,\scrA)
\end{equation}
is a distinguished element which describes the dependence of $\scrA$ on the formal variable $q$. To define it, choose a connection on $\scrA$ as a vector space; then, the derivative of the $A_\infty$-operations with respect to that connection yields a cocycle representing \eqref{eq:kaledin}. 

One can generalize Hochschild (co)homology by allowing coefficients in a bimodule $\scrP$ (the classical case corresponds to taking $\scrP$ to be the diagonal bimodule $\scrA$). The previously mentioned module structure extends to maps (expressing the functoriality of Hochschild homology)
\begin{equation} \label{eq:funct-hh}
\mathit{HH}^*(\scrA,\scrP) \otimes \mathit{HH}_*(\scrA,\scrA) \longrightarrow \mathit{HH}_*(\scrA,\scrP).
\end{equation}
The choice that will be of particular interest for us is the ``inverse dualizing'' bimodule $\scrP = \scrA^!$ \cite[Definition 8.1.6]{kontsevich-soibelman06}. A little confusingly, its cohomology is already itself a Hochschild cohomology group,
\begin{equation} \label{eq:ashriek}
H^*(\scrA^!) = \mathit{HH}^*(\scrA, \scrA \otimes_{\bK} \scrA);
\end{equation}
Let's assume that $\scrA$ is homologically smooth (the diagonal bimodule is perfect): in that case, $\scrA^!$ is also perfect \cite[Proposition 8.1.5]{kontsevich-soibelman06}, and there are natural isomorphisms \cite[Corollary 2.3]{ganatra13}
\begin{equation} \label{eq:shriek-2}
\mathit{HH}^*(\scrA,\scrP) \iso \mathit{HH}_*(\scrA, \scrA^! \otimes_{\scrA} \scrP)
\end{equation}
for perfect $\scrP$ (one recovers \eqref{eq:ashriek} by taking $\scrP = \scrA \otimes_{\bK} \scrA$). Assume additionally that $\scrA^!$ is an invertible bimodule (with respect to tensor product), and denote its inverse by $(\scrA^!)^{-1}$. Then, tensor product with $\scrA^!$ or its inverse preserves perfect bimodules (because of their characterization as compact objects). We can therefore replace $\scrP$ by $(\scrA^!)^{-1} \otimes_{\scrA} \scrP$ in \eqref{eq:shriek-2}, and rewrite the outcome as
\begin{equation} \label{eq:hom-from-shriek}
H^*(\mathit{hom}_{(\scrA,\scrA)}(\scrA^!,\scrP)) \iso \mathit{HH}_*(\scrA,\scrP),
\end{equation}
where the morphism space on the left hand is in the category of bimodules. We will actually use only the special case $\scrP = \scrA$.

Following \cite[Definition 3.2.3]{ginzburg06}, let's say that $\scrA$ is Calabi-Yau (of dimension $n$) if it is homologically smooth and comes with a choice of quasi-isomorphism
\begin{equation} \label{eq:calabi-yau}
\scrA^! \stackrel{\htp}{\longrightarrow} \scrA[-n].
\end{equation}
One then gets induced isomorphisms 
\begin{equation} \label{eq:calabi-yau-hh}
\mathit{HH}^{*+n}(\scrA,\scrA^!) \stackrel{\iso}{\longrightarrow} \mathit{HH}^*(\scrA,\scrA) \stackrel{\iso}{\longrightarrow}
\mathit{HH}_{*-n}(\scrA,\scrA).
\end{equation}
By \eqref{eq:hom-from-shriek}, the map \eqref{eq:calabi-yau} can also be viewed as a $\nu \in \mathit{HH}_{-n}(\scrA,\scrA)$ (there is a stronger version of the Calabi-Yau property, which involves a lift of $\nu$ to negative cyclic homology, see e.g.\ \cite[Definition 6.7]{ganatra-perutz-sheridan15}; but that won't be necessary for our purpose). In those terms, the second map in \eqref{eq:calabi-yau-hh} is given by letting Hochschild cohomology classes act on $\nu$ (through the module structure of Hochschild homology). Similarly, the first map is given by inserting $\nu$ into \eqref{eq:funct-hh}, for $\scrP = \scrA^!$, to get a map $\mathit{HH}^{*+n}(\scrA,\scrA^!) \rightarrow \mathit{HH}_*(\scrA,\scrA^!)$, and then using the isomorphism \eqref{eq:ashriek} (which is independent of the Calabi-Yau structure).

From now on, let's consider the special case when the Kaledin class \eqref{eq:kaledin} vanishes. This holds if and only if $\scrA$ admits an $A_\infty$-connection, which is defined to be a sequence of operations
\begin{equation}
\begin{aligned}
& \partial^0 \in \scrA^1, \\
& \partial^1: \scrA \longrightarrow \scrA, \\
& \partial^2: \scrA \otimes \scrA \longrightarrow \scrA[-1], \\
& \dots
\end{aligned}
\end{equation}
with the following properties: each $\partial^d$, $d \neq 1$, is $\bK$-(multi)linear, whereas $\partial^1$ is a connection; and the relations for a Hochschild cocycle are satisfied. An $A_\infty$-connection on $\scrA$ induces a connection $\nabla$ (in the ordinary sense) on $\mathit{HH}^*(\scrA,\scrA)$, as well as connections $\nabla^{-1}$ on $\mathit{HH}_*(\scrA,\scrA)$ and $\nabla^1$ on $\mathit{HH}^*(\scrA,\scrA^!)$. These connections are compatible with all algebraic structures on those groups. In particular, $\nabla$ satisfies analogues of \eqref{eq:nabla-derivation} and \eqref{eq:nabla-derivation-2}, while $\nabla^{-1}$ commutes with the Connes operator.

\begin{proposition}
Assume that $\scrA$ is Calabi-Yau, and that it carries an $A_\infty$-connection. Let $a \in \mathit{HH}^0(\scrA,\scrA)$ be the unique class such that
\begin{equation}
\nabla^{-1} \nu = -a\nu \in \mathit{HH}_{-n}(\scrA,\scrA).
\end{equation}
Then, if one uses \eqref{eq:calabi-yau-hh} to think of $\nabla^{\pm 1}$ as living on $\mathit{HH}^*(\scrA,\scrA)$, 
\begin{equation} \label{eq:c-term-alg}
\nabla^{\pm 1} x = \nabla x \pm ax.
\end{equation}
\end{proposition}

\begin{proof}[Sketch of proof]
Because the connections $(\nabla,\nabla^{-1})$ are compatible with the module structure of Hochschild homology,
\begin{multline}
\nabla^{-1}(x \nu) = (\nabla x)\nu + x (\nabla^{-1} \nu) = (\nabla x - ax) \nu \\ \in \mathit{HH}_*(\scrA,\scrA),
\;\;\text{ for $x \in \mathit{HH}^*(\scrA,\scrA)$.}
\end{multline}
Similarly, compatibility with \eqref{eq:funct-hh} for $\scrP = \scrA^!$ means that
\begin{multline}
\nabla(x \nu) = (\nabla^1 x) \nu + (x \nabla^{-1} \nu) = (\nabla^1 x - ax) \nu \\  \in \mathit{HH}_*(\scrA,\scrA^!) \iso \mathit{HH}^*(\scrA,\scrA), \;\;\text{ for $x \in \mathit{HH}^*(\scrA,\scrA^!)$.}
\end{multline}
\end{proof}

The formula \eqref{eq:c-term-alg} is obviously parallel to \eqref{eq:c-term}, even though the direction of the argument is reversed. Here, we were given $\nabla$ and $\nabla^{\pm 1}$ with certain properties (for instance, that $\nabla^{-1}$ was compatible with the Connes operator), and then derived the relation \eqref{eq:c-term-alg} from those; in the case of \eqref{eq:c-term}, we obtained $\nabla^c$ by an explicit modification of $\nabla$, and any properties of those connections had to be established based on that.
By using powers of the Serre functor, one could construct homological algebra counterparts of $\nabla^c$ for any integer value of $c$, but we will not discuss that here.

The relation between homological algebra and symplectic cohomology goes via the wrapped Fukaya category $\scrW(E)$. This is a $\bZ$-graded $A_\infty$-category defined over $\bK$ (and expected to be Calabi-Yau, see \cite{ganatra13}). In general, Fukaya categories have curvature ($\mu^0$) terms. Here, we work in the framework of Lefschetz fibrations (as in Section \ref{subsec:define-lefschetz}), and allow only Lefschetz thimbles as objects. In that case, one can arrange that $\mu^0 = 0$. The previous discussion of connections on $A_\infty$-algebras carries over without any difficulties to the categorical case.
%
%
The Hochschild (co)homology of $\scrW(E)$ is related to symplectic cohomology by open-closed string maps \cite{abouzaid10, ganatra13, ritter-smith12}
\begin{align} \label{eq:co}
& \mathit{SH}^*(E) \longrightarrow \mathit{HH}^*(\scrW(E),\scrW(E)), \\
\label{eq:oc}
& \mathit{HH}_*(\scrW(E),\scrW(E)) \longrightarrow \mathit{SH}^{*+n}(E).
\end{align}
The image of $q^{-1}[\omega_E]$ under the composition of \eqref{eq:co} and \eqref{eq:acc} is the Kaledin class of $\scrW(E)$. In particular, if Assumption \ref{th:as-1} holds, the Kaledin class is trivial.

\begin{conjecture} \label{th:constant}
Suppose that Assumption \ref{th:as-1} holds. Then, for a suitable choice of connection $\nabla$ on $\mathit{SH}^*(E)$ as in Proposition \ref{th:1}, and of $A_\infty$-connection on $\scrW(E)$, the map \eqref{eq:co} is covariantly constant. Similarly, \eqref{eq:oc} is covariantly constant if we equip $\mathit{SH}^*(E)$ with the connection $\nabla^{-1}$ from \eqref{eq:c-term}.
\end{conjecture}

It makes sense to think that there should be another canonical map
\begin{equation} 
\label{eq:co2}
\mathit{SH}^*(E) \longrightarrow \mathit{HH}^{*+n}(\scrW(E),\scrW(E)^!).
\end{equation}
The image of $e$ under \eqref{eq:co2} can be thought of as a bimodule map $\scrW(E)[-n] \rightarrow \scrW(E)^!$ (conjecturally, it is the inverse of the isomorphism \eqref{eq:calabi-yau} defining the Calabi-Yau structure). The analogue of Conjecture \ref{th:constant} for that map involves the connection $\nabla^1$. It would then follow that the $q$-dependence of this bimodule map is governed by \eqref{eq:2nd-order}.

\begin{remark}
The material from Section \ref{sec:equivariant} fits into this discussion as follows. Negative cyclic homology $\mathit{HC}_*^-(\scrA)$ is the natural $S^1$-equivariant analogue of $\mathit{HH}_*(\scrA,\scrA)$: it is a module over $\bK[[u]]$, where $u$ is a formal variable of degree $2$, and fits into a long exact sequence
\begin{equation}
\cdots \rightarrow \mathit{HC}_{*-2}^-(\scrA) \stackrel{u}{\longrightarrow} \mathit{HC}_*^-(\scrA) \longrightarrow \mathit{HH}_*(\scrA) \rightarrow \cdots
\end{equation}
Getzler's connection \cite{getzler95} is a canonical $u$-linear map
\begin{equation} \label{eq:gamma-connection}
\begin{aligned}
& \Gamma: \mathit{HC}_*^-(\scrA) \longrightarrow \mathit{HC}_{*+2}^-(\scrA), \\
& \Gamma(f(q)x) = f(q) \Gamma(x) + u (\partial_q f) x,
\end{aligned}
\end{equation}
which is related to the Kaledin class by a diagram
\begin{equation} \label{eq:obst}
\xymatrix{
\ar[d] \mathit{HC}_*^-(\scrA) \ar[rr]^-{\Gamma} && \mathit{HC}_{*+2}^-(\scrA) \ar[d] \\
\mathit{HH}_*(\scrA,\scrA) \ar[rr]^-{k} && \mathit{HH}_{*+2}(\scrA,\scrA).
}
\end{equation}
Assuming that the Kaledin class vanishes, an $A_\infty$-connection on $\scrA$ induces a $u$-linear connection $\nabla_{\mathit{eq}}$ on $\mathit{HC}_*^-(\scrA)$, which satisfies
\begin{equation} \label{eq:alg1}
\Gamma = u\nabla_{\mathit{eq}},
\end{equation} 
and which is a lift of $\nabla^{-1}$:
\begin{equation} \label{eq:alg2}
\xymatrix{
\ar[d] \mathit{HC}_*^-(\scrA) \ar[rr]^-{\nabla_{\mathit{eq}}} && \ar[d] \mathit{HC}_*^-(\scrA)
\\
\mathit{HH}_*(\scrA,\scrA) \ar[rr]^-{\nabla^{-1}} && \mathit{HH}_*(\scrA,\scrA).
}
\end{equation}
Suitable equivariant open-closed string maps (like those for compact manifolds studied in \cite{ganatra-perutz-sheridan15}) should relate $\mathit{HC}_*^-(\scrA)$ to the $u$-adically completed version of equivariant symplectic cohomology. This is not the version we discussed in Section \ref{sec:equivariant}, but it comes with a map from that version, and structures like the Gauss-Manin connection carry over by $u$-adically completing (on the chain level). Therefore, is still makes sense to consider \eqref{eq:alg1} and \eqref{eq:alg2} as algebraic counterparts of Conjecture \ref{th:eq-connection}.
\end{remark}

\subsection{Mirror symmetry}
Homological mirror symmetry relates wrap\-ped Fukaya categories to categories of coherent sheaves (on smooth open Calabi-Yau varieties). Symplectic cohomology is mirror to polyvector field cohomology. We will consider a specific algebro-geometric situation which formally corresponds to the symplectic one encountered before. 

For the sake of familiarity, let's switch the coefficient field to $\bF = \overline{\bC((h))}$, which is the algebraic closure of the Laurent series field (obtained by adjoining roots $h^{1/d}$ of all orders; in symplectic geometry, such a field can be used instead of the Novikov field whenever the symplectic class is rational). Let $X$ be a smooth variety of dimension $n$ over $\bF$. The polyvector field cohomology 
\begin{equation} \label{eq:ht}
HT^*(X) = H^*(X, \Lambda^* TX)
\end{equation}
has the structure of a Gerstenhaber algebra (given by the exterior product and Schouten bracket). 
The dependence of $X$ on $h$ is measured by the Kodaira-Spencer class
\begin{equation} \label{eq:ks}
k \in H^1(X,TX) \subset \mathit{HT}^2(X).
\end{equation}
If $k$ vanishes, $X$ admits a connection in $\partial_h$-direction (acting on the sheaf of functions). This induces a connection on polyvector fields (and their cohomology), which we denote by $\nabla$; and also one on differential forms, which we denote by $\nabla^{-1}$.

Suppose that $X$ has trivial canonical bundle, and comes with an algebraic volume form $\nu$. This gives rise to a BV operator on \eqref{eq:ht}, induced by the sheaf homomorphism
\begin{equation} \label{eq:dd}
\Delta: \Lambda^i TX \stackrel{\nu}{\iso} \Omega^{n-i}_{X}
\stackrel{d}{\longrightarrow} \Omega^{n+1-i}_{X} \stackrel{\nu}{\iso}
\Lambda^{i-1} TX.
\end{equation}
Now consider the case where $X$ comes with a connection, and set
\begin{equation} \label{eq:holomorphic-a}
a = -\frac{\nabla^{-1} \nu}{\nu} \in H^0(X,\scrO_{X}) \subset \mathit{HT}^0(X).
\end{equation}
If we use $\nu$ to identify the differential forms and polyvector fields, then
\begin{equation} \label{eq:nabla-1-t}
\nabla^{-1}x = \nabla x - ax.
\end{equation}
Since $\nabla^{-1}$ commutes with the de Rham differential, \eqref{eq:nabla-1-t} defines a connection on polyvector fields which commutes with the BV operator. This implies that the original connection $\nabla$ satisfies the analogue of \eqref{eq:delta-nabla} for the function \eqref{eq:holomorphic-a}.
More generally, there are natural connections $\nabla^c$ on $\Lambda^* TX \otimes K_{X}^{\otimes -c}$ for any $c \in \bZ$ (reducing to the case of differential forms for $c = -1$). If we use the volume form to trivialize $K_{X}$, these connections can be written as $\nabla^c = \nabla  + ca$, acting on polyvector fields.

%
The specific situation we want to consider is the following one. Start with a fibration over $\bC P^1$ which is log Calabi-Yau. By this, we mean a smooth projective variety $\bar{X}_0$ over $\bC$ and a surjective map
\begin{equation} \label{eq:cy-pencil}
\bar{X}_0 \longrightarrow \bC P^1 = \bC \cup \{\infty\},
\end{equation}
such that $K_{\bar{X}_0}$ is the pullback of $\scrO(-1)$. If we remove the fibre at $\infty$ to form
\begin{equation}
p_0: X_0 \longrightarrow \bC,
\end{equation}
then $X_0$ carries a holomorphic volume form $\nu_0$ (with a simple pole along the divisor we have removed). Extend constants to form $\bar{X} = \bar{X}_0 \times_{\bC} \bF$, and then remove the fibre at $h^{-1}$ to get $X \subset \bar{X}$. This carries an algebraic volume form
\begin{equation}
\nu = f \frac{\nu_0}{1-p_0 h},
\end{equation}
where $f \in \bF^\times$ is an arbitrary scaling factor. Let's write $l = (\partial_h f)f^{-1}$. In this situation, the connection on $X$ is just differentiation $\partial_h$. The function \eqref{eq:holomorphic-a} is therefore
\begin{equation} \label{eq:compute-a}
a = -\frac{\partial_h \nu}{\nu} = \frac{p_0}{p_0 h - 1} - l,
\end{equation}
which obviously satisfies
\begin{equation} \label{eq:a-a2}
\partial_h a + a^2 + 2l a + (\partial_h l + l^2) = 0.
\end{equation}
This equation is the mirror of \eqref{eq:nonlinear-a}. Correspondingly, the section of $K_X^{-1}$ given by
\begin{equation}
\eta = \frac{1}{\nu} = f^{-1} \frac{1-p_0 h}{\nu_0}
\end{equation}
satisfies the mirror of \eqref{eq:2nd-order-2}, which is
\begin{equation} \label{eq:trivial-ode}
\partial_h^2 \eta + 2l \partial_h \eta + (\partial_h l + l^2) \eta = 0.
\end{equation}
In fact, \eqref{eq:trivial-ode} has scalar solutions $1/f$ and $h/f$, of which our $\eta$ was a linear combination (with coefficients in sections of $K_{X_0}^{-1}$). If one wrote down an equation for quotients of solutions of \eqref{eq:trivial-ode} in parallel with \eqref{eq:schwarz}, that equation would just be vanishing of the Schwarzian $S_h$ (functions with that property are exactly the rational automorphisms of the $h$-line). What this means is that under mirror symmetry, the variable $h$ here does not correspond to $q$, but rather to the ``mirror coordinate'' which is one particular solution of \eqref{eq:schwarz} (see the discussion in \cite[Sections 4c and 7b]{seidel15}).




\section{Configuration spaces\label{sec:conf}}

This section and the two following ones serve as a transition between our introductory discussion, in which results were stated with no attempt at explaining the underlying Floer-theoretic ideas, and the subsequent detailed constructions. We will consider an operad-style framework in which abstract versions of the desired properties can be seen to hold, without the technicalities inherent in working with pseudo-holomorphic curves (one could call this a ``photoshopped version'' of the operations in Floer theory).

\subsection{Gerstenhaber algebras\label{subsec:gerstenhaber}}
We begin by recalling classical material: the little discs operad and its algebraic counterpart (see e.g.\ \cite{sinha13} for an exposition).

\begin{setup} \label{th:setup-sigma}
A disc configuration is a collection of points $\zeta_i \in \bC$ together with radii $r_i > 0$, indexed by a finite set $I$. We require that the closed discs of radius $r_i$ around the $\zeta_i$ should be pairwise disjoint, and contained in the open unit disc centered at the origin. We allow one exception to the last-mentioned condition, called the identity configuration: this consists of a single point $\zeta_1 = 0$ at the origin, with $r_1 = 1$. (When representing disc configurations graphically, we will often only draw the centers $\zeta_i$; this causes no major issues, since the choices of $r_i$ form a contractible set.)

Given two disc configurations, indexed by sets $I_1$ and $I_2$, and a choice of $i_1 \in I_1$, one can carry out the following gluing process. Take the second configuration, rescale it by $r_{i_1}$, and then insert it into the first configuration centered at $\zeta_{i_1}$, replacing the point originally located there. The outcome is a configuration indexed by $I = (I_1 \setminus \{i_1\}) \cup I_2$. As the name suggests, gluing with the identity configuration does nothing. 

We will also consider families of disc configurations, parametrized by a smooth compact oriented manifold $P$ (which may have boundary or corners). Given two families with parameter spaces $P_1$ and $P_2$, gluing produces a family over $P_1 \times P_2$.
\end{setup}

Fix some commutative coefficient ring $R$. The algebraic structure we will look at is a graded $R$-module $H^*$ with operations induced by families of configurations. More precisely, we require that for any family with closed $P$, there is an associated $R$-multilinear map
\begin{equation} \stareq
\label{eq:basic-operation}
(H^*)^{\otimes I} \longrightarrow H^{*-\mathrm{dim}(P)}. 
\end{equation}
(The (H) added to the equation numbers helps us record the fact that these structures live on the cohomology level space $H^*$, as opposed to the cochain level operations to be considered later on.) The tensor product of graded $R$-modules in \eqref{eq:basic-operation} includes the usual Koszul signs which relate different orderings of the indexing set $I$. The disjoint union of two parameter spaces (with the same $I$) should result in the sum of the associated operations. If we have a family over a (compact oriented, as always) manifold with boundary, then its restriction to the boundary yields the zero operation. The identity configuration should induce the identity map. Gluing of families should correspond to composition of operations, with appropriate signs. Namely, suppose for simplicity that we have operations $\phi_1$ and $\phi_2$ whose inputs are indexed by $I_1 = \{1,\dots,m_1\}$ and $I_2 = \{1,\dots,m_2\}$, respectively. Then, identifying $(I_1 \setminus \{i_1\}) \cup I_2 = \{1,\dots,m\}$ with $m = m_1+m_2-1$, the operation associated to the glued family is
\begin{equation} \stareq \label{eq:composition-law}
\begin{aligned}
\phi(x_1,\dots,x_m) & = (-1)^{(|\phi_1|+ |x_1| + \cdots + |x_{i_1-1}|) |\phi_2|}
\phi_1(x_1,\dots,x_{i_1-1}, \\  & \qquad \qquad
\phi_2(x_{i_1},\dots,x_{i_1+m_2-1}),x_{i_1+m_2},\dots,x_m).
\end{aligned}
 \end{equation}
%
\begin{figure}
\begin{centering}
\begin{picture}(0,0)%
\includegraphics{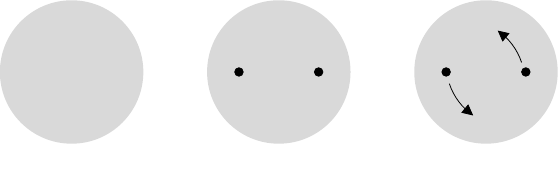}%
\end{picture}%
\setlength{\unitlength}{3355sp}%
\begingroup\makeatletter\ifx\SetFigFont\undefined%
\gdef\SetFigFont#1#2#3#4#5{%
  \reset@font\fontsize{#1}{#2pt}%
  \fontfamily{#3}\fontseries{#4}\fontshape{#5}%
  \selectfont}%
\fi\endgroup%
\begin{picture}(5250,1665)(-1349,-901)
\put(976,-886){\makebox(0,0)[lb]{\smash{{\SetFigFont{10}{12.0}{\rmdefault}{\mddefault}{\updefault}{Product}%
}}}}
\put(2926,-886){\makebox(0,0)[lb]{\smash{{\SetFigFont{10}{12.0}{\rmdefault}{\mddefault}{\updefault}{Bracket}%
}}}}
\put(751,164){\makebox(0,0)[lb]{\smash{{\SetFigFont{10}{12.0}{\rmdefault}{\mddefault}{\updefault}{1}%
}}}}
\put(2701,164){\makebox(0,0)[lb]{\smash{{\SetFigFont{10}{12.0}{\rmdefault}{\mddefault}{\updefault}{1}%
}}}}
\put(-824,-886){\makebox(0,0)[lb]{\smash{{\SetFigFont{10}{12.0}{\rmdefault}{\mddefault}{\updefault}{Unit}%
}}}}
\put(1651,-161){\makebox(0,0)[lb]{\smash{{\SetFigFont{10}{12.0}{\rmdefault}{\mddefault}{\updefault}{2}%
}}}}
\put(3601,-161){\makebox(0,0)[lb]{\smash{{\SetFigFont{10}{12.0}{\rmdefault}{\mddefault}{\updefault}{2}%
}}}}
\end{picture}%
\caption{\label{fig:tqft}The operations \eqref{eq:unit}--\eqref{eq:bracket}.}
\end{centering}
\end{figure}%

The basic operations obtained from this framework are:
\begin{align} 
\stareq \label{eq:unit}
& e \in H^0 && \text{unit}, \\
\stareq \label{eq:pair-of-pants}
& \bullet: H^* \otimes H^* \longrightarrow H^* && \text{pair-of-pants-product}, \\
\stareq \label{eq:bracket}
& [\cdot,\cdot]: H^* \otimes H^* \longrightarrow H^{*-1} && \text{bracket}.
\end{align}
The geometry underlying \eqref{eq:unit}--\eqref{eq:bracket} is shown in Figure \ref{fig:tqft} (where we always set $I = \{1,\dots,m\}$). The most interesting case is that of the bracket, which comes from a family of configurations parametrized by $P = S^1$.  From the general axioms, it follows that the pair-of-pants product is commutative, associative, and $e$ is its unit; and that the bracket satisfies 
\begin{align}
\stareq \label{eq:antisymmetry}
& [x_2,x_1] = (-1)^{|x_1|\cdot |x_2|} [x_1,x_2], \\
\stareq \label{eq:derivation-bracket}
& [x_1,x_2 \bullet x_3] = [x_1,x_2]\bullet x_3 + (-1)^{(|x_1|+1) |x_2|} x_2 \bullet [x_1,x_3], \\
\stareq \label{eq:jacobi} 
& 
\begin{aligned}
& (-1)^{|x_1|} [x_1,[x_2,x_3]] + (-1)^{|x_1|(|x_2|+|x_3|)+|x_2|} [x_2,[x_3,x_1]] \\
& \qquad \qquad \qquad \qquad + (-1)^{|x_3|(|x_1|+|x_2|+1)} [x_3,[x_1,x_2]] = 0, 
\end{aligned}
\\
\stareq \label{eq:e-is-ideal}
& [e, x] = 0.
\end{align}

\begin{remark}
Our sign for the bracket $[x_1,x_2]$ departs from the standard convention for Gerstenhaber algebras 
by $(-1)^{|x_1|}$, see e.g. \eqref{eq:antisymmetry}. This discrepancy will also be visible later on, when we define the bracket in terms of the BV operator \eqref{eq:delta-bracket}.
\end{remark}

As a reminder of how the relations \eqref{eq:antisymmetry}--\eqref{eq:e-is-ideal} are proved, consider for instance \eqref{eq:derivation-bracket}. By gluing, one defines families over $S^1$ which represent each of the terms involved; the equality involved is then given by a cobordism (a family whose parameter space $P$ is the pair-of-pants), see Figure \ref{fig:gamma}.
\begin{figure}
\begin{centering}
\begin{picture}(0,0)%
\includegraphics{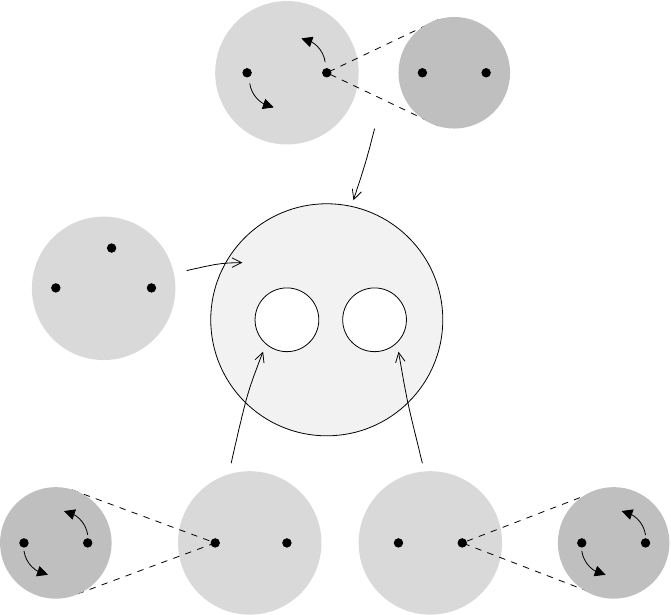}%
\end{picture}%
\setlength{\unitlength}{3355sp}%
\begingroup\makeatletter\ifx\SetFigFont\undefined%
\gdef\SetFigFont#1#2#3#4#5{%
  \reset@font\fontsize{#1}{#2pt}%
  \fontfamily{#3}\fontseries{#4}\fontshape{#5}%
  \selectfont}%
\fi\endgroup%
\begin{picture}(6300,5775)(-4424,-3736)
\put(-3674,-3286){\makebox(0,0)[lb]{\smash{{\SetFigFont{10}{12.0}{\rmdefault}{\mddefault}{\updefault}{2}%
}}}}
\put(-3524,-486){\makebox(0,0)[lb]{\smash{{\SetFigFont{10}{12.0}{\rmdefault}{\mddefault}{\updefault}{2}%
}}}}
\put(-4274,-2986){\makebox(0,0)[lb]{\smash{{\SetFigFont{10}{12.0}{\rmdefault}{\mddefault}{\updefault}{1}%
}}}}
\put(-2249,1439){\makebox(0,0)[lb]{\smash{{\SetFigFont{10}{12.0}{\rmdefault}{\mddefault}{\updefault}{1}%
}}}}
\put(-524,1439){\makebox(0,0)[lb]{\smash{{\SetFigFont{10}{12.0}{\rmdefault}{\mddefault}{\updefault}{2}%
}}}}
\put(151,1139){\makebox(0,0)[lb]{\smash{{\SetFigFont{10}{12.0}{\rmdefault}{\mddefault}{\updefault}{3}
}}}}
\put(-3899,-886){\makebox(0,0)[lb]{\smash{{\SetFigFont{10}{12.0}{\rmdefault}{\mddefault}{\updefault}{1}%
}}}}
\put(-3149,-886){\makebox(0,0)[lb]{\smash{{\SetFigFont{10}{12.0}{\rmdefault}{\mddefault}{\updefault}{3}%
}}}}
\put(-824,-2986){\makebox(0,0)[lb]{\smash{{\SetFigFont{10}{12.0}{\rmdefault}{\mddefault}{\updefault}{2}%
}}}}
\put(1576,-3286){\makebox(0,0)[lb]{\smash{{\SetFigFont{10}{12.0}{\rmdefault}{\mddefault}{\updefault}{3}%
}}}}
\put(-149,-886){\makebox(0,0)[lb]{\smash{{\SetFigFont{10}{12.0}{\rmdefault}{\mddefault}{\updefault}{parameter space $P$}%
}}}}
\put(-1799,-3286){\makebox(0,0)[lb]{\smash{{\SetFigFont{10}{12.0}{\rmdefault}{\mddefault}{\updefault}{3}%
}}}}
\put(901,-2986){\makebox(0,0)[lb]{\smash{{\SetFigFont{10}{12.0}{\rmdefault}{\mddefault}{\updefault}{1}%
}}}}
\end{picture}%
\caption{\label{fig:gamma}The relation \eqref{eq:derivation-bracket}.}
\end{centering}
\end{figure}

\begin{remark} \label{th:universal}
Take the universal family of disc configurations with $I = \{1,\dots,m\}$. Its parameter space $\scrP_m$ is homotopy equivalent to the configuration space of $m$ ordered points in $\bC$. As we have defined it, there is an operation \eqref{eq:basic-operation} for each bordism class in $\mathit{MSO}_*(\scrP_m) \otimes R$. One could add another axiom as follows: if a family of configurations is pulled back from a lower-dimensional parameter space, the associated map \eqref{eq:basic-operation} is zero. This mitigates the complexities of bordism, replacing it with a ``poor man's version of homology''. In particular, if $R$ contains $\bQ$, then operations would only depend on $\mathit{MSO}_*(\scrP_m) \otimes_{\mathit{MSO}_*} R \iso H_*(\scrP_m;R)$. 
\end{remark}

\subsection{Chain level operations}
The familiar next step is to refine the framework from Section \ref{subsec:gerstenhaber} by lifting it to the chain level. Here, we assume that the previous $H^*$ is the cohomology of a chain complex of $R$-modules, denoted by $(C^*,d)$. To any family of disc configurations over $P$, we now associate a map
\begin{equation} \label{eq:basic-chain-map}
(C^*)^{\otimes I} \longrightarrow C^{*-\mathrm{dim}(P)}.
\end{equation}
If $P$ is closed, this must be a chain map, inducing \eqref{eq:basic-operation} on cohomology. More generally, if $\phi$ is a map \eqref{eq:basic-chain-map} (with $I = \{1,\dots,m\}$ for simplicity), and $\psi$ the sum of the operations associated to the restriction of the family to the (closed codimension $1$) boundary faces of $P$, then
\begin{equation} \label{eq:boundary-axiom}
\begin{aligned}
& (-1)^{|\phi|} d\phi(x_1,\dots,x_m) - \phi(dx_1,\dots,x_m) - \cdots \\ & \;\;\cdots - (-1)^{|x_1|+\cdots+|x_{m-1}|} \phi(x_1,\dots,dx_m) + \psi(x_1,\dots,x_m) = 0.
\end{aligned}
\end{equation}
This replaces the cobordism axiom. The axioms concerning disjoint union (of parameter spaces), the identity configuration, and the composition law \eqref{eq:composition-law} remain as before.  Reversing the orientation of $P$ should switch the sign of the associated operation (up to homotopy, this follows from \eqref{eq:boundary-axiom}, but we prefer to impose a strict equality). We will also keep the property, inherent in the formulation with an arbitrary indexing set $I$, that permuting the inputs results in the standard Koszul signs.

\begin{remark}
For a more in-depth study of the theory, one would probably want to assume that the complex $C^*$ is cohomologically well-behaved (let's say, $R$ is a ring of finite global dimension, and $C^*$ is a complex of projective $R$-modules; this would imply that it is both $K$-flat and $K$-projective, in the terminology from \cite{spaltenstein}). However, for our very limited purpose, this is not necessary.
\end{remark}

The basic operations \eqref{eq:unit}--\eqref{eq:bracket} have chain level representatives which are unique up to chain homotopy. We fix such a choice once and for all, and use the same notation for them as before. The previously stated algebraic relations will now hold up to chain homotopies, which one can think of as secondary operations. For instance, the family from Figure \ref{fig:gamma} now gives rise to a map
\begin{equation} \label{eq:gamma-operation}
\begin{aligned}
& \gamma: (C^*)^{\otimes 3} \longrightarrow C^{*-2}, \\
& d\gamma(x_1,x_2,x_3) - \gamma(dx_1,x_2,x_3) - (-1)^{|x_1|} \gamma(x_1,dx_2,x_3) \\ & 
\;\; - (-1)^{|x_1|+|x_2|} \gamma(x_1,x_2,dx_3)
+ [x_1,x_2\bullet x_3] \\ & \;\; - [x_1,x_2] \bullet x_3 - (-1)^{(|x_1|+1)|x_2|} x_2 \bullet [x_1,x_3] = 0.
\end{aligned}
\end{equation}
Concerning such secondary operations, some uniqueness considerations are appropriate, and we use \eqref{eq:gamma-operation} as an example to explain them. From the viewpoint of Remark \ref{th:universal}, the construction of \eqref{eq:gamma-operation} relies on a choice of map $P \rightarrow \scrP_3$, where $P$ is a surface with three boundary components, and where the behaviour of the map on $\partial P$ is fixed. Within this general context there are genuinely different choices, since 
\begin{equation} \label{eq:mo-h}
\mathit{MSO}_2(\scrP_3) \otimes R \iso H_2(\scrP_3;R) \iso R^2. 
\end{equation}
This corresponds to the fact that we could add a multiple of $[x_3,[x_2,x_1]]$ or its cyclic permutations to $\gamma$, and the outcome would still satisfy \eqref{eq:gamma-operation}. When choosing the actual family in Figure \ref{fig:gamma}, we want to make sure that the point $\zeta_3$ always lies in the half-plane to the right of $\zeta_2$. This restricts us to a subspace of $\scrP_3$ for which the counterpart of \eqref{eq:mo-h} vanishes (in fact, that subspace is itself homotopy equivalent to a pair-of-pants). Two choices of $\gamma$ constructed under this restriction differ by a nullhomotopic chain map; we say that $\gamma$ has been defined ``in a homotopically unique way''. For future use, pick a concrete representative $\gamma$ within that class.

\subsection{The differentiation axiom\label{subsec:differentiation-axiom}}
We will now further enrich the previous framework. At first, this will be an extremely simple modification, which adds a distinguished cohomology class; in a second step, that class will take on a more fundamental meaning.

\begin{setup} \label{th:setup-extra}
We allow disc configurations with an optional extra point. Formally, this means that we equip them with the additional datum of some $Z \subset \bC$, with $|Z| \leq 1$, which lies in the closed unit disc, and avoids the interior of the discs around the $\zeta_i$.

In the exceptional instance where we have just one point $\zeta_1 = 0$ with $r_1 = 1$, the $Z$ point can lie anywhere on the unit circle. Let's call this the $r$-family. Given a disc configuration with no $Z$ point, gluing in the $r$-family corresponds to adding a $Z$ point which runs around a circle (either the unit circle, or the $r_i$-circles around the $\zeta_i$).

Suppose that we have a family over $P$ as in Setup \ref{th:setup-sigma}, which means with no extra point; and we also want to exclude the identity configuration. One can then define a new family by ``inserting the extra point in all possible places''. The base of the new family, denoted by $P^\circ$, has $\mathrm{dim}(P^\circ) = \mathrm{dim}(P) + 2$. (Figure \ref{fig:hat} shows how this looks like for the configuration underlying the pair-of-pants product; this time, it is helpful to see the radii $r_i$, and we have drawn them as dashed circles.)
\end{setup}
\begin{figure}
\begin{centering}
\begin{picture}(0,0)%
\includegraphics{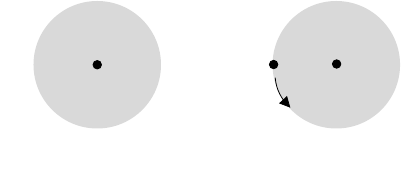}%
\end{picture}%
\setlength{\unitlength}{3355sp}%
\begingroup\makeatletter\ifx\SetFigFont\undefined%
\gdef\SetFigFont#1#2#3#4#5{%
  \reset@font\fontsize{#1}{#2pt}%
  \fontfamily{#3}\fontseries{#4}\fontshape{#5}%
  \selectfont}%
\fi\endgroup%
\begin{picture}(3765,1577)(-1589,-964)
\put(-1574,-886){\makebox(0,0)[lb]{\smash{{\SetFigFont{10}{12.0}{\rmdefault}{\mddefault}{\updefault}{Kodaira-Spencer class}%
}}}}
\put(826,-886){\makebox(0,0)[lb]{\smash{{\SetFigFont{10}{12.0}{\rmdefault}{\mddefault}{\updefault}{endomorphism $r$}%
}}}}
\put(-749,164){\makebox(0,0)[lb]{\smash{{\SetFigFont{10}{12.0}{\rmdefault}{\mddefault}{\updefault}{$Z$}%
}}}}
\put(826,164){\makebox(0,0)[lb]{\smash{{\SetFigFont{10}{12.0}{\rmdefault}{\mddefault}{\updefault}{$Z$}%
}}}}
\end{picture}%
\caption{\label{fig:tqft2}The operations \eqref{eq:k-class} and \eqref{eq:r-map}.}
\end{centering}
\end{figure}

Generalizing \eqref{eq:basic-operation}, each family over a closed $P$ should now give rise to a map
\begin{equation} \stareq \label{eq:basic-operation-z}
(H^*)^{\otimes I} \longrightarrow H^{*-\mathrm{dim}(P)+2|Z|},
\end{equation}
with the same properties as before. The effect of this is to introduce a distinguished element, called the Kodaira-Spencer class 
\begin{equation}
\stareq \label{eq:k-class}
k \in H^2.
\end{equation}
Geometrically, this comes from a single configuration with a $Z$ point (see Figure \ref{fig:tqft2}). We also use the $r$-family to define an endomorphism
\begin{equation}
\stareq \label{eq:r-map} 
r: H^* \longrightarrow H^{*+1},
\end{equation}
but that can be reduced to \eqref{eq:k-class},
\begin{equation} \stareq
r(x) = [k,x]. \label{eq:r-relation}
\end{equation}
The cobordism which proves this is based on a simple idea, which is to ``pull out the $Z$ point'', thinking of it as a copy of $k$ being glued in; see Figure \ref{fig:rho}.

It is straightforward to introduce the chain level version of \eqref{eq:basic-operation-z}. As before, we use the same notation for the chain level representatives of \eqref{eq:k-class} and \eqref{eq:r-map}. Then, Figure \ref{fig:rho} gives rise to a (homotopically unique) secondary operation
\begin{equation} \label{eq:rho-operation}
\begin{aligned}
& \rho: C^* \longrightarrow C^*, \\
& d\rho(x) - \rho(dx) + [k,x] - r(x) = 0.
\end{aligned}
\end{equation}
(When choosing the orientation of the underlying parameter space, our convention is as follows: the ``pulling out'' parameter, which is minimal for $r(x)$ and maximal for $[k,x]$, is the first coordinate; and the anticlockwise rotation of $Z$ is the second coordinate.)
\begin{figure}
\begin{centering}
\begin{picture}(0,0)%
\includegraphics{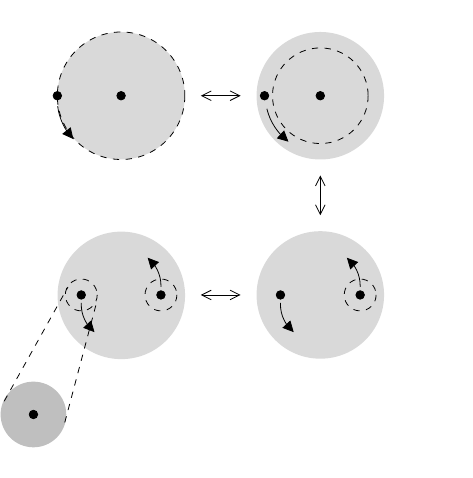}%
\end{picture}%
\setlength{\unitlength}{3355sp}%
\begingroup\makeatletter\ifx\SetFigFont\undefined%
\gdef\SetFigFont#1#2#3#4#5{%
  \reset@font\fontsize{#1}{#2pt}%
  \fontfamily{#3}\fontseries{#4}\fontshape{#5}%
  \selectfont}%
\fi\endgroup%
\begin{picture}(4437,4489)(-4064,-3500)
\put(-3224,-1561){\makebox(0,0)[lb]{\smash{{\SetFigFont{10}{12.0}{\rmdefault}{\mddefault}{\updefault}{\color[rgb]{0,0,0}$1$}%
}}}}
\put(-2774,-2086){\makebox(0,0)[lb]{\smash{{\SetFigFont{10}{12.0}{\rmdefault}{\mddefault}{\updefault}{\color[rgb]{0,0,0}$2$}%
}}}}
\put(-1649,239){\makebox(0,0)[lb]{\smash{{\SetFigFont{10}{12.0}{\rmdefault}{\mddefault}{\updefault}{\color[rgb]{0,0,0}$Z$}%
}}}}
\put(-1799,839){\makebox(0,0)[lb]{\smash{{\SetFigFont{10}{12.0}{\rmdefault}{\mddefault}{\updefault}{\color[rgb]{0,0,0}radius shrinks}%
}}}}
\put(-3674,839){\makebox(0,0)[lb]{\smash{{\SetFigFont{10}{12.0}{\rmdefault}{\mddefault}{\updefault}{\color[rgb]{0,0,0}endomorphism $r$}%
}}}}
\put(-1949,-2611){\makebox(0,0)[lb]{\smash{{\SetFigFont{10}{12.0}{\rmdefault}{\mddefault}{\updefault}{\color[rgb]{0,0,0}points move symmetrically}%
}}}}
\put(-4049,-3436){\makebox(0,0)[lb]{\smash{{\SetFigFont{10}{12.0}{\rmdefault}{\mddefault}{\updefault}{\color[rgb]{0,0,0}Kodaira-Spencer}%
}}}}
\put(-3074,-2611){\makebox(0,0)[lb]{\smash{{\SetFigFont{10}{12.0}{\rmdefault}{\mddefault}{\updefault}{\color[rgb]{0,0,0}bracket}%
}}}}
\put(-1499,-1636){\makebox(0,0)[lb]{\smash{{\SetFigFont{10}{12.0}{\rmdefault}{\mddefault}{\updefault}{\color[rgb]{0,0,0}$Z$}%
}}}}
\put(-3824,-2761){\makebox(0,0)[lb]{\smash{{\SetFigFont{10}{12.0}{\rmdefault}{\mddefault}{\updefault}{\color[rgb]{0,0,0}$Z$}%
}}}}
\end{picture}%
\caption{\label{fig:rho}The chain homotopy $\rho$ from \eqref{eq:rho-operation}.}
\end{centering}
\end{figure}
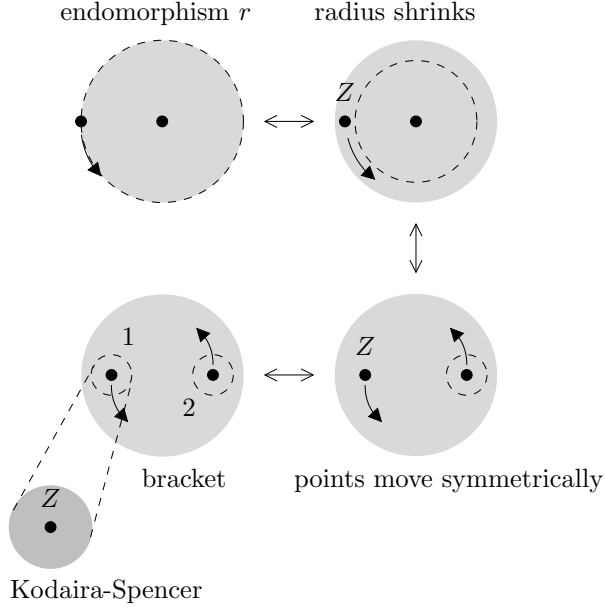

\begin{remark} \label{th:more-z}
In principle, it is possible to generalize our setup to $|Z| > 1$. However, one has to deal with points of $Z$ colliding, and we prefer not to discuss that here. The first consequence of such a generalization, using only $|Z| = 2$, would be that $[k,k] = 0$. More importantly, in the context of the differentiation axiom (to be introduced next), having $|Z| > 1$ would allow higher derivatives.
\end{remark}

Assume from now on that the coefficient ring $R$ comes with a derivation $\partial$. Correspondingly, we require that the graded module $C^*$ should come with a connection $\qabla$ in $\partial$-direction, which means an endomorphism (in each degree) which satisfies the analogue of \eqref{eq:connection-property}, with $\partial$ instead of $\partial_q$. This is not required to be compatible with the differential: instead, we ask that
\begin{equation} \label{eq:diff-axiom-1}
\qabla d x - d \qabla x = r(x). 
\end{equation}

\begin{remark}
An immediate consequence of \eqref{eq:diff-axiom-1} is that $r$ induces the trivial map on $H^*$. In spite of that, $r$ is not nullhomotopic in general (as an $R$-linear map; $\qabla$ is not $R$-linear, hence does not qualify as a nullhomotopy). For instance, suppose that $R = \bC[q]$, with $\partial = \partial_q$. Consider the chain complex
\begin{equation} \label{eq:stupid-r}
C^* = \big\{ R \stackrel{q}{\longrightarrow} R \big\},
\end{equation}
with $\qabla$ given by differentiation with respect to the obvious choice of basis. Then, $r$ maps the first $R$ group identically to the second one, and that is clearly not nullhomotopic. Alternatively (anticipating some of the later discussion) one can argue that if $r$ was nullhomotopic, $H^*$ would carry  a connection in $\partial$-direction; in our example, this contradicts the fact that the cohomology of \eqref{eq:stupid-r} is located at $q =0$.
\end{remark}

Suppose that $\phi$ is the operation \eqref{eq:basic-chain-map} associated to a family of configurations over $P$, with $Z = \emptyset$, and excluding the identity configuration. Let $\phi^\circ$ be the corresponding map for the associated family over $P^\circ$, as described in Setup \ref{th:setup-extra}. We then require that
\begin{multline} \label{eq:diff-axiom-2}
\qabla \phi(x_1, \dots, x_m) - \phi(\qabla x_1,\dots x_m) - \cdots - \phi(x_1,\dots, \qabla x_m) \\ = \phi^\circ(x_1,\dots,x_m).
\end{multline}
The two conditions \eqref{eq:diff-axiom-1} and \eqref{eq:diff-axiom-2} together are referred to as the differentiation axiom. 

\begin{remark}
As a check on the internal consistency of this two-part axiom, consider the case when $P$ is closed. It then follows from the structure of $P^\circ$ (see Figure \ref{fig:hat} for an example) that 
\begin{equation}
\begin{aligned}
& (-1)^{|\phi|} d\phi^\circ(x_1,\dots,x_m) - \phi^\circ(dx_1,\dots,x_m) - \cdots \\ & \qquad \qquad \qquad \cdots - 
(-1)^{|x_1|+\cdots+|x_{m-1}|} \phi^\circ(x_1,\dots,dx_m) \\
& \quad =  -(-1)^{|\phi|} r(\phi(x_1,\dots,x_m)) + \phi(r(x_1),\dots,x_m) + \\ & \qquad \qquad \qquad \cdots + 
(-1)^{|x_1|+\cdots+|x_{m-1}|} \phi(x_1,\dots,r(x_m)).
\end{aligned}
\end{equation}
From \eqref{eq:diff-axiom-1} and the fact that $\phi$ is a chain map, we get
\begin{equation}
\begin{aligned}
& (-1)^{|\phi|} d\big(\qabla \phi(x_1,\dots,x_m) - \phi(\qabla x_1,\dots,x_m) - \cdots - \phi(x_1,\dots, \qabla x_m)\big) 
\\
&\quad - \big(\qabla \phi(dx_1,\dots,x_m) - \phi(\qabla dx_1,\dots,x_m) - \cdots - \phi(dx_1,\dots, \qabla x_m)\big)
\\
&\quad \cdots 
\\
&\quad - (-1)^{|x_1|+\cdots+|x_{m-1}|} \big(\qabla \phi(x_1,\dots,dx_m) - \phi(\qabla x_1,\dots, dx_m) - \cdots \\ & \qquad \qquad \qquad\qquad - \phi(x_1,\dots,\qabla dx_m)\big) 
\\ &  = -(-1)^{|\phi|} r(\phi(x_1,\dots,x_m))
+ \phi(r(x_1),\dots,x_m) + \cdots \\ & \qquad \qquad \qquad \cdots + (-1)^{|x_1|+ \cdots + |x_{m-1}|}
\phi(x_1,\dots,r(x_m)).
\end{aligned}
\end{equation}
This shows that if we apply the differential $d$ to both sides of \eqref{eq:diff-axiom-2}, the outcome is the same; which is not a tautology, since the argument did not use \eqref{eq:diff-axiom-2}, only \eqref{eq:diff-axiom-1}.
\end{remark}
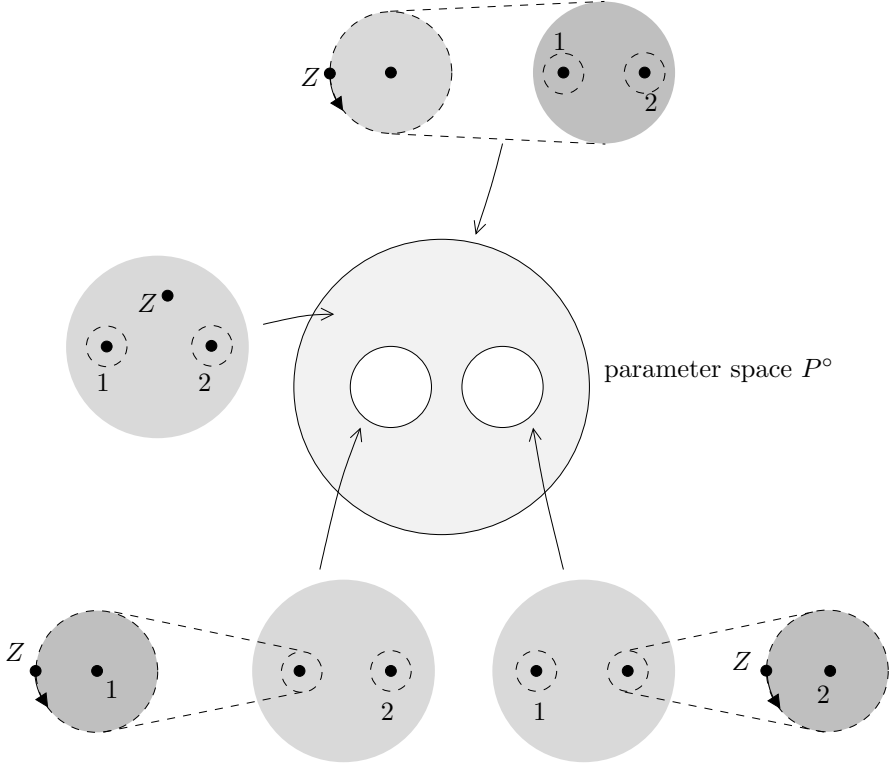
\begin{figure}
\begin{centering}
\begin{picture}(0,0)%
\includegraphics{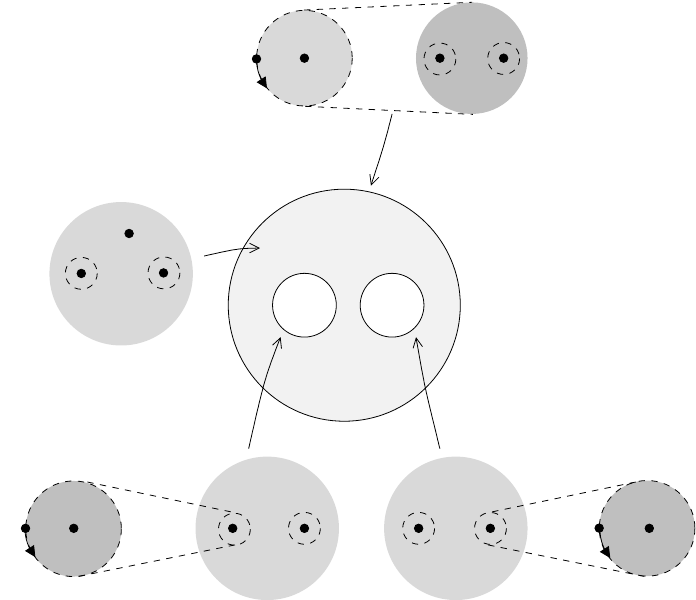}%
\end{picture}%
\setlength{\unitlength}{3355sp}%
\begingroup\makeatletter\ifx\SetFigFont\undefined%
\gdef\SetFigFont#1#2#3#4#5{%
  \reset@font\fontsize{#1}{#2pt}%
  \fontfamily{#3}\fontseries{#4}\fontshape{#5}%
  \selectfont}%
\fi\endgroup%
\begin{picture}(6548,5638)(-4589,-3736)
\put(-3834,-3265){\makebox(0,0)[lb]{\smash{{\SetFigFont{10}{12.0}{\rmdefault}{\mddefault}{\updefault}{$1$}%
}}}}
\put(1522,-3302){\makebox(0,0)[rb]{\smash{{\SetFigFont{10}{12.0}{\rmdefault}{\mddefault}{\updefault}{$2$}%
}}}}
\put(-3899,-991){\makebox(0,0)[lb]{\smash{{\SetFigFont{10}{12.0}{\rmdefault}{\mddefault}{\updefault}{$1$}%
}}}}
\put(-3149,-991){\makebox(0,0)[lb]{\smash{{\SetFigFont{10}{12.0}{\rmdefault}{\mddefault}{\updefault}{$2$}%
}}}}
\put(-149,-886){\makebox(0,0)[lb]{\smash{{\SetFigFont{10}{12.0}{\rmdefault}{\mddefault}{\updefault}{parameter space $P^\circ$}%
}}}}
\put(943,-3055){\rotatebox{360.0}{\makebox(0,0)[rb]{\smash{{\SetFigFont{10}{12.0}{\rmdefault}{\mddefault}{\updefault}{$Z$}%
}}}}}
\put(-4574,-2986){\makebox(0,0)[lb]{\smash{{\SetFigFont{10}{12.0}{\rmdefault}{\mddefault}{\updefault}{$Z$}%
}}}}
\put(-1799,-3427){\makebox(0,0)[lb]{\smash{{\SetFigFont{10}{12.0}{\rmdefault}{\mddefault}{\updefault}{$2$}%
}}}}
\put(-574,-3423){\makebox(0,0)[rb]{\smash{{\SetFigFont{10}{12.0}{\rmdefault}{\mddefault}{\updefault}{$1$}%
}}}}
\put(-3596,-407){\makebox(0,0)[lb]{\smash{{\SetFigFont{10}{12.0}{\rmdefault}{\mddefault}{\updefault}{$Z$}%
}}}}
\put(-524,1529){\makebox(0,0)[lb]{\smash{{\SetFigFont{10}{12.0}{\rmdefault}{\mddefault}{\updefault}{$1$}%
}}}}
\put(151,1079){\makebox(0,0)[lb]{\smash{{\SetFigFont{10}{12.0}{\rmdefault}{\mddefault}{\updefault}{$2$}%
}}}}
\put(-2391,1263){\makebox(0,0)[lb]{\smash{{\SetFigFont{10}{12.0}{\rmdefault}{\mddefault}{\updefault}{$Z$}%
}}}}
\end{picture}%
\caption{\label{fig:hat}Inserting an added point, as in Setup \ref{th:setup-extra}.}
\end{centering}
\end{figure}%


If we consider $\qabla + \rho$ instead of $\qabla$, \eqref{eq:rho-operation} implies that
\begin{equation} \label{eq:s-obstruction}
(\qabla + \rho)(dx) - d(\qabla +\rho)(x) = [k,x].
\end{equation}
In words, the failure of the modified connection to commute with the differential is measured by the Kodaira-Spencer class. We will now impose a final condition, but one which should be thought of in an entirely different way from the previous ones. Whereas the discussion so far described the setup for a general class of theories, we now specialize to instances where the Kodaira-Spencer class vanishes, something one shouldn't expect to be true in general.
%

\begin{assumption} \label{th:kill-k}
There is a $\theta \in C^1$ such that 
\begin{equation} \label{eq:theta}
d\theta = k.
\end{equation}
\end{assumption}

Naturally in view of \eqref{eq:s-obstruction}, the connection can then be modified to be compatible with the differential. This modified connection is 
\begin{equation} \label{eq:nabla}
\nabla x = \qabla x + \rho(x) - [\theta, x].
\end{equation}

\begin{proposition} \label{th:nabla-gerstenhaber}
The connection on $H^*$ induced by $\nabla$ is compatible with the Gerstenhaber algebra structure, meaning that \eqref{eq:nabla-derivation} and \eqref{eq:nabla-derivation-2} hold for it.
\end{proposition}

\begin{proof}
For simplicity, we will consider only compatibility with the product (the proof for the bracket is parallel). Start with the family of configurations shown in Figure \ref{fig:hat}. One can ``pull out'' the $Z$ point, in the same way as when defining $\rho$. The resulting three-dimensional parameter space (a pair-of-pants times an interval) has five codimension $1$ boundary faces. The first face corresponds to the original family. Three more boundary faces appear when the additional point moves either towards the boundary of the unit disc, or to those of the discs around the $\zeta_i$; one can arrange that this corresponds exactly to gluing in a copy of the family underlying $\rho$. The final boundary face occurs when we have completely ``pulled out'' the additional point; one can arrange that this corresponds exactly to the family from Figure \ref{fig:gamma}, with the Kodaira-Spencer element glued in at the first point. Together with \eqref{eq:diff-axiom-2}, this says that the following expression is a nullhomotopic map:
\begin{equation} \label{eq:nullhomotopy-product}
\begin{aligned}
& \qabla (x_1 \bullet x_2) - \qabla x_1 \bullet x_2 - x_1 \bullet \qabla x_2 \\
& \quad + \rho(x_1 \bullet x_2) - \rho(x_1) \bullet x_2 - x_1 \bullet \rho(x_2) - \gamma(k,x_1,x_2) \htp 0.
\end{aligned}
\end{equation}
On the other hand, using \eqref{eq:gamma-operation} and Assumption \ref{th:kill-k}, we find a chain homotopy
\begin{equation} \label{eq:gamma-produces-the-homotopy}
\gamma(k,x_1,x_2) \htp [\theta, x_1 \bullet x_2] - [\theta, x_1] \bullet x_2 - x_1 \bullet [\theta,x_2].
\end{equation}
Inserting that into \eqref{eq:nullhomotopy-product} yields the desired result.
\end{proof}

\section{Framed configuration spaces\label{sec:bv}}

The little disc operad can be enlarged by adding framings \cite{getzler94b, salvatore-wahl03}. The aim of this section is to extend the previous discussion to the framed context, and to explain how that gives rise to the phenomenon encountered in \eqref{eq:delta-nabla}. 

\subsection{The BV operator}
The first step is to introduce geometric data which give rise to a version of the framed little disc operad. For symplectic geometers, an appropriate idea to think of would be ``allowing the asymptotic markers to rotate''.

\begin{setup} \label{th:setup-framed}
A framing of a disc configuration is an additional choice of $\tau_i \in S^1 = \bR/\bZ$, also identified with the complex number $\exp(2\pi i \tau_i)$, for each $i \in I$. If one has an added $Z$ point, as in Setup \ref{th:setup-extra}, that point should not be equipped with a framing. Any disc configuration can be considered as a framed one, by setting all $\tau_i= 0$. For the identity configuration, as well as for the $r$-family (Setup \ref{th:setup-extra}), we allow only this trivial choice of framing.

Given two framed configurations, and an input point $i_1$ of the first one, we always rotate the second configuration by $\tau_{i_1}$ before gluing it into the first one. (Graphically, the $\tau_i$ are represented as a tangent arrow at $\zeta_i$. Actually, in all the pictures in Sections \ref{sec:conf}--\ref{sec:bv}, the complex plane is shown rotated, so that the positive real axis points upwards; for compatibility with standard drawing conventions, read this paper while lying on your side.)
\end{setup}
\begin{figure}
\begin{centering}
\begin{picture}(0,0)%
\includegraphics{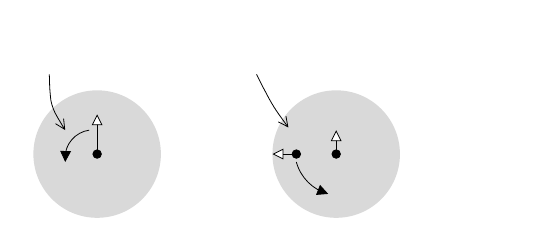}%
\end{picture}%
\setlength{\unitlength}{3355sp}%
\begingroup\makeatletter\ifx\SetFigFont\undefined%
\gdef\SetFigFont#1#2#3#4#5{%
  \reset@font\fontsize{#1}{#2pt}%
  \fontfamily{#3}\fontseries{#4}\fontshape{#5}%
  \selectfont}%
\fi\endgroup%
\begin{picture}(5081,2349)(361,-814)
\put(2551,1364){\makebox(0,0)[lb]{\smash{{\SetFigFont{10}{12.0}{\rmdefault}{\mddefault}{\updefault}{the first point moves, and its}%
}}}}
\put(3601,-61){\makebox(0,0)[lb]{\smash{{\SetFigFont{10}{12.0}{\rmdefault}{\mddefault}{\updefault}{2}%
}}}}
\put(3076,164){\makebox(0,0)[lb]{\smash{{\SetFigFont{10}{12.0}{\rmdefault}{\mddefault}{\updefault}{1}%
}}}}
\put(676,-736){\makebox(0,0)[lb]{\smash{{\SetFigFont{10}{12.0}{\rmdefault}{\mddefault}{\updefault}{BV operator}%
}}}}
\put(3001,-736){\makebox(0,0)[lb]{\smash{{\SetFigFont{10}{12.0}{\rmdefault}{\mddefault}{\updefault}{bracket $[\cdot,\cdot]^{-1}$}%
}}}}
\put(2551,914){\makebox(0,0)[lb]{\smash{{\SetFigFont{10}{12.0}{\rmdefault}{\mddefault}{\updefault}{away from the second point}%
}}}}
\put(376,914){\makebox(0,0)[lb]{\smash{{\SetFigFont{10}{12.0}{\rmdefault}{\mddefault}{\updefault}{framing rotates}%
}}}}
\put(2551,1139){\makebox(0,0)[lb]{\smash{{\SetFigFont{10}{12.0}{\rmdefault}{\mddefault}{\updefault}{framing rotates, always pointing}%
}}}}
\end{picture}%
\caption{\label{fig:bv}The BV operator \eqref{eq:bv-map} and the bracket \eqref{eq:bracket-1}.}
\end{centering}
\end{figure}
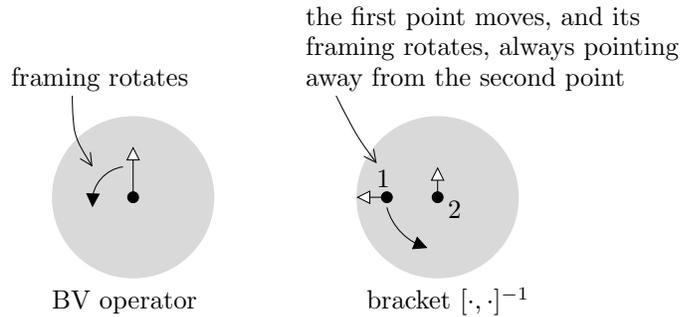%

Suppose that we have a graded $R$-module $H^*$ as before, with operations \eqref{eq:basic-operation} induced by families of framed configurations (also allowing ones with a $Z$ point). In addition to the previous \eqref{eq:unit}--\eqref{eq:bracket} and \eqref{eq:k-class}, we now have the BV operator (Figure \ref{fig:bv})
\begin{equation} \stareq \label{eq:bv-map}
\Delta: H^* \longrightarrow H^{*-1},
\end{equation}
which satisfies
\begin{align}
\stareq \label{eq:delta-e}
& \Delta e = 0, \\
\stareq \label{eq:delta-k}
& \Delta k = 0, \\
\stareq \label{eq:delta-squared}
& \Delta \Delta x = 0, \\
\stareq \label{eq:delta-bracket}
& [x_1,x_2] = \Delta(x_1 \bullet x_2) - (\Delta x_1) \bullet x_2 - (-1)^{|x_1|} x_1 \bullet \Delta x_2, \\
& \stareq \label{eq:delta-bracket-2}
\Delta [x_1,x_2] + [\Delta x_1,x_2] + (-1)^{|x_1} [x_1,\Delta x_2] = 0.
\end{align}
We will find it useful to consider a modified version of the bracket,
\begin{equation} \label{eq:bracket-1} \stareq
[x_1,x_2]^{-1} = [x_1,x_2] + (\Delta x_1) \bullet x_2 = \Delta(x_1 \bullet x_2) - (-1)^{|x_1|} x_1 \bullet \Delta x_2.
\end{equation}
This can be defined directly as shown in Figure \ref{fig:bv}. It is no longer graded commutative, but it commutes with the BV operator for fixed $x_1$:
\begin{equation} \label{eq:bracket-bv-1}
\stareq
[x_1,\Delta x_2]^{-1} = -(-1)^{|x_1|} \Delta [x_1,x_2]^{-1}.
\end{equation}
Because of \eqref{eq:delta-k}, one can write \eqref{eq:r-relation} equivalently as
\begin{equation} \label{eq:r-k-1}
\stareq
r(x) = [k,x]^{-1}.
\end{equation}

\subsection{Chain level operations}
One adds framings to the chain level story in the same way as before. Concerning the chain level versions of the relations \eqref{eq:delta-e}--\eqref{eq:delta-bracket-2}, consider for instance \eqref{eq:delta-squared}. If we glue together two copies of the family underlying $\Delta$, the outcome is a family over the two-torus, with parameters $\tau_1,\tau_2 \in S^1$, but which actually depends only on $\tau_1+\tau_2$. Because of this, the family extends over the solid torus, and such an extension gives rise to a secondary operation
\begin{equation} \label{eq:delta-2}
\begin{aligned}
& \delta: C^* \longrightarrow C^{*-3}, \\
& d\delta x + \delta d x + \Delta \Delta x = 0.
\end{aligned}
\end{equation}

As another example of chain level relations, take \eqref{eq:delta-k}. If we take one of the framed configurations in the family defining $\Delta$, and the configuration defining $k$, and glue them together, the outcome is always the same. Hence, the outcome of the gluing process is a family over a circle which bounds one over the disc, and that yields a cochain
\begin{equation} \label{eq:kappa-class}
\begin{aligned}
& \kappa \in C^0, \\
& d\kappa + \Delta k = 0.
\end{aligned}
\end{equation}

\begin{lemma} \label{th:delta-kappa-delta-kappa}
$\Delta \kappa + \delta k \in C^{-1}$ is nullhomologous.
\end{lemma}

To prove that, one considers the families defining $\Delta \kappa$ and $\delta k$, which are both parametrized by $S^1 \times D^2$, and have common behaviour over the boundary. One constructs a family over $D^2 \times D^2$ whose restriction to the two boundary faces is given by those two families. This is fairly straightforward, and we will not discuss it further here.

We use the family of framed configurations from Figure \ref{fig:bv} to define a chain level representative of $[\cdot,\cdot]^{-1}$. Given this, as well as the standard chain level representative for $r$ from Figure \ref{fig:tqft2}, the equality \eqref{eq:r-k-1} should be replaced by a chain homotopy 
\begin{equation} \label{eq:rho-1-operation}
\begin{aligned}
& \rho^{-1}: C^* \longrightarrow C^*, \\
& d\rho^{-1}(x) - \rho^{-1}(dx) + [k,x]^{-1} - r(x) = 0.
\end{aligned}
\end{equation}
This is parallel to \eqref{eq:rho-operation}, and is defined by a similar family over $P = [0,1] \times S^1$ (Figure \ref{fig:rho-1}). At this point, it may make sense to briefly revisit the uniqueness question for such secondary operations. Let $\scrP_{m;q}^{\mathit{fr}}$ be the universal space of framed disc configurations with $|I| = m$ and $|Z| = q$. This is homotopy equivalent to the ordered configuration space of $m+q$ points times $(S^1)^m$. In principle, to find a homotopy $\rho^{-1}$ satisfying \eqref{eq:rho-1-operation}, one chooses a map from a surface $P$ with two boundary components to $\scrP^{\mathit{fr}}_{1;1}$, with fixed behaviour near the boundary. Since $\scrP^{\mathit{fr}}_{1;1} \htp (S^1)^2$, there are essentially different choices (up to bordism or homology) of such maps; the corresponding operations differ (up to chain homotopy) by multiples of $r(\Delta(x)) \htp - \Delta(r(x))$. However, in our construction (Figure \ref{fig:rho-1}), we impose the condition that the framing marker at the input point should not move. This narrows down the choice to a subspace of $\scrP^{\mathit{fr}}_{1;1}$ homotopy equivalent to $S^1$, which singles out a homotopically unique operation \eqref{eq:rho-1-operation}.

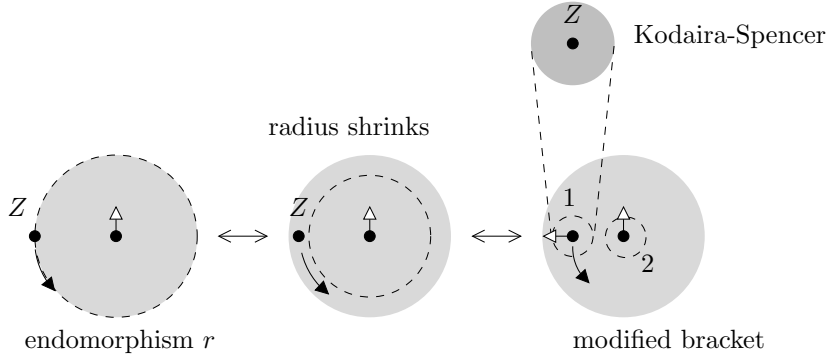
\begin{figure}
\begin{centering}
\begin{picture}(0,0)%
\includegraphics{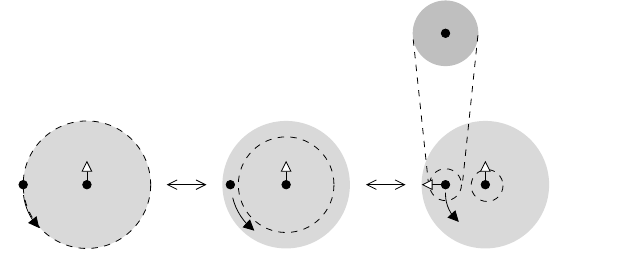}%
\end{picture}%
\setlength{\unitlength}{3355sp}%
\begingroup\makeatletter\ifx\SetFigFont\undefined%
\gdef\SetFigFont#1#2#3#4#5{%
  \reset@font\fontsize{#1}{#2pt}%
  \fontfamily{#3}\fontseries{#4}\fontshape{#5}%
  \selectfont}%
\fi\endgroup%
\begin{picture}(6043,2623)(-3743,-800)
\put(-3728,236){\makebox(0,0)[lb]{\smash{{\SetFigFont{10}{12.0}{\rmdefault}{\mddefault}{\updefault}{\color[rgb]{0,0,0}$Z$}%
}}}}
\put(-3599,-736){\makebox(0,0)[lb]{\smash{{\SetFigFont{10}{12.0}{\rmdefault}{\mddefault}{\updefault}{\color[rgb]{0,0,0}endomorphism $r$}%
}}}}
\put(-1649,239){\makebox(0,0)[lb]{\smash{{\SetFigFont{10}{12.0}{\rmdefault}{\mddefault}{\updefault}{\color[rgb]{0,0,0}$Z$}%
}}}}
\put(-1799,839){\makebox(0,0)[lb]{\smash{{\SetFigFont{10}{12.0}{\rmdefault}{\mddefault}{\updefault}{\color[rgb]{0,0,0}radius shrinks}%
}}}}
\put(451,-736){\makebox(0,0)[lb]{\smash{{\SetFigFont{10}{12.0}{\rmdefault}{\mddefault}{\updefault}{\color[rgb]{0,0,0}modified bracket}%
}}}}
\put(901,1514){\makebox(0,0)[lb]{\smash{{\SetFigFont{10}{12.0}{\rmdefault}{\mddefault}{\updefault}{\color[rgb]{0,0,0}Kodaira-Spencer}%
}}}}
\put(376,1664){\makebox(0,0)[lb]{\smash{{\SetFigFont{10}{12.0}{\rmdefault}{\mddefault}{\updefault}{\color[rgb]{0,0,0}$Z$}%
}}}}
\put(956,-173){\makebox(0,0)[lb]{\smash{{\SetFigFont{10}{12.0}{\rmdefault}{\mddefault}{\updefault}{\color[rgb]{0,0,0}$2$}%
}}}}
\put(376,314){\makebox(0,0)[lb]{\smash{{\SetFigFont{10}{12.0}{\rmdefault}{\mddefault}{\updefault}{\color[rgb]{0,0,0}$1$}%
}}}}
\end{picture}%
\caption{\label{fig:rho-1}The operation $\rho^{-1}$ from \eqref{eq:rho-1-operation}.}
\end{centering}
\end{figure}

Another property of $[\cdot,\cdot]^{-1}$ is \eqref{eq:bracket-bv-1}. The chain level version involves a homotopy
\begin{equation} \label{eq:chi-operation}
\begin{aligned}
& \varpi: C^* \otimes C^* \longrightarrow C^{*-3}, \\
& d\varpi(x_1,x_2) + \varpi(dx_1,x_2) + (-1)^{|x_1|} \varpi(x_1,dx_2) \\ & \qquad \qquad + \Delta[x_1,x_2]^{-1} + (-1)^{|x_1|} [x_1,\Delta x_2]^{-1} = 0.
\end{aligned}
\end{equation}
Both the bracket $[\cdot,\cdot]^{-1}$ and the BV operator are defined by families of framed configurations over $S^1$. Let's call their parameters $\tau_1$ (for the BV operator) and $\tau_2$ (for the bracket). The gluing process corresponding to $-\Delta[x_1,x_2]^{-1}$ results in a family over $(S^1)^2$. Changing parameters to $\tilde\tau_1 = \tau_1+ \tau_2$, $\tilde\tau_2 = \tau_1$ yields the corresponding glued family associated to $(-1)^{|x_1|+1} [x_1,\Delta x_2]^{-1}$ (note that the parameter change is orientation-reversing). The outcome of this consideration is a family over $[0,1] \times (S^1)^2$, shown schematically in Figure \ref{fig:chi}; in this case, homotopical uniqueness is obtained by the prescription that the framing marker at the first point should always point away from the second point.
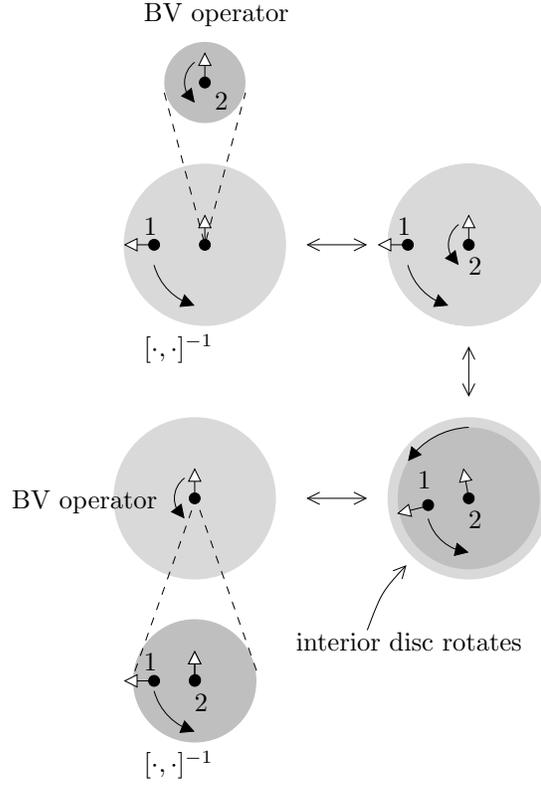
\begin{figure}
\begin{centering}
\begin{picture}(0,0)%
\includegraphics{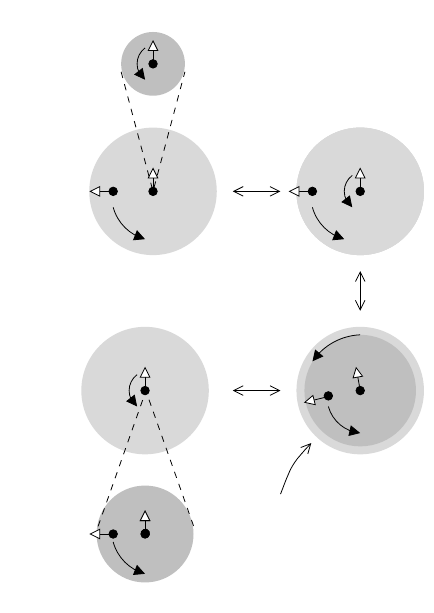}%
\end{picture}%
\setlength{\unitlength}{3355sp}%
\begingroup\makeatletter\ifx\SetFigFont\undefined%
\gdef\SetFigFont#1#2#3#4#5{%
  \reset@font\fontsize{#1}{#2pt}%
  \fontfamily{#3}\fontseries{#4}\fontshape{#5}%
  \selectfont}%
\fi\endgroup%
\begin{picture}(4068,5761)(-389,-3875)
\put(601,1739){\makebox(0,0)[lb]{\smash{{\SetFigFont{10}{12.0}{\rmdefault}{\mddefault}{\updefault}{\color[rgb]{0,0,0}BV operator}%
}}}}
\put(601,-736){\makebox(0,0)[lb]{\smash{{\SetFigFont{10}{12.0}{\rmdefault}{\mddefault}{\updefault}{\color[rgb]{0,0,0}$[\cdot,\cdot]^{-1}$}%
}}}}
\put(1726,-2911){\makebox(0,0)[lb]{\smash{{\SetFigFont{10}{12.0}{\rmdefault}{\mddefault}{\updefault}{\color[rgb]{0,0,0}interior disc rotates}%
}}}}
\put(601,-3811){\makebox(0,0)[lb]{\smash{{\SetFigFont{10}{12.0}{\rmdefault}{\mddefault}{\updefault}{\color[rgb]{0,0,0}$[\cdot,\cdot]^{-1}$}%
}}}}
\put(1126,1089){\makebox(0,0)[lb]{\smash{{\SetFigFont{10}{12.0}{\rmdefault}{\mddefault}{\updefault}{\color[rgb]{0,0,0}$2$}%
}}}}
\put(2476,164){\makebox(0,0)[lb]{\smash{{\SetFigFont{10}{12.0}{\rmdefault}{\mddefault}{\updefault}{\color[rgb]{0,0,0}$1$}%
}}}}
\put(3001,-136){\makebox(0,0)[lb]{\smash{{\SetFigFont{10}{12.0}{\rmdefault}{\mddefault}{\updefault}{\color[rgb]{0,0,0}$2$}%
}}}}
\put(3001,-2011){\makebox(0,0)[lb]{\smash{{\SetFigFont{10}{12.0}{\rmdefault}{\mddefault}{\updefault}{\color[rgb]{0,0,0}$2$}%
}}}}
\put(2626,-1711){\makebox(0,0)[lb]{\smash{{\SetFigFont{10}{12.0}{\rmdefault}{\mddefault}{\updefault}{\color[rgb]{0,0,0}$1$}%
}}}}
\put(976,-3361){\makebox(0,0)[lb]{\smash{{\SetFigFont{10}{12.0}{\rmdefault}{\mddefault}{\updefault}{\color[rgb]{0,0,0}$2$}%
}}}}
\put(601,-3061){\makebox(0,0)[lb]{\smash{{\SetFigFont{10}{12.0}{\rmdefault}{\mddefault}{\updefault}{\color[rgb]{0,0,0}$1$}%
}}}}
\put(601,164){\makebox(0,0)[lb]{\smash{{\SetFigFont{10}{12.0}{\rmdefault}{\mddefault}{\updefault}{\color[rgb]{0,0,0}$1$}%
}}}}
\put(-374,-1861){\makebox(0,0)[lb]{\smash{{\SetFigFont{10}{12.0}{\rmdefault}{\mddefault}{\updefault}{\color[rgb]{0,0,0}BV operator}%
}}}}
\end{picture}%
\caption{\label{fig:chi}The operation $\varpi$ from \eqref{eq:chi-operation}.}
\end{centering}
\end{figure}%

\subsection{The differentiation axiom in the framed context\label{subsec:differentiation-axiom-2}}
We now impose the same differentiation axiom as in Section \ref{subsec:differentiation-axiom},  extended to framed disc configurations. We also impose Assumption \ref{th:kill-k}, but modify the previous construction of connection \eqref{eq:nabla}, taking instead
\begin{equation} \label{eq:nabla-1}
\nabla^{-1} x = \qabla x + \rho^{-1}(x) - [\theta,x]^{-1}.
\end{equation}
This time, to show that $\nabla^{-1}$ is compatible with the differential, one uses \eqref{eq:diff-axiom-1} together with \eqref{eq:rho-1-operation} and \eqref{eq:theta}.

\begin{proposition} \label{th:nabla-bv}
The connection on $H^*$ induced by $\nabla^{-1}$ commutes with the BV operator.
\end{proposition}

\begin{proof} The basic approach is the same as for Proposition \ref{th:nabla-gerstenhaber}. The differentiation axiom describes $\qabla \Delta - \Delta \qabla$ in terms of a three-parameter family of framed disc configurations (with parameter space $[0,1] \times (S^1)^2$). One adds another parameter which ``pulls out'' the additional $Z$ point. The resulting family, drawn schematically in Figure \ref{fig:nabla-delta}, shows that the following expression is a nullhomotopic map:
\begin{equation} \label{eq:qabla-delta}
x \longmapsto \qabla \Delta x - \Delta \qabla x + \rho^{-1} (\Delta x) - \Delta \rho^{-1}(x) - \varpi(k,x).
\end{equation}
Using \eqref{eq:chi-operation} and Assumption \ref{th:kill-k}, we find a homotopy
\begin{equation}
\varpi(k,x) \htp [\theta, \Delta x]^{-1} - \Delta [\theta,x]^{-1}.
\end{equation}
Combining that with \eqref{eq:qabla-delta} yields the desired result.
\end{proof}
\begin{figure}
\begin{centering}
\begin{picture}(0,0)%
\includegraphics{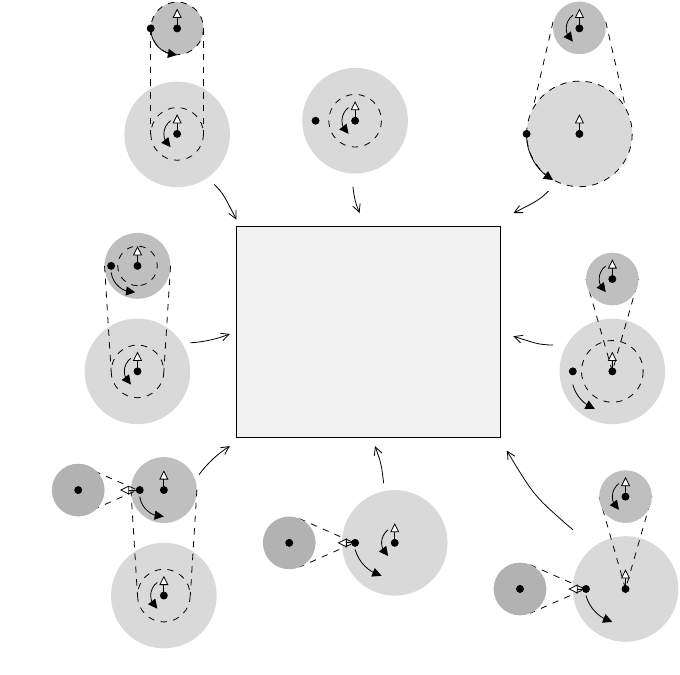}%
\end{picture}%
\setlength{\unitlength}{3355sp}%
\begingroup\makeatletter\ifx\SetFigFont\undefined%
\gdef\SetFigFont#1#2#3#4#5{%
  \reset@font\fontsize{#1}{#2pt}%
  \fontfamily{#3}\fontseries{#4}\fontshape{#5}%
  \selectfont}%
\fi\endgroup%
\begin{picture}(6385,6340)(-239,-5031)
\put(4532, -1){\makebox(0,0)[lb]{\smash{{\SetFigFont{10}{12.0}{\rmdefault}{\mddefault}{\updefault}{\color[rgb]{0,0,0}$Z$}%
}}}}
\put(2670,805){\makebox(0,0)[lb]{\smash{{\SetFigFont{10}{12.0}{\rmdefault}{\mddefault}{\updefault}{\color[rgb]{0,0,0}$\qabla\Delta-\Delta\qabla$}%
}}}}
\put(2732,309){\makebox(0,0)[lb]{\smash{{\SetFigFont{10}{12.0}{\rmdefault}{\mddefault}{\updefault}{\color[rgb]{0,0,0}$Z$}%
}}}}
\put(5525,-560){\makebox(0,0)[lb]{\smash{{\SetFigFont{10}{12.0}{\rmdefault}{\mddefault}{\updefault}{\color[rgb]{0,0,0}$r(\Delta(\cdot))$}%
}}}}
\put(5215,-2857){\makebox(0,0)[lb]{\smash{{\SetFigFont{10}{12.0}{\rmdefault}{\mddefault}{\updefault}{\color[rgb]{0,0,0}$\rho^{-1}(\Delta(\cdot))$}%
}}}}
\put(5277,-4905){\makebox(0,0)[lb]{\smash{{\SetFigFont{10}{12.0}{\rmdefault}{\mddefault}{\updefault}{\color[rgb]{0,0,0}$[k,\Delta(\cdot)]^{-1}$}%
}}}}
\put(5251,-4111){\makebox(0,0)[lb]{\smash{{\SetFigFont{10}{12.0}{\rmdefault}{\mddefault}{\updefault}{\color[rgb]{0,0,0}$1$}%
}}}}
\put(5701,-4411){\makebox(0,0)[lb]{\smash{{\SetFigFont{10}{12.0}{\rmdefault}{\mddefault}{\updefault}{\color[rgb]{0,0,0}$2$}%
}}}}
\put(1367,-3415){\makebox(0,0)[lb]{\smash{{\SetFigFont{9}{10.8}{\rmdefault}{\mddefault}{\updefault}{\color[rgb]{0,0,0}$2$}%
}}}}
\put(3076,-3661){\makebox(0,0)[lb]{\smash{{\SetFigFont{10}{12.0}{\rmdefault}{\mddefault}{\updefault}{\color[rgb]{0,0,0}$1$}%
}}}}
\put(976,-3211){\makebox(0,0)[lb]{\smash{{\SetFigFont{10}{12.0}{\rmdefault}{\mddefault}{\updefault}{\color[rgb]{0,0,0}$1$}%
}}}}
\put(622,-1119){\makebox(0,0)[lb]{\smash{{\SetFigFont{10}{12.0}{\rmdefault}{\mddefault}{\updefault}{\color[rgb]{0,0,0}$Z$}%
}}}}
\put(1180,-622){\makebox(0,0)[lb]{\smash{{\SetFigFont{10}{12.0}{\rmdefault}{\mddefault}{\updefault}{\color[rgb]{0,0,0}$\Delta(r(\cdot))$}%
}}}}
\put(435,-4967){\makebox(0,0)[lb]{\smash{{\SetFigFont{10}{12.0}{\rmdefault}{\mddefault}{\updefault}{\color[rgb]{0,0,0}$\Delta [k,\cdot]^{-1}$}%
}}}}
\put(2980,-4470){\makebox(0,0)[lb]{\smash{{\SetFigFont{10}{12.0}{\rmdefault}{\mddefault}{\updefault}{\color[rgb]{0,0,0}$\varpi(k,\cdot)$}%
}}}}
\put(994,992){\makebox(0,0)[lb]{\smash{{\SetFigFont{10}{12.0}{\rmdefault}{\mddefault}{\updefault}{\color[rgb]{0,0,0}$Z$}%
}}}}
\put(5101,-2086){\makebox(0,0)[lb]{\smash{{\SetFigFont{10}{12.0}{\rmdefault}{\mddefault}{\updefault}{\color[rgb]{0,0,0}$Z$}%
}}}}
\put(4594,-4098){\makebox(0,0)[lb]{\smash{{\SetFigFont{10}{12.0}{\rmdefault}{\mddefault}{\updefault}{\color[rgb]{0,0,0}$Z$}%
}}}}
\put(2422,-3974){\makebox(0,0)[lb]{\smash{{\SetFigFont{10}{12.0}{\rmdefault}{\mddefault}{\updefault}{\color[rgb]{0,0,0}$Z$}%
}}}}
\put(3539,-3912){\makebox(0,0)[lb]{\smash{{\SetFigFont{10}{12.0}{\rmdefault}{\mddefault}{\updefault}{\color[rgb]{0,0,0}$2$}%
}}}}
\put(498,-3477){\makebox(0,0)[lb]{\smash{{\SetFigFont{10}{12.0}{\rmdefault}{\mddefault}{\updefault}{\color[rgb]{0,0,0}$Z$}%
}}}}
\put(-224,-1786){\makebox(0,0)[lb]{\smash{{\SetFigFont{10}{12.0}{\rmdefault}{\mddefault}{\updefault}{\color[rgb]{0,0,0}$\Delta(\rho^{-1}(\cdot))$}%
}}}}
\end{picture}%
\caption{\label{fig:nabla-delta}The nullhomotopy \eqref{eq:qabla-delta}.}
\end{centering}
\end{figure}
%

One can associate to our choice of bounding cochain $\theta$ a cocycle
\begin{equation} \label{eq:a-cocycle}
a = \Delta \theta - \kappa \in C^0,
\end{equation}
where $\kappa$ is the cochain from \eqref{eq:kappa-class}. Using Lemma \ref{th:delta-kappa-delta-kappa}, ones sees that
\begin{equation}
\Delta a = \Delta \Delta \theta -  \Delta \kappa = - \delta(d\theta) + \delta k + \text{\it coboundary} =
\text{\it coboundary}.
\end{equation}
Hence, the cohomology class of $a$ (for which we use the same notation) is annihilated by the BV operator.

\begin{lemma} \label{th:01-connections}
$\nabla - \nabla^{-1}$ is homotopy equivalent to the pair-of-pants product with $a$.
\end{lemma}

\begin{proof}
There is a chain homotopy underlying the first of the cohomology level equalities in \eqref{eq:bracket-1}:
\begin{equation} \label{eq:epsilon-homotopy}
\begin{aligned}
& \eta: C^* \otimes C^* \longrightarrow C^{*-2}, \\
& d\eta(x_1,x_2) - \eta(dx_1,x_2) - (-1)^{|x_1|} \eta(x_1,dx_2) \\ & \qquad \qquad + [x_1,x_2] - [x_1,x_2]^{-1} + (\Delta x_1) \bullet x_2 = 0.
\end{aligned}
\end{equation}
This comes with a secondary nullhomotopy
\begin{equation} \label{eq:secondary-epsilon}
\eta(k,x) - \rho(x) + \rho^{-1}(x) - \kappa \bullet x \htp 0.
\end{equation}
The family behind \eqref{eq:epsilon-homotopy} is shown in Figure \ref{fig:epsilon}. Note that \eqref{eq:epsilon-homotopy}, together with the properties of $\rho$, $\rho^{-1}$ and $\kappa$, ensures that the left hand side of \eqref{eq:secondary-epsilon} is a chain map. The corresponding geometric fact is that the underlying families can be combined to form a family over a closed surface (a torus), as shown in Figure \ref{fig:epsilon2}. To obtain \eqref{eq:secondary-epsilon}, one has to show that this family bounds in $\scrP_{1;1}^{\mathit{fr}} \htp S^1 \times S^1$. For that, it is sufficient to notice that all disc configurations in the family have the property that the framing marker (at the unique $I$ point) points in the same direction.
\begin{figure}
\begin{centering}
\begin{picture}(0,0)%
\includegraphics{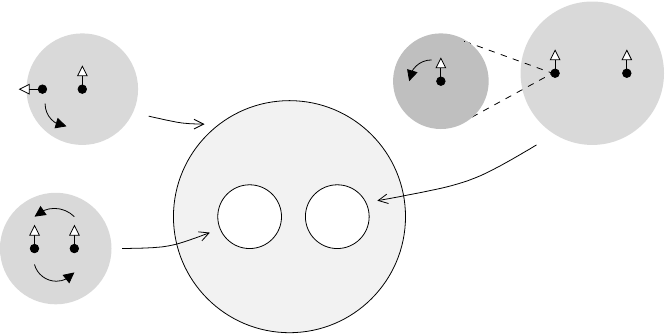}%
\end{picture}%
\setlength{\unitlength}{3355sp}%
\begingroup\makeatletter\ifx\SetFigFont\undefined%
\gdef\SetFigFont#1#2#3#4#5{%
  \reset@font\fontsize{#1}{#2pt}%
  \fontfamily{#3}\fontseries{#4}\fontshape{#5}%
  \selectfont}%
\fi\endgroup%
\begin{picture}(6250,3124)(-4074,-2060)
\put(1876, 89){\makebox(0,0)[rb]{\smash{{\SetFigFont{10}{12.0}{\rmdefault}{\mddefault}{\updefault}{$2$}%
}}}}
\put(-3224, 89){\makebox(0,0)[lb]{\smash{{\SetFigFont{10}{12.0}{\rmdefault}{\mddefault}{\updefault}{$2$}%
}}}}
\put(-3224,-1336){\makebox(0,0)[lb]{\smash{{\SetFigFont{10}{12.0}{\rmdefault}{\mddefault}{\updefault}{$2$}%
}}}}
\put(  1, 53){\makebox(0,0)[lb]{\smash{{\SetFigFont{10}{12.0}{\rmdefault}{\mddefault}{\updefault}{$1$}%
}}}}
\put(-3974,-1186){\makebox(0,0)[lb]{\smash{{\SetFigFont{10}{12.0}{\rmdefault}{\mddefault}{\updefault}{$1$}%
}}}}
\put(-3674,389){\makebox(0,0)[lb]{\smash{{\SetFigFont{10}{12.0}{\rmdefault}{\mddefault}{\updefault}{$1$}%
}}}}
\end{picture}%
\caption{\label{fig:epsilon}The operation $\eta$ from \eqref{eq:epsilon-homotopy}.}
\end{centering}
\end{figure}
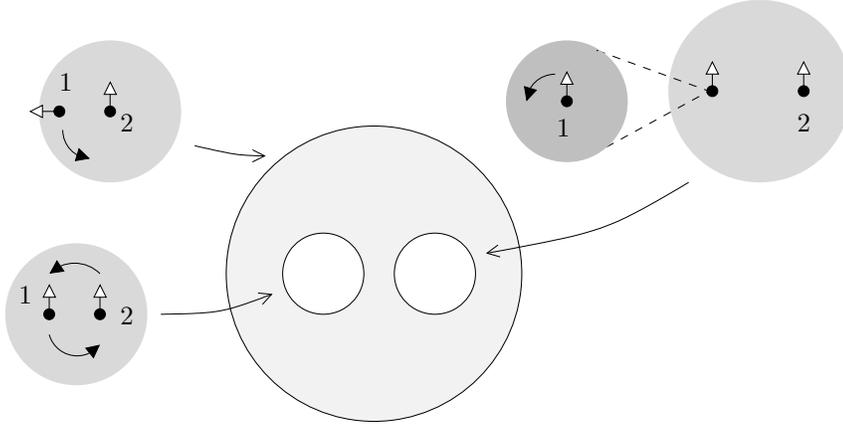
\begin{figure}
\begin{centering}
\begin{picture}(0,0)%
\includegraphics{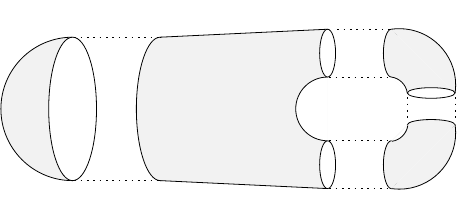}%
\end{picture}%
\setlength{\unitlength}{3355sp}%
\begingroup\makeatletter\ifx\SetFigFont\undefined%
\gdef\SetFigFont#1#2#3#4#5{%
  \reset@font\fontsize{#1}{#2pt}%
  \fontfamily{#3}\fontseries{#4}\fontshape{#5}%
  \selectfont}%
\fi\endgroup%
\begin{picture}(4298,2061)(-3533,-4339)
\put(-1499,-4261){\makebox(0,0)[lb]{\smash{{\SetFigFont{10}{12.0}{\rmdefault}{\mddefault}{\updefault}{$\eta(k,\cdot)$}%
}}}}
\put(376,-2461){\makebox(0,0)[lb]{\smash{{\SetFigFont{10}{12.0}{\rmdefault}{\mddefault}{\updefault}{$\rho^{-1}$}%
}}}}
\put(301,-4261){\makebox(0,0)[lb]{\smash{{\SetFigFont{10}{12.0}{\rmdefault}{\mddefault}{\updefault}{$\rho$}%
}}}}
\put(-3449,-4261){\makebox(0,0)[lb]{\smash{{\SetFigFont{10}{12.0}{\rmdefault}{\mddefault}{\updefault}{$\kappa \bullet \cdot$}%
}}}}
\end{picture}%
\caption{\label{fig:epsilon2}The parameter spaces appearing on the left in \eqref{eq:secondary-epsilon}.}
\end{centering}
\end{figure}
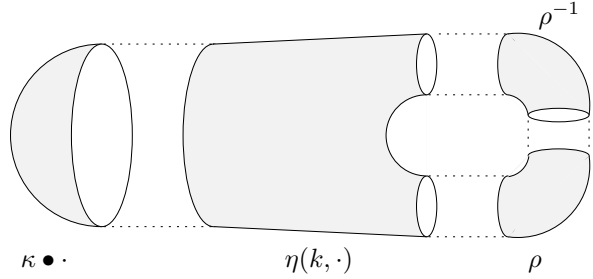%
Now let's see what this says about connections. By definition and \eqref{eq:epsilon-homotopy},
\begin{equation} \label{eq:nabla-vs-nabla-1}
\begin{aligned}
\nabla x - \nabla^{-1} x & = \rho(x) - \rho^{-1}(x) - [\theta,x] + [\theta,x]^{-1} \\ & \htp
\rho(x) - \rho^{-1}(x) + (\Delta \theta) \bullet x - \eta(k,x). 
\end{aligned}
\end{equation}
Then, applying \eqref{eq:secondary-epsilon} achieves the desired result.
\end{proof}

Combining Lemma \ref{th:01-connections} and Proposition \ref{th:nabla-bv} yields the following:

\begin{corollary} \label{th:nabla-delta-cor}
The connection on $H^*$ induced by $\nabla$ satisfies \eqref{eq:delta-nabla}.
\end{corollary}

\section{Floer cohomology\label{sec:floer}}
In contrast with the formal preoccupations of the preceding sections, the next few go straight for the jugular, meaning the relation between enumerative geometry and connections on Floer cohomology. 
The present section sets up the basic Floer-theoretic machinery needed for that purpose.

\subsection{The definition\label{subsec:define-floer}} 
We will work with a class of symplectic manifolds which includes the Lefschetz fibrations from Section \ref{subsec:define-lefschetz} (more precisely, it includes them after possibly deforming the symplectic form). Leaving aside technical details, the key condition is that the first Chern class of our manifold is represented by a symplectic hypersurface with trivial normal bundle.

\begin{setup} \label{th:define-f}
Let $F$ be a closed connected symplectic manifold, and $M \subset F$ a connected codimension two symplectic submanifold. We require that there is a map 
\begin{equation}
F \stackrel{p}{\longrightarrow} \bC P^1 = \bC \cup \{\infty\},
\end{equation}
which has $M = F_0 = p^{-1}(0)$ as a regular fibre, and which is a trivial symplectic fibration over some closed disc (centered at the origin) $B \subset \bC \subset \bC P^1$. This means that there is a (necessarily unique) diffeomorphism
\begin{equation} \label{eq:trivialization-outer}
\xymatrix{
\ar[dr]_{\text{projection}} B \times M\, \ar[rr]^-{\iso} && p^{-1}(B) \ar[dl]^{p} \\ & B
}
\end{equation}
which is symplectic (with respect to the product of $\omega_M = \omega_F|M$ and some rotationally invariant positive $\omega_B$), and which restricts to the inclusion on $\{0\} \times M$.

We require that $c_1(F)$ is Poincar{\'e} dual to $F_0$, and choose a trivialization of the canonical bundle with a pole along that fibre. This means that, for some compatible almost complex structure on $F$, we have a complex volume form $\eta$ on the complement 
\begin{equation} \label{eq:define-e}
E = F \setminus F_0 \stackrel{p}{\longrightarrow} \bC P^1 \setminus \{0\} = \bC^* \cup \{\infty\},
\end{equation}
such that $p \eta$ extends to a complex volume form near $F_0$. 

We also want to fix a cycle representative for $[\omega_F]$. More precisely, we choose (not necessarily symplectic) codimension two submanifolds $\Omega_1,\dots,\Omega_k \subset F$, which are transverse to $F_0$, along with multiplicities $\mu_1,\dots,\mu_k \in \bR$, such that $\Omega = \mu_1 \Omega_1 + \cdots \mu_k \Omega_k$ satisfies
\begin{equation} \label{eq:z-divisor}
[\omega_F] = [\Omega] = \mu_1[\Omega_1] + \cdots + \mu_k[\Omega_k] \in H^2(F;\bR).
\end{equation}
Additionally, we choose a current $\Theta$ on $F$, which is regular (represented by a smooth one-form) outside $\Omega$, and such that
\begin{equation} \label{eq:theta-z}
d\Theta = \omega_F - \Omega.
\end{equation}
Then, if $S$ is a compact oriented surface with boundary, and $u: S \rightarrow F$ a map with $u(\partial S) \cap \Omega = \emptyset$, we can compute its symplectic area by
\begin{equation}
\int_S u^*\omega_F = u \cdot \Omega + \int_{\partial S} u^*\Theta,
\end{equation}
where $u \cdot \Omega = \mu_1 (u \cdot \Omega_1) + \cdots + \mu_k (u \cdot \Omega_k)$.
\end{setup}

The auxiliary geometric data used to define Floer cohomology will be specifically chosen to be compatible with the structure of $F$, and in particular with \eqref{eq:trivialization-outer}.

\begin{setup}
We will use compatible almost complex structures $J$ on $F$ such that $p$ is $J$-holomorphic over $B$. This means that in the partial trivialization \eqref{eq:trivialization-outer}, 
\begin{equation} \label{eq:j}
J|(\{b\} \times M) = i \oplus J_{M,b} \in \mathit{End}(\bC \oplus TM)
\end{equation}
for some family $(J_{M,b})_{b \in B}$ of compatible almost complex structures on $M$. As a consequence of this and the assumption on $c_1(F)$, there are no $J$-holomorphic curves with negative Chern number. If one defines $\eta$ using such an almost complex structure, its restriction to \eqref{eq:trivialization-outer} is
\begin{equation} \label{eq:eta-j-split}
\eta = b^{-1}db \wedge \eta_{M,b}
\end{equation}
for some family $(\eta_{M,b})$, $b \in B$, of $J_{M,b}$-complex volume forms on the fibre.

Similarly, we consider Hamiltonian functions $H \in \smooth(F,\bR)$ whose associated vector field $X$ satisfies, for some $\alpha \in \bR \setminus \bZ$,
\begin{equation} \label{eq:hamiltonian-0}
Dp_x(X) = -\alpha \, 2\pi i b\partial_b \quad \text{for all $x \in p^{-1}(B)$.}
\end{equation}
($2\pi i b\partial_b$ is the standard rotational vector field on the complex plane, normalized so that its flow is one-periodic.) Equivalently, there is a Hamiltonian vector field $X_M$ on the fibre, such that
\begin{equation} \label{eq:hamiltonian}
X = -\alpha \, 2\pi i b\partial_b + X_M \quad \text{on $p^{-1}(B) \iso B \times M$.}
\end{equation}
\end{setup}

Fix $\alpha$, and choose a time-dependent Hamiltonian function $H = (H_t)$, $t \in S^1 = \bR/\bZ$, satisfying \eqref{eq:hamiltonian}. Each one-periodic orbit of the associated vector field $X = (X_{t})$ either lies entirely inside or entirely outside $F_0$; in the second case, such orbits in fact lie outside $p^{-1}(B)$, thanks to the assumption on $\alpha$. From now on, suppose that the one-periodic orbits are nondegenerate, and disjoint from $\Omega$ (this is true for generic $H$). To any orbit $x$, one can associate two quantities. The first one is the action, defined in terms of \eqref{eq:theta-z} as
\begin{equation}
A(x) = \int_{S^1} -x^*\Theta + H_t(x(t)) \, \mathit{dt}.
\end{equation}
The second one is the Conley-Zehnder index
\begin{equation}
i(x) \in \bZ. 
\end{equation}
There are actually two cases to be considered here: if $x$ lies outside $F_0$, we define $i(x)$ using $\eta$ (see e.g.\ \cite{seidel99}); and if $x$ is contained in $F_0$, we define it using $p \eta$ instead (in terms of  \eqref{eq:eta-j-split}, this would be $\mathit{db} \wedge \eta_{M,b}$). For the second case, we can also define an index $i_M(x)$ by working entirely inside $M = F_0$ (and using $\eta_{M,0})$. The two indices are related by
\begin{equation} \label{eq:two-indices}
i(x) = i_M(x) + (2 \lfloor \alpha \rfloor + 2).
\end{equation}

\begin{remark} \label{th:move-pole}
Equivalently, one can work with
\begin{equation} \label{eq:eta-e-b}
\frac{1}{1 - b/p} \eta = \frac{1}{b(p/b-1)} p\eta \quad \text{for some } b \in B \setminus \{0\},
\end{equation}
which is a complex volume form on $F \setminus F_b$. This yields the same index for one-periodic orbits lying outside $F_0$, since one can deform $\eta$ to \eqref{eq:eta-e-b} by moving the pole from $F_0$ to $F_b$. The same holds if $x$ is contained in $F_0$, since one can deform $(-1/b) p\eta$ to \eqref{eq:eta-e-b} by moving the pole from $F_\infty$ to $F_b$. The advantage of \eqref{eq:eta-e-b} is that the pole is now disjoint from all one-periodic orbits.
\end{remark}

Choose a family $J = (J_t)$, $t \in S^1$, of almost complex structures satisfying \eqref{eq:j}. The Hamiltonian Floer equation is
\begin{equation} \label{eq:floer}
\left\{
\begin{aligned}
& u: \bR \times S^1 \longrightarrow F, \\
& \partial_s u + J_t(u)(\partial_t u - X_{t}) = 0, \\
& \textstyle\lim_{s \rightarrow \pm \infty} u(s,\cdot) = x^{\pm},
\end{aligned}
\right.
\end{equation}
where $x^\pm$ are one-periodic orbits. We will consider three quantities associated to each solution $u$. The first is a kind of degree with respect to the map $p$,
\begin{equation} \label{eq:define-degree}
\mathrm{deg}(u) = u \cdot F_b \quad \text{for $b \in B \setminus \{0\}$}.
\end{equation}
The second, closely related, one is the Fredholm index of the associated linearized operator $D_u$. By a standard index formula (which becomes particularly clear from the perspective of Remark \ref{th:move-pole}), this can be written as
\begin{equation} \label{eq:index}
\mathrm{index}(D_u) = i(x^-) - i(x^+) + 2 \,\mathrm{deg}(u).
\end{equation}
The third one is the energy 
\begin{equation} \label{eq:energy}
\begin{aligned}
E(u) & = \int_{\bR \times S^1} \|\partial_s u\|^2 \, \mathit{ds} \wedge \mathit{dt} 
= \int_{\bR \times S^1} u^*\omega_F - dH_t(\partial_s u) \mathit{ds} \wedge \mathit{dt} \\
& = A(x^-) - A(x^+) + u \cdot \Omega.
\end{aligned}
\end{equation}

%
The quantities \eqref{eq:define-degree}--\eqref{eq:energy} are topological, by which we mean that they are unchanged under deformations of $u$. In fact, their definition can be extended to maps $u$ which do not satisfy our Cauchy-Riemann equation (but are still asymptotic to one-periodic orbits).

Let's discuss the behaviour of solutions $u$ of \eqref{eq:floer} with respect to \eqref{eq:trivialization-outer}. Write $v = p(u)$. Take the subset of all $(s,t) \in \bR \times S^1$ where $v(s,t) \in B$. There, $v$ obeys a linear equation
\begin{equation} \label{eq:v-equation}
\partial_s v + i\partial_t v - 2\pi \alpha v = 0,
\end{equation}
As a consequence, $(\partial_s^2 + \partial_t^2)(\log v) = 0$ at all points where $v(s,t) \in B \setminus \{0\}$. By applying the maximum (and minimum) principle to $\log |v| = \mathrm{re}(\log v)$, one gets:

\begin{lemma} \label{th:max-principle-1}
If a solution of \eqref{eq:floer} remains in $p^{-1}(B)$, it is in fact contained in $F_0$.
\end{lemma}

\begin{lemma} \label{th:max-principle-2}
Take a solution of \eqref{eq:floer} both of whose limits $x^{\pm}$ lie outside $F_0$. If that solution enters $p^{-1}(B \setminus \partial B)$ anywhere, it must intersect $F_0$.
\end{lemma}

Solutions of \eqref{eq:v-equation} are of the form
\begin{equation} 
\label{eq:v-tilde}
v(s,t) = \exp(2\pi \alpha s)\, \tilde{v}(\exp(2\pi (s+it))) 
\end{equation}
for holomorphic $\tilde{v}$. Consider the special case of \eqref{eq:floer} where $x^-$ lies in $F_0$, and assume that $u$ is not entirely contained in $F_0$. In that case, we get a nontrivial solution of \eqref{eq:v-equation}, defined for all $s \ll 0$. In the limit $s \rightarrow -\infty$,
\begin{multline} \label{eq:m-minus}
v(s,t) \sim \exp(2\pi \alpha s + 2\pi m^- (s+it)) \\ \text{for some integer $m^- = m^-(u) > -\alpha$.}
\end{multline}
Here, $\sim$ means that the quotient of the two functions converges to a nonzero (complex) constant. In terms of \eqref{eq:v-tilde}, $m^-$ is the order of vanishing (or pole order, if it is negative) of the function $\tilde{v}(z)$ at $z = 0$. For the corresponding case where $x^+$ lies in $F_0$, one considers $s \rightarrow +\infty$ and finds that there,
\begin{multline} \label{eq:m-plus}
v(s,t) \sim \exp(2 \pi \alpha s - 2\pi m^+(s+it)) \\ \text{for some integer $m^+ = m^+(u) > \alpha$.}
\end{multline}
Since $m^\pm$ is an integer but $\alpha$ isn't, the inequalities in \eqref{eq:m-minus}, \eqref{eq:m-plus} can be equivalently rewritten as
\begin{equation} \label{eq:rewritten-inequalities}
m^- \geq -\lfloor \alpha \rfloor, \quad m^+ \geq \lfloor \alpha \rfloor + 1.
\end{equation}

\begin{lemma} \label{th:nonnegative-1}
Consider a solution of \eqref{eq:floer} not entirely contained in $F_0$. Then, $u^{-1}(F_0)$ is finite, and each point of that subset contributes positively to the intersection number $u \cdot F_0$.
\end{lemma}

\begin{proof}
By \eqref{eq:v-tilde}, $u^{-1}(F_0)$ is a discrete subset. Thanks to \eqref{eq:m-minus} and \eqref{eq:m-plus}, it must be compact. Hence, the intersection number is well-defined. The local contribution of each point is the same as the order of vanishing of $\tilde{v}$, which is of course positive.
\end{proof}

\begin{lemma} \label{th:nonnegative-2}
Take a solution of \eqref{eq:floer} not entirely contained in $F_0$. Then
\begin{equation} \label{eq:fb-intersection}
\begin{aligned}
\mathrm{deg}(u) = u \cdot F_0  & \, + 
\begin{cases} 
m^-(u) & \text{if $x^-$ lies in $F_0$,} \\
0 & \text{otherwise}
\end{cases}
\\ & \, +
\begin{cases}
m^+(u) & \text{if $x^+$ lies in $F_0$,} \\
0 & \text{otherwise.}
\end{cases}
\end{aligned}
\end{equation}
\end{lemma}

Note that if both $x^{\pm}$ lie in $F_0$, $m^-(u) + m^+(u) \geq 1$ by \eqref{eq:rewritten-inequalities}.

\begin{proof}
Let's consider the situation where $x^-$ lies in $F_0$. Restrict to a region where $s \ll 0$, on which $v$ is defined. Partially compactify that region by adding a point at the limit $s  \rightarrow -\infty$, and extend $v$ by setting $v(-\infty) = 0$. Consider the degree of the resulting map over $0$. The point $-\infty$ contributes $m^-(u)$, and the other points contribute $u \cdot F_0$. But the outcome is clearly equal to the degree over nearby points $b \neq 0$, where this time $-\infty$ does not contribute. The other parts are similar.
\end{proof}

\begin{lemma} \label{th:transversality}
Fix $H = (H_t)$. Then for generic $J = (J_t)$, the following properties hold. (i) All solutions of \eqref{eq:floer} are regular. (ii) Any $J_t$-holomorphic sphere with zero Chern number avoids $u(s,t)$ for any $u$ such that $\mathrm{index}(D_u) \leq 2$, and any $s \in \bR$. (iii) Any $J_t$-holomorphic sphere with Chern number $1$ avoids $x(t)$, for any one-periodic orbit $x$.
\end{lemma}

\begin{proof}[Sketch of proof]
Overall, this is quite standard, using only stone age methods \cite{floer-hofer-salamon94, hofer-salamon95, mcduff-salamon}. (i) Regularity of solutions $u$ which remain inside $F_0$ is equivalent to regularity as maps to that fibre (the linearized operator $D_u$ splits, and the part in base direction is always invertible). All other solutions must leave $p^{-1}(B)$, by Lemma \ref{th:max-principle-1}, and outside that subset $J$ is unconstrained. (ii) A pseudo-holomorphic sphere with zero Chern number is either contained in a single fibre $F_b$, $b \in B$; or else, it must be disjoint from $p^{-1}(B)$ by positivity of intersections. As before, one can prove generic regularity (for spheres which are not multiply covered) by considering the two cases separately. A similar remark applies to ``transversality of evaluation maps''. (iii) A pseudo-holomorphic sphere with Chern number $1$ must intersect every fibre of $p$, and can never be multiply covered. Hence, transversality is again easy to show.
\end{proof}

We now have all the basic ingredients needed for the relevant version of Hamiltonian Floer cohomology. Suppose that $J$ has been chosen as in Lemma \ref{th:transversality}. We consider solutions of \eqref{eq:floer} with the conditions
\begin{equation} \label{eq:0-deg}
\left\{
\begin{aligned}
& u^{-1}(F_0) = \emptyset, \\
& x^{\pm} = \textstyle\lim_{s \rightarrow\pm\infty} u(s,\cdot) \text{ lie outside } F_0.
\end{aligned}
\right.
\end{equation}
By Lemma \ref{th:max-principle-2}, this means that $u$ avoids $p^{-1}(B \setminus \partial B)$, hence is contained in a compact subset of $E = F \setminus F_0$. If we additionally require that $\mathrm{index}(D_u) = 1$ and put an upper bound on $u \cdot \Omega$ (or equivalently on the energy), there are only finitely many solutions, up to translation in $\bR$-direction.

Let $\mathit{CF}^*(E,H)$ be the graded $\bK$-vector space with one generator, of degree $i(x)$, for each one-periodic orbit $x$ which lies outside $F_0$. The Floer differential has the form
\begin{equation} \label{eq:floer-d}
dx^+ = \sum_{x^-} \Big( \sum_u {\pm} q^{u \cdot \Omega} \Big) x^-,
\end{equation}
where the sum is over the kind of solution considered above. To determine the sign $\pm$ associated to each $u$, some additional choices are required (in the terminology of \cite{seidel04}, a trivialization of the orientation line of each $x$), which we will not explain further. The cohomology of \eqref{eq:floer-d} is denoted by $\mathit{HF}^*(E,H)$.

\subsection{Continuation maps and the BV operator}
Independence of Floer cohomology of the various choices involved in its construction is established using continuation maps \cite{salamon-zehnder92}. Suppose that we have two choices $(H^{\pm},J^{\pm})$, with the same $\alpha$, either of which can be used to define $\mathit{HF}^*(E,H^{\pm})$. To compare them, we choose an interpolating family $(H^C, J^C)$ which depends on $(s,t) \in \bR \times S^1$, lies in the class \eqref{eq:hamiltonian}, and satisfies
\begin{equation} \label{eq:inter-hj}
(H^C_{s,t},J^C_{s,t}) \longrightarrow (H^{\pm}_t,J^{\pm}_t) \quad \text{as $s \rightarrow \pm\infty$.}
\end{equation}
This is slightly weaker than the commonly imposed requirement that $(H^C,J^C)$ should agree with $(H^{\pm},J^{\pm})$ for $\pm s \gg 0$, and we need to be precise about the notion of convergence that appears. Namely, when restricted to any fixed length cylinder $[s,s+\sigma] \times S^1 \iso [0,\sigma] \times S^1$, we want $(H^C,J^C)$ to converge to $(H^{\pm},J^{\pm})$ exponentially fast in any $C^r$-topology, as $s \rightarrow \pm\infty$.

The continuation map equation is the following generalization of \eqref{eq:floer}:
\begin{equation} \label{eq:continuation}
\left\{
\begin{aligned}
& u: \bR \times S^1 \longrightarrow F, \\
& \partial_s u + J^C_{s,t}(u)(\partial_t u - X^C_{s,t}) = 0, \\
& \textstyle\lim_{s \rightarrow \pm \infty} u(s,\cdot) = x^{\pm}.
\end{aligned}
\right.
\end{equation}
Since \eqref{eq:v-equation} still applies to solutions of this equations, all its consequences (Lemmas \ref{th:max-principle-1}--\ref{th:nonnegative-2}) continue to hold. So does the index formula \eqref{eq:index}.  It is useful to distinguish between geometric and topological energy:
\begin{align} \label{eq:geom-energy}
&
E^{\mathit{geom}}(u) = \int_{\bR \times S^1} \|\partial_s u\|^2\, \mathit{ds} \wedge \mathit{dt} \geq 0, \\
& \label{eq:top-energy}
E^{\mathit{top}}(u) = \int_{\bR \times S^1} u^*\omega_F - d(H^C_{s,t}(u)\, \mathit{dt}) \\
\notag
& \qquad \qquad \qquad \qquad = A(x^-) - A(x^+) + u \cdot \Omega.
\end{align}
These are related by
\begin{equation} \label{eq:two-energies}
E^{\mathit{geom}}(u) = E^{\mathit{top}}(u) + \int_{\bR \times S^1} (\partial_s H^C_{s,t})(u) \,
\mathit{ds} \wedge \mathit{dt}.
\end{equation}
An upper bound on $u \cdot \Omega$ yields a bound on the topological energy. Due its exponential decay as $s \rightarrow \infty$, there is an upper bound on $\int \|\partial_s H^C_{s,t}\| \mathit{ds} \wedge \mathit{dt}$. Hence, a bound on $u \cdot \Omega$ also bounds the geometric energy, which gets the Gromov compactness argument off the ground. 

\begin{lemma} \label{th:transversality-2}
Fix $(H^-,J^-)$ and $(H^+,J^+)$. Then, for a generic choice of $(H^C, J^C)$, the following holds. (i) All solutions of \eqref{eq:continuation} are regular. (ii) For any $(s,t)$, all $J_{s,t}^C$-holomorphic spheres with zero first Chern number avoid $u(s,t)$ for any solution $u$ of \eqref{eq:continuation} with $\mathrm{index}(D_u) \leq 1$.
\end{lemma}

This is the counterpart of Lemma \ref{th:transversality}; if anything, the proof is actually simpler this time, since we have allowed perturbations of both the Hamiltonian and the almost complex structure. Note that keeping $(H^{\pm}, J^{\pm})$ fixed does not constrain the value of $(H^C_{s,t}, J^C_{s,t})$ at any point $(s,t) \in \bR \times S^1$, because the two are only related asymptotically.

As in \eqref{eq:floer-d}, one counts isolated solutions of \eqref{eq:continuation} satisfying \eqref{eq:0-deg}. This leads to a chain map, the continuation map 
\begin{equation} \label{eq:continuation-map}
C: \mathit{CF}^*(E,H^+) \longrightarrow \mathit{CF}^*(E,H^-).
\end{equation}
Up to chain homotopy, continuation maps are unique and well-behaved with respect to composition. For $(H^-,J^-) = (H^+,J^+)$ and the constant deformation between them, the continuation map is the identity. These two properties imply that the maps on Floer cohomology induced by \eqref{eq:continuation-map} are (canonical) isomorphisms. With that in mind, we also write $\mathit{HF}^*(E,\alpha)$ for Floer cohomology.

A minor variation on the previous construction yields the BV operator, a chain map
\begin{equation} \label{eq:define-bv}
\Delta: \mathit{CF}^*(E,H^+) \longrightarrow \mathit{CF}^{*-1}(E,H^-).
\end{equation}
This involves a parametrized moduli space, with a single parameter $\tau \in S^1$. The auxiliary data $(H^\Delta,J^\Delta)$ depend on $\tau$ (as well as on $(s,t)$, obviously), with the analogue of \eqref{eq:inter-hj} being 
\begin{equation} \label{eq:hj-delta}
(H^\Delta_{\tau,s,t}, J^\Delta_{\tau,s,t}) \longrightarrow \begin{cases} 
(H^-_t,J^-_t) & s \rightarrow -\infty, \\ 
(H^+_{t+\tau}, J^+_{t+\tau}) & s \rightarrow +\infty.
\end{cases}
\end{equation}
The appropriate version of \eqref{eq:continuation} is an equation for pairs $(\tau,u)$:
\begin{equation} \label{eq:bv-equation}
\left\{
\begin{aligned}
& u: \bR \times S^1 \longrightarrow F, \\
& \partial_s u + J^{\Delta}_{\tau,s,t}(u)(\partial_t u - X^{\Delta}_{\tau,s,t}) = 0, \\
& \textstyle\lim_{s \rightarrow -\infty} u(s,t) = x^-(t), \\ 
& \textstyle\lim_{s \rightarrow +\infty} u(s,t) = x^+(t+\tau).
\end{aligned}
\right.
\end{equation}
As before, only those solutions which satisfy \eqref{eq:0-deg} are taken into account when defining  \eqref{eq:define-bv}. The induced map on Floer cohomology is independent of all choices, and compatible with continuation isomorphisms.

\subsection{Changing the angle of rotation}
The construction of continuation maps can be extended to the case where the Hamiltonians $H^{\pm}$ satisfy \eqref{eq:hamiltonian} for two different angles $\alpha^- \geq \alpha^+$ in $\bR \setminus \bZ$. One first chooses $\alpha_s$, $s \in \bR$, such that
\begin{equation} \label{eq:the-alphas}
\left\{
\begin{aligned}
& \alpha_s = \alpha^- && s \ll 0, \\
& \alpha_s = \alpha^+ && s \gg 0, \\
& d\alpha_s/ds \leq 0 && \text{everywhere.}
\end{aligned}
\right.
\end{equation}
Then, take $(H^C, J^C)$ as before, except that now $H_{s,t}^C$ lies in the class \eqref{eq:hamiltonian} for $\alpha_s$. Consider solutions of the associated equation \eqref{eq:continuation}. The analogue of \eqref{eq:v-equation} is 
\begin{equation} \label{eq:v-cont-equation}
\partial_s v + i\partial_t v - 2\pi \alpha_s v = 0,
\end{equation}
and the appropriate generalization of \eqref{eq:v-tilde} is
\begin{equation} \label{eq:v-tilde-2}
v(s,t) = \exp(2\pi\textstyle\int\! \alpha_s \mathit{ds}) \,\tilde{v}(\exp(2\pi (s+it))).
\end{equation}
In particular, at any point where $v(s,t) \in B \setminus \{0\}$, one has
\begin{equation} \label{eq:subharmonic}
(\partial_s^2 + \partial_t^2) \log |v| = 2\pi\, d\alpha_s/ds \leq 0.
\end{equation}
The minimum principle applies, showing that Lemma \ref{th:max-principle-2} extends to this situation. In contrast, Lemma \ref{th:max-principle-1} carries over only in a special case, namely when $[\alpha^+,\alpha^-] \cap \bZ = \emptyset$. This is a consequence of the following:

\begin{lemma} \label{th:v-regular}
Consider the left hand side of \eqref{eq:v-cont-equation} as a linear operator between suitable Sobolev completions, let's say $W^{1,2}(\bR \times S^1,\bC) \rightarrow L^2(\bR \times S^1,\bC)$. That operator is always onto. It is an isomorphism iff $[\alpha^+, \alpha^-] \cap \bZ = \emptyset$.
\end{lemma}

\begin{proof}
If one writes solutions of \eqref{eq:v-cont-equation} explicitly using \eqref{eq:v-tilde-2}, then the $W^{1,2}$ condition says that the Taylor expansion of $\tilde{v}(z)$ can have nontrivial terms $z^{-k}$  with $k \in [\alpha^+,\alpha^-] \cap \bZ$. This matches the Fredholm index of our operator, which can be determined either directly using a suitable index formula, or by using the relation with spectral flow.
\end{proof}

Lemmas \ref{th:nonnegative-1}--\ref{th:nonnegative-2} continue to hold, with the obvious adaptation that the numbers $\alpha^{\pm}$ should be used to give bounds for $m^{\pm}(u)$. Lemma \ref{th:transversality-2} also generalizes. The proof is largely as before, with one added wrinkle, which concerns solutions $u$ of \eqref{eq:continuation} which remain inside $p^{-1}(B)$. Let's write such a solution as $u = (v,w)$, with respect to the trivialization \eqref{eq:trivialization-outer}. By Lemma \ref{th:v-regular}, $v$ is a regular solution of \eqref{eq:v-cont-equation}; and $w$ is a solution of a continuation-type equation in $M$, for which one can easily show transversality. Combining the two components then yields the desired transversality result for $u$. One now defines a continuation map exactly as in \eqref{eq:continuation-map}.

\begin{lemma} \label{th:no-integer}
Suppose that $[\alpha^+,\alpha^-] \cap \bZ = \emptyset$. Then the continuation map $\mathit{HF}^*(E,\alpha^+) \rightarrow \mathit{HF}^*(E,\alpha^-)$ is an isomorphism.
\end{lemma}

We will not give the proof of this, since it follows a standard pattern (compare e.g.\ \cite[Lemma 3.4]{seidel14b}). One can correlate the two choices of Hamiltonians so that both have the same one-periodic orbits. Again assuming specific choices of almost complex structures, a minimum principle argument shows that the Floer differentials coincide, and that the continuation map is the identity on the chain level.

\subsection{Pseudo-holomorphic thimbles}
We now consider maps defined on a partial compactification of the cylinder, the ``thimble'' Riemann surface
\begin{equation} \label{eq:thimble-surface}
T = (\bR \times S^1) \cup \{\infty\} \iso \bC,
\end{equation}
where the added point closes up the end $s \rightarrow +\infty$. Fix $\alpha > 0$, $\alpha \in \bR \setminus \bZ$, and suppose that $\mathit{HF}^*(E,\alpha)$ has been defined using some $(H,J)$. Choose a nonincreasing  $\alpha_s$, in the same sense as in \eqref{eq:the-alphas}, which goes from $\alpha$ (for $s \ll 0$) to zero (for $s \gg 0$). Next, choose $(H^{\mathit{thimble}},J^{\mathit{thimble}})$ which depend on $(s,t) \in \bR \times S^1$, such that $H^{\mathit{thimble}}_{s,t}$ belongs to the class \eqref{eq:hamiltonian} for $\alpha_s$. For $s \rightarrow -\infty$, one requires this to converge to $(H,J)$, in the same sense as before. We also require that both $H^{\mathit{thimble}}_{s,t} \mathit{dt}$ (as a one-form on $\bR \times S^1 \times M$) and $J^{\mathit{thimble}}$ should extend smoothly over $\infty \in T$. The equation one considers is a version of \eqref{eq:continuation}:
\begin{equation}
\label{eq:thimble}
\left\{
\begin{aligned}
& u: T \longrightarrow F, \\
& \partial_s u + J^{\mathit{thimble}}_{s,t}(u)(\partial_t u - X^{\mathit{thimble}}_{s,t}) = 0, \\
& \textstyle\lim_{s \rightarrow -\infty} u(s,t) = x(t).
\end{aligned}
\right.
\end{equation}
We have written down the Cauchy-Riemann equation at points $(s,t) \in \bR \times S^1$, but it extends smoothly to all of $T$. Lemma \ref{th:v-regular} has an analogue for this situation, leading to the required transversality result. Lemma \ref{th:max-principle-2}--\ref{th:nonnegative-2}, as well as the index formula \eqref{eq:index} and energy formulas \eqref{eq:geom-energy}--\eqref{eq:two-energies}, carry over after one has removed any mention of the $s \rightarrow +\infty$ limit. The simplest use of \eqref{eq:thimble} is to count solutions satisfying
\begin{equation} \label{eq:0-deg-again}
\left\{
\begin{aligned}
& u^{-1}(F_0) = \emptyset, \\
& x = \textstyle\lim_{s \rightarrow-\infty} u(s,\cdot) \text{ lies outside } F_0.
\end{aligned}
\right.
\end{equation}
This yields a cocycle 
\begin{equation} \label{eq:unit-construction}
e \in \mathit{CF}^0(E,H),
\end{equation}
whose Floer cohomology class is independent of all choices. A natural generalization is to introduce evaluation constraints. Let's fix a point $\zeta \in T$, for concreteness say $\zeta = \infty$. Suppose that we are given an oriented manifold $K$ together with a proper map $\kappa: K \rightarrow E$.
Consider pairs $(u,k)$ consisting of a solution $u$ of \eqref{eq:thimble}, \eqref{eq:0-deg-again} and a point $k \in K$, such that
\begin{equation} \label{eq:evaluation-constraint}
u(\zeta) = \kappa(k).
\end{equation}
Assuming suitably generic choices, counting solutions which satisfy this constraint yields a cocycle 
\begin{equation} \label{eq:pss-q}
b_K \in \mathit{CF}^{\mathrm{dim}(E)-\mathrm{dim}(K)}(E,H).
\end{equation}
Generalizing slightly, one can replace $K$ by a proper pseudo-cycle in $E$, let's say one with $\bR$-coefficients. Finally, one can allow proper pseudo-cycles with $\bK$-coefficients, by which we mean formal sums
\begin{equation} \label{eq:k-pseudo-cycle}
K = q^{d_0} K_0 + q^{d_1} K_1 + \cdots \qquad \text{with $d_i \in \bR$, $\textstyle\lim_i d_i = +\infty$,}
\end{equation}
where the $K_j$ are proper pseudo-cycles with $\bR$-coefficients. For \eqref{eq:k-pseudo-cycle}, one defines $b_K$ by adding up the cocycles \eqref{eq:pss-q} associated to $K_0,K_1,\dots$ with corresponding weights. The cohomology class of \eqref{eq:pss-q} is independent of all choices, and represents the image of $[K] \in H^*(E;\bK)$ under the PSS map \cite{piunikhin-salamon-schwarz94}
\begin{equation} \label{eq:pss}
B: H^*(E;\bK) \longrightarrow \mathit{HF}^*(E,\alpha).
\end{equation}
In fact, one can use this as the definition of \eqref{eq:pss}, even though that is not necessarily the most convenient or natural approach (using Morse homology instead of pseudo-cycles would have distinct advantages, one of them being that it leads more easily to a chain level map). However, here we will not actually use the entire map \eqref{eq:pss}, but only specific instances of the classes \eqref{eq:pss-q}.

\begin{remark} \label{th:pseudo-cycle}
While an extensive discussion of pseudo-cycles would be out of place here (see \cite[Section 6.5]{mcduff-salamon}), we do want to briefly explain the notion of properness which appeared above. A proper pseudo-cycle in $E$ is a map $\kappa: K \rightarrow E$, where $K$ is an oriented manifold, such that there exists another $\tilde{\kappa}: \tilde{K} \rightarrow E$, with $\mathrm{dim}(\tilde{K}) \leq \mathrm{dim}(K) - 2$, satisfying the following condition. If $k_1,k_2,\dots \in K$ is a sequence of points which goes to infinity, and such that $\kappa(k_1), \kappa(k_2),\dots$ converges, then the limit of the latter sequence lies in $\tilde{\kappa}(\tilde{K})$. In the special case $\tilde{K} = \emptyset$, this would reduce to saying that the original map $\kappa$ is proper. Also, if we have a pseudo-cycle in $F$, then its intersection with $E$ is a proper pseudo-cycle in $E$. Any proper pseudo-cycle in $E$ determines a class $[K] \in H^*(E;\bZ)$. 

To get a proper pseudo-cycle with $\bR$-coefficients, one equips $K$ with a multiplicity function (a locally constant real-valued function). One can always think of this as a formal sum
\begin{equation} \label{eq:mu-sum}
\sum_j \mu_j K_j 
\end{equation}
of proper pseudo-cycles $K_j$ (in the previously considered sense) with multiplicities $\mu_j$. It makes sense to consider these modulo some equivalence relation, namely: any term in \eqref{eq:mu-sum} with zero multiplicity may be removed freely; two terms that contain the same pseudo-cycle may be combined by adding up their multiplicities; and the orientation of any $K_j$ may be reversed, while changing the sign of $\mu_j$ at the same time.
\end{remark}

\begin{lemma} \label{th:bv-annihilates}
For any Floer cocycle $b_K$ as in \eqref{eq:pss-q}, $\Delta b_K$ is a co\-boundary.
\end{lemma}

\begin{proof}
By considering a suitable parametrized moduli problem for maps on the thimble, we will construct a Floer cochain $\beta_K$ satisfying
\begin{equation}\label{eq:betaq}
d\beta_K + \Delta b_K = 0.
\end{equation}
The parameter space will be a disc, which itself has to be compactified. While this is generally speaking a standard construction, setting it up requires a bit of care. Fix all the data involved in the relevant previously defined structures, namely: the definition of $\mathit{HF}^*(E,\alpha)$ uses some $(H,J)$; the BV operator uses $(H^\Delta, J^\Delta)$, where we take both asymptotics $(H^\pm, J^\pm) = (H,J)$ to be the same for simplicity; finally, $b_K$ uses $(H^{\mathit{thimble}}, J^{\mathit{thimble}})$.

Part of our parameter space is a half-infinite cylinder $\xi = (\rho,\tau)$, where $\rho \gg 0$ and $\tau \in S^1$. For such a parameter value, we construct $(H^\beta_\xi,J^\beta_\xi)$ as follows. Consider 
\begin{align}
\label{eq:patch-1}
& (H^{\Delta}_{\tau,s+\rho,t},J^\Delta_{\tau,s+\rho,t}), && \text{defined for $(s,t)$ on the cylinder, and} \\
\label{eq:patch-2}
& (H^{\mathit{thimble}}_{s,t+\tau}, J^{\mathit{thimble}}_{s,t+\tau}), && \text{defined on the thimble.} 
\end{align}
Take a finite piece of the cylinder where $s \in [-\rho/2-L/2,-\rho/2+L/2]$, for some constant $L/2$. Because of \eqref{eq:hj-delta} and the translation by $\rho$, \eqref{eq:patch-1} is close to $(H_{t+\tau},J_{t+\tau})$ on that region (it converges uniformly exponentially fast as $\rho \rightarrow \infty$). The same holds for \eqref{eq:patch-2}, simply because of its limit. One can therefore use a partition of unity to glue the two data together to a new datum on the thimble, which we denote somewhat informally by
\begin{equation} \label{eq:makes-sense}
(H^\beta_{\rho,\tau,s,t}, J^\beta_{\rho,\tau,s,t}) = 
(H^{\Delta}_{\tau,s,t}, J^{\Delta}_{\tau,s,t}) \#_{\rho}
(H^{\mathit{thimble}}_{s,t+\tau}, J^{\mathit{thimble}}_{s,t+\tau})
\end{equation}
(the subscript in $\#_{\rho}$ is the gluing length, which in our previous description we had realized by translating the first datum by $\rho$ in $s$-direction). Our full parameter space is $\xi \in (\bR \times S^1) \cup \{\xi = -\infty\}$. We extend our choice \eqref{eq:makes-sense} arbitrarily to the rest of that space. 

Consider the parametrized moduli space of pairs $(\xi,u,k)$, where $k$ is as in \eqref{eq:evaluation-constraint}. In the limit where $\xi =(\rho,\tau)$ with $\rho \rightarrow \infty$, we can get broken solutions $(\tau,\tilde{u},u,k)$, where $(\tau,\tilde{u})$ is a point in the space contributing to $\Delta$, and $(u,k)$ (after rotation $t \mapsto t-\tau$) similarly contributes to $b_K$. Assuming transversality, such limits together with bubbling off of Floer trajectories describe the ends of the $1$-dimensional parametrized moduli spaces. By a standard argument, it follows that counting isolated points of the parametrized problem yields the desired cochain \eqref{eq:betaq}.
\end{proof}

\begin{discussion} \label{th:hf-plus}
The map \eqref{eq:pss} fits into a long exact sequence
\begin{equation} \label{eq:red-hf}
\cdots \longrightarrow H^*(E;\bK) \stackrel{B}{\longrightarrow} \mathit{HF}^*(E,\alpha) \longrightarrow \mathit{HF}^*(E,\alpha)_{\mathit{red}} \longrightarrow \cdots
\end{equation}
and the third term $\mathit{HF}^*(E,\alpha)_{\mathit{red}}$ comes with a refined BV operator
\begin{equation} \label{eq:delta-minus-2}
\xymatrix{
\mathit{HF}^*(E,\alpha) \ar[d]_-{\Delta} \ar[r]  &
\mathit{HF}^*(E,\alpha)_{\mathit{red}} \ar[dl]^-{\Delta_{\mathit{red}}}
\\
\mathit{HF}^{*-1}(E,\alpha).
}
\end{equation}
One defines $\mathit{HF}^*(E,\alpha)_{\mathit{red}}$ as the mapping cone for a suitable chain level version of $B$. 
Similarly, a nullhomotopy for the chain level version of $\Delta \circ B$ gives rise to $\Delta_{\mathit{red}}$.
More explicitly, if $C^*(E;\bK)$ stands for a suitable (e.g.\, Morse) complex underlying $H^*(E;\bK)$, then $\mathit{HF}^*(E,\alpha)_{\mathit{red}}$ is the cohomology of a complex of the form $C^{*+1}(E;\bK) \oplus \mathit{CF}^*(E,H)$, whose differential involves the two differentials on the summands and the chain level $B$. The operation $\Delta_{\mathit{red}}$ is induced by a map $C^{*+1}(E;\bK) \oplus \mathit{CF}^*(E,H) \rightarrow \mathit{CF}^{*-1}(E,H)$ which is the chain level $\Delta$ on the second summand, and the abovementioned nullhomotopy on the first summand. For our purpose, the following partial description suffices: a pair $(K,\xi)$ consisting of a proper pseudo-cycle $K$ and a Floer cochain $\xi \in \mathit{CF}^*(E,H)$ with
\begin{equation}
d\xi = b_K
\end{equation}
gives rise to a class 
\begin{equation}
\label{eq:relative-class}
[(-K,\xi)] \in \mathit{HF}^{\mathrm{codim}(K)-1}(E,\alpha)_{\mathit{red}}, 
\end{equation}
which maps to $[K]$ under the boundary map in \eqref{eq:red-hf}. Moreover, 
\begin{equation} \label{eq:combine-into-a-class}
\Delta_{\mathit{red}} [(-K,\xi)] = [\Delta \xi - \beta_K] \in \mathit{HF}^{\mathrm{codim}(K)-2}(E,\alpha).
\end{equation}
\end{discussion}


\section{The Borman-Sheridan class}

This section discusses the Borman-Sheridan class and various ideas surrounding it. The geometric setup remains the same, but we will now start to make more serious use of its specific properties.

\subsection{Definition}
Take the thimble surface \eqref{eq:thimble-surface}, with its marked point $\zeta =\infty$. We will consider solutions of \eqref{eq:thimble} such that
\begin{equation} \label{eq:extra-s-class}
\left\{
\begin{aligned}
& u(\zeta) \in F_0, \\
& \mathrm{deg}(u) = 1, \\
& x = \textstyle\lim_{s \rightarrow-\infty} u(s,\cdot) \text{ lies outside } F_0.
\end{aligned}
\right.
\end{equation}
The expected dimension of \eqref{eq:thimble}, with \eqref{eq:extra-s-class} taken into account, is 
\begin{equation}
\mathrm{index}(D_u) - 2 = i(x) + 2 \deg(u) - 2 = i(x). 
\end{equation}
Hence, counting isolated solutions (under suitable transversality assumptions) yields a cochain 
\begin{equation} \label{eq:bs-1}
s \in \mathit{CF}^0(E,H).
\end{equation}

\begin{lemma} \label{th:key-splitting}
Suppose that $\alpha > 1$. Consider solutions of \eqref{eq:thimble}, \eqref{eq:extra-s-class}. Suppose that a sequence of such solutions converges to a broken solution $(\tilde{u},u)$, consisting of a Floer trajectory $\tilde{u}$ and a solution $u$ of \eqref{eq:thimble}. Then, $\tilde{u}$ avoids $F_0$ entirely, and $u$ again satisfies \eqref{eq:extra-s-class}.
\end{lemma}

\begin{proof}
By definition of a broken solution, there is some one-periodic orbit $\tilde{x}$ such that
\begin{equation} \textstyle
\lim_{s \rightarrow +\infty} \tilde{u}(s,\cdot) = \tilde{x} = \lim_{s \rightarrow -\infty} u(s,\cdot).
\end{equation}
We know that
\begin{equation} \label{eq:sum-is-1}
\deg(\tilde{u}) + \deg(u) = 1.
\end{equation}
If one assumes that $\tilde{x}$ lies in $F_0$, \eqref{eq:fb-intersection} and its analogue for maps on the thimble say that
\begin{align} \label{eq:u1u2-separate}
& \deg(\tilde{u}) > \alpha, \\
& \deg(u) > 1-\alpha. \label{eq:degu-1}
\end{align}
Ordinarily, there should be an exception in \eqref{eq:degu-1} when $u$ is entirely contained in $F_0$; but in that case $\deg(u) = 0$ by definition, whereas $1-\alpha < 0$ by assumption, so the inequality still holds. Adding up \eqref{eq:u1u2-separate} and \eqref{eq:degu-1} yields a contradiction to \eqref{eq:sum-is-1}. We now know that $\tilde{x}$ does not lie in $F_0$, and therefore
\begin{align}
\label{eq:u1u2-b}
& \deg(\tilde{u}) = \tilde{u} \cdot F_0 \geq 0, \\
\label{eq:u1u2-c}
& \deg(u) = u \cdot F_0 \geq 1.
\end{align}
By \eqref{eq:sum-is-1}, equality must hold in \eqref{eq:u1u2-b} and \eqref{eq:u1u2-c}, which yields the desired behaviour.
\end{proof}

\begin{remark} \label{th:broken-big}
It is maybe instructive to look at more general broken solutions, even though (for codimension reasons) they are not relevant for our counting arguments. Suppose that a sequence of solutions of \eqref{eq:thimble}, \eqref{eq:extra-s-class} converges to a broken solution $(\tilde{u}_1,\dots,\tilde{u}_k,u)$, where the first $k$ components are Floer trajectories. As before, we know that
\begin{equation} \label{eq:all-1}
\deg(\tilde{u}_1) + \cdots + \deg(\tilde{u}_k) + \deg(u) = 1.
\end{equation}
Suppose first that $u$ is not contained in $F_0$. By applying Lemma \ref{th:nonnegative-2} to all components, and using \eqref{eq:rewritten-inequalities}, one then finds that
\begin{equation} \label{eq:geqgeq}
\deg(\tilde{u}_1) + \cdots + \deg(\tilde{u}_k) + \deg(u) \geq \tilde{u}_1 \cdot F_0 + \cdots + \tilde{u}_k \cdot F_0 + u \cdot F_0 \geq 1. 
\end{equation}
The first inequality in \eqref{eq:geqgeq} is an equality iff none of our maps has a limit lying in $F_0$. Assuming this holds, the second inequality is an equality iff all $\tilde{u}_1,\dots,\tilde{u}_k$ are disjoint from $F_0$, while $u$ satisfies \eqref{eq:extra-s-class}. The other potential situation is where, for some $j \geq 1$, $(\tilde{u}_{j+1},\dots,\tilde{u}_k,u)$ are all contained in $F_0$, while $\tilde{u}_j$ isn't. However, this means that
\begin{equation}
\deg(\tilde{u}_1) + \cdots + \deg(\tilde{u}_k) + \deg(u) \geq \tilde{u}_1 \cdot F_0 + \cdots + \tilde{u}_j \cdot F_0 + \alpha.
\end{equation}
If $\alpha > 1$, this contradicts \eqref{eq:all-1}. Hence, we see that all broken solutions that appear as limits are built by adding Floer trajectories inside $E$ as extra components. One can further extend this discussion, including sphere bubbles as well, which leads to a complete description of the Gromov compactification. As usual, sphere bubbling turns out to be always a codimension $\geq 2$ phenomenon, hence will not occur in our applications.
\end{remark}
%
%
%
%
%

Lemma \ref{th:key-splitting} implies that \eqref{eq:bs-1} is a cocycle provided that $\alpha  > 1$. Its Floer cohomology class, the Borman-Sheridan class, is independent of all choices. An argument parallel to Lemma \ref{th:bv-annihilates} (and which we will therefore omit) shows the following:

\begin{lemma} \label{th:bs-done}
For $\alpha>1$, the Floer cocycle $\Delta s$ is a coboundary.
\end{lemma}

\subsection{The Borman-Sheridan class and Gromov-Witten theory\label{subsec:secondary}}
The starting point for the following discussion is the Gromov-Witten invariant which is the special case $k = 1$ of \eqref{eq:zk}. Our notation is as follows:
\begin{equation} \label{eq:sphere}
\left\{
\begin{aligned}
& \tilde{S} = \bC P^1, \\
& \tilde{\zeta}_1 = 0, \;\; \tilde{\zeta}_2 = 1, \;\; \tilde{\zeta}_3 = \infty \in \tilde{S}.
\end{aligned}
\right.
\end{equation}
We also fix a point
\begin{equation} \label{eq:dagger}
\dag \in B \setminus (\partial B \cup \{0\}), 
\end{equation}
to be used throughout the rest of the paper. Choose some almost complex structure $\tilde{J}$ as in \eqref{eq:j} (one could allow a family of almost complex structures varying along $\tilde{S}$, but that is not necessary or helpful for our purpose). Suppose that we have a map $u: \tilde{S} \rightarrow F$, which is pseudo-holomorphic with respect to $\tilde{J}$ and satisfies 
\begin{equation} \label{eq:pseudo-sphere-1}
\tilde{u} \cdot F_0 = 1.
\end{equation}
Any intersection point of $\tilde{u}$ with a fibre $F_b$, $b \in B$, contributes positively to the intersection number, and contributes $1$ iff the intersection is transverse. Hence, the part of $\tilde{u}$ lying inside $p^{-1}(B)$ is a pseudo-holomorphic section. In particular, there is an automorphism of $\tilde{S}$ which we can apply to achieve that
\begin{equation} \label{eq:pseudo-sphere-2}
\left\{
\begin{aligned}
& \tilde{u}(\tilde{\zeta}_1) \in F_0, \\
& \tilde{u}(\tilde{\zeta}_2) \in F_\dag.
\end{aligned}
\right.
\end{equation}
From now on, we will consider maps $\tilde{u}$ satisfying \eqref{eq:pseudo-sphere-1}, \eqref{eq:pseudo-sphere-2}. For generic choice of $\tilde{J}$, the space of such maps in any fixed homology class, with evaluation at the point $\tilde{\zeta}_3$, defines a pseudo-cycle of codimension $2$ in $F$. If we add up the contributions of all homology classes with the usual weights $q^{\tilde{u} \cdot \Omega}$, the outcome is a pseudo-cycle with $\bK$-coefficients, which represents $z^{(1)} \in H^2(F;\bK)$. Let's denote this pseudo-cycle by $Z^{(1)}$. Restricting to the open subset where $\tilde{u}(\tilde{\zeta}_3) \notin F_0$ yields a proper pseudo-cycle with $\bK$-coefficients in $E$, denoted by $Z^{(1)}|E$.

\begin{lemma} \label{th:z1-zero}
For $\alpha > 1$, $z^{(1)}|E$ lies in the kernel of the map \eqref{eq:pss}.
\end{lemma}

\begin{proof}
By considering an appropriate parametrized moduli space, we will construct a Floer cochain bounding the cocycle \eqref{eq:pss-q} associated to $Z^{(1)}|E$:
\begin{equation} \label{eq:z-bounds}
d\nu = b_{Z^{(1)}|E}.
\end{equation}

The parameter is a number $\sigma \in \bR$, which appears as the position of a marked point $\zeta_2 = (\sigma,0)$ on the thimble (this is in addition to our usual marked point $\zeta_1 = \infty$). We consider the space of pairs $(\sigma,u)$, where $u$ is a solution of an equation of type \eqref{eq:thimble}, with additional conditions
\begin{equation} \label{eq:z1-kernel}
\left\{
\begin{aligned}
& u(\zeta_1) \in F_0, \\
& u(\zeta_2) \in F_{\dag}, \\
& \mathrm{deg}(u) = 1, \\
& x = \textstyle\lim_{s \rightarrow-\infty} u(s,\cdot) \text{ lies outside } F_0.
\end{aligned}
\right.
\end{equation}

There are two relevant degenerations to consider. In the first one, $\sigma \rightarrow -\infty$, the limit is a broken solution $(\tilde{u},u)$, where $\tilde{u}$ is a Floer trajectory satisfying 
\begin{equation} \label{eq:00-incidence}
\tilde{u}(0,0) \in F_{\dag}.
\end{equation}
The other component $u$ is again a solution of \eqref{eq:thimble}, with $u(\zeta_1) \in F_0$. By the same argument as in Lemma \ref{th:key-splitting}, one finds that the limits of $\tilde{u}$ must lie outside $F_0$, and that $\mathrm{deg}(\tilde{u}) = 0$; but then, $\tilde{u}$ must be disjoint from $p^{-1}(B \setminus \partial B)$ by Lemma \ref{th:max-principle-2}, which is a contradiction to \eqref{eq:00-incidence}. Hence, there are no such limit points. 

The other relevant degeneration happens when $\sigma \rightarrow +\infty$; the limit $(u,\tilde{u})$ consists of a solution $u$ of \eqref{eq:thimble}, which for the same reason as before will be disjoint from $F_0$, together with a sphere bubble which satisfies \eqref{eq:pseudo-sphere-1} and (when appropriately parametrized) \eqref{eq:pseudo-sphere-2}. More precisely, the almost complex structure $\tilde{J}$ which arises is that appearing in \eqref{eq:thimble} for the point $\zeta_1 \in T$. The two components are related by the incidence requirement
\begin{equation}
u(\zeta_1) = \tilde{u}(\tilde{\zeta}_3).
\end{equation}
Hence, counting such pairs $(u,\tilde{u})$ just yields $b_{Z^{(1)}|E}$. From this, a standard argument shows that counting points in the parametrized moduli space leads to the desired structure \eqref{eq:z-bounds}.
\end{proof}

\begin{remark}
Sphere bubbling is normally a codimension $2$ phenomenon, and correspondingly, the relevant gluing problem has two parameters: a real gluing length and a circle-valued gluing angle. In the situation of Lemma \ref{th:z1-zero}, we use $\sigma$ instead of the gluing length as a parameter, and since $\zeta_2$ is constrained to lie on $\bR \times \{0\}$, there is no gluing angle. This explains why it makes sense for sphere bubbling to yield a codimension $1$ boundary stratum of our moduli space.
\end{remark}

As an instance of Discussion \ref{th:hf-plus}, \eqref{eq:z-bounds} gives rise to a class
\begin{equation} \label{eq:z1-rel}
[(-Z^{(1)}|E, \nu)] \in \mathit{HF}^1(E,\alpha)_{\mathit{red}}
\end{equation}
which maps to $z^{(1)}|E$ under the connecting map from \eqref{eq:red-hf}. We also get a Floer cocycle 
\begin{equation} \label{eq:corrected-beta}
\Delta \nu - \beta_{Z^{(1)}|E} \in \mathit{CF}^0(E,H),
\end{equation}
which represents the image of \eqref{eq:z1-rel} under $\Delta_{\mathit{red}}$.

\begin{prop} \label{th:interpret-bs}
The cocycle \eqref{eq:corrected-beta} is cohomologous to $(-1)$ times the Borman-Sheridan cocycle \eqref{eq:bs-1}.
\end{prop}

\begin{proof}
The first step is to combine the ideas from Lemmas \ref{th:bv-annihilates} and \ref{th:z1-zero} (since this is a reprise of previous arguments, we will be rather light on details). Consider a moduli space as in Lemma \ref{th:bv-annihilates}, but with one additional parameter $\sigma$ as in Lemma \ref{th:z1-zero}, again thought of as giving a marked point $\zeta_2 = (\sigma,0)$. Look at maps $u$ satisfying \eqref{eq:z1-kernel}.
\begin{figure}
\begin{centering}
\begin{picture}(0,0)%
\includegraphics{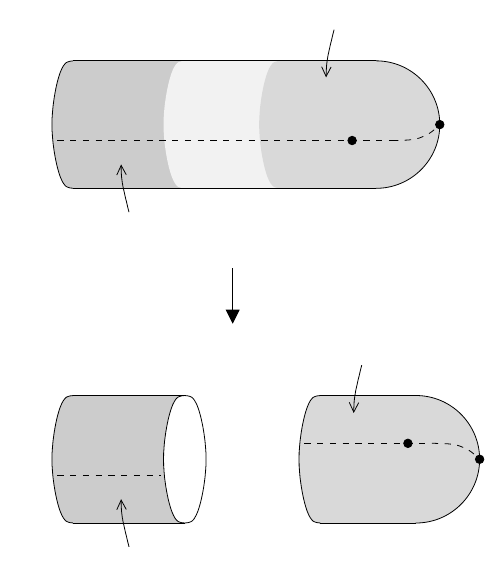}%
\end{picture}%
\setlength{\unitlength}{3355sp}%
\begingroup\makeatletter\ifx\SetFigFont\undefined%
\gdef\SetFigFont#1#2#3#4#5{%
  \reset@font\fontsize{#1}{#2pt}%
  \fontfamily{#3}\fontseries{#4}\fontshape{#5}%
  \selectfont}%
\fi\endgroup%
\begin{picture}(4605,5436)(-2339,-6139)
\put(601,-886){\makebox(0,0)[lb]{\smash{{\SetFigFont{10}{12.0}{\rmdefault}{\mddefault}{\updefault}{$(H^{\mathit{thimble}}_{s,t+\tau}, J^{\mathit{thimble}}_{s,t+\tau})$}%
}}}}
\put(-1724,-2911){\makebox(0,0)[lb]{\smash{{\SetFigFont{10}{12.0}{\rmdefault}{\mddefault}{\updefault}{$(H^{\Delta}_{\tau,s+\rho,t},J^\Delta_{\tau,s+\rho,t})$}%
}}}}
\put(-2324,-2086){\makebox(0,0)[lb]{\smash{{\SetFigFont{10}{12.0}{\rmdefault}{\mddefault}{\updefault}{$t = 0$}%
}}}}
\put(1876,-1786){\makebox(0,0)[lb]{\smash{{\SetFigFont{10}{12.0}{\rmdefault}{\mddefault}{\updefault}{$\zeta_1$}%
}}}}
\put(976,-2236){\makebox(0,0)[lb]{\smash{{\SetFigFont{10}{12.0}{\rmdefault}{\mddefault}{\updefault}{$\zeta_2$}%
}}}}
\put(526,-4036){\makebox(0,0)[lb]{\smash{{\SetFigFont{10}{12.0}{\rmdefault}{\mddefault}{\updefault}{$(H^{\mathit{thimble}}_{s,t}, J^{\mathit{thimble}}_{s,t})$}%
}}}}
\put(  1,-3436){\makebox(0,0)[lb]{\smash{{\SetFigFont{10}{12.0}{\rmdefault}{\mddefault}{\updefault}{degeneration $\rho \rightarrow \infty$}%
}}}}
\put(-1724,-6061){\makebox(0,0)[lb]{\smash{{\SetFigFont{10}{12.0}{\rmdefault}{\mddefault}{\updefault}{$(H^{\Delta}_{\tau,s,t},J^\Delta_{\tau,s,t})$}%
}}}}
\put(-2324,-5236){\makebox(0,0)[lb]{\smash{{\SetFigFont{10}{12.0}{\rmdefault}{\mddefault}{\updefault}{$t = 0$}%
}}}}
\put(  1,-4936){\makebox(0,0)[lb]{\smash{{\SetFigFont{10}{12.0}{\rmdefault}{\mddefault}{\updefault}{$t = \tau$}%
}}}}
\put(1351,-5086){\makebox(0,0)[lb]{\smash{{\SetFigFont{10}{12.0}{\rmdefault}{\mddefault}{\updefault}{$\zeta_2$}%
}}}}
\put(2251,-4936){\makebox(0,0)[lb]{\smash{{\SetFigFont{10}{12.0}{\rmdefault}{\mddefault}{\updefault}{$\zeta_1$}%
}}}}
\end{picture}%
\caption{\label{fig:breaking-thimble}One of the limits from the proof of Proposition \ref{th:interpret-bs}.}
\end{centering}
\end{figure}
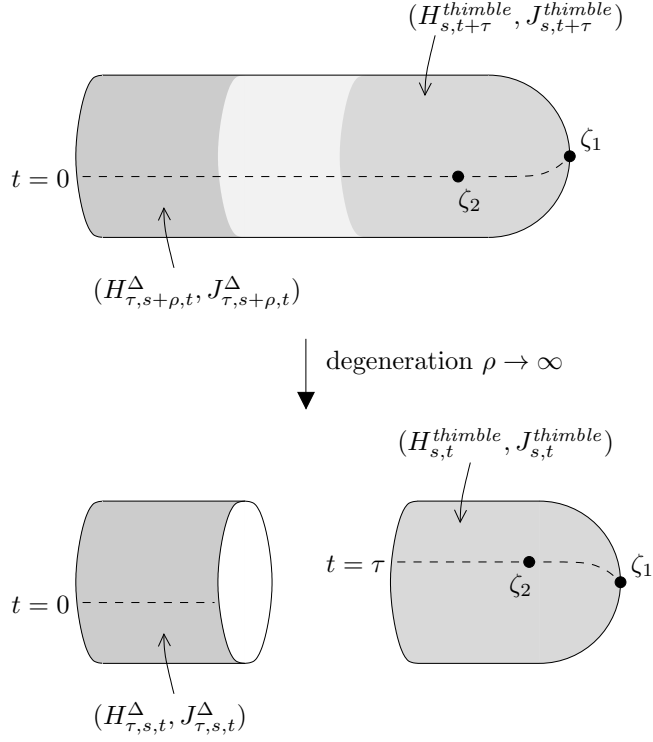%

Three kinds of degenerations of such $(\xi,\sigma,u)$ are relevant for us. The first one is $\sigma \rightarrow \infty$, in which case one gets sphere bubbling as in Lemma \ref{th:z1-zero}. The contribution obtained by counting such configurations is $\beta_{Z^{(1)}|E}$. The second kind of degeneration happens when $\xi = (\rho,\tau)$ with $\rho \rightarrow \infty$, and simultaneously $\sigma \rightarrow -\infty$. This contributes zero, for the same reason as in Lemma \ref{th:z1-zero}. The final possibility is that  $\xi = (\rho,\tau)$ with $\rho \rightarrow \infty$, while $\sigma$ remains bounded (see Figure \ref{fig:breaking-thimble}). The resulting broken solutions are of the form $(\tau,\sigma,\tilde{u},u)$, where $(\tau,\tilde{u})$ satisfies the equation which defines $\Delta$, and $u$ is a map as in \eqref{eq:thimble}, such that
\begin{equation} \label{eq:tau-appears}
\left\{
\begin{aligned}
& u(\zeta_1) \in F_0, \\ 
& u(\sigma,\tau) \in F_{\dag}, \\
& \mathrm{deg}(u) = 1, \\
& \tilde{x} = \textstyle\lim_{s \rightarrow-\infty} u(s,\cdot) \text{ lies outside } F_0.
\end{aligned}
\right.
\end{equation}
The equations for $(\tau,\tilde{u})$ and $(\sigma,u)$ are not independent, because of the appearance of $\tau$ in \eqref{eq:tau-appears}. Hence, the contribution from this degeneration can't be written as a combination of previously introduced ones.

To make up for this shortcoming, we introduce another moduli space. This has parameters $\tau_1 \in [0,1]$, $\sigma \in \bR$ and $\tau_2 \in [0,\tau_1]$, and consists of $(\tilde{u},u)$ such that: $(\tau_1,\tilde{u})$ is again a  point in the moduli space defining $\Delta$, and $u$ is as in \eqref{eq:tau-appears} except that the second part of that equation must be replaced by
\begin{equation}
u(\sigma,\tau_2) \in F_{\dag}.
\end{equation}
The space of parameter values has three boundary components. One is when $\tau_1 = \tau_2$, which yields exactly the same contribution as in \eqref{eq:tau-appears}. The second boundary component $\tau_2 = 0$ contributes $\Delta \nu$. The final boundary component appears when $\tau_1 = 1$. In that case, $(\sigma,\tau_2)$ can be an arbitrary point in $\bR \times S^1$. Now, suppose that when defining \eqref{eq:bs-1}, we have chosen the auxiliary data so that isolated solutions are transverse to $F_{\dag}$. The count (with signs) of points at which any such solution passes through $F_{\dag}$ is $1$. Hence, we get a contribution which equals \eqref{eq:bs-1} (since $\tau_1$ is fixed to be $1$, one can arrange that the only maps $\tilde{u}$ that appear in that case are of the form $\tilde{u}(s,t) = x(t)$, which do not affect the contribution).

By combining both parametrized moduli spaces, and counting isolated points in them (more precisely, by subtracting the contribution of the second space from that of the first one), we therefore obtain a Floer cochain which bounds $\beta_{Z^{(1)}|E} - \Delta \nu - s$.
\end{proof}

\begin{remark}
The sign conventions we have introduced in the formal TQFT context, specifically \eqref{eq:composition-law} and \eqref{eq:boundary-axiom}, are compatible with those that arise in Floer theory from juggling determinant lines of elliptic operators (compare e.g.\ \cite[Section 11]{seidel04} or \cite{zinger13}). We do not normally dig into such details, but Proposition \ref{th:interpret-bs} may be a good point to do so, since the sign in its statement has consequences throughout the paper, for instance \eqref{eq:a-s}.

Points in the first parametrized moduli space are of the form $(\xi,\sigma,u)$. Since $\xi$ takes values in a two-dimensional space, the associated orientation does not change if we reorder the components as $(\sigma,\xi,u)$. Hence, $\sigma \rightarrow +\infty$ corresponds to a ``boundary face at infinity'' which is positively oriented. The other nontrivially contributing ``boundary face at infinity'' corresponds to writing $\xi = (\rho,\tau)$ and sending $\rho \rightarrow +\infty$, which is again positively oriented. As in \eqref{eq:boundary-axiom}, this means that the parametrized moduli space yields a cocycle $\alpha_1 \in \mathit{CF}^{-1}(E,H)$ with
\begin{equation} \label{eq:sign-test-1}
-\!d\alpha_1 + \beta_{Z^{(1)}|E} + \theta = 0,
\end{equation}
where $\theta$ is the contribution from \eqref{eq:tau-appears}. In the second parametrized moduli space, points are of the form $(\tau_1,\sigma,\tau_2,\tilde{u},u)$, or equivalently as far as orientation is concerned, $(\tau_2,\tau_1,\sigma,\tilde{u},u)$. The boundary value $\tau_1 = \tau_2$ is the maximal value of $\tau_2$, hence counts positively as in \eqref{eq:sign-test-1} (it is important that the ordering $(\tau_1,\sigma)$ of the remaining parameters is the same as the one encountered before). On the other hand, $\tau_2 = 0$ is the minimal value, which a priori means that the corresponding boundary points $(\tau_1,\sigma,\tilde{u},u)$ are counted negatively. However, to recover the orientation that corresponds to $\Delta \nu$, one has to reorder the entries as $(\tau_1,\tilde{u},\sigma,u)$, and since $\mathrm{index}(D_{\tilde{u}}) = -1$ at the relevant points, that permutation introduces another sign, which cancels the previous one. The final contribution comes from $\tau_1$ reaching its maximal value, and since the remaining parameters $(\sigma,\tau_2)$ correspond to the standard complex orientation of the cylinder, we get a positive sign. The outcome is that the second moduli space yields an $\alpha_2 \in \mathit{CF}^{-1}(E,H)$ such that
\begin{equation} \label{eq:sign-test-2}
-\!d\alpha_2 + \theta + \Delta \nu + s = 0.
\end{equation}
Combining \eqref{eq:sign-test-1} and \eqref{eq:sign-test-2} yields the desired statement.
\end{remark}

\subsection{Generalizations\label{subsec:higher-s}}
The construction of \eqref{eq:bs-1} admits generalizations in several directions. We only want to consider two examples. 

First, fix distinct points $\zeta_1,\zeta_2 \in T$ on the thimble, and consider solutions of \eqref{eq:thimble} such that
\begin{equation} \label{eq:s2-class}
\left\{
\begin{aligned}
& u(\zeta_1) \in F_0, \\ & u(\zeta_2) \in F_0, \\
& \mathrm{deg}(u) = 2, \\
& x = \textstyle\lim_{s\rightarrow -\infty} u(s,\cdot) \text{ lies outside } F_0.
\end{aligned}
\right.
\end{equation}
This leads to a cochain
\begin{equation} \label{eq:s2}
s^{(2)} \in \mathit{CF}^0(E,H)
\end{equation}
which, by a straightforward analogue of Lemma \ref{th:key-splitting}, is a cocycle if $\alpha > 2$. As the notation is intended to suggest, we will show later on that the Floer cohomology class of \eqref{eq:s2} is the (pair-of-pants) square of the Borman-Sheridan class; see Section \ref{subsec:pair-of-pants}.

The second possibility is to use only one marked point, but to ask for a tangency condition:
\begin{equation} \label{eq:tangency-condition}
\left\{
\begin{aligned}
& u(\zeta) \in F_0, \;\; Dp \circ Du_{\zeta} = 0, \\
& \mathrm{deg}(u) = 2, \\
& x = \textstyle\lim_{s\rightarrow -\infty} u(s,\cdot) \text{ lies outside } F_0.
\end{aligned}
\right.
\end{equation}
This makes sense because, if $u(\zeta) \in F_0$ holds, $Dp \circ Du_\zeta$ is complex-linear, which means that its vanishing is a condition of (real) codimension $2$. As before, counting such solutions leads to a Floer cochain, which is a cocycle if $\alpha > 2$; we denote it by
\begin{equation} \label{eq:s01}
\tilde{s}^{(2)} \in \mathit{CF}^0(E,H).
\end{equation}

\begin{lemma} \label{th:bisections}
Let $z^{(2)} \in \bK$ be the count of ``holomorphic bisections'', which means holomorphic spheres which have degree $2$ over $\bC P^1$, as in \eqref{eq:zk}. Then (assuming $\alpha > 2$ as usual) $s^{(2)}$ represents the same Floer cohomology class as $\tilde{s}^{(2)} + 2z^{(2)} e$.
\end{lemma}

\begin{proof}
The basic idea is similar to that in Proposition \ref{th:interpret-bs}. We introduce one parameter $\rho \geq 0$. One of the marked points is fixed, let's say $\zeta_1 = \infty$, and the other one is $\zeta_2 = (\rho,0)$. One then considers the parametrized version of \eqref{eq:s2-class}.

As $\rho \rightarrow \infty$, the relevant limiting behaviour is sphere bubbling at $\zeta_1$, leading to a principal component $u$ together with a bubble component $\tilde{u}: \tilde{S} \rightarrow F$, which (for a suitable parametrization) satisfies
\begin{equation} \label{eq:tilde-u-para}
\left\{
\begin{aligned}
& \tilde{u}(\tilde{\zeta}_1) \in F_0, \\
& \tilde{u}(\tilde{\zeta}_2) \in F_0, \\
& \tilde{u}(\tilde{\zeta}_3) = u(\zeta_1).
\end{aligned}
\right.
\end{equation}
There are in fact two topologically distinct sub-cases. One is that $\tilde{u}$ is a constant ``ghost bubble''. In that case, the principal component is tangent to $F_0$ at $\zeta_1$. Counting such configurations contributes $\tilde{s}^{(2)}$. The other case is where $\tilde{u}$ is not contained in $F_0$, so that $\tilde{u} \cdot F_0 = 2$. In that case, the principal component satisfies $\deg(u) = u \cdot F_0 = 0$, which means that it avoids $F_0$ altogether. If one counts the maps $\tilde{u}$, thinking of $u(\zeta_1)$ as a fixed (generic) point in $E$, the outcome is $2z^{(2)}$, where the factor of $2$ comes from the possibility of changing the parametrization \eqref{eq:tilde-u-para}, exchanging $\tilde{\zeta}_1$ and $\tilde{\zeta}_2$ while keeping $\tilde{\zeta}_3$ fixed. Hence, the contribution of such pairs $(u,\tilde{u})$ is $2 z^{(2)} e$. 

Note that other kinds of sphere bubbling, for instance where $\tilde{u}$ is a non-constant map to $F_0$, are irrelevant because they are of codimension $\geq 2$. The other potentially problematic limiting behaviour as $\rho \rightarrow \infty$ is where the limit is of the form $(\tilde{u},u)$, where $\tilde{u}$ is a Floer trajectory, and the principal component $u$ is entirely contained in $F_0$. However, that can be excluded by the same arguments as in Lemma \ref{th:key-splitting}.
\end{proof}

\section{A connection\label{sec:floer-connection}}
We will construct, under Assumption \ref{th:psi-eta}, a connection on Floer cohomology, and show that it satisfies a version of \eqref{eq:nablac-s} with $c = -1$; see \eqref{eq:we-got-it}. Eventually, one has to discuss how this construction fits into the contexts of Sections \ref{subsec:sh} and \ref{subsec:differentiation-axiom}; but we postpone that to Sections \ref{sec:operations-on-floer-cohomology}--\ref{sec:symplectic-cohomology}.

\subsection{The setup}
Fix some $\alpha \in \bR \setminus \bZ$, and the data $(H,J)$ required to define the chain complex $\mathit{CF}^*(E,H)$ underlying $\mathit{HF}^*(E,\alpha)$. We denote by $\partial_q$ the operation of differentiating Floer cochains, which means differentiating their coefficients with respect to the obvious basis. It follows from the definition of the differential \eqref{eq:floer-d} that
\begin{equation} \label{eq:differentiate-d}
(\partial_q d - d \partial_q )(x^+) = \sum_{x^-} \Big( \sum_u \pm (u \cdot \Omega) q^{u \cdot \Omega-1} \Big) x^-.
\end{equation}
We can assume, since this is true for generic choices of $(H,J)$, that each solution $u$ with $\mathrm{index}(D_u) = 1$ intersects the components $\Omega_j$ of $\Omega$ transversally. Let's introduce a parameter $\tau \in S^1$, which singles out a point $\zeta = (0,-\tau) \in \bR \times S^1$. We can then consider pairs $(\tau,u)$, where $u$ is a solution of \eqref{eq:floer} such that $u(\zeta) \in \Omega_j$. For each $j$, counting points in the associated parametrized moduli space yields a chain map of degree $1$. Adding up those chain maps, with the weights $\mu_j$ from \eqref{eq:z-divisor}, and then multiplying the whole with $q^{-1}$, recovers \eqref{eq:differentiate-d}. 

\begin{remark}
One may wonder why, in order to obtain the count from \eqref{eq:differentiate-d}, one has to choose orientations of the parameter space so that the marked point $\zeta$ runs negatively around the $S^1$ factor. This has its origin in the sign convention for the differential $d$. Namely, if $u$ is a regular solution of \eqref{eq:floer} with $\mathrm{index}(D_u) = 1$, 
\begin{equation} \label{eq:only-translations}
\mathit{ker}(D_u) = \bR \,\partial_su.
\end{equation}
The general theory of orientation operators associated to one-periodic points $x$ (going back to \cite{floer88}, see also \cite{floer-hofer93}; the terminology is borrowed from \cite{seidel04}) produces orientations of $\mathit{ker}(D_u) \oplus \mathit{coker}(D_u)$ for any $u$, regular or not. In \eqref{eq:floer-d} we count $u$ with $\pm 1$ depending on whether this orientation agrees with the obvious one on the right hand side of \eqref{eq:only-translations} or not. 

An oriented basis of the tangent space to the parametrized moduli space at $(\tau,u)$ is given by $((1,0), (0, \pm \partial_s u))$, or by $((0, \pm \partial_s u), (-1,0))$, where the sign has the same meaning as before. Under the evaluation map $(\tau,u) \mapsto u(0,-\tau)$, the latter basis maps to $(\pm \partial_s u, \partial_t u)$. This agrees with the orientation that contributes to the expression $\pm (u \cdot \Omega)$ in \eqref{eq:differentiate-d}.
\end{remark}

Having reformulated the right hand side of \eqref{eq:differentiate-d} as a parametrized moduli problem, we can consider it as a special case of a more general construction. Namely, take a proper pseudo-cycle $K$ with $\bK$-coefficients in $E$, as in \eqref{eq:k-pseudo-cycle}. Choose $(H^r,J^r)$ which depend on $(\tau,s,t)$, with asymptotics for $\pm s \gg 0$ as in \eqref{eq:inter-hj}. Consider the parametrized space of solutions of the associated equation \eqref{eq:continuation}, with 
\begin{equation} \label{eq:defining-r}
\left\{
\begin{aligned}
& u(\zeta) = \kappa_j(k) \text{ for some $k \in K_j$,} \\
& \mathrm{deg}(u) = 0, \\
& x^{\pm} = \textstyle \lim_{s \rightarrow \pm\infty} u(s,\cdot) \text{ lies outside } F_0.
\end{aligned}
\right.
\end{equation}
Counting points in this space (with real multiplicities, since $K_j$ is a pseudo-cycle with $\bR$-coefficients) gives a chain map $r_{K_j}$. We then define
\begin{equation} \label{eq:o-map}
r_K = \sum_j q^{d_j} \, r_{K_j}: \mathit{CF}^*(E,H^+) \longrightarrow \mathit{CF}^{*+\mathrm{codim}(K)-1}(E,H^-).
\end{equation}
Let's specialize to $K = q^{-1}\Omega|E$. Suppose that we take $(H^r_{\tau,s,t},J^r_{\tau,s,t}) = (H_t,J_t)$. In that case, our previous argument shows that $r_{q^{-1}\Omega|E}$ agrees with the right hand side of \eqref{eq:differentiate-d}. If we still choose $(H^r,J^r)$ to converge to $(H,J)$ as $s \rightarrow \pm\infty$, but otherwise arbitrary, then $r_{q^{-1}\Omega|E}$ is still at least chain homotopic to \eqref{eq:differentiate-d} (the chain homotopy is defined in a straightforward way, by adding another parameter that deforms the auxiliary data). Finally, we can choose different asymptotics $(H^-,J^-)$ and $(H^+,J^+)$, even with different amounts of rotation $\alpha^- \geq \alpha^+$. In that case, the resulting map $r_{q^{-1}\Omega|E}$ is chain homotopic to \eqref{eq:differentiate-d} composed with the continuation map \eqref{eq:continuation-map}. To make this entirely clear, let's write $d^{\pm}$ for the two Floer differentials involved. Then we have chain homotopies
\begin{equation} \label{eq:r-homotopy-omega}
r_{q^{-1}\Omega|E} \htp (\partial_q d^- - d^-\partial_q) \, C \htp C \, (\partial_q d^+ - d^+ \partial_q).
\end{equation}
These homotopies are obtained by moving the marked point $\zeta = (0,-\tau)$ towards one of the ends of the cylinder, and considering the resulting two-parameter moduli space.

From this point onwards, we will work under Assumption \ref{th:psi-eta}. Let's choose a pseudo-chain $A$ such that
\begin{equation} \label{eq:boundary-of-a}
\partial A \iso \psi Z^{(1)} -q^{-1}\Omega - \eta F_0.
\end{equation}
To be precise, $A$ is a pseudo-chain with $\bK$-coefficients in $F$; $\iso$ stands for equivalence of pseudo-cycles; and the inclusion $F_0 \hookrightarrow F$ is itself thought of as a pseudo-cycle. Restricting to $E$ then yields a proper pseudo-chain $A|E$, which satisfies
\begin{equation} \label{eq:a-restricted}
\partial (A|E) = \partial A|E \iso \psi Z^{(1)}|E - q^{-1}\Omega|E.
\end{equation}
By imitating the construction of \eqref{eq:o-map} with $A|E$ instead of $K$, one gets a chain homotopy
\begin{equation} \label{eq:pseudo-homotopy}
\psi r_{Z^{(1)}|E} = r_{\psi Z^{(1)}|E} \htp r_{q^{-1}\Omega|E}.
\end{equation}

\begin{remark} \label{th:pseudo-cycle-2}
We refer to \cite{schwarz99, kahn01, zinger08} for a general discussion of the relation between equivalence classes of pseudo-cycles and homology, and only consider the aspect which is immediately relevant here. A pseudo-chain (or relative pseudo-cycle) in $F$ is a map $a: A \rightarrow F$, where $A$ is an oriented manifold with boundary, which satisfies the same conditions as a pseudo-cycle, and such that $a|\partial A$ is itself a pseudo-cycle. Any pseudo-cycle which is trivial in $H_*(F;\bZ)$ can be represented as such a boundary.

For a pseudo-cycle with $\bR$-coefficients, the same is true up to equivalence (in the sense mentioned in Remark \ref{th:pseudo-cycle}): if its class in $H_*(F;\bR)$ is trivial, the pseudo-cycle is equivalent to the boundary of a pseudo-chain with $\bR$-coefficients. Finally, for the case of $\bK$-coefficients which we have encountered above, one simply argues order by order in the Novikov variable. We have used this implicitly when asserting the existence of \eqref{eq:boundary-of-a}.
\end{remark}

Our strategy will be to show, using the same ideas as in Lemma \ref{th:z1-zero}, that for suitable choices of $\alpha^\pm$, the map $r_{Z^{(1)}|E}$ is chain homotopic to zero. More precisely:

\begin{lemma} \label{th:nu-null-homotopy}
Suppose that $[\alpha^+,\alpha^-] \cap \bZ \neq \emptyset$. Then there is a map 
\begin{equation} \label{eq:nu-null-homotopy}
\begin{aligned}
& \chi: \mathit{CF}^*(E,H^+) \longrightarrow \mathit{CF}^*(E,H^-), \\
& d^- \chi - \chi d^+ + r_{Z^{(1)}|E} = 0.
\end{aligned}
\end{equation}
\end{lemma}

Let's explain why this is useful. From \eqref{eq:r-homotopy-omega} and \eqref{eq:pseudo-homotopy}, we have chain homotopies (for which we now introduce the notation $h^+$ and $r_{A|E}$)
\begin{align}
\label{eq:1st-homotopy}
& d^- h^+ - h^+ d^+ + C\, (\partial_q d^+ - d^+ \partial_q) - r_{q^{-1}\Omega|E} = 0, \\
\label{eq:2nd-homotopy}
& d^- r_{A|E} - r_{A|E} d^+ + r_{q^{-1}\Omega|E} - \psi \, r_{Z^{(1)}|E} = 0.
\end{align}
We combine them with \eqref{eq:nu-null-homotopy} to form
\begin{equation} \label{eq:chi-connection}
\begin{aligned}
& \nabla^{-1}: \mathit{CF}^*(E,H^+) \longrightarrow \mathit{CF}^*(E,H^-), \\
& \nabla^{-1} = C \partial_q - h^+ - r_{A|E} - \psi \chi,
\end{aligned}
\end{equation}
which satisfies 
\begin{align}
& d^- \nabla^{-1} - \nabla^{-1} d^+ = 0, \\
& \nabla^{-1}( fx) = f \nabla^{-1} x + (\partial_q f) C(x). \label{eq:connection-derivation}
\end{align}
Hence, the induced map on cohomology, which we also denote by
\begin{equation} \label{eq:chi-connection-induced}
\nabla^{-1}: \mathit{HF}^*(E,\alpha^+) \longrightarrow \mathit{HF}^*(E,\alpha^-),
\end{equation}
is a connection with respect to the continuation map between those Floer cohomology groups, meaning that it satisfies the cohomology level counterpart of \eqref{eq:connection-derivation}.

\begin{discussion} \label{th:hplus-hminus}
Alternatively, one could use the other homotopy from \eqref{eq:r-homotopy-omega}, which we now denote by $h^-$. It satisfies
\begin{equation}
d^- h^- - h^- d^+ + r_{q^{-1}\Omega|E} - (\partial_q d^- - d^- \partial_q) C = 0.
\end{equation}
This leads to another formula for a connection:
\begin{equation}
\partial_q C + h^- - r_{A|E} - \psi \chi,
\end{equation}
which however turns out to be the same as \eqref{eq:chi-connection} up to chain homotopy. This follows from the existence of a nullhomotopy
\begin{equation} \label{eq:h-plus-and-h-minus}
 \partial_q C - C \partial_q + h^- + h^+\htp 0.
\end{equation}
To construct it, one interprets $\partial_q C - C \partial_q$ geometrically as in \eqref{eq:differentiate-d}: it counts pairs consisting of a solution $u$ of the continuation map equation and a point $(\sigma,\tau) \in \bR \times S^1$ whose image goes through $q^{-1}\Omega|E$. Now, $h^-$ and $h^+$ are obtained from the same kind of moduli problem, but where the marked point is required to lie in one of the two halves of the cylinder. Adding up the two possibilites gives a space which is very similar to that underlying $\partial_q C - C \partial_q$: the choices of auxiliary data may not agree, but one can interpolate between them, and the associated parametrized moduli space gives rise to the homotopy in \eqref{eq:h-plus-and-h-minus}.
\end{discussion}

\subsection{Constructing the nullhomotopy\label{subsec:construct-chi}}
This section gives the proof of Lemma \ref{th:nu-null-homotopy}. Fix $\alpha^\pm$ and $(H^\pm,J^\pm)$ as before. We will introduce a moduli space with two parameters $(\sigma, \tau) \in (-\infty,0) \times S^1$. The parameters determine the position of two marked points 
\begin{equation} \label{eq:two-points-correlated}
\zeta_1 = (0,-\tau), \; \zeta_2 = (\sigma,-\tau) \in \bR \times S^1.
\end{equation}
We choose auxiliary data $(H^\chi,J^\chi)$ which depend on $(\sigma,\tau,s,t)$, and which on the ends, are asymptotically equal to the previously chosen $(H^\pm,J^\pm)$. In fact, we want our data to extend smoothly to $\sigma = 0$, and there, to agree with those used to define $r_{Z^{(1)}|E}$. We impose an additional requirement on those data, which is that $J^\chi_{0,\tau,0,-\tau} = \tilde{J}$ should be the almost complex structure used to define $Z^{(1)}$. We also ask that the data should extend to the other limit $\sigma = -\infty$. It is worth emphasizing what this means: as $\sigma \rightarrow -\infty$, the marked point $\zeta_2$ goes to $-\infty$, but the data $(H^\chi,J^\chi)$ converge uniformly over the entire cylinder (in any $C^r$ topology, with some exponential weights on the ends $s \rightarrow \pm\infty$).

We consider solutions of the associated equation of type \eqref{eq:continuation}, with the conditions
\begin{equation} \label{eq:updown-left}
\left\{
\begin{aligned}
& u(\zeta_1) \in F_0, \\
& u(\zeta_2) \in F_{\dag}, \\
& \mathrm{deg}(u) = 1, \\
& x^{\pm} = \textstyle \lim_{s \rightarrow \pm\infty} u(s,\cdot) \text{ lies outside $F_0$.}
\end{aligned}
\right.
\end{equation}

The limiting behaviour as $\sigma \rightarrow 0$ is bubbling off of a holomorphic sphere. More precisely, the relevant situation is as follows. The bubble is a $\tilde{J}$-holomorphic sphere $\tilde{u}$ which (suitably parametrized) satisfies \eqref{eq:pseudo-sphere-1}, \eqref{eq:pseudo-sphere-2}. The principal component of the limit satisfies $\mathrm{deg}(u) = 0$; and the two components are joined by the incidence condition $u(\zeta_1) = \tilde{u}(\tilde{\zeta}_3)$. This is exactly what ones sees when spelling out the definition of $r_{Z^{(1)}|E}$, and has the desired codimension $1$.
More complicated potential degenerations can be excluded because they have higher codimension.

For $\sigma \rightarrow -\infty$, the simplest degeneration would lead to a configuration consisting of two components $(\tilde{u},u)$, of which the first one is a Floer trajectory for $(H^-,J^-)$, and the second is a solution of an equation \eqref{eq:continuation} for the limit datum $(H^{\chi}_{-\infty,\tau},J^\chi_{-\infty,\tau})$. With $\tilde{\zeta} = \zeta = (0,-\tau)$, these would satisfy
\begin{equation} \label{eq:two-part-incidence}
\left\{
\begin{aligned}
& u(\zeta) \in F_0, \\
& \tilde{u}(\tilde{\zeta}) \in F_{\dag}, \\
& \mathrm{deg}(u) + \mathrm{deg}(\tilde{u}) = 1, \\
& x^- = \textstyle \lim_{s \rightarrow -\infty} \tilde{u}(s,\cdot) \text{ lies outside $F_0$,} \\
& \textstyle \lim_{s \rightarrow +\infty} \tilde{u}(s,\cdot) = \tilde{x} = \lim_{s \rightarrow -\infty} u(s,\cdot), \\
& x^+ = \textstyle \lim_{s \rightarrow +\infty} \tilde{u}(s,\cdot) \text{ lies outside $F_0$.}
\end{aligned}
\right.
\end{equation}
By Lemma \ref{th:nonnegative-2} and its analogue for the continuation map equation,
\begin{multline} \label{eq:degu-u}
\mathrm{deg}(u) + \mathrm{deg}(\tilde{u}) = u \cdot F_0 + \tilde{u} \cdot F_0 \\ +
\begin{cases} m^+(\tilde{u}) + m^-(u) & \text{if $\tilde{x}$ lies in $F_0$,} \\
0 & \text{otherwise.} \end{cases}  \qquad \qquad
\end{multline}
In the first case, $m^+(\tilde{u}) + m^-(u) \geq 1$ by \eqref{eq:rewritten-inequalities}; since $u \cdot F_0 \geq 1$, that would mean that \eqref{eq:degu-u} is at least $2$, which contradicts \eqref{eq:two-part-incidence}. Hence, it follows that $\tilde{x}$ lies outside $F_0$, and that in fact, $\mathrm{deg}(u) = 1$, $\mathrm{deg}(\tilde{u}) = 0$. But then, $\tilde{u}$ will never enter $p^{-1}(B \setminus \partial B)$, which is again a contradiction to \eqref{eq:two-part-incidence}. Hence, such limits do not after all exist.

Among the other potential degenerations as $\sigma \rightarrow -\infty$, there is one which deserves a separate discussion, since it leads to the restriction on $\alpha^{\pm}$ in Lemma \ref{th:nu-null-homotopy}. Namely, suppose that the limit consists of three components $(\tilde{u}^-,u,\tilde{u}^+)$ with the following properties. The principal component $u$ lies entirely in $F_0$; The other components $\tilde{u}^{\pm}$ are Floer trajectories for $(H^\pm,J^\pm)$, but do not quite play symmetrical roles: $\tilde{u}^-$ carries a marked point with an evaluation constraint, as in \eqref{eq:two-part-incidence}; $\tilde{u}^+$ does not, and is considered up to translation in $s$-direction. To be more specific, the counterpart of \eqref{eq:two-part-incidence} is
\begin{equation} \label{eq:three-part-incidence}
\left\{
\begin{aligned}
& u(\bR \times S^1) \subset F_0, \\
& \tilde{u}^-(\tilde{\zeta}) \in F_{\dag} \quad \text{for $\tilde{\zeta} = (0,-\tau)$}, \\
& \mathrm{deg}(\tilde{u}^-) + \mathrm{deg}(\tilde{u}^+) = 1, \\
& x^- = \textstyle \lim_{s \rightarrow -\infty} \tilde{u}^-(s,\cdot) \text{ lies outside $F_0$,} \\
& \textstyle \lim_{s \rightarrow +\infty} \tilde{u}^-(s,\cdot) = \tilde{x}^- = \lim_{s \rightarrow -\infty} u(s,\cdot) \text{ lies in $F_0$,} \\
& \textstyle \lim_{s \rightarrow +\infty} u(s,\cdot) = \tilde{x}^+ = \lim_{s \rightarrow -\infty} \tilde{u}^+(s,\cdot) \text{ lies in $F_0$,} \\
& x^+ = \textstyle \lim_{s \rightarrow +\infty} \tilde{u}^+(s,\cdot) \text{ lies outside $F_0$.}
\end{aligned}
\right.
\end{equation}
Ordinarily, since such a limit has three components, one would expect to be able to rule it out on the basis having of codimension $2$. In this case, $u$ originally came with an evaluation condition $u(\zeta) \in F_0$, which now holds tautologically (hence is non-transverse), which means that the dimension count has to be reconsidered. Assuming suitably generic choices, the actual dimension of the space of solutions $(\tau,\tilde{u}^-,u,\tilde{u}^+)$ is
\begin{equation} \label{eq:codim-computation}
\begin{aligned}
& 
\overbrace{1}^{\tau} + 
\overbrace{i(x^-) - i(\tilde{x}^-) + 2\,\mathrm{deg}(\tilde{u}^-)}^{\tilde{u}^-} 
- \overbrace{2}^{\tilde{u}^-(\tilde{\zeta}) \in F_\dag}
+ \overbrace{i_M(\tilde{x}^-) - i_M(\tilde{x}_+)}^u \\
& + \overbrace{i(\tilde{x}^+) - i(x^+) + 2\,\mathrm{deg}(\tilde{u}^+) - 1}^{\tilde{u}^+} \\
& \qquad \qquad
= (i(x^-) - i(x^+)) - 2 (\lfloor \alpha^- \rfloor - \lfloor \alpha^+ \rfloor),
\end{aligned}
\end{equation}
where the $\lfloor \alpha^\pm \rfloor$ terms come from \eqref{eq:two-indices}. The first term in the last line of \eqref{eq:codim-computation} is the dimension of the main moduli space, and the second term is therefore the codimension of the boundary stratum under consideration. Our assumption $[\alpha^+,\alpha^-] \cap \bZ \neq \emptyset$ is equivalent to $\lfloor \alpha^- \rfloor - \lfloor \alpha^+ \rfloor > 0$, which ensures codimension $\geq 2$ in \eqref{eq:codim-computation}, enough to make such degenerations irrelevant. However, it turns out that in this case there is a simpler topological argument which excludes their existence. Namely, Lemma \ref{th:nonnegative-2} and \eqref{eq:rewritten-inequalities} lead to a contradiction with \eqref{eq:three-part-incidence}:
\begin{equation}
\mathrm{deg}(\tilde{u}^-) + \mathrm{deg}(\tilde{u}^+) \geq m^+(\tilde{u}^-) + m^-(\tilde{u}^+)
\geq  \lfloor \alpha^- \rfloor + 1 - \lfloor \alpha^+ \rfloor \geq 2.
\end{equation}
We have considered only the simplest possible limits, but more complicated ones can be disregarded for codimension reasons. The upshot is that counting isolated points in our parametrized moduli space indeed provides a map $\chi$ with the desired property \eqref{eq:nu-null-homotopy}, concluding our proof of Lemma \ref{th:nu-null-homotopy}.

\subsection{Applying the connection to the identity element}
Our next goal is:

\begin{proposition} \label{th:nabla-e}
The connection \eqref{eq:chi-connection} has the property that, for $\psi$ as in \eqref{eq:express-o},
\begin{equation} \label{eq:nabla-applied-to-e}
\nabla^{-1} e = \psi s + \text{\it coboundary}.
\end{equation}
\end{proposition}
\begin{figure}
\begin{centering}
\begin{picture}(0,0)%
\includegraphics{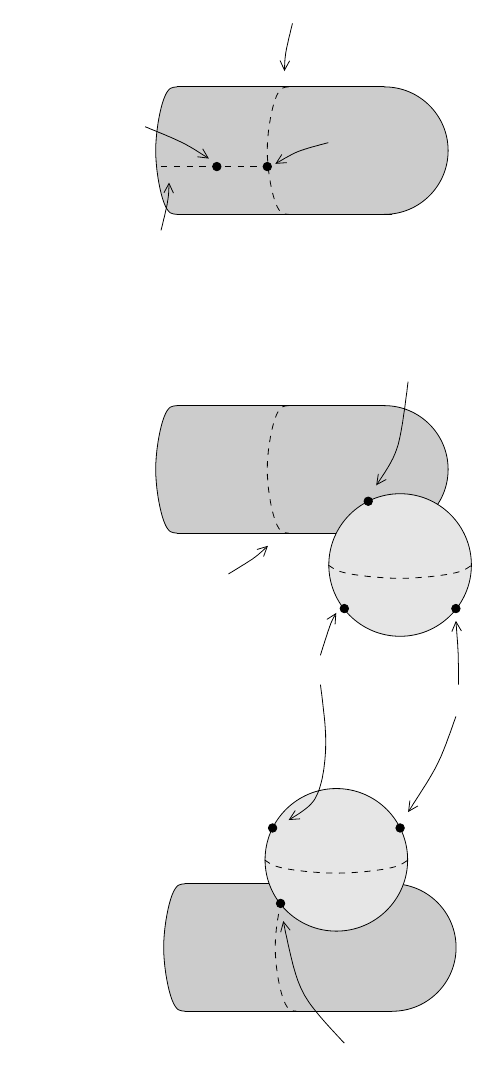}%
\end{picture}%
\setlength{\unitlength}{3355sp}%
\begingroup\makeatletter\ifx\SetFigFont\undefined%
\gdef\SetFigFont#1#2#3#4#5{%
  \reset@font\fontsize{#1}{#2pt}%
  \fontfamily{#3}\fontseries{#4}\fontshape{#5}%
  \selectfont}%
\fi\endgroup%
\begin{picture}(4530,10090)(-3314,-10541)
\put(-2849,-1561){\makebox(0,0)[lb]{\smash{{\SetFigFont{10}{12.0}{\rmdefault}{\mddefault}{\updefault}{\color[rgb]{0,0,0}$u(\zeta_2) \in F_\dag$}%
}}}}
\put(-1049,-586){\makebox(0,0)[lb]{\smash{{\SetFigFont{10}{12.0}{\rmdefault}{\mddefault}{\updefault}{\color[rgb]{0,0,0}$\zeta_1$ moves on this circle}%
}}}}
\put(-674,-10486){\makebox(0,0)[lb]{\smash{{\SetFigFont{10}{12.0}{\rmdefault}{\mddefault}{\updefault}{\color[rgb]{0,0,0}$\zeta$ lies on the circle}%
}}}}
\put(-3299,-7561){\makebox(0,0)[lb]{\smash{{\SetFigFont{10}{12.0}{\rmdefault}{\mddefault}{\updefault}{\color[rgb]{0,0,0}\underline{common boundary component:}}%
}}}}
\put(-3299,-6061){\makebox(0,0)[lb]{\smash{{\SetFigFont{10}{12.0}{\rmdefault}{\mddefault}{\updefault}{\color[rgb]{0,0,0}$\zeta$ lies to the right of this circle}%
}}}}
\put(151,-7111){\makebox(0,0)[lb]{\smash{{\SetFigFont{10}{12.0}{\rmdefault}{\mddefault}{\updefault}{\color[rgb]{0,0,0}this point maps to $F_0$}%
}}}}
\put(1201,-4861){\makebox(0,0)[lb]{\smash{{\SetFigFont{10}{12.0}{\rmdefault}{\mddefault}{\updefault}{\color[rgb]{0,0,0}$\epsilon_+$}%
}}}}
\put(-1424,-3961){\makebox(0,0)[lb]{\smash{{\SetFigFont{10}{12.0}{\rmdefault}{\mddefault}{\updefault}{\color[rgb]{0,0,0}$u(\zeta)$ lies on a pseudo-holomorphic sphere}%
}}}}
\put(-149,-1786){\makebox(0,0)[lb]{\smash{{\SetFigFont{10}{12.0}{\rmdefault}{\mddefault}{\updefault}{\color[rgb]{0,0,0}$u(\zeta_1) \in F_0$}%
}}}}
\put(1201,-1861){\makebox(0,0)[lb]{\smash{{\SetFigFont{10}{12.0}{\rmdefault}{\mddefault}{\updefault}{\color[rgb]{0,0,0}$\epsilon_-$}%
}}}}
\put(-3299,-2836){\makebox(0,0)[lb]{\smash{{\SetFigFont{10}{12.0}{\rmdefault}{\mddefault}{\updefault}{\color[rgb]{0,0,0}$\zeta_2$ moves on the half-infinite line determined by $\zeta_1$}%
}}}}
\put(-1199,-6811){\makebox(0,0)[lb]{\smash{{\SetFigFont{10}{12.0}{\rmdefault}{\mddefault}{\updefault}{\color[rgb]{0,0,0}this point maps to $F_\dag$}%
}}}}
\end{picture}%
\caption{\label{fig:two-epsilons}The cochains $\epsilon_{\pm}$ from \eqref{eq:two-epsilons}.}
\end{centering}
\end{figure}
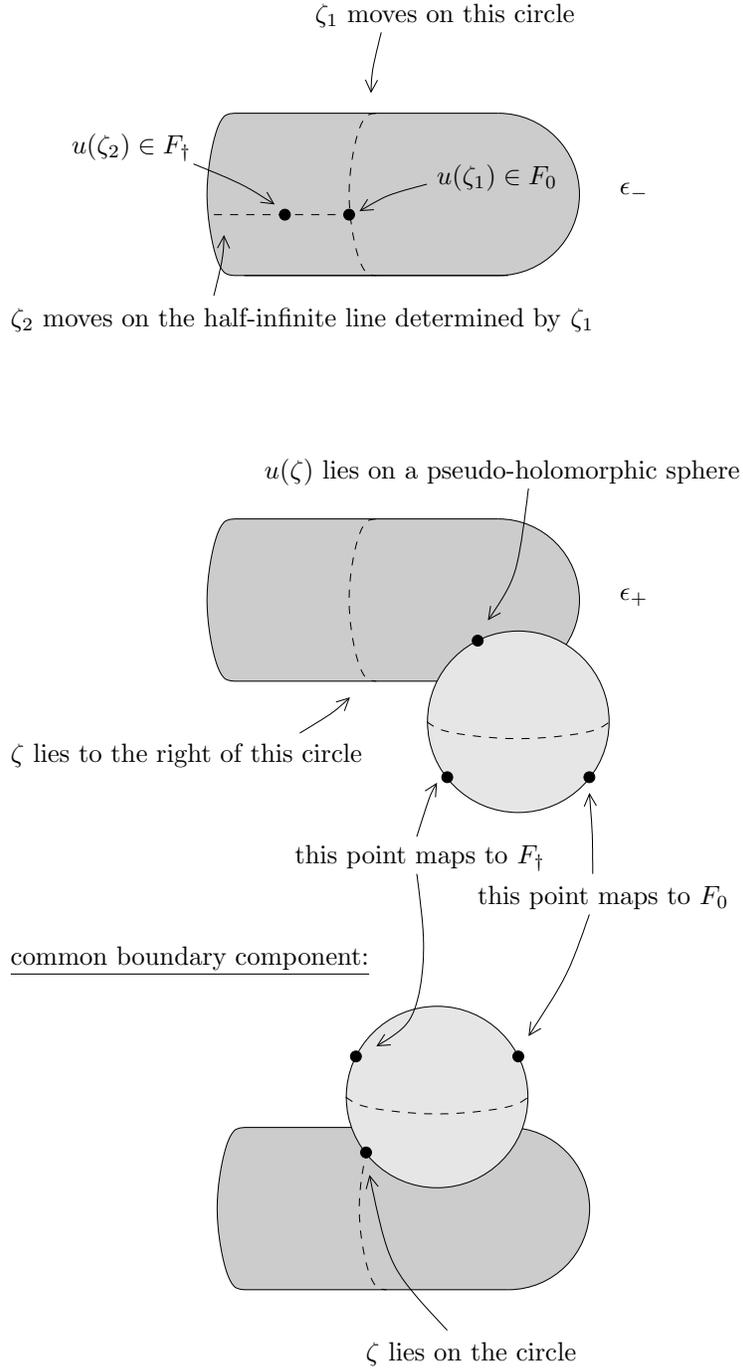%

The proof hinges on introducing an intermediate object, a cocycle (see Figure \ref{fig:two-epsilons} for a schematic description)
\begin{equation} \label{eq:two-epsilons}
\epsilon_- + \epsilon_+ \in \mathit{CF}^0(E,H),
\end{equation}
which will then be compared to both sides of \eqref{eq:nabla-applied-to-e}. The construction of $\epsilon_-$ is very similar to that in Lemma \ref{th:nu-null-homotopy}, which means that it uses the same parameters $(\sigma,\tau)$ and associated marked points $(\zeta_1,\zeta_2)$, except that we are now working on the thimble surface \eqref{eq:thimble}. Accordingly, we consider auxiliary data $(H^{\epsilon_-},J^{\epsilon_-})$ on that surface, with limit $(H,J)$ over the end. As before, we require that these data should extend smoothly to $\sigma = -\infty$ and $\sigma = 0$. In the case of $\sigma = 0$, we impose the previous restriction that $J^{\epsilon_-}_{0,\tau,0,-\tau} = \tilde{J}$ should be the almost complex structure used to define $Z^{(1)}$. The maps $u$ are required to satisfy  \eqref{eq:updown-left} (we only have one limit instead of two, of course). As for the ends of one-dimensioal moduli spaces, we get bubbling off of $\tilde{J}$-holomorphic spheres as $\sigma \rightarrow 0$. The other limit $\sigma \rightarrow -\infty$ does not contribute, provided that $\alpha > 1$, for the same reason as in Section \ref{subsec:construct-chi}.

To define $\epsilon_+$, one considers another moduli space of maps on the thimble, where the parameters are given by a choice of point $\zeta \in T$ lying in the interior of the circle $\{0\} \times S^1$ (which means that either $\zeta = (\sigma,-\tau)$ with $\sigma > 0$, or $\zeta = \infty$). We want data $(H^{\epsilon_+},J^{\epsilon_+})$ as before, extending smoothly to the case $\sigma = 0$, and 
which in that case match up with $(H^{\epsilon_-},J^{\epsilon_-})$. Moreover, we now impose the condition $J^{\epsilon_+}_{\zeta,\zeta} = \tilde{J}$ for all $\zeta$ (to clarify: for every value of the parameter $\zeta$, we have a family of almost complex structures varying over the thimble, and we constrain that family at one point, whose position is given by $\zeta$). The maps $u$ should satisfy
\begin{equation} \label{eq:epsilon-plus}
\left\{
\begin{aligned}
& u(\zeta) \in Z^{(1)}|E, \\
& \mathrm{deg}(u) = 0, \\
& x = \mathrm{lim}_{s \rightarrow -\infty} u(s,\cdot) \text{ lies outside $F_0$.}
\end{aligned}
\right.
\end{equation}
The first line of \eqref{eq:epsilon-plus} is shorthand: we consider the moduli spaces such that $u(\zeta)$ goes through some component of $Z^{(1)}$ (because $\mathrm{deg}(u) = 0$, $u(\zeta)$ automatically lies in $E$), and add them up with suitable multiplicities. Another way to see this is to write the moduli space as one of pairs $(u,\tilde{u})$, where $\tilde{u}$ is a $\tilde{J}$-holomorphic sphere as in \eqref{eq:pseudo-sphere-1}, \eqref{eq:pseudo-sphere-2}, connected to $u$ by the incidence condition $u(\zeta) = \tilde{u}(\zeta_3)$ (this is what we drew in Figure \ref{fig:two-epsilons}). 

The important point is that at $\sigma = 0$, the two moduli spaces match. Hence, the associated contributions to $d\epsilon_-$ and $d\epsilon_+$ will cancel, which ensures that \eqref{eq:two-epsilons} is a cocycle (provided that $\alpha>1$).

\begin{lemma} \label{th:many}
The cocycle \eqref{eq:two-epsilons} is related to the connection \eqref{eq:chi-connection} by
\begin{equation} \label{eq:many}
\nabla^{-1} e = \psi (\epsilon_- + \epsilon_+) + \text{\it coboundary}.
\end{equation}
\end{lemma}

It is worth while clarifying in which Floer cochain group this statement is supposed to hold. We choose $\alpha^+ > 0$, and then $\alpha^-$ such that $[\alpha^+,\alpha^-] \cap \bZ \neq \emptyset$. On the left side of \eqref{eq:many}, $e \in \mathit{CF}^*(H^+,J^+)$, and then $\nabla^{-1} e \in \mathit{CF}^*(H^-,J^-)$. On the right side, one constructs \eqref{eq:two-epsilons} using $\alpha = \alpha^-$ (which is $>1$ by assumption) and $(H,J) = (H^-,J^-)$. 
\begin{figure}
\begin{centering}
\begin{picture}(0,0)%
\includegraphics{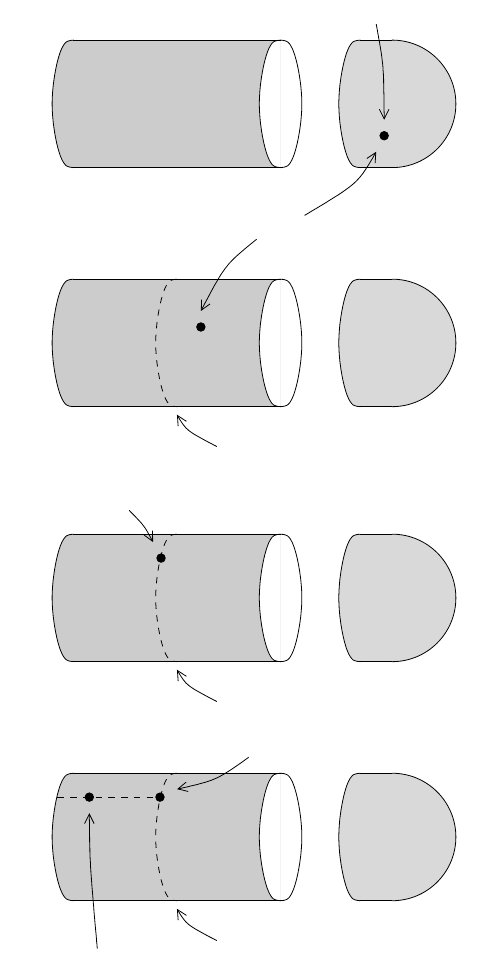}%
\end{picture}%
\setlength{\unitlength}{3355sp}%
\begingroup\makeatletter\ifx\SetFigFont\undefined%
\gdef\SetFigFont#1#2#3#4#5{%
  \reset@font\fontsize{#1}{#2pt}%
  \fontfamily{#3}\fontseries{#4}\fontshape{#5}%
  \selectfont}%
\fi\endgroup%
\begin{picture}(4605,9211)(-2114,-13250)
\put(-2099,-5011){\makebox(0,0)[lb]{\smash{{\SetFigFont{10}{12.0}{\rmdefault}{\mddefault}{\updefault}{\color[rgb]{0,0,0}(i)}%
}}}}
\put(-2099,-7261){\makebox(0,0)[lb]{\smash{{\SetFigFont{10}{12.0}{\rmdefault}{\mddefault}{\updefault}{\color[rgb]{0,0,0}(ii)}%
}}}}
\put(2476,-4936){\makebox(0,0)[lb]{\smash{{\SetFigFont{10}{12.0}{\rmdefault}{\mddefault}{\updefault}{\color[rgb]{0,0,0}$c(\partial_q(e))$}%
}}}}
\put(2476,-7261){\makebox(0,0)[lb]{\smash{{\SetFigFont{10}{12.0}{\rmdefault}{\mddefault}{\updefault}{\color[rgb]{0,0,0}$h(e)$}%
}}}}
\put(226,-11011){\makebox(0,0)[lb]{\smash{{\SetFigFont{10}{12.0}{\rmdefault}{\mddefault}{\updefault}{\color[rgb]{0,0,0}the image of $\zeta_1$ lies on $F_0$}%
}}}}
\put(2476,-9661){\makebox(0,0)[lb]{\smash{{\SetFigFont{10}{12.0}{\rmdefault}{\mddefault}{\updefault}{\color[rgb]{0,0,0}$r_{A|E}(e)$}%
}}}}
\put(-1349,-8461){\makebox(0,0)[lb]{\smash{{\SetFigFont{10}{12.0}{\rmdefault}{\mddefault}{\updefault}{\color[rgb]{0,0,0}$\zeta$ lies in the half-cylinder to the right of this circle}%
}}}}
\put( 76,-10711){\makebox(0,0)[lb]{\smash{{\SetFigFont{10}{12.0}{\rmdefault}{\mddefault}{\updefault}{\color[rgb]{0,0,0}$\zeta$ lies on this circle}%
}}}}
\put(  1,-12886){\makebox(0,0)[lb]{\smash{{\SetFigFont{10}{12.0}{\rmdefault}{\mddefault}{\updefault}{\color[rgb]{0,0,0}$\zeta_1$ lies on this circle}%
}}}}
\put(-2099,-9661){\makebox(0,0)[lb]{\smash{{\SetFigFont{10}{12.0}{\rmdefault}{\mddefault}{\updefault}{\color[rgb]{0,0,0}(iii)}%
}}}}
\put(-2099,-11911){\makebox(0,0)[lb]{\smash{{\SetFigFont{10}{12.0}{\rmdefault}{\mddefault}{\updefault}{\color[rgb]{0,0,0}(iv)}%
}}}}
\put(-1874,-8761){\makebox(0,0)[lb]{\smash{{\SetFigFont{10}{12.0}{\rmdefault}{\mddefault}{\updefault}{\color[rgb]{0,0,0}the image of $\zeta$ goes through $A|E$}%
}}}}
\put(-524,-6211){\makebox(0,0)[lb]{\smash{{\SetFigFont{10}{12.0}{\rmdefault}{\mddefault}{\updefault}{\color[rgb]{0,0,0}the image of $\zeta$ goes through $q^{-1}\Omega|E$}%
}}}}
\put(2476,-11911){\makebox(0,0)[lb]{\smash{{\SetFigFont{10}{12.0}{\rmdefault}{\mddefault}{\updefault}{\color[rgb]{0,0,0}$\chi(e)$}%
}}}}
\put(226,-4186){\makebox(0,0)[lb]{\smash{{\SetFigFont{10}{12.0}{\rmdefault}{\mddefault}{\updefault}{\color[rgb]{0,0,0}$\zeta$ can be anywhere on the thimble}%
}}}}
\put(-1649,-13186){\makebox(0,0)[lb]{\smash{{\SetFigFont{10}{12.0}{\rmdefault}{\mddefault}{\updefault}{\color[rgb]{0,0,0}the image of $\zeta_2$ lies on $F_\dag$}%
}}}}
\end{picture}%
\caption{\label{fig:nabla-e}Applying \eqref{eq:chi-connection} to the identity element.}
\end{centering}
\end{figure}
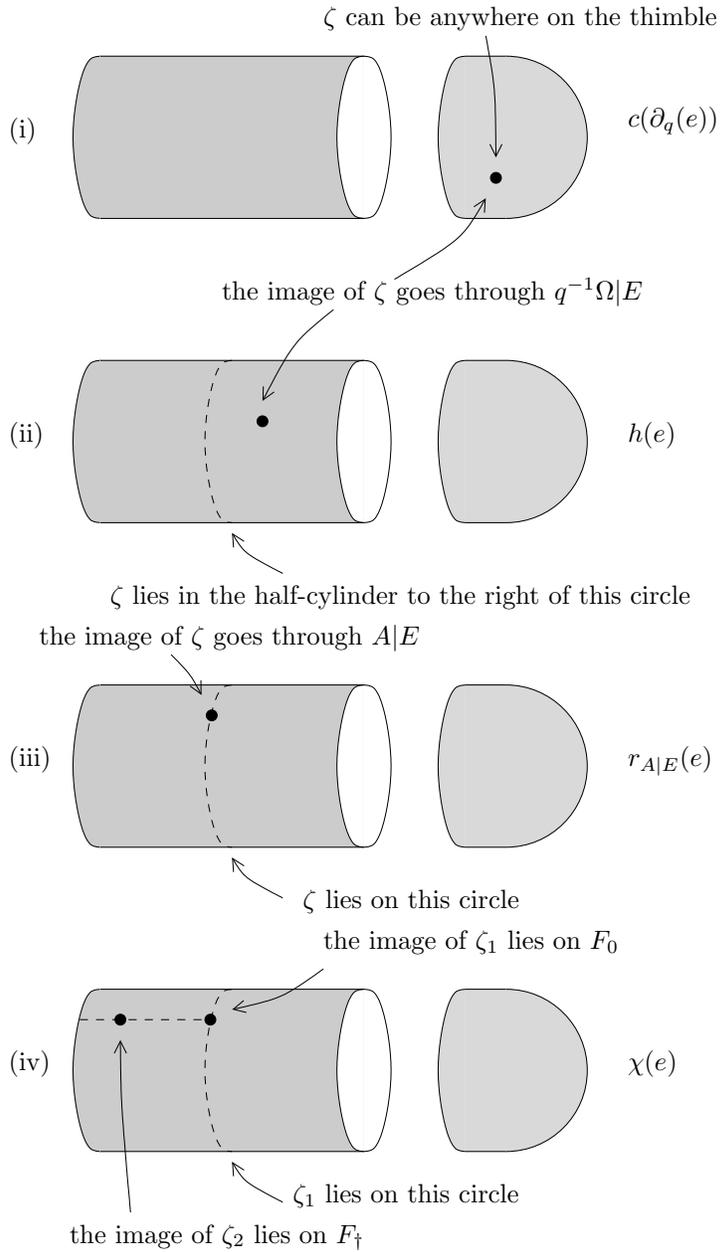%

\begin{proof}
As usual, the argument constructs an explicit coboundary which measures the discrepancy between the two sides of \eqref{eq:many}. Even though none of the steps is surprising or difficult, the number of terms threatens to be notationally overwhelming, and we will therefore use a graphical shorthand description. The starting point is to spell out what $\nabla^{-1} e$ is, which we have done in Figure \ref{fig:nabla-e}. All the terms are obvious from \eqref{eq:chi-connection} except for $\partial_q e$, which we interpret as thimbles with a marked point, in exact analogy with what we have done for \eqref{eq:differentiate-d}. To construct the coboundary, one adds up contributions from four parametrized moduli spaces, whose construction is outlined in Figure \ref{fig:1st-step}--\ref{fig:4th-step}, and where the contribution of the last of the four should be multiplied by $\psi$.
\begin{figure}
\begin{centering}
\begin{picture}(0,0)%
\includegraphics{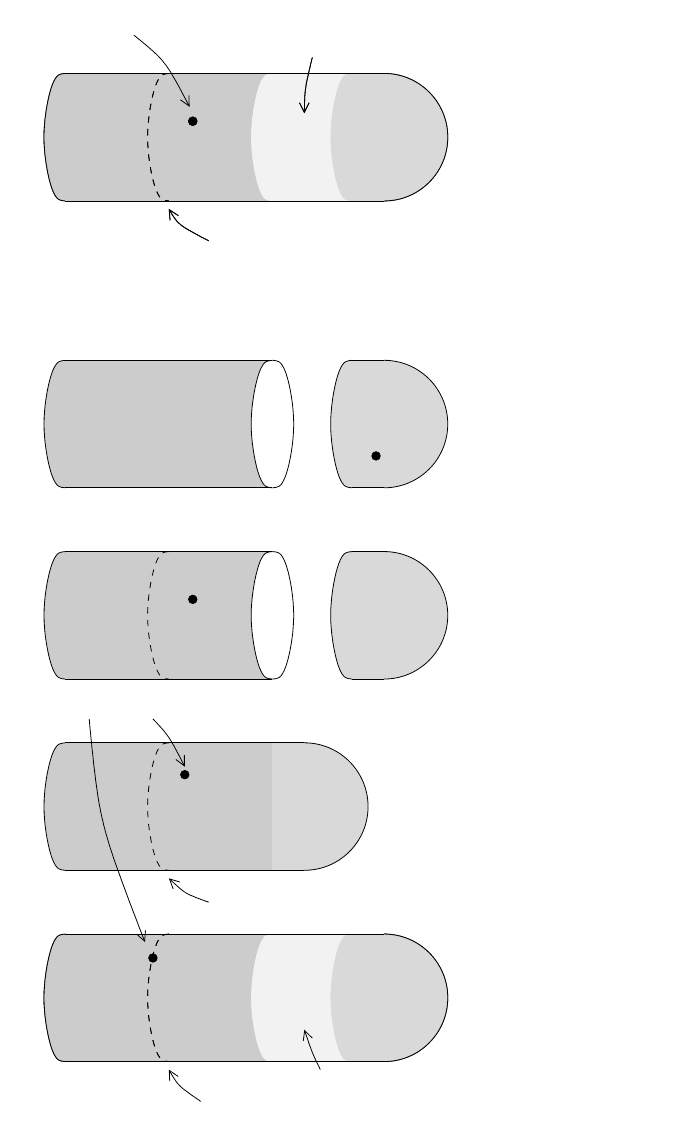}%
\end{picture}%
\setlength{\unitlength}{3355sp}%
\begingroup\makeatletter\ifx\SetFigFont\undefined%
\gdef\SetFigFont#1#2#3#4#5{%
  \reset@font\fontsize{#1}{#2pt}%
  \fontfamily{#3}\fontseries{#4}\fontshape{#5}%
  \selectfont}%
\fi\endgroup%
\begin{picture}(6474,10639)(-2039,-11666)
\put(  1,-9586){\makebox(0,0)[lb]{\smash{{\SetFigFont{10}{12.0}{\rmdefault}{\mddefault}{\updefault}{\color[rgb]{0,0,0}$\zeta$ lies to the right of this circle}%
}}}}
\put(-2024,-10411){\makebox(0,0)[lb]{\smash{{\SetFigFont{10}{12.0}{\rmdefault}{\mddefault}{\updefault}{\color[rgb]{0,0,0}(iv)}%
}}}}
\put(-2024,-8611){\makebox(0,0)[lb]{\smash{{\SetFigFont{10}{12.0}{\rmdefault}{\mddefault}{\updefault}{\color[rgb]{0,0,0}(iii)}%
}}}}
\put(-2024,-6811){\makebox(0,0)[lb]{\smash{{\SetFigFont{10}{12.0}{\rmdefault}{\mddefault}{\updefault}{\color[rgb]{0,0,0}(ii)}%
}}}}
\put(-2024,-5011){\makebox(0,0)[lb]{\smash{{\SetFigFont{10}{12.0}{\rmdefault}{\mddefault}{\updefault}{\color[rgb]{0,0,0}(i)}%
}}}}
\put(-1349,-7711){\makebox(0,0)[lb]{\smash{{\SetFigFont{10}{12.0}{\rmdefault}{\mddefault}{\updefault}{\color[rgb]{0,0,0}the image of $\zeta$ goes through $q^{-1}\Omega|E$}%
}}}}
\put(226,-11311){\makebox(0,0)[lb]{\smash{{\SetFigFont{10}{12.0}{\rmdefault}{\mddefault}{\updefault}{\color[rgb]{0,0,0}an additional parameter stretches this part}%
}}}}
\put(2551,-4936){\makebox(0,0)[lb]{\smash{{\SetFigFont{10}{12.0}{\rmdefault}{\mddefault}{\updefault}{\color[rgb]{0,0,0}$C(\partial_q(e))$}%
}}}}
\put(2551,-6811){\makebox(0,0)[lb]{\smash{{\SetFigFont{10}{12.0}{\rmdefault}{\mddefault}{\updefault}{\color[rgb]{0,0,0}$h^+(e)$}%
}}}}
\put(-1874,-3961){\makebox(0,0)[lb]{\smash{{\SetFigFont{10}{12.0}{\rmdefault}{\mddefault}{\updefault}{\color[rgb]{0,0,0}\underline{the relevant boundary/degeneration contributions are:}}%
}}}}
\put(451,-1486){\makebox(0,0)[lb]{\smash{{\SetFigFont{10}{12.0}{\rmdefault}{\mddefault}{\updefault}{\color[rgb]{0,0,0}an additional parameter stretches this part}%
}}}}
\put(-374,-11611){\makebox(0,0)[lb]{\smash{{\SetFigFont{10}{12.0}{\rmdefault}{\mddefault}{\updefault}{\color[rgb]{0,0,0}$\zeta$ lies on this circle}%
}}}}
\put(  1,-3436){\makebox(0,0)[lb]{\smash{{\SetFigFont{10}{12.0}{\rmdefault}{\mddefault}{\updefault}{\color[rgb]{0,0,0}$\zeta$ lies to the right of this circle}%
}}}}
\put(  1,-3436){\makebox(0,0)[lb]{\smash{{\SetFigFont{10}{12.0}{\rmdefault}{\mddefault}{\updefault}{\color[rgb]{0,0,0}$\zeta$ lies to the right of this circle}%
}}}}
\put(-1499,-1186){\makebox(0,0)[lb]{\smash{{\SetFigFont{10}{12.0}{\rmdefault}{\mddefault}{\updefault}{\color[rgb]{0,0,0}the image of $\zeta$ goes through $q^{-1}\Omega|E$}%
}}}}
\end{picture}%
\caption{\label{fig:1st-step}The first ingredient in the proof of Lemma \ref{th:many}.}
\end{centering}
\end{figure}%

In words, what we do in the first moduli space (Figure \ref{fig:1st-step}) is to take the two surfaces (a cylinder and a thimble) involved in defining $C(\partial_q(e))$, and glue them together (introducing a gluing parameter). Moreover, once we carry out the gluing, the marked point is allowed to move freely over the glued surface, as long as it remains to the right of the circle inherited from $\{0\} \times S^1 \subset \bR \times S^1$. The maps $u$ involved are all assumed to satisfy $\mathrm{deg}(u) = 0$, and therefore remain inside $E$. There are two different (codimension $1$) degenerations as the gluing parameter goes to infinity, depending on which component of the limit the marked point ends up in, see Figure \ref{fig:1st-step}(i) and (ii); those two components contribute $C(\partial_q(e))$ and $h^+(e)$, respectively, where $h^+$ is as in \eqref{eq:1st-homotopy}. There are two other relevant boundary components: one where the gluing length has reached its (arbitrarily chosen) minimum, Figure \ref{fig:1st-step}(iii); and one where the marked point moves as far to the left as is allowed, Figure \ref{fig:1st-step}(iv). In those two last-mentioned cases, the resulting contribution can not be expressed in terms of previously introduced operations.
\begin{figure}
\begin{centering}
\begin{picture}(0,0)%
\includegraphics{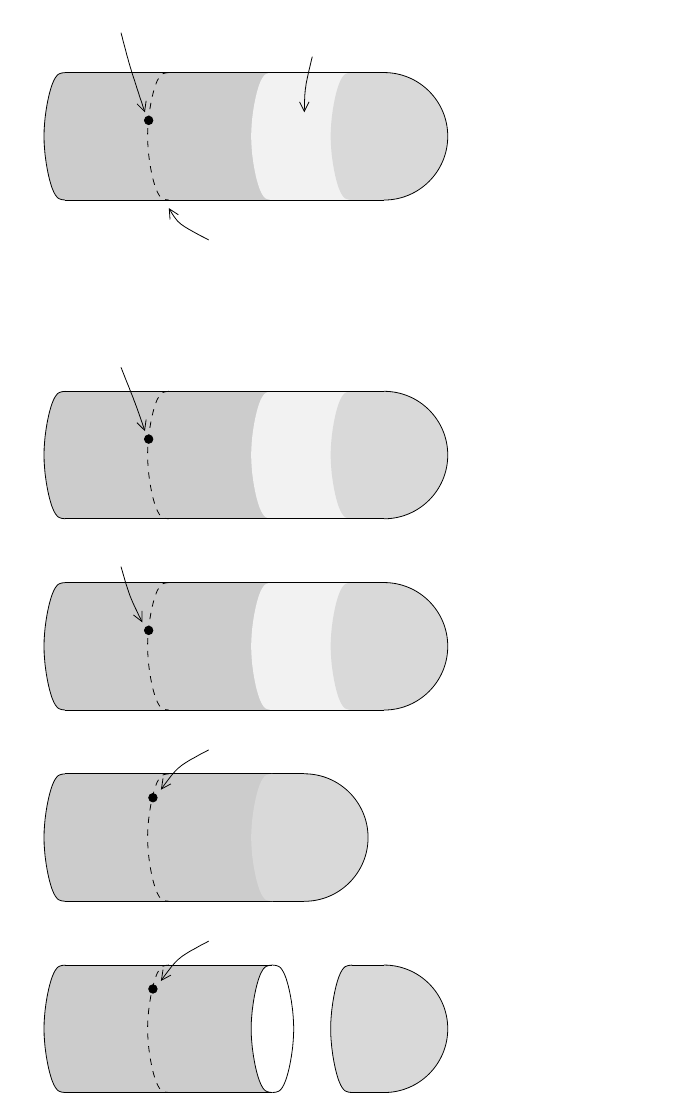}%
\end{picture}%
\setlength{\unitlength}{3355sp}%
\begingroup\makeatletter\ifx\SetFigFont\undefined%
\gdef\SetFigFont#1#2#3#4#5{%
  \reset@font\fontsize{#1}{#2pt}%
  \fontfamily{#3}\fontseries{#4}\fontshape{#5}%
  \selectfont}%
\fi\endgroup%
\begin{picture}(6474,10284)(-2039,-11248)
\put(451,-1411){\makebox(0,0)[lb]{\smash{{\SetFigFont{10}{12.0}{\rmdefault}{\mddefault}{\updefault}{\color[rgb]{0,0,0}an additional parameter stretches this part}%
}}}}
\put(-2024,-5236){\makebox(0,0)[lb]{\smash{{\SetFigFont{10}{12.0}{\rmdefault}{\mddefault}{\updefault}{\color[rgb]{0,0,0}(i)}%
}}}}
\put(-2024,-7036){\makebox(0,0)[lb]{\smash{{\SetFigFont{10}{12.0}{\rmdefault}{\mddefault}{\updefault}{\color[rgb]{0,0,0}(ii)}%
}}}}
\put(-2024,-8911){\makebox(0,0)[lb]{\smash{{\SetFigFont{10}{12.0}{\rmdefault}{\mddefault}{\updefault}{\color[rgb]{0,0,0}(iii)}%
}}}}
\put(-2024,-10636){\makebox(0,0)[lb]{\smash{{\SetFigFont{10}{12.0}{\rmdefault}{\mddefault}{\updefault}{\color[rgb]{0,0,0}(iv)}%
}}}}
\put(  1,-9811){\makebox(0,0)[lb]{\smash{{\SetFigFont{10}{12.0}{\rmdefault}{\mddefault}{\updefault}{\color[rgb]{0,0,0}the image of $\zeta$ goes through $A|E$}%
}}}}
\put(2551,-5236){\makebox(0,0)[lb]{\smash{{\SetFigFont{10}{12.0}{\rmdefault}{\mddefault}{\updefault}{\color[rgb]{0,0,0}see Fig. \ref{fig:1st-step}(iv)}%
}}}}
\put(  1,-3361){\makebox(0,0)[lb]{\smash{{\SetFigFont{10}{12.0}{\rmdefault}{\mddefault}{\updefault}{\color[rgb]{0,0,0}$\zeta$ lies on this circle}%
}}}}
\put(  1,-8011){\makebox(0,0)[lb]{\smash{{\SetFigFont{10}{12.0}{\rmdefault}{\mddefault}{\updefault}{\color[rgb]{0,0,0}the image of $\zeta$ goes through $A|E$}%
}}}}
\put(-1349,-6211){\makebox(0,0)[lb]{\smash{{\SetFigFont{10}{12.0}{\rmdefault}{\mddefault}{\updefault}{\color[rgb]{0,0,0}the image of $\zeta$ goes through $\psi Z^{(1)}|E$}%
}}}}
\put(-1349,-4336){\makebox(0,0)[lb]{\smash{{\SetFigFont{10}{12.0}{\rmdefault}{\mddefault}{\updefault}{\color[rgb]{0,0,0}the image of $\zeta$ goes through $q^{-1}\Omega|E$}%
}}}}
\put(-1874,-3886){\makebox(0,0)[lb]{\smash{{\SetFigFont{10}{12.0}{\rmdefault}{\mddefault}{\updefault}{\color[rgb]{0,0,0}\underline{the relevant boundary/degeneration contributions are:}}%
}}}}
\put(-1349,-1111){\makebox(0,0)[lb]{\smash{{\SetFigFont{10}{12.0}{\rmdefault}{\mddefault}{\updefault}{\color[rgb]{0,0,0}the image of $\zeta$ goes through $A|E$}%
}}}}
\put(2551,-10636){\makebox(0,0)[lb]{\smash{{\SetFigFont{10}{12.0}{\rmdefault}{\mddefault}{\updefault}{\color[rgb]{0,0,0}$r_{A|E}(e)$}%
}}}}
\end{picture}%
\caption{\label{fig:2nd-step}The second ingredient in the proof of Lemma \ref{th:many}.}
\end{centering}
\end{figure}%

The construction of the second moduli space (Figure \ref{fig:2nd-step}) uses the same idea as the first one, but now applied to $r_{A|E}(e)$, in the notation from \eqref{eq:2nd-homotopy}. Since here, the evaluation at the marked point is assumed to go through the proper pseudo-chain $A|E$, we obtain two boundary components corresponding to the terms in \eqref{eq:a-restricted}, shown in Figure \ref{fig:2nd-step}(i) and (ii); the first of them produces the same contribution (up to sign) as one of our previous degenerations, Figure \ref{fig:1st-step}(iv). Another contribution comes from reaching the minimal gluing length, Figure \ref{fig:2nd-step}(iii). The final one appears when the gluing length goes to infinity, which by design recovers $r_{A|E}(e)$.
\begin{figure}
\begin{centering}
\begin{picture}(0,0)%
\includegraphics{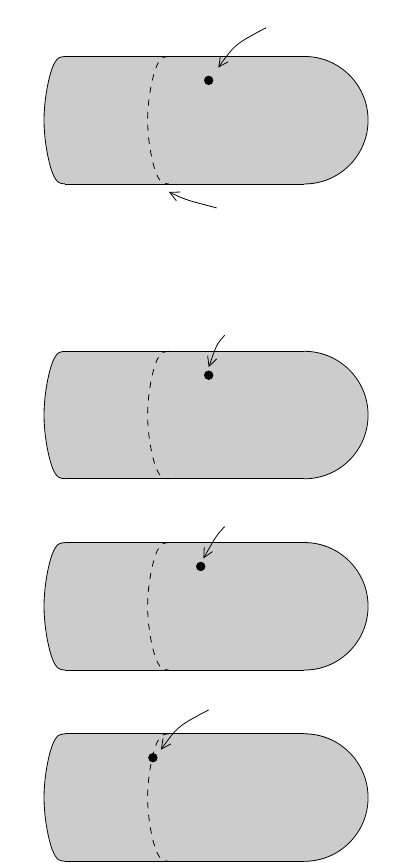}%
\end{picture}%
\setlength{\unitlength}{3355sp}%
\begingroup\makeatletter\ifx\SetFigFont\undefined%
\gdef\SetFigFont#1#2#3#4#5{%
  \reset@font\fontsize{#1}{#2pt}%
  \fontfamily{#3}\fontseries{#4}\fontshape{#5}%
  \selectfont}%
\fi\endgroup%
\begin{picture}(3930,8109)(-2039,-9448)
\put(1801,-5236){\makebox(0,0)[lb]{\smash{{\SetFigFont{10}{12.0}{\rmdefault}{\mddefault}{\updefault}{\color[rgb]{0,0,0}see Fig. \ref{fig:1st-step}(iii)}%
}}}}
\put(  1,-6211){\makebox(0,0)[lb]{\smash{{\SetFigFont{10}{12.0}{\rmdefault}{\mddefault}{\updefault}{\color[rgb]{0,0,0}the image of $\zeta$ goes through $\psi Z^{(1)}|E$}%
}}}}
\put(  1,-8011){\makebox(0,0)[lb]{\smash{{\SetFigFont{10}{12.0}{\rmdefault}{\mddefault}{\updefault}{\color[rgb]{0,0,0}the image of $\zeta$ goes through $A|E$}%
}}}}
\put(  1,-4411){\makebox(0,0)[lb]{\smash{{\SetFigFont{10}{12.0}{\rmdefault}{\mddefault}{\updefault}{\color[rgb]{0,0,0}the image of $\zeta$ goes through $q^{-1}\Omega|E$}%
}}}}
\put(-1874,-3886){\makebox(0,0)[lb]{\smash{{\SetFigFont{10}{12.0}{\rmdefault}{\mddefault}{\updefault}{\color[rgb]{0,0,0}\underline{the relevant boundary/degeneration contributions are:}}%
}}}}
\put(1876,-8836){\makebox(0,0)[lb]{\smash{{\SetFigFont{10}{12.0}{\rmdefault}{\mddefault}{\updefault}{\color[rgb]{0,0,0}see Fig. \ref{fig:2nd-step}(iii)}%
}}}}
\put(1801,-7036){\makebox(0,0)[lb]{\smash{{\SetFigFont{10}{12.0}{\rmdefault}{\mddefault}{\updefault}{\color[rgb]{0,0,0}$\psi\epsilon_+$}%
}}}}
\put(-299,-1486){\makebox(0,0)[lb]{\smash{{\SetFigFont{10}{12.0}{\rmdefault}{\mddefault}{\updefault}{\color[rgb]{0,0,0}the image of $\zeta$ goes through $A|E$}%
}}}}
\put(-2024,-5236){\makebox(0,0)[lb]{\smash{{\SetFigFont{10}{12.0}{\rmdefault}{\mddefault}{\updefault}{\color[rgb]{0,0,0}(i)}%
}}}}
\put(-2024,-7036){\makebox(0,0)[lb]{\smash{{\SetFigFont{10}{12.0}{\rmdefault}{\mddefault}{\updefault}{\color[rgb]{0,0,0}(ii)}%
}}}}
\put(-2024,-8911){\makebox(0,0)[lb]{\smash{{\SetFigFont{10}{12.0}{\rmdefault}{\mddefault}{\updefault}{\color[rgb]{0,0,0}(iii)}%
}}}}
\put( 76,-3511){\makebox(0,0)[lb]{\smash{{\SetFigFont{10}{12.0}{\rmdefault}{\mddefault}{\updefault}{\color[rgb]{0,0,0}$\zeta$ lies to the right of this circle}%
}}}}
\end{picture}%
\caption{\label{fig:3rd-step}The third ingredient in the proof of Lemma \ref{th:many}.}
\end{centering}
\end{figure}%

The third moduli space (Figure \ref{fig:3rd-step}) does not involve any neck-stretch\-ing, and the marked point lies in a bounded subset of the thimble Riemann surface. As in the previous two cases, the maps involved have $\mathrm{deg}(u) = 0$, hence remain inside $E$. The three boundary contributions appear for very straightforward reasons; two of them, Figures \ref{fig:3rd-step}(i) and (iii), match previously encountered ones, whereas the remaining one reproduces $\epsilon_+$ up to the factor of $\psi$.
\begin{figure}
\begin{centering}
\begin{picture}(0,0)%
\includegraphics{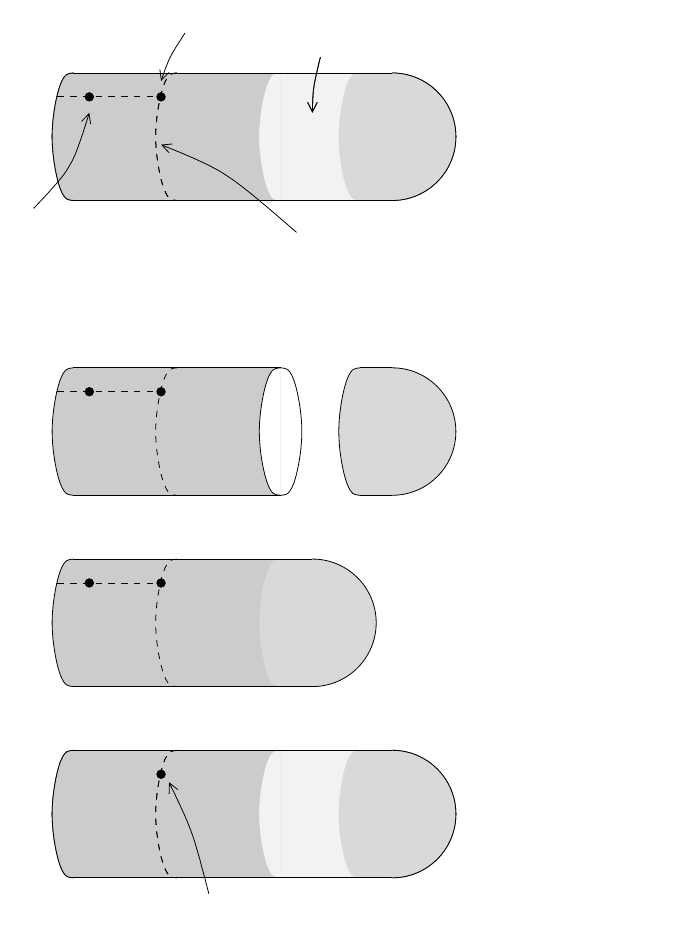}%
\end{picture}%
\setlength{\unitlength}{3355sp}%
\begingroup\makeatletter\ifx\SetFigFont\undefined%
\gdef\SetFigFont#1#2#3#4#5{%
  \reset@font\fontsize{#1}{#2pt}%
  \fontfamily{#3}\fontseries{#4}\fontshape{#5}%
  \selectfont}%
\fi\endgroup%
\begin{picture}(6549,8691)(-2114,-9655)
\put(451,-1411){\makebox(0,0)[lb]{\smash{{\SetFigFont{10}{12.0}{\rmdefault}{\mddefault}{\updefault}{\color[rgb]{0,0,0}an additional parameter stretches this part}%
}}}}
\put(-2024,-5011){\makebox(0,0)[lb]{\smash{{\SetFigFont{10}{12.0}{\rmdefault}{\mddefault}{\updefault}{\color[rgb]{0,0,0}(i)}%
}}}}
\put(-2024,-6811){\makebox(0,0)[lb]{\smash{{\SetFigFont{10}{12.0}{\rmdefault}{\mddefault}{\updefault}{\color[rgb]{0,0,0}(ii)}%
}}}}
\put(-2024,-8611){\makebox(0,0)[lb]{\smash{{\SetFigFont{10}{12.0}{\rmdefault}{\mddefault}{\updefault}{\color[rgb]{0,0,0}(iii)}%
}}}}
\put(-449,-9586){\makebox(0,0)[lb]{\smash{{\SetFigFont{10}{12.0}{\rmdefault}{\mddefault}{\updefault}{\color[rgb]{0,0,0}The image of this point goes through $Z^{(1)}$}%
}}}}
\put(676,-3286){\makebox(0,0)[lb]{\smash{{\SetFigFont{10}{12.0}{\rmdefault}{\mddefault}{\updefault}{\color[rgb]{0,0,0}$\zeta_1$ lies on this circle}%
}}}}
\put(-1874,-3886){\makebox(0,0)[lb]{\smash{{\SetFigFont{10}{12.0}{\rmdefault}{\mddefault}{\updefault}{\color[rgb]{0,0,0}\underline{the relevant boundary/degeneration contributions are:}}%
}}}}
\put(2476,-5011){\makebox(0,0)[lb]{\smash{{\SetFigFont{10}{12.0}{\rmdefault}{\mddefault}{\updefault}{\color[rgb]{0,0,0}$\chi(e)$}%
}}}}
\put(2476,-6811){\makebox(0,0)[lb]{\smash{{\SetFigFont{10}{12.0}{\rmdefault}{\mddefault}{\updefault}{\color[rgb]{0,0,0}$\epsilon_-$}%
}}}}
\put(2476,-8536){\makebox(0,0)[lb]{\smash{{\SetFigFont{10}{12.0}{\rmdefault}{\mddefault}{\updefault}{\color[rgb]{0,0,0}Fig. \ref{fig:2nd-step}(ii), up to $\psi$}%
}}}}
\put(-1349,-1111){\makebox(0,0)[lb]{\smash{{\SetFigFont{10}{12.0}{\rmdefault}{\mddefault}{\updefault}{\color[rgb]{0,0,0}the image of $\zeta_1$ lies on $F_0$}%
}}}}
\put(-2099,-3136){\makebox(0,0)[lb]{\smash{{\SetFigFont{10}{12.0}{\rmdefault}{\mddefault}{\updefault}{\color[rgb]{0,0,0}the image of $\zeta_2$ lies on $F_\dag$}%
}}}}
\end{picture}%
\caption{\label{fig:4th-step}The fourth ingredient in the proof of Lemma \ref{th:many}.}
\end{centering}
\end{figure}%

The Riemann surfaces in the final moduli space (Figure \ref{fig:4th-step}) can again be thought of as being obtained by gluing together two pieces, this time those defining $\chi(e)$. Here, the maps have $\mathrm{deg}(u) = 1$. Besides the limit where the gluing length goes to infinity, shown in Figure \ref{fig:4th-step}(i), in which one recovers $\chi(e)$, we have two other contributions to consider: one where the gluing length reaches its minimum, which yields $\epsilon_-$, and another one where the two marked points collide (Figure \ref{fig:4th-step}(ii) and (iii), respectively). In fact, there is another possible limit, which is where one of the marked points goes to $-\infty$, but its contribution is trivial, for the same reason as in the definition of \eqref{eq:nu-null-homotopy} or of $\epsilon_-$.

Inspection shows that when taking all four moduli spaces into account, all codimension $1$ contributions vanish except those which occur on both sides of \eqref{eq:many}, at least up to sign. To confirm the correctness of the signs, one has to inspect all the orientations, something which we will omit here. \end{proof}

\begin{lemma} \label{th:many-2}
The cocycle \eqref{eq:two-epsilons} is cohomologous to the Bor\-man-Sheri\-dan cocycle \eqref{eq:bs-1}.
\end{lemma}

\begin{proof}
We introduce a moduli space similar to that underlying $\epsilon_-$, but with an additional parameter. The outcome will be a cochain $\iota$ satisfying
\begin{equation} \label{eq:two-etas}
d\iota + \tilde{s} - (\epsilon_- + \epsilon_+) = 0,
\end{equation}
where $\tilde{s}$ is cohomologous to $s$. The construction is shown in Figure \ref{fig:eta}, but it still makes sense to spend a little time discussing it.
\begin{figure}
\begin{centering}
\begin{picture}(0,0)%
\includegraphics{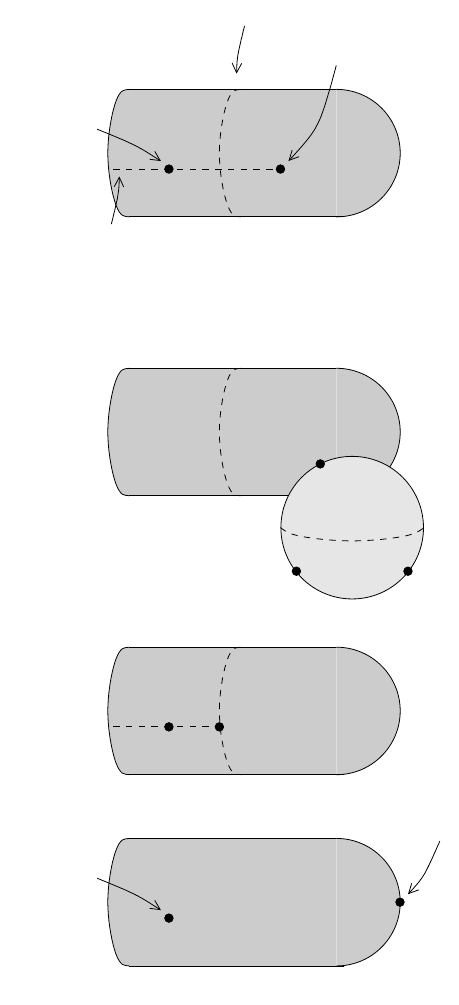}%
\end{picture}%
\setlength{\unitlength}{3355sp}%
\begingroup\makeatletter\ifx\SetFigFont\undefined%
\gdef\SetFigFont#1#2#3#4#5{%
  \reset@font\fontsize{#1}{#2pt}%
  \fontfamily{#3}\fontseries{#4}\fontshape{#5}%
  \selectfont}%
\fi\endgroup%
\begin{picture}(4305,9432)(-2864,-9871)
\put(-2849,-8611){\makebox(0,0)[lb]{\smash{{\SetFigFont{10}{12.0}{\rmdefault}{\mddefault}{\updefault}{\color[rgb]{0,0,0}$u(\zeta_2) \in F_\dag$}%
}}}}
\put(-1724,-9811){\makebox(0,0)[lb]{\smash{{\SetFigFont{10}{12.0}{\rmdefault}{\mddefault}{\updefault}{\color[rgb]{0,0,0}$\zeta_1$ is fixed, while $\zeta_2$ is unconstrained}%
}}}}
\put(-2849,-3511){\makebox(0,0)[lb]{\smash{{\SetFigFont{10}{12.0}{\rmdefault}{\mddefault}{\updefault}{\color[rgb]{0,0,0}\underline{the relevant boundary/degeneration contributions are:}}%
}}}}
\put(-2849,-2761){\makebox(0,0)[lb]{\smash{{\SetFigFont{10}{12.0}{\rmdefault}{\mddefault}{\updefault}{\color[rgb]{0,0,0}$\zeta_2$ moves on the half-infinite line determined by $\zeta_1$}%
}}}}
\put(1426,-4486){\makebox(0,0)[lb]{\smash{{\SetFigFont{10}{12.0}{\rmdefault}{\mddefault}{\updefault}{\color[rgb]{0,0,0}$\epsilon_+$}%
}}}}
\put(-2849,-1561){\makebox(0,0)[lb]{\smash{{\SetFigFont{10}{12.0}{\rmdefault}{\mddefault}{\updefault}{\color[rgb]{0,0,0}$u(\zeta_2) \in F_\dag$}%
}}}}
\put(-1049,-586){\makebox(0,0)[lb]{\smash{{\SetFigFont{10}{12.0}{\rmdefault}{\mddefault}{\updefault}{\color[rgb]{0,0,0}$\zeta_1$ lies to the right of this circle}%
}}}}
\put( 76,-961){\makebox(0,0)[lb]{\smash{{\SetFigFont{10}{12.0}{\rmdefault}{\mddefault}{\updefault}{\color[rgb]{0,0,0}$u(\zeta_1) \in F_0$}%
}}}}
\put(1351,-7111){\makebox(0,0)[lb]{\smash{{\SetFigFont{10}{12.0}{\rmdefault}{\mddefault}{\updefault}{\color[rgb]{0,0,0}$\epsilon_-$}%
}}}}
\put(1051,-8236){\makebox(0,0)[lb]{\smash{{\SetFigFont{10}{12.0}{\rmdefault}{\mddefault}{\updefault}{\color[rgb]{0,0,0}$u(\zeta_1) \in F_0$}%
}}}}
\put(1351,-8911){\makebox(0,0)[lb]{\smash{{\SetFigFont{10}{12.0}{\rmdefault}{\mddefault}{\updefault}{\color[rgb]{0,0,0}$\tilde{s}$}%
}}}}
\end{picture}%
\caption{\label{fig:eta}The construction of $\iota$ from \eqref{eq:two-etas}.}
\end{centering}
\end{figure}%

Concretely, the position of the two marked points is now $\zeta_1 = (\rho,-\tau)$ and $\zeta_2 = (\sigma,-\tau)$, where $\rho \in (0,\infty)$ is the new parameter, and $\sigma \in (-\infty,\rho)$, $\tau \in S^1$.
Concerning the auxiliary data $(H^{\iota},J^{\iota})$, in the limit $\rho \rightarrow 0$ they should agree with $(H^{\epsilon_-}, J^{\epsilon_-})$, which gives a contribution of $\epsilon_-$. There are three other limits that one has to discuss. As usual, having the two marked points collide leads to bubbling off of a holomorphic sphere at $\zeta_1$. One can arrange that the resulting contribution is exactly $\epsilon_+$. Next, the limit $\rho \rightarrow -\infty$ has zero contribution, as usual.

The most interesting limit is $\rho \rightarrow \infty$, which one thinks of as follows: $\zeta_1 = \infty \in T$, whereas the position of $\zeta_2 \in \bR \times S^1 \subset T$ can be arbitrary. We claim that this yields a cochain $\tilde{s}$ in the same homology class as $s$. To show that, one introduces another parameter space, which deforms the auxiliary data until they are no longer dependent on the position of $\zeta_2$. Then, what one is counting are maps with $u(\zeta_1) \in F_0$, with a multiplicity which depends on the possible positions of $\zeta_2$. However, the signed count of such possibilities is precisely $u \cdot F_{\dag} = \mathrm{deg}(u)$, which is $1$ by definition (we have seen the same kind of argument before, towards the end of the proof of Proposition \ref{th:interpret-bs}). Hence, the multiplicity issue turns out to be trivial, and we recover $s$.
\end{proof}

Combining Lemmas \ref{th:many} and \ref{th:many-2} establishes Proposition \ref{th:nabla-e}.

\subsection{Applying the connection to the Borman-Sheridan element}
The following result is the core of our discussion of $\nabla^{-1}$. 

\begin{proposition} \label{th:main-computation}
Applying the connection \eqref{eq:chi-connection} to the Bor\-man-Sheri\-dan cocycle \eqref{eq:bs-1} yields the following relation involving \eqref{eq:s01} and the functions from \eqref{eq:express-o}:
\begin{equation} \label{eq:thats-it}
\nabla^{-1} s + \eta s = 2\psi\, \tilde{s}^{(2)} + \text{\it coboundary}.
\end{equation}
\end{proposition}

In view of Lemma \ref{th:bisections}, \eqref{eq:thats-it} can equivalently be written in terms of \eqref{eq:s2}, in a way which makes the relation with the $c = -1$ case of \eqref{eq:nablac-s} more evident:
\begin{equation} \label{eq:we-got-it}
\nabla^{-1} s - 2\psi s^{(2)} + \eta s + 4 z^{(2)}\psi\, e = \text{\it coboundary}.
\end{equation}
\begin{figure}
\begin{centering}
\begin{picture}(0,0)%
\includegraphics{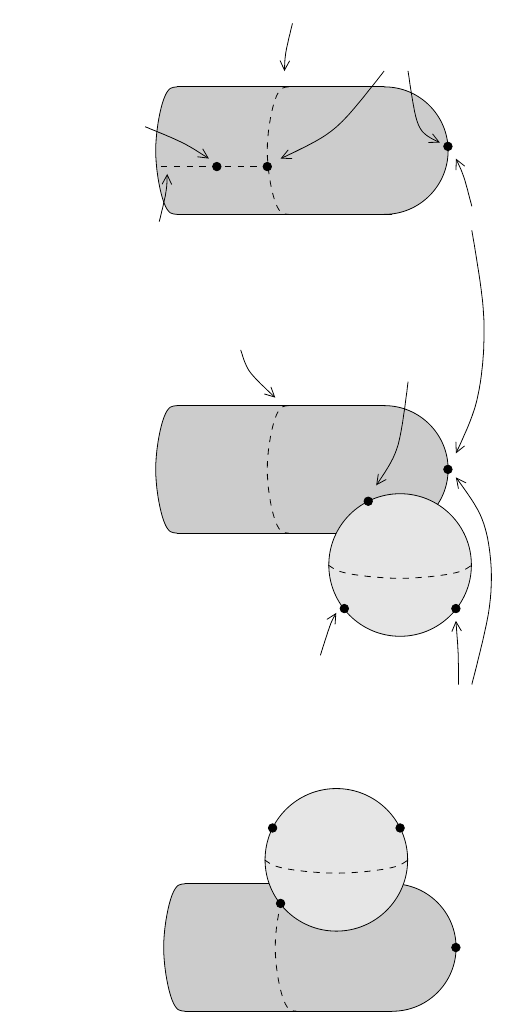}%
\end{picture}%
\setlength{\unitlength}{3355sp}%
\begingroup\makeatletter\ifx\SetFigFont\undefined%
\gdef\SetFigFont#1#2#3#4#5{%
  \reset@font\fontsize{#1}{#2pt}%
  \fontfamily{#3}\fontseries{#4}\fontshape{#5}%
  \selectfont}%
\fi\endgroup%
\begin{picture}(4755,9522)(-3314,-9973)
\put(-824,-3961){\makebox(0,0)[lb]{\smash{{\SetFigFont{10}{12.0}{\rmdefault}{\mddefault}{\updefault}{\color[rgb]{0,0,0}$u(\zeta_2)$ lies on a holomorphic sphere}%
}}}}
\put(-1049,-586){\makebox(0,0)[lb]{\smash{{\SetFigFont{10}{12.0}{\rmdefault}{\mddefault}{\updefault}{\color[rgb]{0,0,0}$\zeta_2$ moves on this circle}%
}}}}
\put(  1,-1036){\makebox(0,0)[lb]{\smash{{\SetFigFont{10}{12.0}{\rmdefault}{\mddefault}{\updefault}{\color[rgb]{0,0,0}$u(\zeta_1), u(\zeta_2) \in F_0$}%
}}}}
\put(1426,-1861){\makebox(0,0)[lb]{\smash{{\SetFigFont{10}{12.0}{\rmdefault}{\mddefault}{\updefault}{\color[rgb]{0,0,0}$\sigma_-$}%
}}}}
\put(901,-2536){\makebox(0,0)[lb]{\smash{{\SetFigFont{10}{12.0}{\rmdefault}{\mddefault}{\updefault}{\color[rgb]{0,0,0}$\zeta_1$ is fixed}%
}}}}
\put(1426,-4861){\makebox(0,0)[lb]{\smash{{\SetFigFont{10}{12.0}{\rmdefault}{\mddefault}{\updefault}{\color[rgb]{0,0,0}$\sigma_+$}%
}}}}
\put(-1199,-6811){\makebox(0,0)[lb]{\smash{{\SetFigFont{10}{12.0}{\rmdefault}{\mddefault}{\updefault}{\color[rgb]{0,0,0}this point maps to $F_\dag$}%
}}}}
\put(226,-7111){\makebox(0,0)[lb]{\smash{{\SetFigFont{10}{12.0}{\rmdefault}{\mddefault}{\updefault}{\color[rgb]{0,0,0}these points map to $F_0$}%
}}}}
\put(-1574,-3661){\makebox(0,0)[lb]{\smash{{\SetFigFont{10}{12.0}{\rmdefault}{\mddefault}{\updefault}{\color[rgb]{0,0,0}$\zeta_2$ lies to the right of this circle}%
}}}}
\put(-2740,-1561){\makebox(0,0)[lb]{\smash{{\SetFigFont{10}{12.0}{\rmdefault}{\mddefault}{\updefault}{\color[rgb]{0,0,0}$u(\zeta_3) \in F_\dag$}%
}}}}
\put(-2699,-2761){\makebox(0,0)[lb]{\smash{{\SetFigFont{10}{12.0}{\rmdefault}{\mddefault}{\updefault}{\color[rgb]{0,0,0}$\zeta_3$ moves on the half-infinite line determined by $\zeta_1$}%
}}}}
\put(-3299,-7561){\makebox(0,0)[lb]{\smash{{\SetFigFont{10}{12.0}{\rmdefault}{\mddefault}{\updefault}{\color[rgb]{0,0,0}\underline{common boundary component:}}%
}}}}
\end{picture}%
\caption{\label{fig:two-sigmas}The cochains $\sigma_{\pm}$ from \eqref{eq:two-sigmas}.}
\end{centering}
\end{figure}%
The proof is an adaptation of that of Proposition \ref{th:nabla-e}. Again, the two sides of \eqref{eq:thats-it} will each separately be related to a cocycle (assuming $\alpha > 2$, shown in Figure \ref{fig:two-sigmas})
\begin{equation} \label{eq:two-sigmas}
\sigma_- + \sigma_+ \in \mathit{CF}^0(E,H).
\end{equation}
To define $\sigma_{\pm}$, one modifies our previous construction of $\epsilon_{\pm}$ by adding another marked point at $\infty \in T$, whose image is required to go through $F_0$, and increasing $\mathrm{deg}(u)$ by $1$.

\begin{lemma} \label{th:many-3}
The cocycle \eqref{eq:two-sigmas} is related to the connection \eqref{eq:chi-connection} by
\begin{equation} \label{eq:what-is-sigma-for}
\nabla^{-1} s = \psi(\sigma_- + \sigma_+) - \eta s + \text{\it coboundary}.
\end{equation}
\end{lemma}

\begin{proof}
The proof follows closely that of Lemma \ref{th:many}. The modifications one has to apply to the moduli spaces involved are the same as when transforming \eqref{eq:two-epsilons} into \eqref{eq:two-sigmas}, and much of the argument carries over without any essential changes. Therefore, we will mostly just do a quick review, slowing down only to focus on the new aspect of \eqref{eq:what-is-sigma-for}, namely the $\eta s$ term.

The starting point is the analogue of Figure \ref{fig:nabla-e}, appropriately modified by replacing $e$ with $s$; which means that the thimble component acquires an additional marked point, and is now of degree $1$. The same applies to Figure \ref{fig:1st-step}. Note that now, since the maps will no longer remain in $E$, we must use $q^{-1}\Omega$ instead of its restriction to $E$. Similarly, in Figure \ref{fig:2nd-step}, one uses $A$ instead of its restriction to $E$, and hence relies on \eqref{eq:boundary-of-a} instead of \eqref{eq:a-restricted}. Potentially, the term $\eta F_0$ in \eqref{eq:boundary-of-a} could cause an additional boundary contribution to appear, which would consist of thimbles with two distinct marked points (one at $\infty$, and another on the circle $\{0\} \times S^1$) whose images go through $F_0$. However, our maps have $\mathrm{deg}(u) = 1$, hence can't have two distinct intersection points with $F_0$, ruling out this contribution.
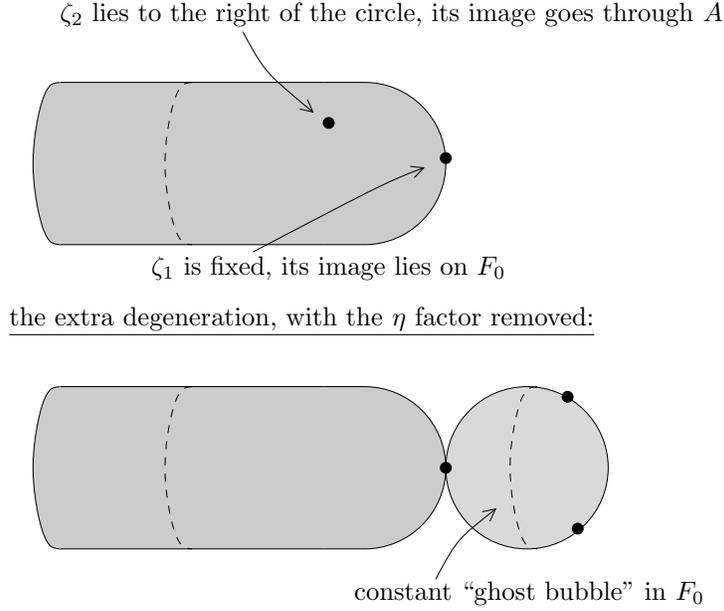
\begin{figure}
\begin{centering}
\begin{picture}(0,0)%
\includegraphics{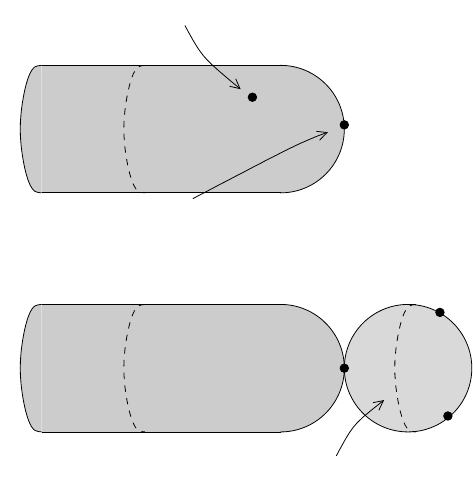}%
\end{picture}%
\setlength{\unitlength}{3355sp}%
\begingroup\makeatletter\ifx\SetFigFont\undefined%
\gdef\SetFigFont#1#2#3#4#5{%
  \reset@font\fontsize{#1}{#2pt}%
  \fontfamily{#3}\fontseries{#4}\fontshape{#5}%
  \selectfont}%
\fi\endgroup%
\begin{picture}(4448,4491)(-1814,-8380)
\put(751,-8311){\makebox(0,0)[lb]{\smash{{\SetFigFont{10}{12.0}{\rmdefault}{\mddefault}{\updefault}{\color[rgb]{0,0,0}constant ``ghost bubble'' in $F_0$}%
}}}}
\put(-1799,-6286){\makebox(0,0)[lb]{\smash{{\SetFigFont{10}{12.0}{\rmdefault}{\mddefault}{\updefault}{\color[rgb]{0,0,0}\underline{the extra degeneration, with the $\eta$ factor removed:}}%
}}}}
\put(-1424,-4036){\makebox(0,0)[lb]{\smash{{\SetFigFont{10}{12.0}{\rmdefault}{\mddefault}{\updefault}{\color[rgb]{0,0,0}$\zeta_2$ lies to the right of the circle, its image goes through $A$}%
}}}}
\put(-749,-5911){\makebox(0,0)[lb]{\smash{{\SetFigFont{10}{12.0}{\rmdefault}{\mddefault}{\updefault}{\color[rgb]{0,0,0}$\zeta_1$ is fixed, its image lies on $F_0$}%
}}}}
\end{picture}%
\caption{\label{fig:h5}Origin of the $\eta s$ term on the right side of \eqref{eq:what-is-sigma-for}.}
\end{centering}
\end{figure}%

In the analogue of Figure \ref{fig:3rd-step}, we again have to consider the impact of the extra term in \eqref{eq:boundary-of-a}. One might want to argue as before that this vanishes, but there are points in the compactification of our moduli space to which this argument does not apply. To illustrate that point, the space itself is shown in Figure \ref{fig:h5} at the top, and the problematic degenerations, which make up what we will call the ``ghost bubble stratum'', at the bottom. The dimension of the entire moduli space is $i(x) + 2\,\mathrm{deg}(u) + 2 - 3 = i(x) + 1$, and the ``ghost bubble stratum'' has dimension $i(x)$. We may assume that, if $\zeta_2$ is sufficiently close to $\zeta_1$, the auxiliary data involved (Hamiltonian and almost complex structures) are independent of the position of $\zeta_2$. In that case, and assuming suitable transversality, the situation is as follows. Given any map $u$ with $u(\zeta_1) \in F_0$, the points $\zeta_2$ sufficiently close to $\zeta_1$, and such that $u(\zeta_2)$ lies in a given component of $A$, form a one-dimensional submanifold of $T$. That submanifold has endpoints corresponding to the various parts of $\partial A$, and one kind of such endpoints corresponds precisely to our ``ghost bubbles''. In other words, this argument shows that in the compactification of a one-dimensional moduli space, the ``ghost bubbles'' are ordinary boundary points. Counting such points then makes a contribution which is at least cohomologus to $\eta s$ (like the $\tilde{s}$ term encountered in the proof of Lemma \ref{th:many-2}).
\end{proof}

\begin{lemma} \label{th:many-4}
The cocycle \eqref{eq:two-sigmas} is cohomologous to $2\tilde{s}^{(2)}$.
\end{lemma}
\begin{figure}
\begin{centering}
\begin{picture}(0,0)%
\includegraphics{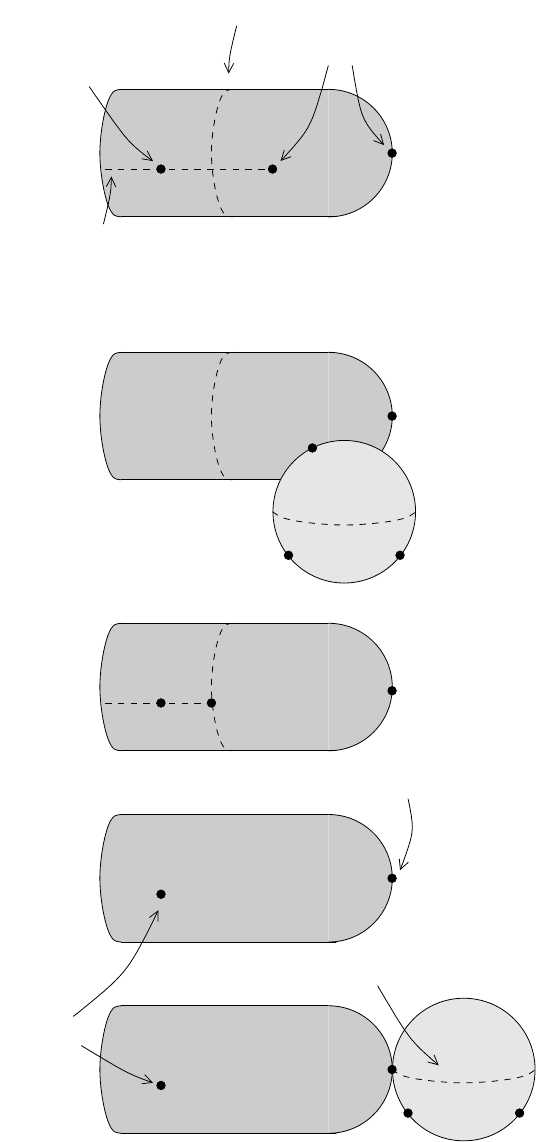}%
\end{picture}%
\setlength{\unitlength}{3355sp}%
\begingroup\makeatletter\ifx\SetFigFont\undefined%
\gdef\SetFigFont#1#2#3#4#5{%
  \reset@font\fontsize{#1}{#2pt}%
  \fontfamily{#3}\fontseries{#4}\fontshape{#5}%
  \selectfont}%
\fi\endgroup%
\begin{picture}(5052,10726)(-2789,-11165)
\put(-2190,-1111){\makebox(0,0)[lb]{\smash{{\SetFigFont{10}{12.0}{\rmdefault}{\mddefault}{\updefault}{\color[rgb]{0,0,0}$u(\zeta_3) \in F_\dag$}%
}}}}
\put(-1049,-586){\makebox(0,0)[lb]{\smash{{\SetFigFont{10}{12.0}{\rmdefault}{\mddefault}{\updefault}{\color[rgb]{0,0,0}$\zeta_2$ lies to the right of this circle}%
}}}}
\put( 76,-961){\makebox(0,0)[lb]{\smash{{\SetFigFont{10}{12.0}{\rmdefault}{\mddefault}{\updefault}{\color[rgb]{0,0,0}$u(\zeta_1), u(\zeta_2) \in F_0$}%
}}}}
\put(1351,-8911){\makebox(0,0)[lb]{\smash{{\SetFigFont{10}{12.0}{\rmdefault}{\mddefault}{\updefault}{\color[rgb]{0,0,0}$2\tilde{s}^2 + (\mathit{coboundary})$}%
}}}}
\put(1351,-6886){\makebox(0,0)[lb]{\smash{{\SetFigFont{10}{12.0}{\rmdefault}{\mddefault}{\updefault}{\color[rgb]{0,0,0}$\sigma_-$}%
}}}}
\put(1426,-4336){\makebox(0,0)[lb]{\smash{{\SetFigFont{10}{12.0}{\rmdefault}{\mddefault}{\updefault}{\color[rgb]{0,0,0}$\sigma_+$}%
}}}}
\put(-2474,-2761){\makebox(0,0)[lb]{\smash{{\SetFigFont{10}{12.0}{\rmdefault}{\mddefault}{\updefault}{\color[rgb]{0,0,0}$\zeta_3$ moves on the half-infinite line determined by $\zeta_2$}%
}}}}
\put(-2624,-3361){\makebox(0,0)[lb]{\smash{{\SetFigFont{10}{12.0}{\rmdefault}{\mddefault}{\updefault}{\color[rgb]{0,0,0}\underline{the relevant boundary/degeneration contributions are:}}%
}}}}
\put(-1049,-7861){\makebox(0,0)[lb]{\smash{{\SetFigFont{10}{12.0}{\rmdefault}{\mddefault}{\updefault}{\color[rgb]{0,0,0}$u(\zeta_1) \in F_0$, with tangency condition}%
}}}}
\put(-2774,-10186){\makebox(0,0)[lb]{\smash{{\SetFigFont{10}{12.0}{\rmdefault}{\mddefault}{\updefault}{\color[rgb]{0,0,0}$u(\zeta_2) \in F_\dag$}%
}}}}
\put(-1424,-9586){\makebox(0,0)[lb]{\smash{{\SetFigFont{10}{12.0}{\rmdefault}{\mddefault}{\updefault}{\color[rgb]{0,0,0}pseudo-holomorphic sphere intersecting $F_0$ twice}%
}}}}
\end{picture}%
\caption{\label{fig:many-4}The proof of Lemma \ref{th:many-4}.}
\end{centering}
\end{figure}
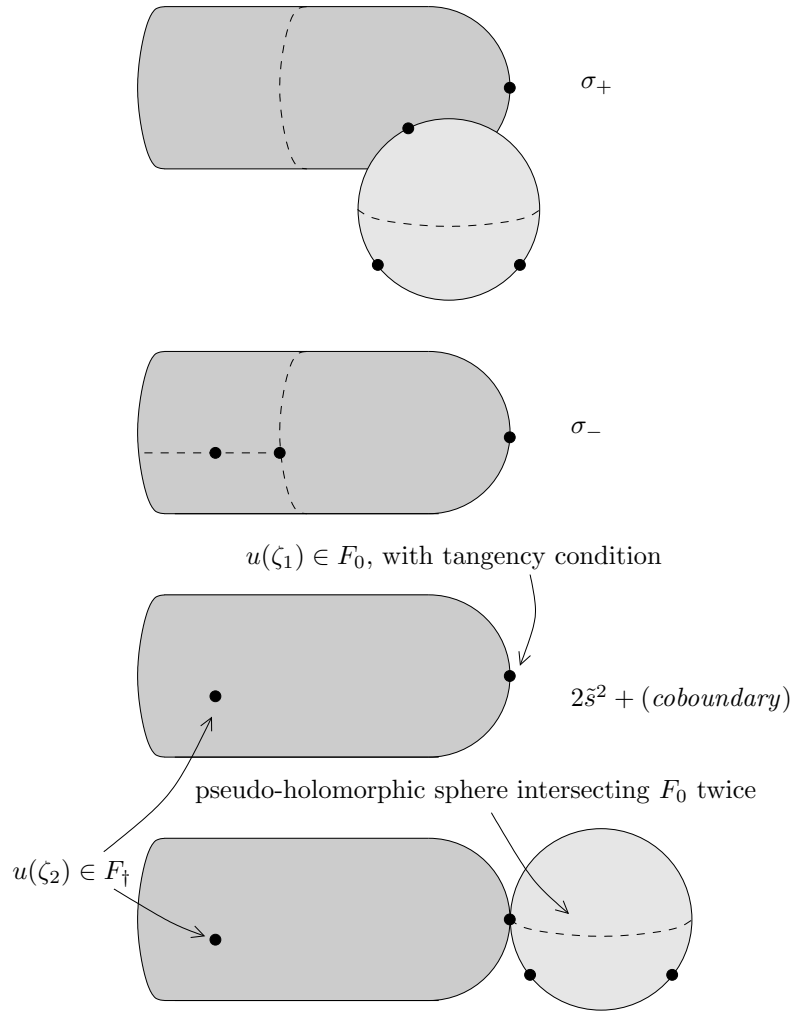

\begin{proof}
This is a modification of the proof of Lemma \ref{th:many-2}, entirely analogous to what we have done previously. We summarize it in Figure \ref{fig:many-4}. The only part that requires extra attention is the limit where the two marked points collide. In that case, there are two possibilities, in parallel with Lemma \ref{th:bisections}. One is the appearance of a constant ``ghost bubble'', in which case (for topological reasons) the principal component $u$ is tangent to $F_0$ at $\zeta_1$. We can think of the condition $u(\zeta_2) \in F_{\dag}$ as counting the possible positions of $\zeta_2$, in analogy with the argument from Lemma \ref{th:many-2}. Since $\mathrm{deg}(u) = 2$, the multiplicity is $2$ this time, hence a count leads to (a cocycle cohomologous to) $2\tilde{s}^{(2)}$. The other possibility (shown at the bottom of Figure \ref{fig:many-4}) is bubbling off of a nonconstant pseudo-holomorphic sphere, with two points on $F_0$. In that case, the principal component satisfies $\mathrm{deg}(u) = 0$, which by the  same argument as before, means that the relevant contribution is zero.
\end{proof}

The combination of Lemmas \ref{th:many-3} and \ref{th:many-4} establishes Proposition \ref{th:main-computation}.

\begin{remark}
We close our discussion with a technical point, which applies to Proposition \ref{th:main-computation} and all earlier computations regarding specific Floer cochains of degree $0$ (Proposition \ref{th:interpret-bs}, Lemma \ref{th:bisections}, Proposition \ref{th:nabla-e}). Namely, inspection of the geometry of our setup shows that it is always possible to choose $H$ so that $\mathit{CF}^*(E,H)$ is concentrated in degrees $\geq 0$. In that case, any relation in $\mathit{CF}^0(E,H)$ that holds up to coboundaries actually holds strictly. In other words, all the zero-dimensional moduli spaces that define those coboundaries would turn out to be empty for dimension reasons. For the same reason, Lemma \ref{th:bs-done} would then become trivial. We have not made use of that option, to avoid giving the false impression that a specific choice of $H$ is an important part of the construction.
\end{remark}

\section{Multilinear operations\label{sec:operations-on-floer-cohomology}}

Up to this point, our Floer-theoretic discussion has involved only two Riemann surfaces: the cylinder and the thimble. One can embed this into a more general TQFT structure, of which we will now recall a version (related expositions are \cite{schwarz95,seidel07,ritter10}). Our main goal is to reinterpret the connection introduced in Section \ref{sec:floer-connection} in terms of that structure, which brings out the similarity with the abstract theory from Sections \ref{sec:conf}--\ref{sec:bv}.

\subsection{Riemann surfaces}
We will use only genus zero surfaces, which are decorated in such a way that they produce operations with one output (in principle, one could work in greater generality; but this level is sufficient for our purpose, and allows some expository simplifications).

\begin{setup} \label{th:worldsheet}
Let $\bar{S}$ be a compact Riemann surface isomorphic to the sphere, with marked points $\zeta^-$ and $\zeta^+_i$, $i \in I$, for some finite set $I$. Let $S$ be the complement of the marked points. A set of tubular ends for $S$ consists of proper holomorphic embeddings with disjoint images,
\begin{equation} \label{eq:epsilon-maps}
\left\{
\begin{aligned}
& \epsilon^-: (-\infty,s^-] \times S^1 \longrightarrow S, \\
& \epsilon^+_i: [s^+_i,\infty) \times S^1 \longrightarrow S,
\end{aligned}
\right.
\end{equation}
for some $s^-, s^+_i \in \bR$. Each embedding \eqref{eq:epsilon-maps} should map to a punctured disc in $\bar{S}$ surrounding the corresponding marked point. By a worldsheet, we mean such a surface $S$ together with a choice of tubular ends, as well as a one-form $\beta$ satisfying 
\begin{equation} \label{eq:sub-closed}
d\beta \leq 0,
\end{equation}
and such that, for some $\alpha^+_i, \alpha^- \in \bR \setminus \bZ$, 
\begin{equation} \label{eq:beta-on-the-ends}
\left\{
\begin{aligned}
& (\epsilon^-)^* \beta = \alpha^-\, \mathit{dt}, \\
& (\epsilon^+_i)^* \beta = \alpha^+_i \, \mathit{dt}.
\end{aligned}
\right.
\end{equation}
\end{setup}

In terms of $\bar{S}$, the tubular ends equip it with distinguished framings (half-lines in the tangent spaces at the marked points). These are defined as
\begin{equation} \label{eq:framings-from-ends}
\left\{
\begin{aligned}
& \textstyle \lim_{s \rightarrow -\infty} D(\epsilon^-)_{s,0}(\bR^+ \cdot \partial_s) \subset T\bar{S}_{\zeta^-}, 
\\
& \textstyle \lim_{s \rightarrow +\infty} D(\epsilon^+_i)_{s,0}(\bR^- \cdot \partial_s) \subset T\bar{S}_{\zeta^+_i}.
\end{aligned}
\right.
\end{equation}
The framing is the ``homotopically essential'' part of choosing a tubular end (the remaining choices belong to a weakly contractible space).

\begin{setup} \label{th:worldsheet-2}
Given a worldsheet $S$, we consider pairs $(K^S,J^S)$ of the following kind. $K^S$ is a one-form on $S$ taking values in $\smooth(F,\bR)$, with the property that for any tangent vector $\xi$, $K^S(\xi)$ satisfies \eqref{eq:hamiltonian} for the angle $\alpha = \beta(\xi)$. $J^S$ is a family of almost complex structures on $F$ parametrized by $S$, all belonging to the class \eqref{eq:j}. Moreover, these should be compatible with the choice of tubular ends; this means that there are families $(H^-,J^-)$ and $(H^+_i,J^+_i)$ parametrized by $S^1$, such that (exponentially fast in any $C^r$ topology)
\begin{equation} \label{eq:kj-tubular}
\left\{
\begin{aligned}
& (\epsilon^-)^* (K^S,J^S) \rightarrow (H^-_t\, \mathit{dt}, J^-_t) \quad \text{as $s \rightarrow -\infty$}, \\
& (\epsilon^+_i)^* (K^S,J^S) \rightarrow (H^+_{i,t} \, \mathit{dt}, J^+_{i,t}) \quad \text{as $s \rightarrow +\infty$.}
\end{aligned}
\right.
\end{equation}
\end{setup}

To $K^S$, one can associate a one-form with values in (Hamiltonian) vector fields on $F$, which we denote by $Y^S$. The associated inhomogeneous Cauchy-Riemann equation is
\begin{equation} \label{eq:cauchy-riemann}
\left\{
\begin{aligned}
& u: S \longrightarrow F, \\
& (Du - Y^S)^{0,1} = 0, \\
& \textstyle \lim_{s \rightarrow -\infty} u(\epsilon^-(s,t)) = x^-(t), \\
& \textstyle \lim_{s \rightarrow +\infty} u(\epsilon^+_i(s,t)) = x^+_i(t).
\end{aligned}
\right.
\end{equation}
In the second line, we consider 
\begin{equation}
Du - Y_{z,u(z)}^S: TS_z \longrightarrow TF_{u(z)}, 
\end{equation}
and then take its complex antilinear part with respect to $J_{z,u(z)}^S$. The limits $x^+_i$ and $x^-$ are one-periodic orbits of the Hamiltonians appearing in \eqref{eq:kj-tubular}. The equation \eqref{eq:kj-tubular} generalizes the equations \eqref{eq:floer}, \eqref{eq:continuation} on the cylinder, as well as \eqref{eq:thimble} on the thimble. The appropriate generalization of \eqref{eq:index} is
\begin{equation}
\mathrm{index}(D_u) = i(x^-) - \sum_{i \in I} i(x^+_i) + 2\, \mathrm{deg}(u),
\end{equation}
where $\mathrm{deg}(u)$ is as in \eqref{eq:define-degree}. Similarly, instead of \eqref{eq:geom-energy}--\eqref{eq:two-energies} one has
\begin{align} \label{eq:geom-energy-2}
&
E^{\mathit{geom}}(u) = \int_{S} \half \|Du-Y^S\|^2 \geq 0, \\
& \label{eq:top-energy-2}
E^{\mathit{top}}(u) = \int_S u^*\omega_F - d(\tilde{u}^*K^S)
 = A(x^-) - \sum_i A(x^+_i) + u \cdot \Omega, \\
& \label{eq:two-energies-2}
E^{\mathit{geom}}(u) = E^{\mathit{top}}(u) + \int_S \tilde{u}^*R^S.
\end{align}
Here, $\tilde{u}: S \rightarrow S \times F$ is the graph of $u$. We can consider $K^S$ as a one-form $S \times F$ which vanishes along the tangent directions to $F$, and that explains the expression $\tilde{u}^*K^S$ in \eqref{eq:top-energy-2}. Similarly, the expression $R^S$ in \eqref{eq:two-energies-2} is a two-form on $S \times F$ which vanishes if one contracts it with a vector tangent to $F$. We will not write it down explicitly; the only thing that's relevant at the moment is that $R^S$ is compactly supported, which means that $\int \tilde{u}^*R^S$ can be bounded by a constant which is independent of $u$.

As before, we write $v = p(u): S \rightarrow \bC P^1$. At any point $z$ where $v(z) \in B$, this satisfies
\begin{equation} \label{eq:general-v-equation}
(Dv + 2\pi i v \beta)^{0,1} = \bar\partial v + 2\pi i \beta^{0,1} v = 0.
\end{equation}
Writing $\beta^{0,1} = \bar\partial f$, one has an analogue of \eqref{eq:v-tilde-2},
\begin{equation} \label{eq:f-twist}
v(z) = \exp(-2\pi i f) \tilde{v}(z).
\end{equation}
In parallel with \eqref{eq:subharmonic}, it follows that at any point where $v(z) \in B \setminus \{0\}$, the Laplacian (as a two-form on $S$) satisfies
\begin{equation} \label{eq:subharmonic-2}
\begin{aligned}
\Delta \log |v| &\, = 2\pi \,\Delta \mathrm{im}(f) = \pi \partial \bar\partial (f - \bar{f})
 \\ &
= \pi (\partial \beta^{0,1} + \bar\partial \beta^{1,0}) = \pi\, d\beta \leq 0.
\end{aligned} 
\end{equation}
The appropriate generalization of Lemma \ref{th:v-regular} is:

\begin{lemma} \label{th:v-regular-2}
Set
\begin{equation} \label{eq:degree-of-z}
\delta = \lfloor \alpha^- \rfloor - \sum_i (\lfloor \alpha^+_i \rfloor + 1).
\end{equation}
Consider the left hand side of \eqref{eq:general-v-equation} as linear operator between suitable Sobolev completions, let's say $W^{1,2}(S,\bC) \rightarrow L^2(S, \Omega^{0,1}_S)$. Then, that operator is surjective if and only if $\delta \geq -1$, and injective if and only if $\delta \leq -1$.
\end{lemma}

\begin{proof}
As mentioned before, one can write $\beta^{0,1} = \bar\partial f$, simply because $S$ is an open Riemann surface. Since $S$ is of genus zero, one can find such an $f$ which satisfies (writing $i = \sqrt{-1}$ to avoid confusion with the indices $i \in I$)
\begin{equation}
\left\{
\begin{aligned}
&
f(\epsilon^-(s,t)) = \sqrt{-1}\,\alpha^-s + (\text{\it bounded holomorphic}), \\
&
f(\epsilon^+_i(s,t)) = \sqrt{-1}\,\alpha^+_i s + (\text{\it bounded holomorphic}).
\end{aligned}
\right.
\end{equation}
In terms of \eqref{eq:f-twist}, if $v$ is a solution of class $W^{1,2}$, then $\tilde{v}$ must belong to $\scrO_{\bar{S}}(Z)$, where
\begin{equation}
Z = \lfloor \alpha^- \rfloor  \zeta^-  - \sum_i  (\lfloor \alpha^+_i \rfloor + 1) \zeta^+_i.
\end{equation}
This divisor has degree exactly $\delta$, hence yields a linear system of (complex) dimension $\max\{0,\delta+1\}$. On the other hand, a standard index formula shows that the (real) Fredholm index of $D_u$ is $2\delta+2$. Both statements follow from those two facts.
\end{proof}

The counterpart of Lemma \ref{th:max-principle-1} holds only if \eqref{eq:degree-of-z} is negative; from \eqref{eq:subharmonic-2} one gets a generalization of Lemma \ref{th:max-principle-2}; and Lemma \ref{th:nonnegative-1} also still holds. The appropriate version of Lemma \ref{th:nonnegative-2} says that if $u$ is not contained in $F_0$, 
\begin{equation} \label{eq:fb-intersection-3}
\begin{aligned}
\mathrm{deg}(u) = u \cdot F_0 + \, & 
\begin{cases} 
m^-(u) & \text{if $x^-$ lies in $F_0$,} \\
0 & \text{otherwise}
\end{cases}
\\ + \sum_{i \in I} &
\begin{cases}
m^+_i(u) & \text{if $x^+_i$ lies in $F_0$,} \\
0 & \text{otherwise.}
\end{cases}
\end{aligned}
\end{equation}
Here, $m^-(u)$ and $m^+_i(u)$ are defined as in \eqref{eq:m-minus} and \eqref{eq:m-plus}, respectively. The analogue of Lemma \ref{th:transversality-2} goes as follows:

\begin{lemma} \label{th:transversality-3}
Fix $(H^-,J^-)$ and $(H^+_i,J^+_i)$, and assume that \eqref{eq:degree-of-z} satisfies $\delta \geq -1$. Then, for a generic choice of $(K^S, J^S)$, the following holds. (i) All solutions of \eqref{eq:cauchy-riemann} are regular. (ii) For any $z \in S$, all $J_z^S$-holomorphic spheres with zero first Chern number avoid $u(z)$ for any solution $u$ of \eqref{eq:cauchy-riemann} with $\mathrm{index}(D_u) \leq 1$. 
\end{lemma}

If one drops the condition on $\delta$, the same statements still hold for those solutions which are not contained in $p^{-1}(B)$.

\subsection{The pair-of-pants product\label{subsec:pair-of-pants}}
A classical application of TQFT ideas is to take $I = \{1,2\}$, which means that $S$ is a three-punctured sphere, and to equip it with a one-form satisfying $d\beta = 0$, which means that 
\begin{equation} \label{eq:alpha-add-up}
\alpha^- = \alpha^+_1 + \alpha^+_2. 
\end{equation}
Let's count solutions of \eqref{eq:cauchy-riemann} such that
\begin{equation} \label{eq:pair-of-pants-0}
\left\{
\begin{aligned}
& \, \mathrm{deg}(u) = 0, \\
& \, \text{all three limits lie outside $F_0$.}
\end{aligned}
\right.
\end{equation}
The counterpart of Lemma \ref{th:max-principle-2} implies that all such solutions remain in a compact subset of $E$. By counting them, one obtains the chain level pair-of-pants product
\begin{equation} \label{eq:define-product}
\bullet: \mathit{CF}^*(E,H^+_1) \otimes \mathit{CF}^*(E,H^+_2) \longrightarrow \mathit{CF}^*(E,H^-),
\end{equation}
which is independent of all choices up to chain homotopy. We will denote the surface used to define this by $S^{\mathit{pants}}$, and the auxiliary data that enter into this construction by $(K^\mathit{pants},J^\mathit{pants})$.

Now suppose $\alpha^+_1 = \alpha^+_2 = \alpha > 1$, so that the Borman-Sheridan cocycle is defined. We take the thimble $T$ with the datum $(H^\mathit{thimble}, J^\mathit{thimble})$ which underlies the definition of that cocycle, and glue one copy of it to each input end of our pair-of-pants (using the same gluing length $\rho \gg 0$ both times). The outcome is a surface isomorphic to the thimble, schematically denoted by
\begin{equation} \label{eq:glue-2-thimbles}
S^{\mathit{pants}} \#_{\rho}\, 2 T.
\end{equation}
This comes with two marked points $\zeta_1, \zeta_2$. We choose auxiliary data on these surfaces which, as $\rho \rightarrow \infty$, converge to the previously chosen $(K^{\mathit{pants}}, J^{\mathit{pants}})$ and (two copies of) $(H^\mathit{thimble}, J^\mathit{thimble})$. Correspondingly, one considers a parametrized moduli space of $(\rho,u)$, where the map $u$ is defined on \eqref{eq:glue-2-thimbles} and subject to the topological conditions
\begin{equation} \label{eq:pair-of-pants-1}
\left\{
\begin{aligned}
& \mathrm{deg}(u) = 2, \\
& u(\zeta_1) \in F_0, \;\; u(\zeta_2) \in F_0, \\ 
& \text{the limit of $u$ lies outside $F_0$.}
\end{aligned}
\right.
\end{equation}
For a fixed value of $\rho$, this is the same space previously used to define \eqref{eq:s2}. As $\rho \rightarrow \infty$, the most relevant limiting configuration consists of a triple $(u,\tilde{u}^+_1, \tilde{u}^+_2)$, where $u$ is a map on the pair-of-pants, whereas the $\tilde{u}^+_k$ are maps on the thimble. They obviously satisfy
\begin{equation} \label{eq:double-thimble}
\left\{
\begin{aligned}
& \mathrm{deg}(u) + \mathrm{deg}(\tilde{u}^+_1) + \mathrm{deg}(\tilde{u}^+_2) = 2, \\
& \tilde{u}^+_1(\zeta) \in F_0, \\
& \tilde{u}^+_2(\zeta) \in F_0, \\
& \textstyle x_- = \lim_{s \rightarrow -\infty} u(\epsilon^-(s,t)) \text{ lies outside $F_0$}, \\
& \textstyle \tilde{x}^+_1 = \lim_{s \rightarrow +\infty} u(\epsilon^+_1(s,t)) = \lim_{s \rightarrow -\infty} \tilde{u}^+_1(s,t), \\
& \textstyle \tilde{x}^+_2 = \lim_{s \rightarrow +\infty} u(\epsilon^+_2(s,t)) = \lim_{s \rightarrow -\infty} \tilde{u}^+_2(s,t). \\
\end{aligned}
\right.
\end{equation}
where $\zeta = +\infty$ is the usual marked point on the thimble. 

{\bf Case 1:} {\em Both $\tilde{u}^+_k$ are contained in $F_0$.} This would mean $\mathrm{deg}(\tilde{u}^+_k) = 0$. Moreover, both $\tilde{x}^+_k$ lie in $F_0$, hence
\begin{equation}
\mathrm{deg}(u) \geq 2(\lfloor \alpha \rfloor + 1) \geq 4 
\end{equation}
by \eqref{eq:fb-intersection-3} and \eqref{eq:rewritten-inequalities}. This leads to a contradiction with \eqref{eq:double-thimble}, hence is impossible.

{\bf Case 2:} {\em Exactly one $\tilde{u}^+_k$ is contained in $F_0$.} Let's assume that this is $\tilde{u}^+_1$. Then $\mathrm{deg}(\tilde{u}^+_1) = 0$, and for the same reason as before, 
\begin{equation}
\begin{aligned}
& \mathrm{deg}(u) + \mathrm{deg}(\tilde{u}^+_2) \geq 
u \cdot F_0 +
\tilde{u}^+_2 \cdot F_0 + (\lfloor \alpha \rfloor + 1)  \\
& \qquad \qquad + \begin{cases} 1 & \text{if $\tilde{x}^+_2$ lies in $F_0$,} \\ 0 & \text{otherwise,} \end{cases}
\end{aligned}
\end{equation}
which is again a contradiction, since $\tilde{u}^+_2$ must intersect $F_0$ by assumption.

{\bf Case 3:} {\em Neither of the $\tilde{u}^+_k$ is contained in $F_0$.} One then has
\begin{equation}
\begin{aligned}
& \mathrm{deg}(u) + \mathrm{deg}(\tilde{u}^+_1) + \mathrm{deg}(\tilde{u}^+_2) \geq
u \cdot F_0 + \tilde{u}^+_1 \cdot F_0 + \tilde{u}^+_2 \cdot F_0 \\ & \qquad \qquad +
\begin{cases} 1 & \text{if $\tilde{x}^+_1$ lies in $F_0$,} \\ 0 & \text{otherwise} \end{cases} +
\begin{cases} 1 & \text{if $\tilde{x}^+_2$ lies in $F_0$,} \\ 0 & \text{otherwise}. \end{cases}
\end{aligned}
\end{equation}
Clearly, the only possibility is as follows: neither $\tilde{x}^+_k$ lies in $F_0$; $u \cdot F_0 = 0$ (hence $u$ is contained in $E$); and $\tilde{u}^+_k \cdot F_0 = \mathrm{deg}(\tilde{u}^+_k) = 1$. Counting such configurations yields the pair-of-pants square $s \bullet s$ of the Borman-Sheridan cocycle.

Among the other limiting configurations that could occur a priori for $\rho \rightarrow \infty$, one deserves particular mention. Namely, suppose that we have four components $(\tilde{u}^-, u,\tilde{u}^+_1,\tilde{u}^+_2)$ (whose degrees of course sum up to $2$). Here, $\tilde{u}^-$ is a Floer trajectory, $u$ is a map on the pair-of-pants, and $\tilde{u}^+_k$ are thimbles. The interesting case is when $u$ is contained in $p^{-1}(B)$. Transversality (Lemma \ref{th:transversality-3}) may or may not hold for such $u$, depending on whether $\alpha$ is less or greater than $\frac32$. However, we don't actually need this: by Lemma \ref{th:v-regular-2}, we must have
\begin{equation}
\mathrm{deg}(\tilde{u}^-)  \geq \lfloor 2\alpha \rfloor + 1 \geq 3,
\end{equation}
a contradiction. There are similar arguments for more complicated limiting configurations  (compare Remark \ref{th:broken-big}). As a consequence, counting points in our parametrized moduli space leads to:

\begin{lemma} \label{th:square-of-borman-sheridan}
The cocycle $s^{(2)}$ from \eqref{eq:s2} is cohomologus to $s \bullet s$.
\end{lemma}

\subsection{A bracket\label{subsec:bracket-and-connection}}
We will be concerned with a specific family of worldsheets, parametrized by $\tau \in S^1$, and where $S^1$ is given the opposite of the standard orientation. The Riemann surfaces and tubular ends are
\begin{equation} \label{eq:three-ends}
\left\{
\begin{aligned}
& S_\tau = (\bR \times S^1) \setminus \{(0,-\tau)\}, \\
& \epsilon^-_\tau(s,t) = (s,t), \\
& \epsilon^+_{\tau,1}(s,t) = (0,-\tau) - \exp(-2\pi (s+it)),  \\
& \epsilon^+_{\tau,2}(s,t) = (s,t). \\
\end{aligned}
\right.
\end{equation}
If one thinks in terms of the framings \eqref{eq:framings-from-ends}, we have two marked points on $\bar{S} = \bR \times S^1 \cup \{\pm \infty\}$, namely $\zeta^- = -\infty$ and $\zeta^+_2 = \infty$, which are independent of $\tau$, and so are their framings. The other point $\zeta^+_1 = (0,-\tau)$ moves around a circle, but its framing always points in $-\partial_s$-direction, which means towards $\zeta^-$. 

Fix constants \eqref{eq:alpha-add-up}, and equip $S_\tau$ with a corresponding closed one-form. Assume additionally that we have fixed auxiliary data which define the Floer cochain complexes associated to the three ends, denoted by $(H^-,J^-)$ and $(H^+_k,J^+_k)$. We then choose, on each $S_\tau$ and smoothly depending on $\tau$, auxiliary data $(K_{S_\tau}, J_{S_\tau})$. Let's denote the entire family of such data by $(K^{\mathit{Lie}},J^{\mathit{Lie}})$. We consider a parametrized moduli space of maps on $S_\tau$ which satisfy \eqref{eq:pair-of-pants-0}, and use that to define a chain map
\begin{equation} \label{eq:define-the-1-bracket}
[\cdot,\cdot]^{-1}: \mathit{CF}^*(E,H^+_1) \otimes \mathit{CF}^*(E,H^+_2) \longrightarrow \mathit{CF}^{*-1}(E,H^-).
\end{equation}
As usual, we use descriptive notation for the Riemann surfaces involved in this construction, writing $S_\tau = S^{\mathit{Lie}(-1)}_\tau$.

\begin{discussion}
One can get an alternative picture of the same construction by rotating $S_\tau$ in $S^1$-direction, so as to remove its dependence on $\tau$. After this change of coordinates, the situation looks as follows:
\begin{equation}
\left\{
\begin{aligned}
& S = (\bR \times S^1) \setminus \{(0,0)\}, \\
& \epsilon^-_\tau(s,t) = (s,t+\tau), \\
& \epsilon^+_{\tau,1}(s,t) = - \exp(-2\pi (s+it)), \\
& \epsilon^+_{\tau,2}(s,t) = (s,t+\tau).
\end{aligned}
\right.
\end{equation}
The $\tau$-dependence now comes from rotating the ends at $\zeta^-$ and $\zeta_2^+$. A gluing argument based on this observation provides a chain homotopy
\begin{equation} \label{eq:bracket-1-again}
[x_1,x_2]^{-1} \htp \Delta (x_1 \bullet x_2) - (-1)^{|x_1|} x_1 \bullet \Delta x_2.
\end{equation}
We omit the details.
\end{discussion}

Our next task is to explain what happens when one applies the bracket to a class in the image of the PSS map \eqref{eq:pss}. Assume from now on that $\alpha^+_1 > 0$. Let $K$ be a proper pseudo-cycle with $\bK$-coefficients in $E$. Recall the previously defined structures from \eqref{eq:pss-q} and \eqref{eq:o-map},
\begin{align}
& b_K \in \mathit{CF}^{\mathrm{codim}(K)}(E,H^+_1), \\
& r_K: \mathit{CF}^*(E,H^+_2) \longrightarrow \mathit{CF}^{*+\mathrm{codim}(K)-1}(E,H^-).
\end{align}

\begin{lemma} \label{th:r-is-b}
$r_K$ is chain homotopic to $(-1)^{\mathrm{codim}(K)}[b_K,\cdot]^{-1}$.
\end{lemma}

\begin{proof}
As usual, we will prove this by direct construction of a map
\begin{equation} \label{eq:rho-for-a-cycle}
\begin{aligned}
& \rho^{-1}_K: \mathit{CF}^*(E,H^+_2) \longrightarrow \mathit{CF}^{*+\mathrm{codim}(K)-2}(E,H^-), \\
& d\rho^{-1}_K(x) - (-1)^{\mathrm{codim}(K)} \rho_K^{-1}(dx) \\
& \qquad \qquad + (-1)^{\mathrm{codim}(K)}[b_K,x]^{-1} - r_K(x) = 0.
\end{aligned}
\end{equation}
(To simplify the notation, we have used $d$ for the Floer differentials in both complexes involved.) This is given by a moduli space with parameters $(\lambda,\tau) \in \bR^+ \times S^1$. The associated Riemann surfaces are constructed as follows: for $\lambda \gg 0$, 
\begin{equation} \label{eq:lambda-tau}
S_{\lambda,\tau} = S^{\mathit{Lie}(-1)}_\tau \#_{\lambda} T. 
\end{equation}
More precisely, start with the Riemann surface underlying the bracket (and its auxiliary data), and glue in the thimble (equipped with the auxiliary data used to define $b_K$) into the $\zeta_1^+$ end. The outcome \eqref{eq:lambda-tau} is a cylinder with one marked point (which depends on $\tau$), and this (with its auxiliary data) can easily be deformed to the corresponding Riemann surface in the family underlying $r_K$, which is what we define $S_{\lambda,\tau}$ to be for $\lambda = 0$. (It may be more intuitive to think of starting with the given surfaces defining $r_K$, and ``pulling out'' the marked point, thereby creating a neck that stretches as $\lambda \rightarrow \infty$.) The argument then follows the standard pattern.
\end{proof}

At this point, we assume that $\alpha_1^+>1$, re-impose Assumption \ref{th:psi-eta}, and choose a proper pseudo-chain $A$ as in \eqref{eq:boundary-of-a}. By the same construction as in \eqref{eq:pss-q}, this gives rise to a Floer cochain 
\begin{equation} \label{eq:b-of-a}
\begin{aligned}
& b_{A|E} \in \mathit{CF}^1(E,H^+_1), \\
& d b_{A|E} = \psi b_{Z^{(1)}|E} - b_{q^{-1}\Omega|E}.
\end{aligned}
\end{equation}
One therefore has chain homotopies
\begin{equation} \label{eq:chain-of-homotopies}
r_{q^{-1}\Omega|E} \htp [b_{q^{-1}\Omega|E},\cdot]^{-1} \htp \psi [b_{Z^{(1)}|E},\cdot]^{-1} \htp 0,
\end{equation}
which come from: Lemma \ref{th:r-is-b}; from the bracket with $b_{A|E}$; and from Lemma \ref{th:z1-zero}. This is not our first encounter with such a nullhomotopy: \eqref{eq:pseudo-homotopy} and Lemma \ref{th:nu-null-homotopy} (whose assumption is satisfied in the case relevant here, because of $\alpha_1^+>1$) yield
\begin{equation} \label{eq:chain-of-homotopies-2}
r_{q^{-1}\Omega|E} \htp \psi r_{Z^{(1)}|E} \htp 0.
\end{equation}

\begin{lemma} \label{th:0-is-0}
The two nullhomotopies \eqref{eq:chain-of-homotopies} and \eqref{eq:chain-of-homotopies-2} are essentially equivalent, which means that their difference (a chain map of degree $-1$) is itself nullhomotopic.
\end{lemma}

\begin{proof}
Explicitly, putting in all the necessary notation, the claim is that the map
\begin{equation} \label{eq:its-null}
\begin{aligned}
& \mathit{CF}^*(E,H_2^+) \longrightarrow \mathit{CF}^*(E,H^-), \\
& x \longmapsto \rho_{q^{-1}\Omega|E}^{-1}(x) + [b_{A|E},x]^{-1} + r_{A|E}(x) - \psi [\nu,x]^{-1} + \psi \chi(x) 
\end{aligned}
\end{equation}
where  $\rho^{-1}$ is from \eqref{eq:rho-for-a-cycle}, $r_{A|E}$ is from \eqref{eq:2nd-homotopy}, $\nu$ is from \eqref{eq:z-bounds}, and $\chi$ is from \eqref{eq:nu-null-homotopy}, is nullhomotopic. The nullhomotopy will be constructed in two parts, of which the first one is not surprising. Namely, the argument from Lemma \ref{th:r-is-b} can be generalized to proper pseudo-chains, in which case the expression in \eqref{eq:rho-for-a-cycle} acquires an additional term involving the boundary of the pseudo-chain.
%
Specializing to $A|E$, one gets
\begin{equation}
\rho_{q^{-1}\Omega|E}^{-1} + [b_{A|E},\cdot]^{-1} + r_{A|E} \htp \psi \rho_{Z^{(1)}|E}^{-1}.
\end{equation}
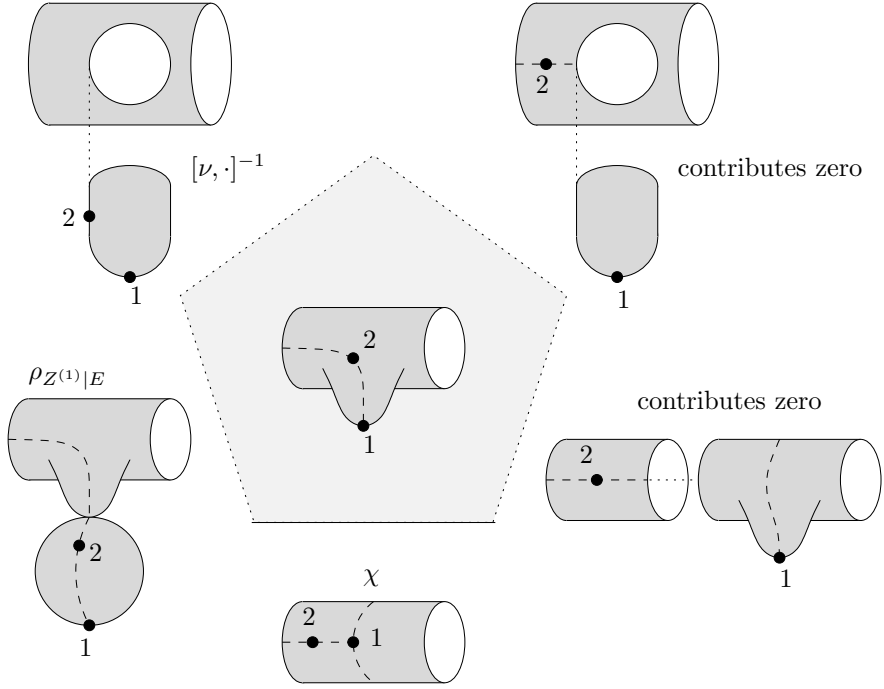
\begin{figure}
\begin{centering}
\begin{picture}(0,0)%
\includegraphics{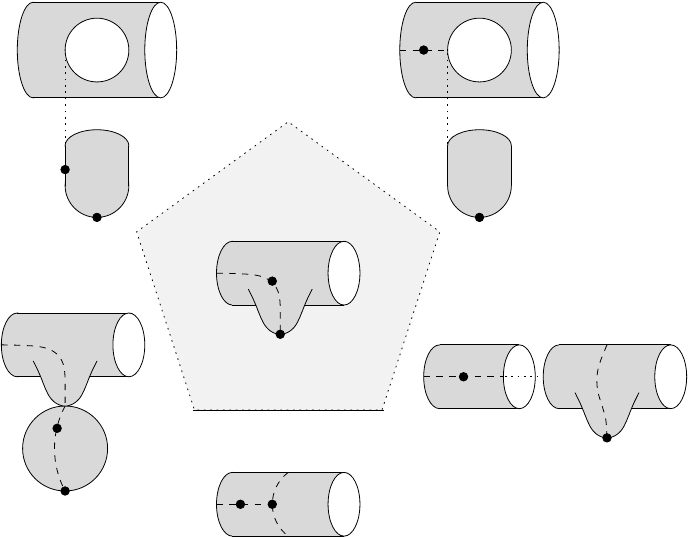}%
\end{picture}%
\setlength{\unitlength}{3355sp}%
\begingroup\makeatletter\ifx\SetFigFont\undefined%
\gdef\SetFigFont#1#2#3#4#5{%
  \reset@font\fontsize{#1}{#2pt}%
  \fontfamily{#3}\fontseries{#4}\fontshape{#5}%
  \selectfont}%
\fi\endgroup%
\begin{picture}(6470,5049)(1039,-4723)
\put(4951,-351){\makebox(0,0)[lb]{\smash{{\SetFigFont{10}{12.0}{\rmdefault}{\mddefault}{\updefault}{\color[rgb]{0,0,0}$2$}%
}}}}
\put(6001,-961){\makebox(0,0)[lb]{\smash{{\SetFigFont{10}{12.0}{\rmdefault}{\mddefault}{\updefault}{\color[rgb]{0,0,0}contributes zero}%
}}}}
\put(2401,-961){\makebox(0,0)[lb]{\smash{{\SetFigFont{10}{12.0}{\rmdefault}{\mddefault}{\updefault}{\color[rgb]{0,0,0}$[\nu,\cdot]^{-1}$}%
}}}}
\put(3676,-3961){\makebox(0,0)[lb]{\smash{{\SetFigFont{10}{12.0}{\rmdefault}{\mddefault}{\updefault}{\color[rgb]{0,0,0}$\chi$}%
}}}}
\put(1201,-2461){\makebox(0,0)[lb]{\smash{{\SetFigFont{10}{12.0}{\rmdefault}{\mddefault}{\updefault}{\color[rgb]{0,0,0}$\rho_{Z^{(1)}|E}$}%
}}}}
\put(3676,-2236){\makebox(0,0)[lb]{\smash{{\SetFigFont{10}{12.0}{\rmdefault}{\mddefault}{\updefault}{\color[rgb]{0,0,0}$2$}%
}}}}
\put(5551,-1936){\makebox(0,0)[lb]{\smash{{\SetFigFont{10}{12.0}{\rmdefault}{\mddefault}{\updefault}{\color[rgb]{0,0,0}$1$}%
}}}}
\put(5701,-2686){\makebox(0,0)[lb]{\smash{{\SetFigFont{10}{12.0}{\rmdefault}{\mddefault}{\updefault}{\color[rgb]{0,0,0}contributes zero}%
}}}}
\put(6751,-4001){\makebox(0,0)[lb]{\smash{{\SetFigFont{10}{12.0}{\rmdefault}{\mddefault}{\updefault}{\color[rgb]{0,0,0}$1$}%
}}}}
\put(5276,-3086){\makebox(0,0)[lb]{\smash{{\SetFigFont{10}{12.0}{\rmdefault}{\mddefault}{\updefault}{\color[rgb]{0,0,0}$2$}%
}}}}
\put(3726,-4436){\makebox(0,0)[lb]{\smash{{\SetFigFont{10}{12.0}{\rmdefault}{\mddefault}{\updefault}{\color[rgb]{0,0,0}$1$}%
}}}}
\put(3226,-4286){\makebox(0,0)[lb]{\smash{{\SetFigFont{10}{12.0}{\rmdefault}{\mddefault}{\updefault}{\color[rgb]{0,0,0}$2$}%
}}}}
\put(3676,-3011){\makebox(0,0)[lb]{\smash{{\SetFigFont{10}{12.0}{\rmdefault}{\mddefault}{\updefault}{\color[rgb]{0,0,0}$1$}%
}}}}
\put(1651,-3811){\makebox(0,0)[lb]{\smash{{\SetFigFont{10}{12.0}{\rmdefault}{\mddefault}{\updefault}{\color[rgb]{0,0,0}$2$}%
}}}}
\put(1576,-4511){\makebox(0,0)[lb]{\smash{{\SetFigFont{10}{12.0}{\rmdefault}{\mddefault}{\updefault}{\color[rgb]{0,0,0}$1$}%
}}}}
\put(1951,-1911){\makebox(0,0)[lb]{\smash{{\SetFigFont{10}{12.0}{\rmdefault}{\mddefault}{\updefault}{\color[rgb]{0,0,0}$1$}%
}}}}
\put(1451,-1336){\makebox(0,0)[lb]{\smash{{\SetFigFont{10}{12.0}{\rmdefault}{\mddefault}{\updefault}{\color[rgb]{0,0,0}$2$}%
}}}}
\end{picture}%
\caption{\label{fig:last}The moduli space underlying \eqref{eq:pentagon-triangle}.}
\end{centering}
\end{figure}

The second part of our proof is therefore to provide a homotopy
\begin{equation} \label{eq:pentagon-triangle}
\rho_{Z^{(1)}|E}^{-1} - [\nu,\cdot]^{-1} + \chi \htp 0.
\end{equation}
To construct that, we take the idea of pulling out a marked point (from Lemma \ref{th:r-is-b}) and apply it to the moduli space underlying $\chi$ (from Section \ref{subsec:construct-chi}); more precisely, the point we pull out is $\zeta_1$ from \eqref{eq:two-points-correlated}. The outcome is a three-dimensional parameter space with one boundary component. Figure \ref{fig:last} shows a schematic picture of that space together with the relevant compactification, where the limits lie  (this is a dimensionally reduced picture; the actual space is an $S^1$-bundle over it). Some degenerations contribute zero, for topological reasons (incidence conditions would force each of the two maps to have positive degree, in the sense of \eqref{eq:define-degree}, which is impossible). The remaining ones, together with the boundary, give rise to the three terms in \eqref{eq:pentagon-triangle}.

Rather than discussing the technical aspects of the construction, which offer nothing substantially new at this point, we want to focus on one fundamental issue. Gluing together the surfaces involved in the expression $[\nu,\cdot]^{-1}$ (for a fixed large gluing lenght) yields a family of surfaces which are all cylinders with two marked points $\zeta_1,\zeta_2$. The specific choice of framing in the definition of $[\cdot,\cdot]^{-1}$ allows us to arrange that $\zeta_2$ always lies on the straight horizontal half-line connecting $\zeta_1$ to the negative end (in more topological terms, as $\zeta_1$ moves around the circle, $\zeta_2-\zeta_1$ describes a small loop with winding number zero around the origin). This explains why this particular bracket arises.
\end{proof}

Let's introduce some notation:
\begin{align}
\label{eq:k-notation}
& k = b_{q^{-1}\Omega|E}, \\
\label{eq:kappa-notation} 
& \kappa = \beta_{q^{-1}\Omega|E}, \\
\label{eq:theta-new}
& \theta = -b_{A|E} - \psi \nu, \\
\label{eq:a-new}
& a = \Delta \theta - \kappa,
\end{align}
where \eqref{eq:kappa-notation} is a special case of the cochains introduced in Lemma \ref{th:bv-annihilates}. We can rewrite \eqref{eq:betaq} and \eqref{eq:z-bounds} as
\begin{align}
& d \kappa + \Delta k = 0, \\
& d \theta = k.
\end{align}
From the viewpoint of Discussion \ref{th:hf-plus}, this means that we get a class
\begin{equation} \label{eq:t-new-class}
t = [(-q^{-1}\Omega|E,\theta)] \in \mathit{HF}^1(E,\alpha^+_1)_{\mathit{red}},
\end{equation}
which maps to $q^{-1}[\omega_E]$ under the connecting map; and
\begin{equation} \label{eq:a-new-class}
a = \Delta_{\mathit{red}} t \in \mathit{HF}^0(E,\alpha_1^+).
\end{equation}

\begin{lemma} \label{th:determine-a}
The cocycle \eqref{eq:a-new} is cohomologous to $-\psi s$.
\end{lemma}

\begin{proof}
By Proposition \ref{th:interpret-bs}, 
\begin{equation} \label{eq:what-is-a}
\begin{aligned}
a & = -\Delta b_{A|E} - \psi \Delta \nu - \beta_{q^{-1}\Omega|E} \\
& = - \psi s + (\psi \beta_{Z^{(1)}|E} - \beta_{q^{-1}\Omega|E} -\Delta b_{A|E}) + \text{\it coboundary}.
\end{aligned}
\end{equation}
We have to show that the term in brackets in \eqref{eq:what-is-a} is nullhomologous. That is actually an instance of a general fact, valid for any proper pseudo-chain: in the same way as in Lemma \ref{th:bv-annihilates}, one can find a Floer cochain $\beta_{A|E}$ satisfying
\begin{equation}
d\beta_{A|E} + \Delta b_{A|E} = \beta_{\partial A|E} = \psi \beta_{Z^{(1)}|E} - \beta_{q^{-1}\Omega|E}.
\end{equation}
\end{proof}

Finally, we can use Lemma \ref{th:0-is-0} to rewrite the connection \eqref{eq:chi-connection-induced}, for $\alpha^+ = \alpha^+_2$, as  
\begin{equation} \label{eq:alt-nabla-1}
\begin{aligned}
& \mathit{HF}^*(E,\alpha_2^+) \longrightarrow \mathit{HF}^*(E,\alpha^-), \\
& [x] \longmapsto [C(\partial_q x) - h^+(x) + \rho^{-1}_{q^{-1}\Omega|E}(x) - [\theta,x]^{-1}].
\end{aligned}
\end{equation}

\begin{remark}
At this point, the similarity with the axiomatic constructions from Section \ref{sec:bv} is already evident. The family \eqref{eq:three-ends} defining $[\cdot,\cdot]^{-1}$ matches that from Figure \ref{fig:bv} (including the orientation convention for the parameter spaces). The formula \eqref{eq:bracket-1-again} agrees exactly with that in \eqref{eq:epsilon-homotopy}. There is also a partial similarity between \eqref{eq:nabla-1} and \eqref{eq:alt-nabla-1} (the difference is the appearance of continuation maps: geometrically, we can't get connections on Floer cohomology groups for a single value of $\alpha$).
\end{remark}

\section{Symplectic cohomology\label{sec:symplectic-cohomology}}

This section (finally) formally introduces symplectic cohomology, as well as its standard structure of operations, in the version suitable for our setup. Combining the abstract ideas from Sections \ref{subsec:differentiation-axiom} and \ref{subsec:differentiation-axiom-2} with some Floer theory from Section \ref{subsec:bracket-and-connection}, we carry out the construction of connections on symplectic cohomology. Finally, we need to discuss the compatibility between this and the previous material.

\subsection{The definition\label{subsec:define-sh}}
We consider a more general geometric situation than in Section \ref{subsec:define-floer}; it includes Lefschetz fibrations over $\bC$ with nontrivial monodromy at infinity. This is essentially for didactic purposes: we want to show that, for the abstract theory, the precise structure of our manifolds is less relevant. We retain the Calabi-Yau assumption, but only for its technical role in pseudo-holomorphic curve theory: in principle, it could be modified or dropped, with the usual implications (loss of the $\bZ$-grading, and technical complications in transversality arguments). Similarly, the theory could be set up in other situations, such as symplectic manifolds with contact type boundary (again, with modifications on the technical level).

\begin{setup} \label{th:sh-setup}
Let $E$ be an open symplectic manifold. We require that this comes with a proper map $p: E \rightarrow \bC$, which is a locally trivial symplectic fibration over $A = \{b \in \bC\,:\,|b| \geq r_0\}$, for some $r_0$. This means that every point $x \in p^{-1}(A)$ is a regular point, with $TE_x^v = \mathit{ker}(D\pi_x) \subset TE_x$ a symplectic subspace, and that its symplectic orthogonal complement $TE_x^h$ satisfies
\begin{equation}
(\omega_E - p^*\omega_{\bC}) \,|\, TE_x^h =  0
\end{equation}
for some choice (fixed once and for all) of rotationally invariant symplectic form $\omega_{\bC}$ on the base. In addition we will fix a sequence of annuli, called barriers,
\begin{equation}
W_j = \{r_j^- \leq |b| \leq r_j^+\} \subset A \subset \bC,
\end{equation}
where $r_0 < r_1^- < r_1^+ < r_2^- < \cdots$, going to $\infty$. Let $W = W_1 \cup W_2 \cup \cdots$ be their union. We also require that $c_1(E)$ should be zero, and choose a complex volume form $\eta$ (for some compatible almost complex structure). Finally, we fix a (proper) cycle representative $\Omega$ for $[\omega_E]$, in the same sense as in \eqref{eq:z-divisor}.
\end{setup}

\begin{setup} \label{th:winding-hj}
We will use compatible almost complex structures $J$ on $E$ such that $p$ is $J$-holomorphic on $p^{-1}(W)$. Similarly, we use Hamiltonian functions $H \in \smooth(E,\bR)$ whose associated vector field $X$ satisfies
\begin{equation} \label{eq:ak-rotation}
Dp_x(X) = a_j (2\pi i b \partial_b) \quad \text{for $x \in p^{-1}(W_j)$, with some $a_j \in \bR$.}
\end{equation}
The simplest way to achieve that is to take $X|p^{-1}(W_j)$ to be the unique horizontal (meaning, lying in $TE^h$) lift of $a_j(2\pi i b\partial_b)$. Indeed, for most of our discussion, it will be entirely sufficient to restrict to such choices.
\end{setup}

As usual, much of the behaviour of the solutions of the associated Cauchy-Riemann equations comes from elementary arguments involving their projection to $\bC$. The following three Lemmas summarize those arguments (compare \cite[Section 4.4]{seidel12b}, which applies the same idea in a slightly different geometric context).

\begin{lemma} \label{th:barrier-1}
Let $S$ be a compact connected Riemann surface, with nonempty boundary, carrying a one-form $\gamma$ such that $d\gamma \leq 0$. Fix some $j$, and an $r \in (r_j^-,r_j^+)$. Suppose that $v: S \rightarrow \bC$ satisfies
\begin{equation} \label{eq:partial-dbar}
\bar\partial v - 2\pi iv \gamma^{0,1} = 0 \quad \text{on $v^{-1}(W_j)$.}
\end{equation}
Suppose moreover that $|v(z)| = r$ for $z \in \partial S$, and that $v$ intersects the circle of radius $r$ transversally at each boundary point. Then there is some point $z \in S$ where $|v(z)| < r$.
\end{lemma}

\begin{proof} 
Let's suppose that on the contrary, $|v(z)| \geq r$ everywhere. Let $\xi$ be a vector tangent to $\partial S$, positive with respect to the boundary orientation. Then $Dv(i\xi)$ must have a positive radial component. From that and \eqref{eq:partial-dbar}, which applies at all boundary points of $S$, we get
\begin{equation}
\begin{aligned}
0 < \langle v, Dv(i\xi) \rangle_{\bR} & = \langle v, Dv(i\xi) - 2\pi i v \gamma(i\xi) \rangle_{\bR} \\
& = \langle iv, 2 \pi i v \gamma(\xi) - Dv(\xi) \rangle_{\bR} \\
& = r^2 (2\pi \gamma(\xi) - \xi.\mathrm{arg}(v)). 
\end{aligned}
\end{equation}
Integrating out yields the desired contradiction to $d\gamma \leq 0$:
\begin{equation} \label{eq:barrier-inequality}
0 < \int_{\partial S} 2\pi \gamma - d\mathrm{arg}(v) = \int_S 2\pi d\gamma.
\end{equation}
\end{proof}

\begin{lemma} \label{th:barrier-2}
Let $(S,\gamma)$ be as in the previous Lemma. Consider a map $v: S \rightarrow \bC$ which satisfies \eqref{eq:partial-dbar} for some $j$, and such that $|v(z)| < r_j^+$ for all $z \in \partial S$. Then, $|v(z)| < r_j^+$ everywhere.
\end{lemma}

\begin{proof}
One can find an $r \in (r_j^-,r_j^+)$ such that $|v(z)| < r$ for all $z \in \partial S$, and such that $v$ intersects the circle of radius $r$ transversally. If the intersection is empty, we have $|v(z)| < r < r_j^+$ everywhere. Otherwise, take $\tilde{S}$ to be one of the connected components of $\{z \in S \;:\; |v(z)| \geq r\}$. By applying Lemma \ref{th:barrier-1} to $\tilde{S}$, one obtains a contradiction.
\end{proof}

\begin{lemma} \label{th:barrier-3}
Let $S$ be as in Setup \ref{th:worldsheet}, except that we denote its one-form by $\gamma$ (instead of $\beta$), and where $\alpha^-$, $\alpha_i^+$ can be arbitrary real numbers. Consider a map $v: S \rightarrow \bC$ which satisfies \eqref{eq:partial-dbar} for some $j$, and such that $v^{-1}(W_j) \subset S$ is compact. Define
\begin{equation} \label{eq:caricature-action}
B(\zeta^-) = \begin{cases} 
\alpha^- - w^- & \text{if } |v(\epsilon^-(s,t))| > r_j^+ \text{ for all $s \ll 0$}, \\
0 & \text{if } |v(\epsilon^-(s,t))| < r_j^- \text{ for all $s \ll 0$},
\end{cases}
\end{equation}
where $w^- = w^-(v) \in \bZ$ is the winding number of the loop $v(\epsilon^-(s,\cdot))$ around the origin, for $s \ll 0$. The same idea, applied to the other ends, yields numbers $B(\zeta^+_i)$. Then, one necessarily has
\begin{equation} \label{eq:b-energy}
B(\zeta^-) \geq \sum_i B(\zeta^+_i).
\end{equation}
Moreover, equality in \eqref{eq:b-energy} can hold only if $v^{-1}(W_j \setminus \partial W_j) = \emptyset$.
\end{lemma}

\begin{proof}
If $|v(z)| \leq r_j^-$ for all $z \in S$, the desired inequality \eqref{eq:b-energy} holds trivially, since both sides are zero. If $|v(z)| \geq r_j^+$ for all $z \in S$, the inequality is true because $w^- = \sum_i w^+_i$ and, by \eqref{eq:sub-closed}, $\alpha^- \geq \sum_i \alpha^+_i$. We therefore focus on the remaining situation, where we expect a strict inequality in \eqref{eq:b-energy}.

Choose some $r \in (r_j^-,r_j^+)$ such that $v$ intersects the circle of radius $r$ transversally, and so that the intersection is not empty. Let 
\begin{equation}
\tilde{S} = \{z \in S\;:\; |v(z)| \geq r\}. 
\end{equation}
Then, \eqref{eq:barrier-inequality} yields
\begin{equation} \label{eq:cut-stokes}
0 < \int_{\partial \tilde{S}} 2\pi \gamma - d\arg(v) = \int_{\tilde{S}} 2\pi \, d\gamma + 2\pi \Big( B(\zeta^-) - \sum_i B(\zeta^+_i) \Big).
\end{equation}
To see the equality in \eqref{eq:cut-stokes}, one cuts off parts of the cylindrical ends of $\tilde{S}$, and then applies Stokes to the resulting compact surface. The winding number term comes from integrating $-d\arg(v)$ over the boundary circles associated to the ends, and the $\alpha^{\pm}$ similarly comes from integrating $\gamma$. Since $\int 2\pi d\gamma \leq 0$ by \eqref{eq:sub-closed}, the desired inequality follows.
\end{proof}

We will construct symplectic cohomology as the Floer cohomology of a single time-dependent Hamiltonian (for similar definitions, see \cite{abouzaid10, ganatra13}). More precisely:

\begin{setup} \label{th:strange-h}
Choose numbers
\begin{equation} \label{eq:the-ak}
\left\{
\begin{aligned}
& 0 < a_1 \leq a_2 \leq \dots \in \bR \setminus \bZ, \\
& \textstyle \lim_j a_j = \infty.
\end{aligned}
\right.
\end{equation}
We use a time-dependent Hamiltonian $H = (H_t)$ which, for each time $t$ and any $j$, satisfies \eqref{eq:ak-rotation} with the given $a_j$. As a consequence, any one-periodic orbit $x$ of the associated Hamiltonian vector field is disjoint from $p^{-1}(W)$. We require three additional properties to hold for all $x$. They should be nondegenerate; disjoint from $\Omega$; and finally,
\begin{equation} \label{eq:winding-number-inequality}
\text{if $|p(x(t))| > r_j^+$, the loop $p(x)$ winds $>a_j$ times around $0$.}
\end{equation}
\end{setup}

To show that these properties can be satisfied, let's first choose a function $\psi(r)$, $r \geq 0$, with 
\begin{equation}
\left\{
\begin{aligned}
& \psi(r) = 0 \quad \text{for $r \leq r_0$}, \\
& \psi'(r) = a_j \quad \text{if $r_j^- \leq r \leq r_j^+$,} \\
& \psi''(r) \geq 0 \quad \text{everywhere}.
\end{aligned}
\right.
\end{equation}
Take our standard rotational vector field $2\pi i b\partial_b$, multiply it by $\psi'(|b|)$, and then lift that to a vector field $\tilde{X}$ on $E$, as follows. On the preimage of the closed unit disc, $\tilde{X}$ shold be zero, and elsewhere, it should be horizontal (a section of $TE^h$). These conditions determine the lift uniquely. Moreover, the resulting vector field is Hamiltonian, and we fix a Hamiltonian function $\tilde{H}$ which induces it. Because of the convexity of $\psi$, this certainly satisfies \eqref{eq:winding-number-inequality}. Now take $H$ to be a sufficiently small time-dependent perturbation of $\tilde{H}$, which does not change the function on $p^{-1}(W)$. That perturbation will still satisfy \eqref{eq:winding-number-inequality}, and for a generic choice, it will achieve nondegeneracy. 

\begin{lemma} \label{th:c0-estimate}
Take $H$ as in Setup \ref{th:strange-h} and $J = (J_t)$ as in Setup \ref{th:winding-hj}. Then, the following holds for all solutions $u$ of the associated Floer equation \eqref{eq:floer}: if $|p(x^+)| < r_j^-$ for some $j$, then $|p(u)| < r_j^+$ everywhere.
\end{lemma}

\begin{proof}
By construction, the projection $v = p(u)$ satisfies \eqref{eq:partial-dbar} for any $j$. More precisely, the Riemann surface involved is $S = \bR \times S^1$, with $\gamma = a_j\mathit{dt}$ determined by \eqref{eq:the-ak}.

Suppose first that the negative limit $x^-$ satisfies $|p(x^-)| > r_j^+$. In the notation from \eqref{eq:caricature-action}, the property \eqref{eq:winding-number-inequality} says that $B(\zeta^-) < 0$, whereas $B(\zeta^+) = 0$. This violates Lemma \ref{th:barrier-3}, yielding a contradiction.

The previous argument, together with the fact that the one-periodic orbits are disjoint from $p^{-1}(W_j)$, shows that $|p(x^-)| < r_j^-$. Of course, the other limit already satisfies the same condition. We can therefore find a large compact subset $S \subset \bR \times S^1$, which is a Riemann surface with boundary, such that $|v(z)| < r_j^+$ for all $z \in \overline{(\bR \times S^1) \setminus S}$. Applying Lemma \ref{th:barrier-2} to the connected components of $S$ yields the desired result.
\end{proof}

Results like Lemma \ref{th:c0-estimate} are usually called $C^0$-bounds in the literature on symplectic cohomology. Once one has obtained such a bound, one can define $\mathit{CF}^*(E,H)$ and its differential $d$ as before \eqref{eq:floer-d}. Symplectic cohomology is its cohomology, $\mathit{SH}^*(E) = \mathit{HF}^*(E,H)$.

\begin{remark}
The Floer complex is defined as the direct sum of copies of $\bK$ associated to one-periodic orbits. In particular, infinite sums of the form
\begin{equation}
x_0 q^{d_0} + x_1 q^{d_1} + \cdots
\end{equation}
involving infinitely many distinct one-periodic orbits $x_i$ do not define Floer cochains, even if the 
$d_i$ go to infinity.
\end{remark}

The well-defined\-ness of symplectic cohomology is again established using continuation maps \eqref{eq:continuation-map}. Suppose that we have two collection of numbers $(a_j^{\pm})$ as in Setup \ref{th:strange-h}, satisfying
\begin{equation} \label{eq:ak-inequality}
a_j^- \geq a_j^+ \quad \text{for all $j$},
\end{equation}
and associated functions $H^{\pm}$. Fix a function $\phi: \bR \rightarrow [0,1]$ with $\phi(s) = 0$ for $s \ll 0$, $\phi(s) = 1$ for $s \gg 0$, and $\phi'(s) \geq 0$ for all $s$. In a preliminary step, consider 
\begin{equation}
\tilde{H}^C_{s,t} =  (1-\phi(s)) H^-_t + \phi(s) H^+_t.
\end{equation}
The actual $H^C$ will be a perturbation of $\tilde{H}^C$, but one which leaves it unchanged on $p^{-1}(W)$, and which still converges to $H^{\pm}$ exponentially fast as $s \rightarrow \pm\infty$. If $u$ is a solution of the associated continuation map equation, then $v = p(u)$ satisfies \eqref{eq:partial-dbar} for any $j$, with
\begin{equation}
\gamma = ((1-\phi(s))a_j^- + \phi(s)a_j^+) \mathit{dt}.
\end{equation}
Lemma \ref{th:c0-estimate} and its proof therefore carry over without any changes, leading one to define a continuation map with the usual uniqueness and composition properties. Before continuing, we want to collect some more basic observations.

\begin{discussion} \label{th:properties}
(i) For any $j$, write $\mathit{CF}^*(E,H)_{\leq j} \subset \mathit{CF}^*(E,H)$ for the subspace generated by one-periodic orbits which satisfy $|p(x)| < r_{j+1}^-$. Lemma \ref{th:c0-estimate} says that these subspaces are preserved by the differential, hence constitute an increasing filtration of the chain complex $\mathit{CF}^*(E,H)$. Denote their cohomology groups by $\mathit{HF}^*(E,H)_{\leq j}$. The continuation maps are compatible with this filtration (as are the homotopies that relate different choices of continuation maps). 

(ii) When $a_j^- = a_j^+$, one can define continuation maps in either direction, which are homotopy inverses of each other. Moreover, all of this is compatible with the previously introduced filtrations. Hence, for fixed $(a_j)$, both $\mathit{HF}^*(E,H)$ and $\mathit{HF}^*(E,H)_{\leq j}$ are well-defined up to canonical isomorphisms, and so are the continuation maps which relate one choice $(a_j^-)$ to another one $(a_j^+)$. To signal this, we will temporarily change notation to $\mathit{HF}^*(E,a_1,a_2,\dots)$ and $\mathit{HF}^*(E,a_1,a_2,\dots)_{\leq j}$. 

(iii) Suppose that there is some $j^*$ such that $a_j^+ = a_j^-$ for $j \leq j^*$. Then, the continuation map 
\begin{equation}
\mathit{HF}^*(E,a_1^+,a_2^+,\dots)_{\leq j^*} \longrightarrow \mathit{HF}^*(E,a_1^-,a_2^-,\dots)_{\leq j^*}
\end{equation}
is an isomorphism. In fact, by suitably correlating the choice of Hamiltonians, one can achieve that the underlying chain map is already an isomorphism.

(iv) In all the arguments above, one is free to ignore finitely many $j$. For instance, the sequence $(a_j)$ only has to be eventually nondecreasing. Similarly, one can define continuation maps as long as \eqref{eq:ak-inequality} holds for all but finitely many $j$. As a particularly simple application, one can arbitrarily change finitely many $a_j$, and $\mathit{HF}^*(E,a_1,a_2,\dots)$ remains the same up to canonical isomorphism.
\end{discussion}

\begin{lemma} \label{th:define-sh}
The continuation map is always an isomorphism.
\end{lemma}

\begin{proof}
Given $(a_j^\pm)$ as in \eqref{eq:ak-inequality}, and some $j^*$, one can find another choice $(a_j^0)$ such that 
\begin{equation}
\left\{
\begin{aligned}
& a_j^+ \leq a_j^0 \leq a_j^- \quad \text{ for all $j$,} \\
& a_j^0 = a_j^- \quad \text{ for $j \leq j^*$}, \\
& a_j^0 = a_j^+ \quad \text{ for all but finitely many $j$.}
\end{aligned}
\right.
\end{equation}
Consider the following commutative diagram, where all the vertical arrows are continuation maps: 
\begin{equation} \label{eq:filtered-continuation}
\xymatrix{
\mathit{HF}^*(E,a_1^+,a_2^+,\dots)_{\leq j^*} \ar[r] \ar[d]
& \mathit{HF}^*(E,a_1^+,a_2^+,\dots) \ar[d]^{\iso} \\ 
\mathit{HF}^*(E,a_1^0,a_2^0,\dots)_{\leq j^*} \ar[d]_-{\iso} \ar[r]
& \mathit{HF}^*(E,a_1^0,a_2^0,\dots) \ar[d] \\ 
\mathit{HF}^*(E,a_1^-,a_2^-,\dots)_{\leq j^*} \ar[r] &
\mathit{HF}^*(E,a_1^-,a_2^-,\dots) 
}
\end{equation}
The observation that two of these arrows are isomorphisms comes from Discussion \ref{th:properties}(iii) and (iv). Any class in $\mathit{HF}^*(E,a_1^-,a_2^-,\dots)$ comes from $\mathit{HF}^*(E,a_1^-,a_2^-,\dots)_{\leq j^*}$ for some $j^*$, and then diagram-chasing in \eqref{eq:filtered-continuation} exhibits a preimage in $\mathit{HF}^*(E,a_1^+,a_2^+,\dots)$. 

Similarly, pick an element of $\mathit{HF}^*(E,a_1^+,a_2^+,\dots)$ which maps to zero in $\mathit{HF}^*(E,a_1^-,a_2^-,\dots)$. Then, one can find some $j^*$ such that our element comes from $\mathit{HF}^*(E,a_1^+,a_2^+,\dots)_{\leq j^*}$ and already maps to zero in $\mathit{HF}^*(E,a_1^-,a_2^-,\dots)_{\leq j^*}$. Diagram-chasing shows that the original element was zero.
\end{proof}

\subsection{Operations on symplectic cohomology}
Symplectic cohomology carries a BV operator \eqref{eq:define-bv}. Along the same lines, it has an identity element \eqref{eq:unit-construction}. Given a proper pseudo-cycle $K$ with $\bK$-coefficents, one can define cocycles \eqref{eq:pss-q}, secondary cochains \eqref{eq:betaq}, and a chain map \eqref{eq:o-map}. In particular, Discussion \ref{th:hf-plus} carries over immediately to our context, leading to reduced symplectic cohomology \eqref{eq:plus-sequence} and its operator $\Delta_{\mathit{red}}$ from \eqref{eq:delta-minus}.

We now turn to operations associated to other surfaces (in parallel with Section \ref{sec:operations-on-floer-cohomology}). In the case of symplectic cohomology, the ends of our surface $S$ should come with choices of sequences $(a^-_j)$ and $(a^+_{i,j})$, respectively, such that
\begin{equation}
a^-_j \geq \sum_{i} a^+_{i,j} \quad \text{for all $j$.}
\end{equation}
One also supposes that $S$ comes with one-forms $\gamma_j$, each of which satisfies $d\gamma_j \leq 0$ as well as an analogue of \eqref{eq:beta-on-the-ends}:
\begin{equation} \label{eq:beta-on-the-ends-2}
\left\{
\begin{aligned}
& (\epsilon^-)^* \gamma_j|W_j = a^-_j \mathit{dt}, \\
& (\epsilon^+_i)^* \gamma_j|W_j = a^+_{i,j} \, \mathit{dt}.
\end{aligned}
\right.
\end{equation}
Given functions $H^-$ and $H^+_i$ associated to the ends, one chooses $K^S$ (in parallel with Setup \ref{th:worldsheet-2}) such that for each tangent vector $\xi$ on $S$, $K^S(\xi)$ belongs to the class of functions \eqref{eq:ak-rotation} with $a_j = \gamma_j(\xi)$. One then considers solutions of the same equation \eqref{eq:cauchy-riemann} as before (taking values in $E$), and gets a generalization of Lemma \ref{th:c0-estimate}:

\begin{lemma} \label{th:c0-estimate-2}
If, for some $j$, all $x^+_i$ satisfy $|p(x^+_i)| < r_j^-$, then $|p(u)| < r_j^+$ everywhere.
\end{lemma}

Given that, is is straightforward to equip symplectic cohomology with versions of the pair-of-pants product \eqref{eq:define-product} and bracket \eqref{eq:define-the-1-bracket}. Slightly more generally, suppose that we modify the last part of \eqref{eq:three-ends} by making one of the ends rotate in dependence of the modular parameter $\tau$:
\begin{equation} \label{eq:specific-end}
\epsilon_{\tau,1}^+(s,t) = (0,-\tau) - \exp(-2\pi(s+it) + 2\pi i(c+1)\tau)
\end{equation}
for some $c \in\bZ$. This gives rise to a family of bracket operations
\begin{equation} \label{eq:c-bracket}
[\cdot,\cdot]^c: \mathit{CF}^*(E,H_1^+) \otimes \mathit{CF}^*(E,H_2^+) \longrightarrow \mathit{CF}^{*-1}(E,H^-).
\end{equation}
The special case $[\cdot,\cdot] = [\cdot,\cdot]^0$ yields the bracket which, with the pair-of-pants product, forms the Gerstenhaber algebra structure on symplectic cohomology. We will use the same notational convention (of omitting the superscript for $c = 0$) several times from now on, without specifically pointing that out.
Since the only difference is the choice of framing at one end, it is not hard to produce chain homotopies which relate the brackets \eqref{eq:c-bracket} for different choices of $c$, in parallel with \eqref{eq:epsilon-homotopy}:
\begin{equation} \label{eq:all-brackets}
\begin{aligned}
& \eta^c: \mathit{CF}^*(E,H_1^+) \otimes \mathit{CF}^*(E,H_2^+) \longrightarrow \mathit{CF}^{*-2}(E,H^-), \\
& d\eta^c(x_1,x_2) - \eta^c(d x_1,x_2) - (-1)^{|x_1|} \eta^c(x_1,d x_2) \\
& \qquad + [x_1,x_2]^c - [x_1,x_2]^{c-1} + (\Delta x_1) \bullet x_2 = 0.
\end{aligned}
\end{equation}
For the proof of Lemma \ref{th:r-is-b}, the specific choice of framing at $\zeta_1^+$ is irrelevant; the same argument produces maps for any $c$,
\begin{equation} \label{eq:rho-for-a-cycle-2}
\begin{aligned}
& \rho^c_K: \mathit{CF}^*(E,H^+_2) \longrightarrow \mathit{CF}^{*+\mathrm{codim}(K)-2}(E,H^-), \\
& d \rho^c_K(x) - (-1)^{\mathrm{codim}(K)} \rho_K^c(d x) \\ & \qquad \qquad + (-1)^{\mathrm{codim}(K)}[b_K,x]^c - r_K(x) = 0.
\end{aligned}
\end{equation}
Note that the following is a chain map (of degree $\mathrm{codim}(K)-2$):
\begin{equation} \label{eq:m-map}
x \longmapsto \eta^c(b_K,x) + (-1)^{\mathrm{codim}(K)} (-\rho_K^c(x) + \rho^{c-1}_K(x))  - \beta_K \bullet x.
\end{equation}
Imitating the argument from \eqref{eq:secondary-epsilon} in our framework (which means that, instead of the operation of inserting disc configurations into each other, one has to use gluing in a Floer-theoretic sense), one sees that \eqref{eq:m-map} is nullhomotopic for $c = 0$; and it is straightforward to extend that to all $c \in \bZ$. We omit the details.

\subsection{Connections on symplectic cohomology}
Let's specialize to $K = q^{-1}\Omega$, and use the notation \eqref{eq:k-notation}, \eqref{eq:kappa-notation}. At this point, we impose Assumption \ref{th:as-1} and choose a bounding cochain, which means some $\theta \in \mathit{CF}^1(E,H_1^+)$ such that $d\theta = k$. As in \eqref{eq:t-new-class} and \eqref{eq:a-new-class}, this gives rise to classes
\begin{align}
& t = [(-q^{-1}\Omega, \theta)] \in \mathit{SH}^1(E)_{\mathit{red}}, \label{eq:infinite-t} \\
& a = [\Delta \theta - \kappa] = \Delta_{\mathit{red}} t \in \mathit{SH}^0(E). \label{eq:infinite-a}
\end{align}
One has a chain homotopy as in \eqref{eq:1st-homotopy}, defined in exactly the same way. As a consequence, one can introduce connections
\begin{equation} \label{eq:alt-nabla-m}
\begin{aligned}
& \nabla^c: \mathit{SH}^*(E) \longrightarrow \mathit{SH}^*(E), \\
& [x] \longmapsto [C(\partial_q x) - h^+(x) + \rho^c_{q^{-1}\Omega}(x) - [\theta,x]^c].
\end{aligned}
\end{equation}
One can show (but we omit the argument, which is fairly tedious) that these connections commute with continuation maps, hence are independent of the auxiliary data used to define symplectic cohomology. Moreover, using the fact that \eqref{eq:m-map} is nullhomotopic, it is straightforward to show the following result, which is modelled on our previous Lemma \ref{th:01-connections}:

\begin{lemma} \label{th:01-connections-2}
$\nabla^c x = \nabla x + c a \bullet x$ for all $c \in \bZ$.
\end{lemma}

Our next task is to show that $\nabla$ is compatible with the Gerstenhaber algebra structure, meaning that it satisfies \eqref{eq:nabla-derivation} and \eqref{eq:nabla-derivation-2}. The proof is a translation (Floerization) of Proposition \ref{th:nabla-gerstenhaber}. The main difficulty in carrying out such a translation is the appearance of continuation map terms in \eqref{eq:alt-nabla-m}. To explain how one addresses that, we will consider just the first property \eqref{eq:nabla-derivation}, for which the technical form of the statement is as follows:

\begin{proposition} \label{th:ninth-circle}
Take Floer complexes $\mathit{CF}^*(E,H_i^+)$ ($i = 1,2,3$) underlying symplectic cohomology, with their sequences $(a_{i,j}^+)$. Suppose that the bounding cochain $\theta$ lies in $\mathit{CF}^1(E,H_1^+)$. Define $\mathit{CF}^*(E,H^-)$ with underlying sequence $(a_j^- = 3a^+_{1,j} + a^+_{2,j} + a^+_{3,j})$. Then, the chain map
\begin{equation} \label{eq:p1}
\begin{aligned}
& \mathit{CF}^*(E,H_2^+) \otimes \mathit{CF}^*(E,H_3^+) \longrightarrow \mathit{CF}^*(E,H^-), \\
& (x_2,x_3) \longmapsto
C \partial_q(Cx_2 \bullet Cx_3) - h^+(Cx_2 \bullet Cx_3) \\ & \qquad \qquad + \rho_{q^{-1}\Omega}(Cx_2 \bullet Cx_3) - [\theta, Cx_2 \bullet Cx_3]
\end{aligned}
\end{equation}
is homotopic to the sum of the following two maps:
\begin{align}
\label{eq:p2}
& C\big( (\partial_q (Cx_2) + h^-(x_2) + \rho_{q^{-1}\Omega}(x_2) - [\theta,x_2]) \bullet Cx_3 \big), \\
\label{eq:p3}
& C\big(Cx_2 \bullet (\partial_q (Cx_3) + h^-(x_3) + \rho_{q^{-1}\Omega}(x_3) - [\theta,x_3])\big).
\end{align}
\end{proposition}

In the expressions above, the same notation is used for related but different maps, which calls for some explanation. In \eqref{eq:p1}, we consider $Cx_i \in \mathit{CF}^*(E,H^>_i)$, which are Floer cochain groups associated to sequences $(a^>_{i,j} = a^+_{1,j} + a^+_{i,j})$. Their pair-of-pants product is $C x_2 \bullet C x_3 \in \mathit{CF}^*(E,H^<)$, where the Floer cochain group is associated to the sequence $(a^<_j = 2a^+_{1,j} + a^+_{2,j} + a^+_{3,j})$. One then applies a form of the connection \eqref{eq:alt-nabla-m} to get from that group to $\mathit{CF}^*(E,H^-)$. In \eqref{eq:p2}, the connection which appears goes from $\mathit{CF}^*(E,H^+_2)$ to $\mathit{CF}^*(E,H^>_2)$, and its formula has been tweaked as in Discussion \ref{th:hplus-hminus}. The pair-of-pants product is the same as in \eqref{eq:p1}, and the continuation map which is applied last goes from $\mathit{CF}^*(E,H^<)$ to $\mathit{CF}^*(E,H^-)$. The situation for \eqref{eq:p3} is analogous.  The reason for ``padding'' \eqref{eq:p1}--\eqref{eq:p3} with extra copies of continuation maps is to ensure that only one kind of pair-of-pants product appears in the terms involving differentiation.
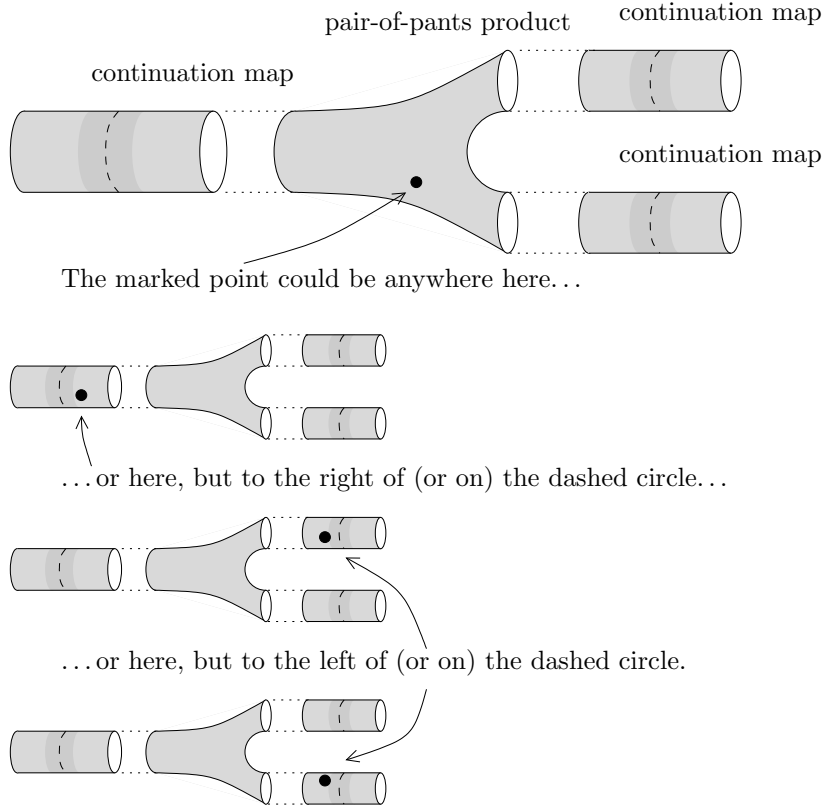
\begin{figure}
\begin{centering}
\begin{picture}(0,0)%
\includegraphics{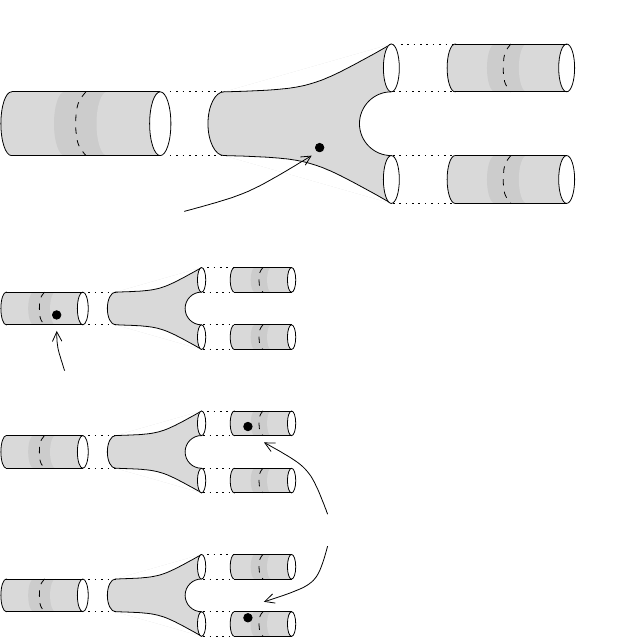}%
\end{picture}%
\setlength{\unitlength}{3355sp}%
\begingroup\makeatletter\ifx\SetFigFont\undefined%
\gdef\SetFigFont#1#2#3#4#5{%
  \reset@font\fontsize{#1}{#2pt}%
  \fontfamily{#3}\fontseries{#4}\fontshape{#5}%
  \selectfont}%
\fi\endgroup%
\begin{picture}(6030,5982)(-4132,-8122)
\put(376,-2311){\makebox(0,0)[lb]{\smash{{\SetFigFont{10}{12.0}{\rmdefault}{\mddefault}{\updefault}{\color[rgb]{0,0,0}continuation map}%
}}}}
\put(-1799,-2386){\makebox(0,0)[lb]{\smash{{\SetFigFont{10}{12.0}{\rmdefault}{\mddefault}{\updefault}{\color[rgb]{0,0,0}pair-of-pants product}%
}}}}
\put(-3524,-2761){\makebox(0,0)[lb]{\smash{{\SetFigFont{10}{12.0}{\rmdefault}{\mddefault}{\updefault}{\color[rgb]{0,0,0}continuation map}%
}}}}
\put(376,-3361){\makebox(0,0)[lb]{\smash{{\SetFigFont{10}{12.0}{\rmdefault}{\mddefault}{\updefault}{\color[rgb]{0,0,0}continuation map}%
}}}}
\put(-3749,-5761){\makebox(0,0)[lb]{\smash{{\SetFigFont{10}{12.0}{\rmdefault}{\mddefault}{\updefault}{\color[rgb]{0,0,0}\dots or here, but to the right of (or on) the dashed circle\dots}%
}}}}
\put(-3749,-4286){\makebox(0,0)[lb]{\smash{{\SetFigFont{10}{12.0}{\rmdefault}{\mddefault}{\updefault}{\color[rgb]{0,0,0}The marked point could be anywhere here\dots}%
}}}}
\put(-3749,-7111){\makebox(0,0)[lb]{\smash{{\SetFigFont{10}{12.0}{\rmdefault}{\mddefault}{\updefault}{\color[rgb]{0,0,0}\dots or here, but to the left of (or on) the dashed circle.}%
}}}}
\end{picture}%
\caption{\label{fig:ccc}The first two lines of \eqref{eq:to-show}.}
\end{centering}
\end{figure}

\begin{proof}[Proof of Proposition \ref{th:ninth-circle}]
Our claim amounts to the nullhomotopy
\begin{equation} \label{eq:to-show}
\begin{aligned}
&  C \big( \partial_q(Cx_2 \bullet Cx_3) - \partial_q (Cx_2) \bullet Cx_3 - Cx_2 \bullet \partial_q( Cx_3) \big) \\ & - h^+(C(x_2) \bullet C(x_3)) - C( h^-(x_2) \bullet C(x_3)) - C(C(x_2) \bullet h^-(x_3)) \\ & + \rho_{q^{-1}\Omega}(C(x_2) \bullet C(x_3)) -
C(\rho_{q^{-1}\Omega}(x_2) \bullet C(x_3)) \\ & - C(C(x_2) \bullet \rho_{q^{-1}\Omega}(x_3))  + [\theta, C(x_2) \bullet C(x_3)] - C([\theta,x_2] \bullet C(x_3)) 
\\ & - C(C(x_2) \bullet [\theta,x_3]) \htp 0.
\end{aligned}
\end{equation}
In the same spirit as \eqref{eq:differentiate-d}, the first two lines of \eqref{eq:to-show} can be given the following geometric interpretation. Think of the disconnected Riemann surface $S = S^- \cup S^0 \cup S^+_2 \cup S^+_3$, where $S^0$ is an open pair-of-pants, and the other components are copies of $\bR \times S^1$ (see Figure \ref{fig:ccc}). Let's suppose that these components come equipped with the auxiliary data that define the pair-of-pants product and continuation maps
\begin{equation}
\left\{
\begin{aligned}
& \mathit{CF}^*(E,H^<) \longrightarrow \mathit{CF}^*(E,H^-), \\
& \mathit{CF}^*(E,H^>_2) \otimes \mathit{CF}^*(E,H^>_3) \longrightarrow \mathit{CF}^*(E,H^<), \\
& \mathit{CF}^*(E,H^+_i) \longrightarrow \mathit{CF}^*(E,H^>_i).
\end{aligned}
\right.
\end{equation}
Let $\epsilon^-$ and $\epsilon^+_i$ be the ends of $S^0$. One considers solutions of suitable pair-of-pants and continuation map equations defined on each component, with matching limits:
\begin{equation}
\left\{
\begin{aligned}
& u^0: S \longrightarrow E, \\
& u^-: S^- = \bR \times S^1 \longrightarrow E, \\
& u^+_i: S^+_i = \bR \times S^1 \longrightarrow E, \\
& \textstyle \lim_{s \rightarrow \infty} u^-(s,\cdot) = \lim_{s \rightarrow -\infty}
u^0(\epsilon^-(s,\cdot)), \\
& \textstyle \lim_{s \rightarrow \infty} u^0(\epsilon^+_i(s,\cdot)) = \lim_{s \rightarrow -\infty} u^+_i(s,\cdot).
\end{aligned}
\right.
\end{equation}
Counting such solutions defines the operation $C(Cx_2 \bullet Cx_3)$. Now suppose that our surface comes with an additional marked point $\zeta$, which can be
\begin{equation}
\left\{
\begin{aligned}
& \text{an arbitrary $\zeta \in S$,} \\
& \text{or $\zeta = (s,t) \in S^-$ with $s \geq 0$,} \\
& \text{or $\zeta = (s,t) \in S^+_i$ with $s \leq 0$.}
\end{aligned}
\right.
\end{equation}
We require that the image of that point should go through $q^{-1}\Omega$, and count solutions accordingly. That (depending on which component of $S$ the point lies in) gives rise to four different expressions, which are exactly the four first terms of \eqref{eq:to-show} (see again Figure \ref{fig:ccc}).

One can construct an analogue of \eqref{eq:gamma-operation}, namely a map
\begin{equation} \label{eq:new-gamma}
\begin{aligned}
& \gamma: \mathit{CF}^*(E,H_1^+) \otimes \mathit{CF}^*(E,H_2^+) \otimes \mathit{CF}^*(E,H_3^+) \longrightarrow \mathit{CF}^{*-2}(E,H^-), \\
& d\gamma(x_1,x_2,x_3) - \gamma(dx_1,x_2,x_3) - (-1)^{|x_1|} \gamma(x_1,dx_2,x_3) \\
& - (-1)^{|x_1|+|x_2|} \gamma(x_1,x_2,dx_3) + [x_1,C(x_2) \bullet C(x_3)] \\ & - C([x_1,x_2] \bullet C(x_3)) - (-1)^{|x_2|(|x_1|+1)} C(C(x_2) \bullet [x_1,x_3]) = 0.
\end{aligned}
\end{equation}
Again, the definition involves four summands, which one can think of as being constructed starting from the same surface $S$ and marked point $\zeta$ as before. One removes $\zeta$, obtaining another end, and equips the result with auxiliary data that are asymptotic to those associated to $\mathit{CF}^*(E,H_1^+)$ over that end (see Figure \ref{fig:ccc2} for a schematic description).
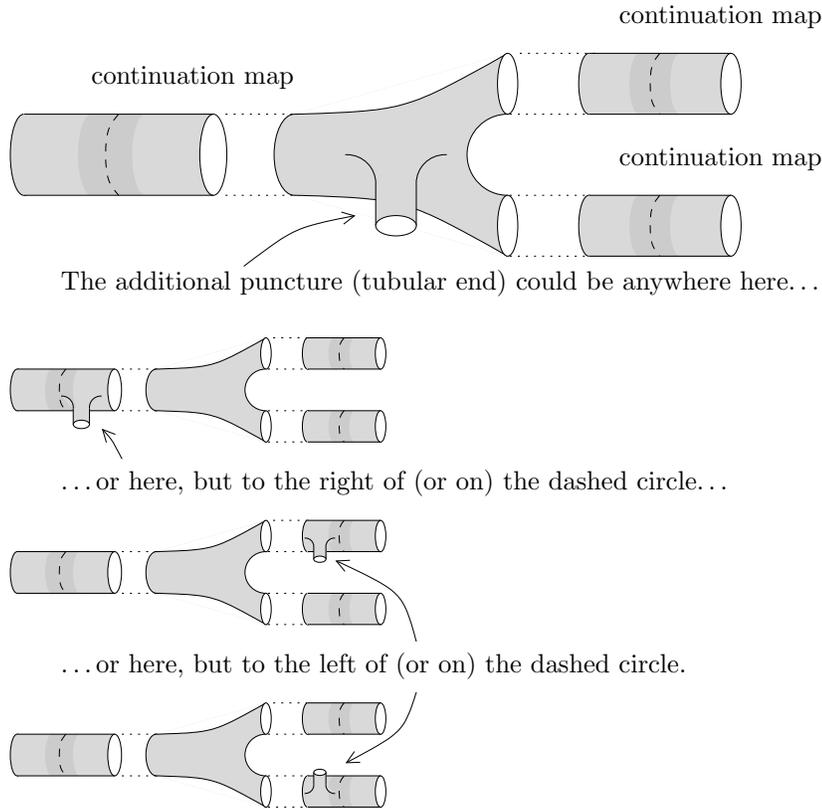
\begin{figure}
\begin{centering}
\begin{picture}(0,0)%
\includegraphics{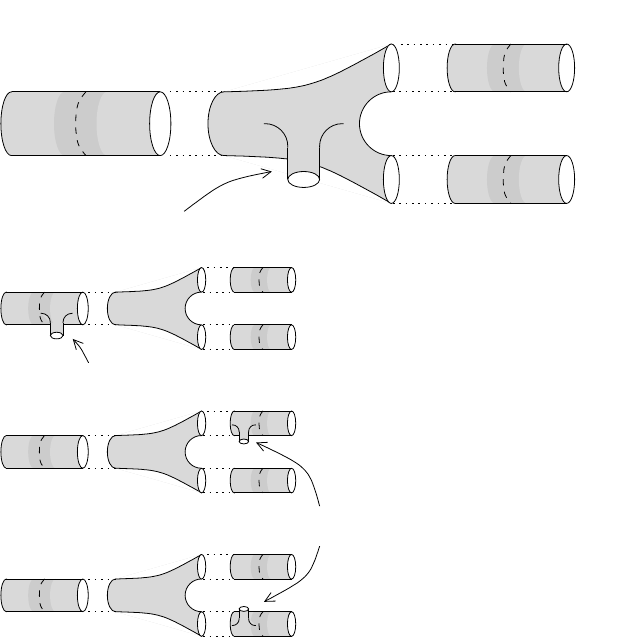}%
\end{picture}%
\setlength{\unitlength}{3355sp}%
\begingroup\makeatletter\ifx\SetFigFont\undefined%
\gdef\SetFigFont#1#2#3#4#5{%
  \reset@font\fontsize{#1}{#2pt}%
  \fontfamily{#3}\fontseries{#4}\fontshape{#5}%
  \selectfont}%
\fi\endgroup%
\begin{picture}(6030,5982)(-4132,-8122)
\put(-3524,-2761){\makebox(0,0)[lb]{\smash{{\SetFigFont{10}{12.0}{\rmdefault}{\mddefault}{\updefault}{\color[rgb]{0,0,0}continuation map}%
}}}}
\put(376,-3361){\makebox(0,0)[lb]{\smash{{\SetFigFont{10}{12.0}{\rmdefault}{\mddefault}{\updefault}{\color[rgb]{0,0,0}continuation map}%
}}}}
\put(376,-2311){\makebox(0,0)[lb]{\smash{{\SetFigFont{10}{12.0}{\rmdefault}{\mddefault}{\updefault}{\color[rgb]{0,0,0}continuation map}%
}}}}
\put(-3749,-4286){\makebox(0,0)[lb]{\smash{{\SetFigFont{10}{12.0}{\rmdefault}{\mddefault}{\updefault}{\color[rgb]{0,0,0}The additional puncture (tubular end) could be anywhere here\dots}%
}}}}
\put(-3749,-5761){\makebox(0,0)[lb]{\smash{{\SetFigFont{10}{12.0}{\rmdefault}{\mddefault}{\updefault}{\color[rgb]{0,0,0}\dots or here, but to the right of (or on) the dashed circle\dots}%
}}}}
\put(-3749,-7111){\makebox(0,0)[lb]{\smash{{\SetFigFont{10}{12.0}{\rmdefault}{\mddefault}{\updefault}{\color[rgb]{0,0,0}\dots or here, but to the left of (or on) the dashed circle.}%
}}}}
\end{picture}%
\caption{\label{fig:ccc2}The construction of \eqref{eq:new-gamma}.}
\end{centering}
\end{figure}

Without changing the statement, one can replace the last three terms in \eqref{eq:to-show} with $-\gamma(k,x_2,x_3)$, following what we've done in \eqref{eq:nullhomotopy-product}. At this point, the required construction is clear: one starts from Figure \ref{fig:ccc} and ``pulls out'' the additional marked point, a process which in the limit yields Figure \ref{fig:ccc2} with a copy of $k$ inserted at the extra end. The associated three-parameter moduli space has other relevant boundary components, namely those where the ``pulling out'' process it being carried out at a point which lies on the boundary of its allowed domain (in Figures \ref{fig:ccc}--\ref{fig:ccc2}, the dashed circles); the contributions of those precisely yield the terms involving $\rho_{q^{-1}\Omega}$ in \eqref{eq:to-show}.
\end{proof}

We omit the proof of the appropriate version of \eqref{eq:nabla-derivation-2}, which is parallel; and this concludes our discussion of Proposition \ref{th:1}. The Floerization of Proposition \ref{th:nabla-bv} is this:

\begin{proposition} \label{th:ninth-circle-2}
Take Floer complexes $\mathit{CF}^*(E,H_i^+)$ ($i = 1,2$) underlying symplectic cohomology, with their sequences $(a_{i,j}^+)$. Suppose that the bounding cochain $\theta$ lies in $\mathit{CF}^1(E,H_1^+)$. Define $\mathit{CF}^*(E,H^-)$ with underlying sequence $(a_j^- = 2a^+_{1,j} + a^+_{2,j})$. Then, the following two chain maps are homotopic:
\begin{equation} \label{eq:b1}
\begin{aligned}
& \mathit{CF}^*(E,H_2^+) \longrightarrow \mathit{CF}^*(E,H^-), \\
& x_2 \longmapsto C\partial_q (\Delta C x_2) - h^+(\Delta C x_2) + \rho^{-1}_{q^{-1}\Omega}(\Delta Cx_2) - [\theta,\Delta Cx_2]^{-1}, \\
& x_2 \longmapsto C \Delta\big(\partial_q C(x_2) + h^-(x_2) + \rho^{-1}_{q^{-1}\Omega}(x_2) - [\theta,x_2]^{-1}\big).
\end{aligned}
\end{equation}
\end{proposition}

\begin{proof}
This follows the model of Proposition \ref{th:ninth-circle}. Our claim is
\begin{equation} \label{eq:to-show-2}
\begin{aligned}
& C( \partial_q \Delta - \Delta \partial_q)(Cx_2)
- h^+(\Delta C x_2) - C \Delta h^-(x_2) \\
& \quad + \rho^{-1}_{q^{-1}\Omega}(\Delta Cx_2) - C \Delta \rho^{-1}_{q^{-1}\Omega}(x_2) \\
& \quad - [\theta, \Delta Cx_2]^{-1} + C \Delta [\theta,x_2]^{-1} \htp 0. 
\end{aligned}
\end{equation}
The first line of \eqref{eq:to-show-2} can be geometrically interpreted as counting configurations as indicated in Figure \ref{fig:ccc3}. Let's replace the marked point in that picture by an additional end (whose framing agrees with that used to define $[\cdot,\cdot]^{-1}$). The outcome is a version of \eqref{eq:chi-operation}, namely a map
\begin{equation}
\begin{aligned}
& \varpi: \mathit{CF}^*(H_1^+) \otimes \mathit{CF}^*(E,H_2^+) \longrightarrow \mathit{CF}^{*-3}(E,H^-), \\
& d\varpi(x_1,x_2) + \varpi(dx_1,x_2) + (-1)^{|x_1|} \varpi(x_1,dx_2) \\ & \qquad \qquad + C\Delta[x_1,x_2]^{-1} + (-1)^{|x_1|} [x_1,\Delta C x_2]^{-1} = 0,
\end{aligned}
\end{equation}
which one can use to replace the last line in \eqref{eq:to-show-2} with $-\varpi(k,x)$. Interpolating between those two constructions yields the desired nullhomotopy.
\end{proof}
\begin{figure}
\begin{centering}
\begin{picture}(0,0)%
\includegraphics{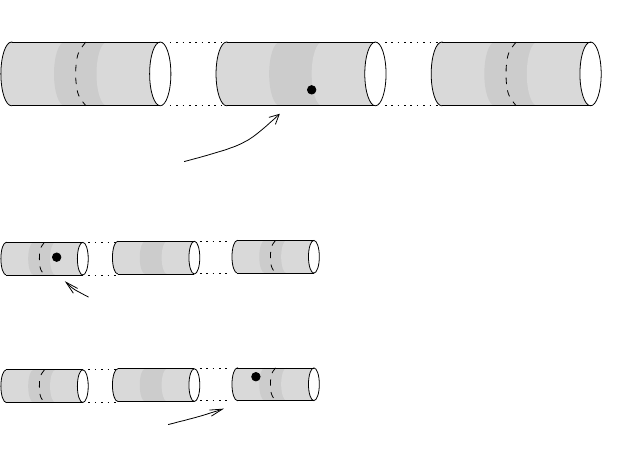}%
\end{picture}%
\setlength{\unitlength}{3355sp}%
\begingroup\makeatletter\ifx\SetFigFont\undefined%
\gdef\SetFigFont#1#2#3#4#5{%
  \reset@font\fontsize{#1}{#2pt}%
  \fontfamily{#3}\fontseries{#4}\fontshape{#5}%
  \selectfont}%
\fi\endgroup%
\begin{picture}(5805,4294)(-4132,-6884)
\put(151,-2761){\makebox(0,0)[lb]{\smash{{\SetFigFont{10}{12.0}{\rmdefault}{\mddefault}{\updefault}{\color[rgb]{0,0,0}continuation map}%
}}}}
\put(-3974,-2761){\makebox(0,0)[lb]{\smash{{\SetFigFont{10}{12.0}{\rmdefault}{\mddefault}{\updefault}{\color[rgb]{0,0,0}continuation map}%
}}}}
\put(-1724,-2761){\makebox(0,0)[lb]{\smash{{\SetFigFont{10}{12.0}{\rmdefault}{\mddefault}{\updefault}{\color[rgb]{0,0,0}BV operator}%
}}}}
\put(-3749,-5611){\makebox(0,0)[lb]{\smash{{\SetFigFont{10}{12.0}{\rmdefault}{\mddefault}{\updefault}{\color[rgb]{0,0,0}\dots or here, but to the right of (or on) the dashed circle\dots}%
}}}}
\put(-3749,-6811){\makebox(0,0)[lb]{\smash{{\SetFigFont{10}{12.0}{\rmdefault}{\mddefault}{\updefault}{\color[rgb]{0,0,0}\dots or here, but to the left of (or on) the dashed circle.}%
}}}}
\put(-3749,-4286){\makebox(0,0)[lb]{\smash{{\SetFigFont{10}{12.0}{\rmdefault}{\mddefault}{\updefault}{\color[rgb]{0,0,0}The marked point could be anywhere here\dots}%
}}}}
\end{picture}%
\caption{\label{fig:ccc3}The geometric meaning of the first two lines of \eqref{eq:to-show-2}.}
\end{centering}
\end{figure}
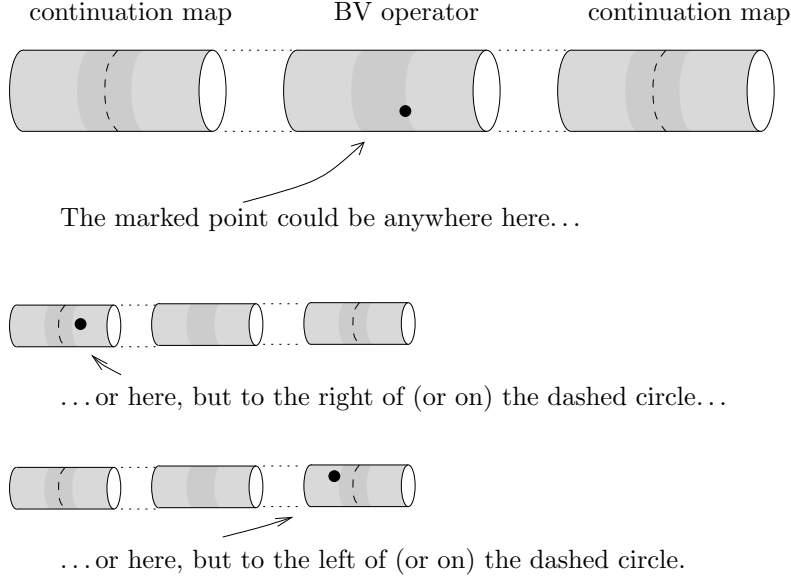%

Exactly as in  Corollary \ref{th:nabla-delta-cor}, combining Lemma \ref{th:01-connections-2} and Proposition \ref{th:ninth-circle-2} shows that $\nabla$ satisfies \eqref{eq:delta-nabla}, hence proves Proposition \ref{th:a-element}.

\subsection{Comparing the two versions of Floer cohomology}
We will now address the relation between symplectic cohomology and the Floer cohomology groups from Section \ref{sec:floer}. To begin with, we remain in the situation from Setup \ref{th:sh-setup}, but consider a wider class of Hamiltonian functions. Namely, we replace the second part of \eqref{eq:the-ak} with the more general requirement that
\begin{equation}
\textstyle\lim_j a_j = \alpha \in (0,\infty] \setminus \bZ.
\end{equation}
Given such $(a_j)$, we use time-dependent Hamiltonians as in Setup \ref{th:strange-h}, and consider the associated Floer cohomology groups. As before, there are continuation maps between two such groups, assuming that $a_j^- \geq a_j^+$ for all $j$.

\begin{discussion}
(i) As a first application of continuation maps, one shows that the Floer cohomology groups only depend on $(a_j)$, up to canonical isomorphism. We therefore again adopt the temporary notation $\mathit{HF}^*(E,a_1,a_2,\dots)$. Slightly more generally, changing finitely many of the $(a_j)$ does not affect Floer cohomology.

(ii) Take some $(a_j^+)$ with limit $\alpha^+ < \infty $. Choose $\alpha^- \geq \alpha^+$ such that $[\alpha^+,\alpha^-] \cap \bZ = \emptyset$. This implies that there is some $j^*$ such that $[a_j^+,\alpha^-] \cap \bZ = \emptyset$ for all $j \geq j^*$. Let's define
\begin{equation}
a_j^- = \begin{cases} a_j^+ &  j \leq j^*, \\
\alpha^- & j > j^*.
\end{cases}
\end{equation}
By an argument in the style of Lemma \ref{th:no-integer}, one can show that the continuation map $\mathit{HF}^*(E,a_1^+,a_2^+,\dots) \rightarrow \mathit{HF}^*(E,a_1^-,a_2^-,\dots)$ is an isomorphism.
\end{discussion}

By combining those two ingredients, one sees that the Floer cohomology depends only on $\alpha$, and in fact, only on $\lfloor \alpha \rfloor \in \{0,1,\dots,\infty\}$. We will therefore use the notation $\mathit{HF}^*(E,\alpha)$, where obviously $\mathit{HF}^*(E,\infty) = \mathit{SH}^*(E)$. As a special case of the general continuation map construction, we have maps
\begin{equation} \label{eq:c-to-infty}
C: \mathit{HF}^*(E,\alpha) \longrightarrow \mathit{SH}^*(E) \quad \text{for any $\alpha<\infty$.}
\end{equation}
A filtration argument similar to the proof of Lemma \ref{th:define-sh} shows that in the limit, these maps induce an isomorphism
\begin{equation}
\underrightarrow{\lim}_{\alpha<\infty}\, \mathit{HF}^*(E,\alpha) \iso \mathit{SH}^*(E).
\end{equation}

\begin{discussion} \label{th:discuss-c}
(i) Continuation maps, PSS maps, and BV operators are all compatible. As a rather unsurprising consequence, we have commutative diagrams relating some of the parallel constructions carried out in both versions of Floer cohomology:
\begin{equation}
\xymatrix{
\cdots \ar[r] & \ar@{=}[d] H^*(E;\bK) \ar[r] & \mathit{HF}^*(E,\alpha) \ar[r] \ar[d] & \mathit{HF}^*(E,\alpha)_{\mathit{red}} \ar[r] \ar[d] & \cdots \\
\cdots \ar[r] & H^*(E;\bK) \ar[r] & \mathit{SH}^*(E) \ar[r] & \mathit{SH}^*(E)_{\mathit{red}} \ar[r] & \cdots
}
\end{equation}
and
\begin{equation} \label{eq:delta-red-red}
\xymatrix{
\ar[d] \mathit{HF}^*(E,\alpha)_{\mathit{red}} \ar[rr]^-{\Delta_{\mathit{red}}} && \mathit{HF}^{*-1}(E,\alpha) \ar[d] \\
\mathit{SH}^*(E)_{\mathit{red}} \ar[rr]^-{\Delta_{\mathit{red}}} && \mathit{SH}^{*-1}(E)
}
\end{equation}

(ii) It is convenient to have a simplified description of \eqref{eq:c-to-infty} in some cases. Namely, given some $\alpha \in (0,\infty) \setminus \bZ$, let's choose $a^+_1 = a^+_2 = \cdots = \alpha$. One can arrange that all one-periodic orbits of $(H_t^+)$ lie in $\{|p(x)| < r_1^-\}$. In the notation from Discussion \ref{th:properties}, this means that
\begin{equation}
\mathit{CF}^*(E,H^+) = \mathit{CF}^*(E,H^+)_{\leq 0}.
\end{equation}
Take $\alpha^-_1 = \alpha \leq \alpha^-_2 \leq \cdots$, going to infinity, and choose $(H^-,J^-)$ so that it agrees with $(H^+,J^+)$ on the region $\{|p(x)| \leq r_1^+\}$. In view of Lemma \ref{th:c0-estimate}, all Floer trajectories with limits in that region remain entirely inside it, so that one has an equality of chain complexes
\begin{equation} \label{eq:leq0}
\mathit{CF}^*(E,H^-)_{\leq 0} = \mathit{CF}^*(E,H^+)_{\leq 0}.
\end{equation}
This means that $\mathit{CF}^*(E,H^+)$ sits inside $\mathit{CF}^*(E,H^-)$ as a subcomplex. When choosing the auxiliary data $(H^C,J^C)$ for the continuation map, one can similarly assume that $(H^C_{s,t},J^C_{s,t}) = (H^+_t,J^+_t) = (H^-_t,J^-_t)$ on $\{|p(x)| \leq r_1^+\}$. An application of Lemma \ref{th:c0-estimate-2} ensures that all solutions of the continuation map equation remain in the same region. By translation invariance, the only isolated solutions are ones of the form $u(s,t) = x(t)$. Hence, the chain map $C: \mathit{CF}^*(E,H^+) \rightarrow \mathit{CF}^*(E,H^-)$ is in fact the inclusion of the subcomplex.

To see why this is helpful, consider for instance the question of compatibility of \eqref{eq:c-to-infty} with the ring structure, which means the commutativity of the diagram
\begin{equation} \label{eq:pp-compatible}
\xymatrix{
\mathit{HF}^*(E,\alpha_1) \otimes \mathit{HF}^*(E,\alpha_2) \ar[rr]^-{\text{pair-of-pants}} \ar[d]_{C \otimes C}
&& \mathit{HF}^*(E,\alpha_1+\alpha_2) \ar[d]^C \\
\mathit{SH}^*(E) \otimes \mathit{SH}^*(E) \ar[rr]^-{\text{pair-of-pants}} &&
\mathit{SH}^*(E).
}
\end{equation}
One can arrange that on the chain level, the groups in the top line of \eqref{eq:pp-compatible} sit inside  those on the bottom as subcomplexes, with the continuation maps just being inclusions; and moreover, again using Lemma \ref{th:c0-estimate-2}, that the pair-of-pants product is strictly compatible with the inclusions, so that the desired compatibility holds trivially. This strategy is not specific to the product: it applies to all sorts of operations.
\end{discussion}

Start with $F$ as in Setup \ref{th:define-f} and its associated $E$, but carry out a coordinate change $b \mapsto 1/b$ on the base, so that we then have a map $p: E \rightarrow \bC$, which satisfies the conditions from Setup \ref{th:define-sh}. If we now choose $a_1 = a_2 = \cdots = \alpha \in (0,\infty) \setminus \bZ$, and take specific choices of Hamiltonians, the resulting Floer cohomology groups $\mathit{HF}^*(E,\alpha)$ agree with those previously defined in Section \ref{subsec:define-floer}. In particular, by taking \eqref{eq:bs-1} and using the continuation map, we can define the Borman-Sheridan class \eqref{eq:s-element} in symplectic cohomology.

Now suppose that Assumption \ref{th:psi-eta} is satisfied. In the chain complex underlying $\mathit{HF}^*(E,\alpha_1^+)$ with $\alpha_1^+ > 1$, we get a cochain which bounds $b_{q^{-1}\Omega|E}$, namely \eqref{eq:theta-new}. Through a suitable continuation map, this gives rise to a cochain in  the complex underlying $\mathit{SH}^*(E)$, with the corresponding property. By construction, the associated class \eqref{eq:infinite-t} is the image of the corresponding class \eqref{eq:t-new-class} under  continuation maps. In view of \eqref{eq:delta-red-red}, the same thing holds for \eqref{eq:a-new-class} and \eqref{eq:infinite-a}. In particular, Lemma \ref{th:determine-a} implies Lemma \ref{th:bounding}.

Let's use our bounding cochains to define connections \eqref{eq:alt-nabla-m} on symplectic cohomology. Then, these fit into a commutative diagram with \eqref{eq:alt-nabla-1}, 
\begin{equation} \label{eq:final-diagram}
\xymatrix{
\ar[d]_-{C} \mathit{HF}^*(E,\alpha^+_2) \ar[rr]^-{\nabla^{-1}} && \mathit{HF}^*(E,\alpha^-) \ar[d]^-{C} \\
\mathit{SH}^*(E) \ar[rr]^-{\nabla^{-1}} && \mathit{SH}^*(E).
}
\end{equation}
Of course, the formulae for the two connections look identical, but some explanation is still needed. As in Discussion \ref{th:discuss-c}(ii), one can arrange that the two continuation maps in \eqref{eq:final-diagram} are realized by inclusions of subcomplexes. All the chain level operations that define the connection on symplectic cohomology can be set up to preserve those subcomplexes, and then \eqref{eq:final-diagram} commutes trivially.

In view of this compatibility statement, Proposition \ref{th:main-computation} and Lemma \ref{th:square-of-borman-sheridan} imply that the connection $\nabla^{-1}$ on symplectic cohomology satisfies \eqref{eq:nablac-s}. Since we know how the connections $\nabla^c$ for different $c$ are related (Lemma \ref{th:01-connections-2}), it follows that \eqref{eq:nablac-s} holds for all $c$, and in particular $c = 0$, which proves Theorem \ref{th:sigma1}. 

\begin{remark}
One of the computations from Section \ref{sec:floer-connection}, namely Proposition \ref{th:nabla-e}, has not been used here. In fact, its main role was to serve as a model for the similar, but more complicated,  Proposition \ref{th:main-computation}. Still, it is interesting to have a direct proof of that formula, since the alternative derivation (given in Section \ref{sec:results}: one starts with \eqref{eq:nabla-e-is-0}, which implies \eqref{eq:nabla-c-e-2}, and then inserts Lemma \ref{th:bounding} into that formula, leading to \eqref{eq:nabla-c-e} for all $\nabla^c$) is not particularly geometric.
\end{remark}


\begin{thebibliography}{10}

\bibitem{abouzaid10}
M.~Abouzaid.
\newblock A geometric criterion for generating the {F}ukaya category.
\newblock {\em Publ. Math. IHES}, 112:191--240, 2010, MR2737980, Zbl 1215.53078.

\bibitem{albers-cieliebak-frauenfelder14}
P.~Albers and K.~Cieliebak and U.~Frauenfelder.
\newblock Symplectic {T}ate homology.
\newblock {\em Proc. London Math. Soc.} 112:169--205, 2016, MR3458149,  Zbl 1337.57059.

\bibitem{bourgeois-oancea09}
F.~Bourgeois and A.~Oancea.
\newblock An exact sequence for contact- and symplectic homology.
\newblock {\em Invent. Math.}, 175:611--680, 2009, MR2471597, Zbl 1167.53071.
\newblock See also: Erratum, {\em Invent. Math.}, 200:1065--1076, 2015, MR3348144, Zbl 1317.53119.

\bibitem{cieliebak-floer-hofer95}
K.~Cieliebak, A.~Floer, and H.~Hofer.
\newblock Symplectic homology {II}: a general construction.
\newblock {\em Math. Z.}, 218:103--122, 1995, MR1312580, Zbl 0869.58011.

\bibitem{diogo12}
L.~Diogo.
\newblock {\em Filtered {F}loer and symplectic homology via {G}romov-{W}itten
  theory}.
\newblock PhD thesis, Stanford Univ., 2012. 

\bibitem{diogo-lisi15}
L.~Diogo and S.~Lisi.
\newblock Symplectic homology for complements of smooth divisors. Preprint arXiv:1804.08014. 

\bibitem{floer88}
A.~Floer.
\newblock Symplectic fixed points and holomorphic spheres.
\newblock {\em Commun. Math. Phys.}, 120:575--611, 1989, MR0987770, Zbl 0755.58022.

\bibitem{floer-hofer93}
A.~Floer and H.~Hofer.
\newblock Coherent orientations for periodic orbit problems in symplectic
  geometry.
\newblock {\em Math. Z.}, 212:13--38, 1993, MR1200162, Zbl 0789.58022.

\bibitem{floer-hofer-salamon94}
A.~Floer, H.~Hofer, and D.~Salamon.
\newblock Transversality in elliptic {M}orse theory for the symplectic action.
\newblock {\em Duke Math. J.}, 80:251--292, 1995, MR1360618, Zbl 0846.58025.

\bibitem{ganatra13}
S.~Ganatra.
\newblock Symplectic cohomology and duality for the wrapped {F}ukaya category.
\newblock {\em PhD Thesis}, MIT, 2012, MR3121862. 

\bibitem{ganatra-perutz-sheridan15}
S.~Ganatra, T.~Perutz, and N.~Sheridan. 
\newblock Mirror symmetry: from categories to curve-counts.
\newblock Preprint arXiv:1510.03839. 

\bibitem{ganatra-pomerleano}
S.~Ganatra and D.~Pomerleano.
\newblock A log {PSS} morphism with applications to Lagrangian embeddings.
\newblock Preprint arXiv:1611.06849. 

\bibitem{getzler94b}
E.~Getzler.
\newblock {B}atalin-{V}ilkovisky algebras and 2d {T}opological {F}ield
  {T}heories.
\newblock {\em Commun. Math. Phys.}, 159:265--285, 1994, MR125698, Zbl 0807.17026.

\bibitem{getzler95}
E.~Getzler.
\newblock Cartan homotopy formula and the {G}auss-{M}anin connection in cyclic homology. 
In {\em Quantum deformations of algebras and their representations}, pages 65--78. 
Bar-Ilan University, 1993, MR1261901, Zbl 0844.18007.

\bibitem{ginzburg06}
V.~Ginzburg.
\newblock Calabi-{Y}au algebras.
\newblock Preprint arXiv:math.AG/0612139. 

\bibitem{givental}
A.~Givental.
\newblock Equivariant {G}romov-{W}itten invariants.
\newblock {\em Internat. Math. Res. Notices}, 13:613--663, 1996.

\bibitem{groman15}
Y.~Groman.
\newblock {F}loer theory on open manifolds.
\newblock Preprint arXiv:1510.04265. 

\bibitem{harris13}
R.~Harris.
\newblock Distinguishing between exotic symplectic structures.
\newblock {\em J. Topology}, 6:1--29, 2013, MR3029419, Zbl 1287.53074.

\bibitem{hofer-salamon95}
H.~Hofer and D.~Salamon.
\newblock Floer homology and {N}ovikov rings.
\newblock In {\em The {F}loer memorial volume},
  pages 483--524. Birkh{\"a}user, 1995, MR1362838,  Zbl 0842.58029.

\bibitem{jones-petrack90}
J.~Jones and S.~Petrack.
\newblock The fixed point theorem in equivariant cohomology.
\newblock {\em Trans. Amer. Math. Soc}, 322:35--49, 1990, MR1010411, Zbl 0723.55003.

\bibitem{kahn01}
P.~Kahn.
\newblock Pseudohomology and homology.
\newblock Preprint arXiv:math/0111223. 

\bibitem{kaledin07}
D.~Kaledin.
\newblock Some remarks on formality in families.
\newblock {\em Mosc. Math. J.}, 7:643--652, 766, 2007, MR2372207, Zbl 1163.18006.

\bibitem{kontsevich-soibelman06}
M.~Kontsevich and Y.~Soibelman.
\newblock Notes on {$A_\infty$}-algebras, {$A_\infty$}-categories, and
  noncommutative geometry.
\newblock In {\em Homological Mirror Symmetry: New Developments and
  Perspectives}, pages 153--219. Springer, 2008, MR2596638, Zbl 1202.81120.

\bibitem{kummer36}
E.~Kummer.
\newblock {\"U}ber die hypergeometrische Reihe.
\newblock {\em J. reine angewandte Math. (Crelle)}, 15:39--83, 1836, MR1578088, ERAM 015.0528cj.
 
 \bibitem{li}
 J.~Li.
 \newblock Stable morphisms to singular schemes and relative stable morphisms. 
 \newblock {\em J. Differential Geom.}, 57:509--578, 2001, MR1882667,  Zbl 1076.14540.
  
\bibitem{lunts10}
V.~Lunts.
\newblock Formality of {DG} algebras (after {K}aledin).
\newblock {\em J. Algebra}, 323:878--898, 2010, MR2578584, Zbl 1227.18013.

\bibitem{maulik-pandharipande06}
D.~Maulik and R.~Pandharipande.
\newblock A topological view of {G}romov-{W}itten theory.
\newblock {\em Topology}, 45:887--918, 2006, MR2248516, Zbl 1112.14065.

\bibitem{mcduff-salamon}
D.~McDuff and D.~Salamon.
\newblock {\em {$J$}-holomorphic curves and quantum cohomology}.
\newblock Amer. Math. Soc., 1994, MR1286255, Zbl 0809.53002.

\bibitem{mclean12}
M.~McLean.
\newblock Symplectic homology of {L}efschetz fibrations and {F}loer homology of the monodromy map.
\newblock {\em Selecta Math.}, 18:473--512, 2012, MR2960024, Zbl 1253.53084.

\bibitem{piunikhin-salamon-schwarz94}
S.~Piunikhin, D.~Salamon, and M.~Schwarz.
\newblock Symplectic {F}loer-{D}onaldson theory and quantum cohomology.
\newblock In {\em Contact and symplectic geometry}, pages
  171--200. Cambridge Univ. Press, 1996, MR1432464, Zbl 0874.53031.

\bibitem{riccati24}
J.F.~Riccati.
\newblock Animadversiones in aequationes differentiales secundi gradus.
\newblock {\em Acta Eruditorum, Suppl.}, 8:66--73, 1724. 

\bibitem{ritter10}
A.~Ritter.
\newblock Topological quantum field theory structure on symplectic cohomology.
\newblock {\em J. Topology}, 6:391--489, 2013, MR3065181, Zbl 1298.53093.

\bibitem{ritter-smith12}
A.~Ritter and I.~Smith.
\newblock The monotone wrapped {F}ukaya category and the open-closed string map.
\newblock {\em Selecta Math.}, 23:533--642, 2017, MR3595902, Zbl 1359.53068.

\bibitem{salamon-zehnder92}
D.~Salamon and E.~Zehnder.
\newblock Morse theory for periodic solutions of {H}amiltonian systems and the
  {M}aslov index.
\newblock {\em Comm. Pure Appl. Math.}, 45:1303--1360, 1992, MR1181727, Zbl 0766.58023.

\bibitem{salvatore-wahl03}
P.~Salvatore and N.~Wahl.
\newblock Framed discs operads and {B}atalin-{V}ilkovisky algebras.
\newblock {\em Q. J. Math.}, 54:213--231, 2003, MR1989873, Zbl 1072.55006.

\bibitem{schwarz95}
M.~Schwarz.
\newblock {\em Cohomology operations from {$S^1$}-cobordisms in {F}loer
  homology}.
\newblock PhD thesis, {ETH} {Z}{\"u}rich, 1995. 

\bibitem{schwarz99}
M.~Schwarz.
\newblock Equivalences for {M}orse homology.
\newblock In {\em Geometry and Topology in Dynamics 
(Winston-Salem, NC, 1998/San Antonio, TX, 1999)}, Amer. Math. Soc., 1999, MR1732382, Zbl 0951.55009.

\bibitem{seidel99}
P.~Seidel.
\newblock Graded {L}agrangian submanifolds.
\newblock {\em Bull. Soc. Math. France}, 128:103--146, 2000, MR1765826, Zbl 0992.53059.

\bibitem{seidel00b}
P.~Seidel.
\newblock More about vanishing cycles and mutation.
\newblock In {\em {S}ymplectic {G}eometry and {M}irror {S}ymmetry ({P}roceedings of the 4th
  {KIAS} Annual International Conference)}, pages 429--465. World Scientific, 2001, MR1882336, Zbl 1079.14529.

\bibitem{seidel07}
P.~Seidel.
\newblock A biased survey of symplectic cohomology.
\newblock In {\em Current Developments in {M}athematics ({H}arvard, 2006)},
  pages 211--253. Intl.\ Press, 2008, MR2459307, Zbl 1165.57020.

\bibitem{seidel04}
P.~Seidel.
\newblock {\em {F}ukaya categories and {P}icard-{L}efschetz theory}.
\newblock European Math.\ Soc., 2008, MR2441780, Zbl 1159.53001.

\bibitem{seidel12b}
P.~Seidel.
\newblock {F}ukaya $A_\infty$-structures associated to {L}efschetz fibrations. {I}.
\newblock {\em J. Symplectic Geom.}, 10:325--388, 2012, MR2983434, Zbl 1267.53094.

\bibitem{seidel14b}
P.~Seidel.
\newblock {F}ukaya $A_\infty$-structures associated to {L}efschetz fibrations. {II}.
\newblock In {\em Algebra, geometry, and physics in the 21st century. Kontsevich festschrift}, pages 295--364.
Birkh{\"a}user/Springer, 2017, MR3727564, Zbl 1390.53098.

\bibitem{seidel15}
P.~Seidel.
\newblock Fukaya {$A_\infty$}-structures associated to {L}efschetz fibrations. {II} 1/2.
\newblock {\em Adv. Theor. Math. Phys.}, 20:883--944, 2016, MR3573028, Zbl 1359.53062.

\bibitem{seidel17}
P.~Seidel.
\newblock Connections on equivariant {H}amiltonian {F}loer cohomology.
\newblock {\em  Comment. Math. Helv.}, 93:587--944, 2018, MR3854903, Zbl 06966824.

\bibitem{sinha13}
D.~Sinha.
\newblock The (non-equivariant) homology of the little disks operad.
\newblock In {\em {OPERADS} 2009}, pages 253--279. Soc. Math. France, 2013, MR3203375, Zbl 1277.18012.

\bibitem{spaltenstein}
N.~Spaltenstein.
\newblock Resolutions of unbounded complexes.
\newblock {\em Compositio Math.}, 65:121--154, 1988,  MR0932640, Zbl 0636.18006

\bibitem{viterbo97a}
C.~Viterbo.
\newblock Functors and computations in {F}loer homology with applications, {P}art {I}.
\newblock {\em Geom. Funct. Anal.}, 9:985--1033, 1999, MR1726235, Zbl 0954.57015.

\bibitem{zhao14}
J.~Zhao.
\newblock Periodic symplectic cohomologies. Preprint arXiv:1405.2084. 
{\em J. Symplectic Geom.}, to appear. 

\bibitem{zhao16}
J.~Zhao.
\newblock {\em Periodic symplectic cohomology and obstructions to exact Lagrangian immersions.} 
PhD thesis, Columbia University, 2016, MR3527182. 

\bibitem{zinger08}
A.~Zinger.
\newblock Pseudocycles and integral homology.
\newblock {\em Trans. Amer. Math. Soc.}, 360:2741--2765, 2008, MR2373332, Zbl 1213.57031.

\bibitem{zinger13}
A.~Zinger.
\newblock The determinant line bundle for {F}redholm operators: construction, properties, and classification.
\newblock {\em Math. Scand.} 116:203--268, 2016, MR3515189, Zbl 1354.58032.
\end{thebibliography}

\end{document}